\documentclass[reqno,11pt,twoside]{article}
\usepackage[a4paper,hmargin=1in,vmargin=1.35in]{geometry}
\usepackage{amsmath,amssymb,amsthm}
\usepackage{bm}
\usepackage{graphicx}
\usepackage{mathtools}
\usepackage{etoolbox}
\usepackage{ebgaramond}
\usepackage[T1]{fontenc}
\numberwithin{equation}{section}
\newtheorem{theorem}{Theorem}[section]
\newtheorem{proposition}[theorem]{Proposition}
\newtheorem{corollary}[theorem]{Corollary}
\newtheorem{lemma}[theorem]{Lemma}

\theoremstyle{definition}
\newtheorem{definition}[theorem]{Definition}
\newtheorem{example}[theorem]{Example}

\theoremstyle{remark}
\newtheorem{remark}[theorem]{Remark}

 \newcommand{\addQEDstyle}[2]{\AtBeginEnvironment{#1}{\pushQED{\qed}\renewcommand{\qedsymbol}{#2}}\AtEndEnvironment{#1}{\popQED}}
 \addQEDstyle{example}{$\triangleleft$}
  \addQEDstyle{definition}{$\triangleleft$}
 \addQEDstyle{remark}{$\blacktriangleleft$}
\usepackage{scalerel}
\let\oldaleph\aleph
\def\aleph{{\ThisStyle{\scalebox{1.5}[1.1]{%
  \raisebox{-.2\LMpt}{$\SavedStyle\oldaleph$}}\kern-.2pt}}}
\newcommand{\balpha}{\bm{\alpha}}
\newcommand{\bell}{\bm{\ell}}
\newcommand{\bzeta}{\bm{\zeta}}
\newcommand{\baleph}{\bm{\aleph}}
\newcommand{\bbeth}{\bm{\beth}}

\newcommand{\Cvect}{\mathbf{X}}
\newcommand{\CS}{\mathcal{C}}
\newcommand{\Odd}{\mathcal{O}}

\DeclareMathOperator{\Tr}{Tr}
\DeclareMathOperator{\Id}{Id}
\DeclareMathOperator{\supp}{supp}
\DeclareMathOperator{\Span}{span}
\renewcommand{\Re}{\operatorname{Re}}
\renewcommand{\Im}{\operatorname{Im}}
\DeclareMathOperator{\changes}{\mathbf{Ch}}
\renewcommand{\csc}{\operatorname{cosec}}

\DeclareMathOperator{\codim}{codim}
\DeclareMathOperator{\Rot}{Rot}

\newcommand{\Sct}[1]{\mathfrak{S}_{#1}}
\newcommand{\Eangles}{\mathcal{E}}
\newcommand{\Sangles}{\mathcal{S}}
\newcommand{\pmset}[1]{\mathfrak{Z}^{#1}}
\newcommand{\Conj}{\mathbb{C}^2_\mathrm{conj}}
\renewcommand{\emptyset}{\varnothing}
\newcommand{\UC}{\hat{\mathbb{C}}_*}
\newcommand{\hbe}{\hat{\bm{\mathfrak{t}}}}
\newcommand{\hbv}{\hat{\bm{\mathfrak{v}}}}
\newcommand{\hbw}{\hat{\bm{\mathfrak{w}}}}
\newcommand{\er}{\mathrm{e}}
\newcommand{\ir}{\mathrm{i}}
\newcommand{\dr}{\mathrm{d}}
\newcommand{\tI}{\mathrm{I}}
\newcommand{\tII}{\mathrm{II}}
\newcommand{\out}{\mathrm{out}}
\newcommand{\inn}{\mathrm{in}}
\renewcommand{\tilde}{\widetilde}
\renewcommand{\hat}{\widehat}
\newcommand{\N}{\hat{\mathbf{N}}}

\newcommand{\T}{\hat{\mathtt T}}
\newcommand{\U}{\hat{\mathtt U}}
\newcommand{\R}{\hat{\mathtt R}}
\newcommand{\w}{\hat{\mathbf w}}
\newcommand{\Sh}{\hat{\mathtt S}}
\newcommand{\NP}{\mathcal{N}_\mathcal{P}}

\newcommand{\NPq}{\mathcal{N}_\mathcal{P}^q}

\newcommand{\Dir}{\mathfrak{D}}
\newcommand{\Codd}{\hat{\Cvect}_\mathrm{odd}}

\newcommand{\Riesz}{\mathcal{R}_\mathcal{P}}
\newcommand{\Pro}{{\bm{\Pi}}}
\newcommand{\sha}{\sharp}
\newcommand{\SL}{\mathbf{SL}}
\newcommand{\DL}{\mathbf{DL}}
\makeatletter
\renewcommand*\env@matrix[1][\arraystretch]{%
  \edef\arraystretch{#1}%
  \hskip -\arraycolsep
  \let\@ifnextchar\new@ifnextchar
  \array{*\c@MaxMatrixCols c}}
\makeatother
\renewcommand*{\arraystretch}{1.25}
\makeatletter
\renewcommand{\pod}[1]{\allowbreak\mathchoice
  {\if@display \mkern 18mu\else \mkern 8mu\fi (#1)}
  {\if@display \mkern 18mu\else \mkern 8mu\fi (#1)}
  {\mkern4mu(#1)}
  {\mkern4mu(#1)}
}
\makeatother
\renewcommand\footnotemark{}
\usepackage[nobottomtitles,pagestyles]{titlesec}
\titleformat{\section}
{\normalfont\large\bfseries}
{\filcenter\thesection.\ }{1ex}{\filcenter}
\titleformat{\subsection}[runin]
{\normalfont\normalsize\bfseries}{\thesubsection.}{1ex}{}[.]
\widenhead*{0.5in}{0.5in}
\renewpagestyle{plain}[\bfseries]{\footrule\setfoot{}{\thepage}{}}
\newpagestyle{mystyle}[\bfseries]{
\headrule
\sethead[\thepage][][Asymptotics of Steklov eigenvalues for curvilinear polygons]{Michael Levitin, Leonid Parnovski, Iosif Polterovich, and David A. Sher}{}{\thepage}
}

\usepackage[colorlinks,allcolors=blue]{hyperref}
\begin{document}
\title{Sloshing, Steklov and corners:\\Asymptotics of  Steklov eigenvalues  for  curvilinear polygons}
\author{Michael Levitin\hspace{-3ex}
\thanks{{\bf ML:} Department of Mathematics and Statistics,
University of Reading, Whiteknights, PO Box 220, Reading RG6 6AX, UK;
M.Levitin@reading.ac.uk; \url{http://www.michaellevitin.net}}
\and Leonid Parnovski\hspace{-3ex}
\thanks{{\bf LP:} Department of Mathematics, University College London, 
Gower Street, London WC1E 6BT, UK;
leonid@math.ucl.ac.uk}
\and Iosif Polterovich\hspace{-3ex}
\thanks{{\bf IP:} D\'e\-par\-te\-ment de math\'ematiques et de
sta\-tistique, Univer\-sit\'e de Mont\-r\'eal CP 6128 succ
Centre-Ville, Mont\-r\'eal QC  H3C 3J7, Canada;
iossif@dms.umontreal.ca; \url{http://www.dms.umontreal.ca/\~iossif}}
\and David A. Sher
\thanks{{\bf DAS:} Department of Mathematical Sciences, DePaul University, 2320 N Kenmore Ave, Chicago, IL 60614, USA;
dsher@depaul.edu}
}
\date{\small Proceedings of the LMS\qquad doi 10.1112/plms.12461\\final version June 2022}
\maketitle
\begin{abstract} 
We obtain  asymptotic formulae for the Steklov eigenvalues and eigenfunctions of curvilinear polygons in terms of their side lengths and angles. These formulae are quite precise: the errors tend to zero as the spectral parameter tends to infinity. The Steklov problem on planar domains with corners is closely linked to the classical sloshing and sloping beach problems in hydrodynamics; as we show it is also related to quantum graphs.  Somewhat surprisingly, the arithmetic properties of the angles of a curvilinear polygon have a significant effect on  the boundary  behaviour of the Steklov eigenfunctions. Our proofs are based on an explicit construction of quasimodes. We use a variety of methods, including ideas from spectral geometry, layer potential analysis, and some new techniques tailored to our  problem.
\end{abstract}  
\thispagestyle{plain}
\pagestyle{mystyle}
\tableofcontents
\clearpage\section{Introduction}\label{sec:intro}
\subsection{Preliminaries}
Let $\Omega\subset\mathbb{R}^2$ be a bounded connected planar domain with connected Lipschitz  boundary $\partial\Omega$, and let $|\partial\Omega|$ denote its perimeter. 
Consider the Steklov eigenvalue problem
\begin{equation}\label{eq:steklovproblem}
\Delta u =0\quad \text{in }\Omega,\qquad\qquad \frac{\partial u}{\partial n}=\lambda u\quad \text{on }\partial \Omega,
\end{equation}
with $\lambda$ being the spectral parameter, and $\dfrac{\partial u}{\partial n}$ being the exterior normal derivative. 
The spectral problem \eqref{eq:steklovproblem} may be understood in the sense of the normalised quadratic form
\[
\frac{\|\operatorname{grad}u\|_{L^2(\Omega)}^2}{\|u\|^2_{L^2(\partial\Omega)}},\qquad u\in H^1(\Omega).
\]

Let 
\[
\mathcal{D}_{\Omega}:H^{1/2}(\partial\Omega)\to H^{-1/2}(\partial\Omega),\qquad \mathcal{D}_{\Omega} f:=\left.\frac{\partial \mathcal{H}_\Omega f}{\partial n}\right|_{\Omega}
\] 
denote the \emph{Dirichlet-to-Neumann map}, where  $\mathcal{H}_\Omega f$ stands for the harmonic extension of $f$ to $\Omega$. The spectrum of $\mathcal{D}_{\Omega}$ coincides with that of the Steklov problem. 
The spectrum is discrete, 
\[
0=\lambda_1(\Omega)<\lambda_2(\Omega)\le\dots \le \lambda_m(\Omega) \le \dots,
\]
with the only limit point at $+\infty$. The corresponding eigenfunctions $u_m$ have the property that their boundary traces 
$u_m|_{\partial\Omega}$ form an orthogonal basis in $L^2(\partial \Omega)$.
 If the boundary $\partial \Omega$ is piecewise $C^1$, the Steklov eigenvalues have the following asymptotics (see \cite{Ag2006}):
\begin{equation}\label{eq:agranovich}
\lambda_m=\frac{\pi m}{ |\partial\Omega|} +o(m)\qquad\text{as }m\to+\infty.
\end{equation}
Moreover, if the boundary is smooth, then $\mathcal{D}_{\Omega}$ is a pseudodifferential operator of order one, and the remainder estimate could be significantly improved 
\cite{Ros, Ed}:
\begin{equation}
\label{RGM}
\lambda_{2m}=\lambda_{2m+1}+O\left(m^{-\infty}\right)=\frac{2\pi m}{|\partial\Omega|}+O\left(m^{-\infty}\right), \qquad m\to +\infty
\end{equation}
(see also \cite{GPPS} for the case of a disconnected $\partial\Omega$).

The asymptotic formula \eqref{RGM} immediately implies
\begin{proposition}
\label{RGMa}
Let $\Omega_\mathrm{I}$ and $\Omega_\mathrm{II}$ be  two smooth simply connected planar domains of the same perimeter. Then
\begin{equation}
\lambda_m(\Omega_\mathrm{I})-\lambda_m(\Omega_\mathrm{II})=O(m^{-\infty}),
\end{equation}
\end{proposition}

For non-smooth domains such as polygons,  formula \eqref{RGM} and Proposition \ref{RGMa} are no longer valid, see e.g.\ \cite[section 3]{GP2017}. Building upon the approach introduced in \cite{sloshing}, in the present paper we develop the techniques that allow to improve the asymptotic  formula \eqref{eq:agranovich} significantly when $\Omega$ is a curvilinear polygon. 

\subsection{Curvilinear polygons. Exceptional and special angles}
To fix notation, let $\mathcal{P}=\mathcal{P}(\boldsymbol{\alpha},\bell)$  be a (simply connected) curvilinear polygon in $\mathbb{R}^2$ with  $n$ vertices $V_1,\dots,V_n$ numbered clock-wise, corresponding internal angles $0<\alpha_j<\pi$ at $V_j$, and smooth sides $I_j$  of length $\ell_j$ joining $V_{j-1}$ and $V_j$.  Here,  $\balpha=(\alpha_1,\dots,\alpha_n)\in\Pi^n$, where 
\[
\Pi:=(0,\pi),
\] 
$\bell=(\ell_1,\dots,\ell_n)\in\mathbb{R}_+^n$,  and we will use cyclic subscript identification $n+1\equiv 1$. Our choice of orientation ensures that an internal angle $\alpha_j$ is measured from $I_j$ to $I_{j+1}$ in the counter-clockwise direction, as in Figure  \ref{fig:fig1}. The perimeter of $\mathcal{P}$ is  $|\partial\mathcal{P}|=\ell_1+\dots+\ell_n$.
\begin{figure}[htb]
\begin{center}
\includegraphics{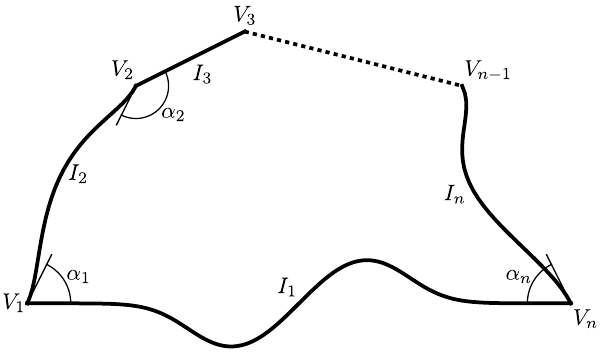}
\end{center}
\caption{A curvilinear polygon\label{fig:fig1}}
\end{figure} 

In what follows we will have to distinguish the cases when some of the polygon's angles belong to the sets of \emph{exceptional} and \emph{special} angles.
\begin{definition}\label{defn:angles}
Let  $\alpha=\dfrac{\pi}{j}$, $j\in \mathbb{N}$. We say that the angle $\alpha$ is {\it exceptional} if $j$ is even, and {\it special} if $j$ is odd, and denote the corresponding sets
\[
\Eangles=\left\{\frac{\pi}{2k}\,\middle|\, k\in\mathbb{N}\right\},\qquad\Sangles=\left\{\frac{\pi}{2k+1}\,\middle|\, k\in\mathbb{N}\right\}.
\]
We  call an exceptional angle $\alpha=\dfrac{\pi}{2k}$ \emph{odd} or \emph{even} depending on whether $k$ is odd or even, respectively, and define its \emph{parity} $\Odd(\alpha)$ to be 
\[
\Odd(\alpha):=\cos\left(\dfrac{\pi^2}{2\alpha}\right)=(-1)^k.
\] 
Similarly, we call a special angle $\alpha=\dfrac{\pi}{2k+1}$ \emph{odd} or \emph{even} depending on whether $k$ is odd or even, respectively, and  define its \emph{parity} $\Odd(\alpha)$ to be 
\[
\Odd(\alpha):=\sin\left(\dfrac{\pi^2}{2\alpha}\right)=(-1)^k.
\]
\end{definition}

\begin{definition}\label{defn:exceptpoly} A curvilinear polygon without any exceptional angles will be called a \emph{non-exceptional polygon}, otherwise it will be called an \emph{exceptional polygon}.
\end{definition}

\subsection{Main results}

The main purpose of this paper is to describe sharp asymptotic behaviour as $m\to\infty$ of the Steklov eigenvalues $\lambda_m$ of a curvilinear polygon $\mathcal{P}=\mathcal{P}(\balpha,\bell)$. More precisely, we show that the Steklov spectrum can be approximated as $m\to\infty$ by a sequence of \emph{quasi-eigenvalues} $\sigma_m$,  which are computable  in terms of side lengths $\bell$ and angles $\balpha$.  

The quasi-eigenvalues $\sigma_m$ can in fact be defined in several equivalent ways, each having its own merit. Originally they are defined in Section \ref{sec:statements} in terms of the so called \emph{vertex} and \emph{side transfer matrices}, in two different ways depending on the presence of exceptional angles. This is done according to Definitions  \ref{def:quasi} and \ref{def:quasimult}  in the non-exceptional case, and according to  Definitions  \ref{def:quasiexc}  and \ref{def:quasiexcmult}  in the exceptional case. This is the most natural definition arising from the construction of corresponding quasimodes. Later, Theorems  \ref{thm:polygoneqn0} and \ref{thm:polygoneqn1} state that the quasi-eigenvalues can be found as roots of some explicit trigonometric polynomials which also depend only upon the geometry of the curvilinear polygon. This approach is the most convenient computationally.  Theorem \ref{thm:quantum} states that $\sigma_m$ can be  viewed alternatively as the square roots of the eigenvalues of a particular  \emph{quantum graph Laplacian}. Here the metric graph is cyclic and is modelled on the boundary of $\mathcal{P}$, while the matching conditions at the vertices are determined by the angles. 
 This interpretation allows us to relate to the well developed theory of quantum graphs, see \cite{BK13} and references therein, and also \cite{BE, KoSm, KN2}. It also leads to another one, \emph{variational}, interpretation of quasi-eigenvalues, see Remark \ref{rem:varprinciple}, and allows us an easy proof of Theorem \ref{thm:polygoneqn0}. We note that in a different but somewhat reminiscent setting of a periodic problem involving Dirichlet-to-Neumann type maps, a relation to a quantum graph problem was already observed in \cite{KuKu}, see also \cite{KuKu99}. We emphasise, however, that we do not directly use the quantum graph analogy in the construction of our quasimodes, see  Remark \ref{rem:notquant}.
Finally, yet another equivalent way to define the quasi-eigenvalues is presented in subsections \ref{subs:noexc} and \ref{subs:exepenum} in terms of the \emph{lifts} of the vertex and side transfer matrices acting on the universal cover $\UC$ of the punctured plane, see subsection \ref{subsec:rep}. This definition is indispensable for the delicate analysis required to establish the correct \emph{enumeration} of quasi-eigenvalues and their monotonicity properties.  

With the definitions of quasi-eigenvalues  in place our main result is
 
\begin{theorem}\label{thm:main} Let  $\mathcal{P}=\mathcal{P}(\balpha,\bell)$ be a curvilinear polygon. Let $\{\sigma_m\}$ denote the sequence of quasi-eigenvalues  ordered increasingly with account of multiplicities. Then there exists $\varepsilon_0>0$ such that for any $\varepsilon\in(0,\varepsilon_0)$, the Steklov eigenvalues of  $\mathcal{P}$ satisfy
\[
\lambda_m=\sigma_m+O(m^{-\varepsilon})\qquad\text{as }m\to\infty.
\]
\end{theorem}
\begin{remark} We give an explicit formula for $\varepsilon_0$, depending only on the angles of $\mathcal P$, in Remark \ref{rem:defofepsilon}.
\end{remark}

As an immediate consequence of Theorem \ref{thm:main}, we obtain 
\begin{corollary}\label{col:same} 
Let $\mathcal{P}_\mathrm{I}(\balpha,\bell)$ and $\mathcal{P}_\mathrm{II}(\balpha,\bell)$ be two curvilinear polygons with the same angles $\balpha$ and the same side lengths $\bell$. Then
\[
\lambda_m(\mathcal{P}_\mathrm{I})-\lambda_m(\mathcal{P}_\mathrm{II})=O(m^{-\varepsilon})\qquad \text{as }m\to+\infty.
\]
\end{corollary}

We also describe the asymptotic behaviour of the Steklov eigenfunctions on the boundary. Up to a small error, they are  given by trigonometric functions of frequency $\sigma_m$ along each edge.

\begin{theorem}\label{thm:maineigenfunction} Fix $\delta>0$. Then there exists $C>0$ such that for all $m$ with 
$\sigma_{m-1}+\delta\leq\sigma_m\leq\sigma_{m+1}-\delta$, there exist constants $a_{m,j}$ and $b_{m,j}$ such that for all $j$,
\[ \|(u_m|_{I_j})(s_j)-a_{m,j}\cos(\sigma_m s_j)-b_{m,j}\sin(\sigma_m s_j)\|_{L^2(I_j)}\leq Cm^{-\varepsilon},\]
where $s_j$ is an arc length coordinate along $I_j$, and $\varepsilon$ is as in Theorem \ref{thm:main}.
\end{theorem}
\begin{remark} The assumption on $m$ is made only so that the theorem is easy to state, as it removes the possibility of clustering of eigenvalues and quasi-eigenvalues. Theorem \ref{thm:maineigenfunctiongen} is a more general version, without this assumption.
\end{remark}
\begin{remark} The coefficients $a_{m,j}$ and $b_{m,j}$ are related to each other by imposing \emph{matching conditions} at the vertices, and may be found explicitly in the same way as the quasi-eigenvalues. See Section \ref{subsec:bqm} for details.
\end{remark}

\subsection{Examples}\label{sec:examples}
The following examples give the flavour of the main results; they are further discussed in more detail and illustrated by numerics in Section \ref{sec:numerics}.

\begin{example}[Each angle is either special or exceptional]\label{ex:specialorexceptional} Let $\mathcal{P}(\balpha,\bell)$ be a curvilinear $n$-gon in which each angle is either special or exceptional
 in the sense of Definition \ref{defn:angles}. In this case we can use Theorem \ref{thm:main} together with Definitions \ref{def:quasi} and  \ref{def:quasiexc} directly without the use of trigonometric polynomials.  We will distinguish two cases. 
 \begin{itemize}
 \item[(a)] All angles are special, that is $\alpha_j=\dfrac{\pi}{2k_j+1}$, $k_j\in\mathbb{N}$, $j=1,\dots,n$. In this case we have the quasi-eigenvalues
 \begin{equation}\label{eq:allspecial}
 \begin{aligned}
 \sigma_1=0, \sigma_{2m}=\sigma_{2m+1}=\frac{2\pi m}{|\partial\mathcal{P}|},\quad m\in\mathbb{N},\qquad&\text{if }\sum_{j=1}^n k_j\text{ is even},\\ 
\sigma_{2m-1}=\sigma_{2m}=\frac{2\pi\left(m-\frac{1}{2}\right)}{|\partial\mathcal{P}|},\quad m\in\mathbb{N},\qquad&\text{if }\sum_{j=1}^n k_j\text{ is odd}.
 \end{aligned}
 \end{equation}
  \item[(b)]   Suppose that there are $K$ exceptional angles  $\alpha^\mathcal{E}_\kappa=\alpha_{E_\kappa}=\frac{\pi}{2k_\kappa}$, with $k_\kappa\in\mathbb{N}$, $\kappa=1,\dots,K$, $1\le E_1<E_2<\dots<E_K\le n$, and all the other angles are special.  We assume the cyclic enumeration of exceptional angles $E_{K+1}=E_1$. Let us denote also by $L_\kappa$ the \emph{total length} of the boundary pieces between exceptional angles $\alpha^\mathcal{E}_{\kappa-1}$ and $\alpha^\mathcal{E}_\kappa$.  
  
 Let 
 \[
 \mathfrak{K}_\mathrm{odd}:=\left\{\kappa\in\{1,\dots,K\}: \Odd\left(\alpha^\mathcal{E}_\kappa\right)\ne\Odd\left(\alpha^\mathcal{E}_{\kappa-1}\right)\right\},
 \] 
be the set of indices $\kappa$ such that $k_{\kappa}-k_{\kappa-1}$ is odd, and  let
\[
 \mathfrak{K}_\mathrm{even}:=\left\{\kappa\in\{1,\dots,K\}: \Odd\left(\alpha^\mathcal{E}_\kappa\right)=\Odd\left(\alpha^\mathcal{E}_{\kappa-1}\right)\right\}.
 \] 
Then $\sigma=0$ is a quasi-eigenvalue of multiplicity $\frac{\#\mathfrak{K}_\mathrm{odd}}{2}$, and the positive quasi-eigenvalues $\sigma$ form the set
\[
\left(\bigcup_{\kappa\in \mathfrak{K}_\mathrm{even}}\left\{\frac{\pi}{L_\kappa}\left(m-\frac{1}{2}\right)\mid m\in\mathbb{N}\right\}\right)\cup 
\left(\bigcup_{\kappa\in \mathfrak{K}_\mathrm{odd}}\left\{\frac{\pi}{L_\kappa}m\mid m\in\mathbb{N}\right\}\right),
\]
with account of multiplicities.
 \end{itemize}
Example \ref{ex:eigenfunctions} and Proposition \ref{prop:equidistrib} below also show strikingly different asymptotic behaviour of eigenfunctions in these two cases.

As an illustration, we consider the following two particular cases of right-angled triangles (see also cases (a$_3$) and (b$_4$) in Example \ref{ex:quasi-regular} below, and Example \ref{ex:eigenfunctions}):
\begin{itemize}
\item[($T_1$)] The isosceles right-angled triangle $T_1=\mathcal{P}\left(\left(\frac{\pi}{4},\frac{\pi}{4}, \frac{\pi}{2}\right),\left(1,\sqrt{2},1\right)\right)$. All angles are exceptional, two of them even, and one odd. There is a single quasi-eigenvalue at $\sigma=0$, a subsequence of quasi-eigenvalues $\sigma=\pi m$, $m\in\mathbb{N}$ of multiplicity two, and a subsequence of single quasi-eigenvalues $\sigma=\frac{\pi}{\sqrt{2}}\left(m-\frac{1}{2}\right)$, $m\in\mathbb{N}$.
\item[($T_2$)] The right-angled triangle $T_2=\mathcal{P}\left(\left(\frac{\pi}{3},\frac{\pi}{6}, \frac{\pi}{2}\right),\left(1,2,\sqrt{3}\right)\right)$. Two angles are odd exceptional and one is  odd special. There are two subsequences of single quasi-eigenvalues 
\[ 
\sigma=\frac{\pi}{3}\left(m-\frac{1}{2}\right)\quad\text{and}\quad\sigma=\frac{\pi}{\sqrt{3}}\left(m-\frac{1}{2}\right),\qquad m\in\mathbb{N}.
\]
\end{itemize}
 \end{example}

 \begin{remark}
Note that \emph{even special} angles do not affect the quasi-eigenvalues in both cases considered in Example \ref{ex:specialorexceptional}.  In particular, in case (a) with all even special angles the quasi-eigenvalues $\sigma$ are the same as for a smooth domain with the same perimeter, compare with \eqref{RGM}. This remains true for any curvilinear polygon --- a vertex with an even special angle can be removed and the two adjacent sides treated as a single side without affecting the quasi-eigenvalues.
\end{remark}

\begin{example}[Quasi-regular curvilinear polygon]\label{ex:quasi-regular} Consider a quasi-regular curvilinear $n$-gon $\mathcal{P}=\mathcal{P}_n(\alpha,\ell)$, namely, a curvilinear polygon whose angles are all equal to $\alpha$ and all sides have the same length $\ell$. Its perimeter is obviously  $|\partial\mathcal{P}|=n\ell$. Then we have the following two cases depending on whether $\alpha$ is exceptional. 
 \begin{itemize}
 \item[(a)] $\alpha\not\in\Eangles$. Then we have the following
 \begin{proposition}\label{prop:regnonexc} 
 Let $\mathcal{P}=\mathcal{P}_n(\alpha,\ell)$ be a quasi-regular curvilinear $n$-gon with a non-exceptional angle $\alpha$. Then the set  of quasi-eigenvalues $\sigma$  is 
 given by
 \[
\left\{\frac{\pm\arccos\left(\sin\left(\frac{\pi^2}{2\alpha}\right)\cos\left(\frac{2\pi q}{n}\right)\right)+2\pi m}{\ell}, m\in\mathbb{N}\cup\{0\}, q=0,1,\dots,\left[\frac{n}{2}\right]\right\}\cap[0,+\infty)
\]
(understood as a set of unique values without multiplicities).  All the quasi-eigenvalues should be taken with multiplicity two, except in the following cases when they are single:
 \begin{itemize}
 \item[(i)] $\alpha$ is not special and $q=0$.
  \item[(ii)] $\alpha$ is not special, $n$ is even, and $q=\frac{n}{2}$.
 \item[(iii)] $\alpha$ is even special, $q=0$, and $m=0$, which corresponds to the quasi-eigenvalue $0$.
 \item[(iv)] $\alpha$ is odd special, $n$ is even, $q=\frac{n}{2}$, and $m=0$, which corresponds to the quasi-eigenvalue $0$.
 \end{itemize}
 \end{proposition}
 The proof of Proposition \ref{prop:regnonexc}  is presented  in Section \ref{sec:numerics}.
 \item[(b)] $\alpha\in\Eangles$.  This case is already covered by  Example \ref{ex:specialorexceptional}(b) with $K=K_\mathrm{even}=n$: all the quasi-eigenvalues have multiplicity $n$ and are given by
 \[
 \ell\sigma_{n(m-1)+1}= \ell\sigma_{n(m-1)+2}=\dots= \ell\sigma_{nm}=\pi \left(m-\frac{1}{2}\right),\qquad m\in\mathbb{N}.
 \]
 \end{itemize}

 The following particular cases are illustrative:
 \begin{itemize}
\item[(a$_1$)]  $\mathcal{P}_1\left(\alpha,1\right)$, a one-gon (a droplet) with the angle $\alpha$ and perimeter one. Then the set of quasi-eigenvalues is
\[
\left\{\pm\left(\frac{\pi}{2}-\frac{\pi^2}{2\alpha}\right)+2\pi m, m\in\mathbb{N}\cup\{0\}\right\}\cap[0,+\infty).
\]
The same formula works also in the case $\alpha\in\Eangles$.
 \item[(a$_3$)] $\mathcal{P}_3\left(\frac{\pi}{3},1\right)$, the equilateral triangle of side one (this case is also covered by Example \ref{ex:specialorexceptional}(a) as all angles are odd special). Then 
\[
\sigma_{2m-1}=\sigma_{2m}=\frac{(2m-1)\pi}{3},\qquad m\in\mathbb{N}.
\]
 \item[(b$_4$)] $\mathcal{P}_4\left(\frac{\pi}{2},1\right)$, the square of side one (this case is also covered by Example \ref{ex:specialorexceptional}(b) as all angles are even exceptional). Then 
\[
\sigma_{4m-3}=\sigma_{4m-2}=\sigma_{4m-1}=\sigma_{4m}=\left(m-\frac{1}{2}\right)\pi,\qquad m\in\mathbb{N}.
\]
 \item[(a$_5$)] $\mathcal{P}_5\left(\frac{3\pi}{5},1\right)$, the regular pentagon of side one. Then there are four subsequences of quasi-eigenvalues of multiplicity two,
\[
\begin{alignedat}{2}
\sigma&=-\arccos\left(\frac{\pm\sqrt{5}-1}{8}\right)+2\pi m, \qquad &&m\in\mathbb{N},\\
\sigma&=\arccos\left(\frac{\pm\sqrt{5}-1}{8}\right)+2\pi m, \qquad &&m\in\mathbb{N}\cup\{0\},
\end{alignedat}
\]
and two subsequences of quasi-eigenvalues of multiplicity one,
\[
\begin{alignedat}{2}
\sigma&=-\frac{\pi}{3}+2\pi m, \qquad &&m\in\mathbb{N},\\
\sigma&=\frac{\pi}{3}+2\pi m, \qquad &&m\in\mathbb{N}\cup\{0\}.
\end{alignedat}
\]
 \end{itemize}  
 The case  (b$_4$) agrees with the results of \cite[Section 3]{GP2017} obtained by separation of variables.
 \end{example}
 
 \begin{example}[Eigenfunction behaviour]\label{ex:eigenfunctions}
 The cases of all-special and all-exceptional angles also illustrate the dependence of the boundary behaviour of eigenfunctions on the arithmetic properties of the angles, via the following 
\begin{proposition}\label{prop:equidistrib} Let $\mathcal P$ be a curvilinear polygon.
\begin{itemize}
\item [(a)] If all angles are special, then the boundary eigenfunctions $u_m|_{\partial\mathcal P}$ are equidistributed in the sense that for any arc $I\subseteq\partial\mathcal{P}$, not necessarily a side,
\[
\lim_{m\to\infty}\frac{\|u_m\|_{L^2(I)}}{\|u_m\|_{L^2(\partial\mathcal{P})}}=\frac{|I|}{|\partial\mathcal{P}|}.
\]
\item [(b)] If all angles are exceptional, then the boundary traces of  eigenfunctions, $u_m|_{\mathcal\partial{P}}$, are \emph{not} equidistributed in the following sense. Pick $\delta>0$. Then for all $m$ with 
\begin{equation}\label{eq:nocluster}
\sigma_{m-1}+\delta\leq\sigma_m\leq\sigma_{m+1}-\delta,
\end{equation} 
there exists an edge $I_{M(m)}$ such that
\[
\|u_m\|_{L^2(\partial\mathcal{P}\setminus I_{M(m)})}=O\left(m^{-2\varepsilon}\right),
\]
with an implied constant in the right-hand side depending upon $\delta$.
\end{itemize}
\end{proposition}
For the proof of Proposition \ref{prop:equidistrib}, see the end of Section \ref{subsec:orthogcons}.

\begin{remark}\label{rem:unionedges} There are other versions of Proposition \ref{prop:equidistrib}(b) if some (at least two) but not all angles are exceptional. To state these versions we would need to use the language of \emph{exceptional components} in Section \ref{sec:exceptional}, see e.g.\ Theorem \ref{thm:maineigenfunctiongen} and Corollary \ref{cor:split}.
\end{remark}

\begin{remark} If all angles are exceptional and all lengths are pairwise incommensurable, then it is easy to show that the proportion of quasi-eigenvalues $\sigma_m$ which do \emph{not} satisfy the hypothesis of (b) tends to zero as $\delta\to 0$.
\end{remark}
 
\begin{remark}\label{rem:symmetry} Condition \eqref{eq:nocluster} is essential in Proposition \ref{prop:equidistrib}(b). Indeed, let $\mathcal{P}$ be a  two-gon with two exceptional angles, and suppose that $\mathcal{P}$ is symmetric with respect to the line $V_1V_2$. Then each eigenfunction is either symmetric or anti-symmetric with respect to this line and therefore cannot concentrate on one side. This happens because each quasi-eigenvalue $\sigma\ne 0$ has in this case multiplicity two.  The boundary behaviour of eigenfunctions of the right-angled isosceles triangle $T_1$, shown below, gives another example demonstrating this phenomenon.
\end{remark}
 
We illustrate  Proposition \ref{prop:equidistrib} by showing, in Figures \ref{fig:fig19} and \ref{fig:fig20}, the numerically computed boundary traces $u_m|_{\mathcal\partial{P}}$  for the equilateral triangle $\mathcal{P}_3$ from Example \ref{ex:quasi-regular}(a$_3$) (all angles are special) and for the right-angled isosceles triangle $T_1$ from Example \ref{ex:specialorexceptional} (all angles are exceptional); see Section \ref{sec: numsetup} for details of the numerical procedure. In both cases we plot two eigenfunctions $u_{18}$ and $u_{19}$.  For the equilateral triangle, these eigenfunctions correspond to  the eigenvalues $\lambda_{18}\approx 17.8023$ and $\lambda_{19}\approx 19.8968$, which in turn correspond to the quasi-eigenvalues $\sigma_{18}=\frac{17\pi}{3}$  and  $\sigma_{19}=\frac{19\pi}{3}$ (both of which are in fact double, $\sigma_{17}=\sigma_{18}$ and $\sigma_{19}=\sigma_{20}$).  For the right-angled isosceles triangle, these eigenfunctions correspond to the eigenvalues $\lambda_{18}\approx 15.708$ and $\lambda_{19}\approx 16.6608$, which in turn correspond to the quasi-eigenvalues $\sigma_{18}=5\pi$ (which is in fact double, $\sigma_{17}=\sigma_{18}$) and  $\sigma_{19}=\frac{15\pi}{2\sqrt{2}}$ (which is single). 

\begin{figure}[htb]
\begin{center}
\includegraphics{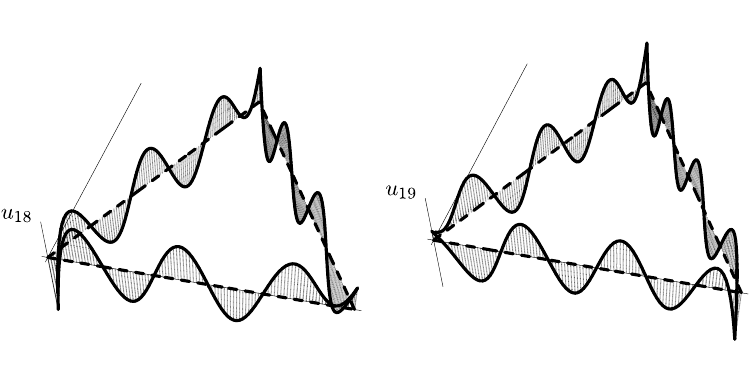}
\end{center}
\caption{Boundary traces  of  $u_{18}$ and $u_{19}$ for the equilateral triangle\label{fig:fig19}}
\end{figure} 

\begin{figure}[htb]
\begin{center}
\includegraphics{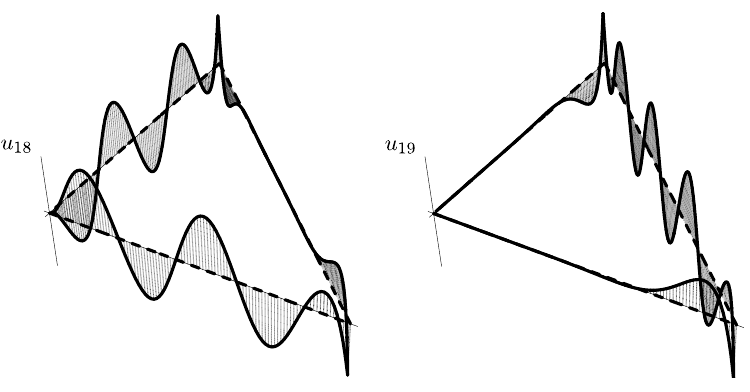}
\end{center}
\caption{Boundary traces  of  $u_{18}$ and $u_{19}$ for the right-angled isosceles triangle\label{fig:fig20}}
\end{figure} 

It is easily seen that in the case of the equilateral triangle the eigenfunctions are more or less equally distributed on all sides, whereas in the exceptional case in Figure \ref{fig:fig20} the eigenfunction $u_{18}$ is mostly concentrated on the union of two sides (and not on one side, cf.\  Remark \ref{rem:symmetry} and Corollary \ref{cor:split}; note that the corresponding quasi-eigenvalue is double),  and the eigenfunction $u_{19}$ is mostly concentrated on the hypothenuse.  
 \end{example}

\subsection{Plan of the paper and further directions}

We begin in Section \ref{sec:statements} by defining and studying the sequence $\{\sigma_m\}$ of quasi-eigenvalues which appears in Theorem \ref{thm:main}. 
The quasi-eigenvalues and, importantly, their \emph{multiplicities} are originally defined in terms of a combination of  \emph{vertex transfer} matrices $\mathtt{A}(\alpha_j)$ and \emph{side transfer} matrices $\mathtt{B}(\ell_j)$, which play a central role throughout the paper; see Definitions \ref{def:quasi}, \ref{def:quasimult},  \ref{def:quasiexc}, and \ref{def:quasiexcmult}. We then give two alternative characterisations of this sequence. On one hand, the quasi-eigenvalues coincide with the roots of  certain trigonometric polynomials, see Theorems \ref{thm:polygoneqn0} and \ref{thm:polygoneqn1}. On the other hand, the sequence of quasi-eigenvalues is also the spectrum of a particular eigenvalue problem on the boundary of our polygon, viewed as a quantum graph, see Theorem \ref{thm:quantum}. Section \ref{sec:statements} also contains statements of the results on Riesz means and the heat trace, see Theorem \ref{thm:riesz} and Corollary \ref{cor:heat}, as well as a discussion of quasi-eigenvalues of auxiliary \emph{zigzag domains}.

The rest of the paper principally contains the proofs of the main results.

In Section \ref{sec:sector}, we recall from \cite{sloshing} the construction of the Peters solutions \cite{Pet50}  of \emph{sloping beach problems} (that is,  mixed Robin-Neumann and Robin-Dirichlet problems) in an infinite  sector.  These solutions are then combined, via symmetry, to give so called \emph{scattering Peters solutions} of a pure Robin problem, see Theorem \ref{thm:mainthmssector}. This naturally gives rise to the previously defined vertex transfer matrices $\mathtt{A}(\alpha)$.

Section \ref{sec:quasieigenvalues} describes the  \emph{quasimode construction}, and finally makes apparent the reasons for our definitions of quasi-eigenvalues $\{\sigma_m\}$.  We construct approximate Steklov eigenfunctions on a curvilinear polygon, first in the straight boundary case, then in the partially curvilinear case (with boundary straight in a neighbourhood of each corner), and finally in the fully curvilinear case. The arguments use the Peters solutions of Section \ref{sec:sector} as building blocks. We conclude by proving that near each sufficiently large quasi-eigenvalue $\sigma_m$ there exists a distinct Steklov eigenvalue $\lambda_{i_m}$, see Theorem \ref{thm:approx}, and by stating and proving Theorem \ref{thm:maineigenfunctiongen} on the boundary behaviour of eigenfunctions.

In Section \ref{sec:completeness} we address the delicate issue  of \emph{enumeration} of quasi-eigenvalues, namely, by proving that we may take $i_m=m$. Note that this does not follow from quantum graph or any previously discussed techniques (see also Remark \ref{rem:notquant}), and requires development of a new machinery. In Section \ref{sec:completeness} we concentrate on the case of partially curvilinear polygons and prove that for such polygons  $|\sigma_m-\lambda_m|=o(1)$. A key element of the proof is the lifting of vectors and matrices onto the universal cover $\UC$ of the punctured complex plane and a construction based on the change of argument on $\UC$. The proof proceeds via a gluing construction: we cut our polygon through its side into a union of zigzag domains, establish correct enumeration for each of those by comparison with eigenvalue asymptotics for the Steklov--Dirichlet and Steklov--Neumann problems \cite{sloshing}, see Definition \ref{defn:naturalenumeration} and Proposition \ref{lemma1}, and then glue zigzag domains together via the Dirichlet--Neumann bracketing. Note that cutting a polygon through the vertices rather than through the sides may appear more natural. However, comparing the contribution from corners to the eigenvalue asymptotics for mixed Steklov--Dirichlet and Steklov--Neumann problems, one can see that the bracketing in that case does not yield accurate enough estimates.
 
Sections \ref{sec:trigpolynoms} and \ref{sec:quantum} explore various consequences of the alternative characterisations of $\{\sigma_m\}$. In the former, we prove Theorem \ref{thm:polygoneqn1} by explicitly writing down the trigonometric polynomials whose roots are $\sigma_m$. In the latter, we establish the quantum graph analogy, and use it to prove Theorem \ref{thm:polygoneqn0}, as well as the  results on the Riesz means.

In Section \ref{sec:layer}, we extend our results to fully curvilinear polygons. This is done by taking advantage of the well-known relationship between the Dirichlet-to-Neumann operator and layer potentials. A careful analysis of the kernels of single- and double-layer potential operators on curvilinear polygons, inspired by the work of Costabel \cite{cost},  allows us to show that a small change in the boundary curvature and its derivatives induces only a small change in the Steklov spectrum. From there, we use a deformation argument to complete the proof.

Finally, Section \ref{sec:numerics} contains some numerical calculations of the Steklov spectrum in specific examples, which provide an  illustration of our results and suggest further avenues for exploration.

We want to emphasise that the most  crucial and novel points of this paper are  the construction of the scattering Peters solutions in Section \ref{sec:sector}, and the enumeration argument  of Section  \ref{sec:completeness} based on step-by-step comparison between zigzag problems and the sloshing problem of \cite{sloshing}.  
Sections \ref{subsec:proofapprox1}--\ref{subsec:asymptef} and \ref{sec:layer} contain mostly fine-tuned technical details and may be omitted in the first reading. 

\begin{remark} The present article  is the second in a series of papers  concerned with the study of Steklov-type eigenvalue problems on planar domains with corners. Our preceding work \cite{sloshing} 
focused on spectral asymptotics for the sloshing problem. As was mentioned above, the  methods and results of \cite{sloshing}  have been instrumental for a number of arguments used in this  article.  

In a separate publication \cite{inversesteklov}, written jointly with S. Krymski,  we apply the results of the present article to the study of the inverse spectral problem for curvilinear polygons. In particular we show there that, generically, the side lengths of a curvilinear polygon and some information about its angles can be reconstructed from its Steklov spectrum. 
\end{remark}

\begin{remark} The results and most of the methods of this paper are specifically two-dimensional. For some related recent advances in higher dimensions see \cite{Iv, GLPS}.
\end{remark}

\subsection*{Acknowledgements}
\addcontentsline{toc}{subsection}{Acknowledgements}
The authors are grateful to  P. Kurasov  for pointing out that the notions introduced in subsections \ref{sec:verttran} and \ref{sec:quasi} 
can be interpreted in the language of quantum graphs. 
We are  thankful as well to M. Costabel for helpful advice regarding the results presented in Section \ref{sec:layer}.  We would also like to thank  G. Berkolaiko, J. Bolte, A. Girouard,  Y. Kannai, P. Kuchment, K.-M. Perfekt, and  U. Smilansky for useful discussions, as well as J. Lagac\'e for comments on the preliminary version of the paper. We are grateful to the anonymous referees for their careful reading of the  manuscript and numerous valuable suggestions.\label{page:ack}

Parts of this project were completed  while the authors   (in various combinations) attended workshops  at  Casa Matem\'atica Oaxaca (\emph{Dirichlet-to-Neumann maps: spectral theory, inverse problems and applications}, May--June, 2016), at the American Institute of Mathematics (\emph{Steklov eigenproblems}, April--May 2018),  and at the Banff International Research Station (\emph{Spectral geometry: theory, numerical analysis and applications}, July 2018), as well as a research program  
(\emph{Spectral methods in mathematical physics}, January--February 2019)  at the Institut Mittag-Leffler.  
Parts of this work were done when ML and LP were visiting  the Weizmann Institute of Science, and ML, LP and DS were visiting the Centre de recherches math\'{e}matiques, Montreal. We are grateful to all these institutions for their hospitality.

The research of LP was partially supported by   EPSRC grants EP/J016829/1 and EP/P024793/1.
The research of IP was partially supported by NSERC, the Canada Research Chairs program and the Weston Visiting Professorship program at the Weizmann Institute of Science.
The research of DS was partially supported by a  Faculty Summer Research Grant from DePaul University.
\clearpage\section{Quasi-eigenvalues. Definitions and further statements}\label{sec:statements}

\subsection{Vertex and side transfer matrices}
\label{sec:verttran}
Given an angle $\alpha$, set
\begin{equation}\label{eq:mualphadef}
\mu_\alpha:=\frac{\pi^2}{2\alpha}.
\end{equation}
For $\alpha\not\in\Eangles$, set
\begin{equation}\label{eq:a1a2defn}
a_1(\alpha):=\csc\mu_\alpha=\csc\frac{\pi^2}{2\alpha},\qquad a_2(\alpha):=\cot\mu_\alpha=\cot\frac{\pi^2}{2\alpha},
\end{equation}
and consider the  matrix
\begin{equation}\label{eq:Adef1}
\mathtt{A}(\alpha) := \begin{pmatrix}
a_1(\alpha)&-\ir a_2(\alpha)\\
\ir a_2(\alpha)&a_1(\alpha)
\end{pmatrix}=
\begin{pmatrix}
\csc\frac{\pi^2}{2\alpha}&-\ir \cot\frac{\pi^2}{2\alpha}\\
\ir \cot\frac{\pi^2}{2\alpha}&\csc\frac{\pi^2}{2\alpha}
\end{pmatrix}.
\end{equation}
For the reasons that will be explained later, the matrix $\mathtt{A}(\alpha)$ is called a {\it vertex transfer matrix} at the corner with angle $\alpha$.

\begin{remark}\label{rem:Aproperties} Note that
\begin{itemize}
\item[(a)] for exceptional angles $\alpha\in\Eangles$ the vertex transfer matrix is not defined since its entries blow up;
\item[(b)] for a non-exceptional $\alpha\not\in\Eangles$, $\det\mathtt{A}(\alpha)=1$, $\mathtt{A}^*(\alpha)=\mathtt{A}(\alpha)$, and $(\mathtt{A}(\alpha))^{-1}=\overline{\mathtt{A}(\alpha)}$;
\item[(c)] for special angles $\alpha\in\Sangles$ the vertex transfer matrix is equal to $\Odd(\alpha) \Id$, see Definition \ref{defn:angles};
\item[(d)] the eigenvalues of $\mathtt{A}(\alpha)$ are
\begin{equation}\label{eq:eta12}
\begin{split}
\eta_1(\alpha)&:=a_1(\alpha)-a_2(\alpha)=\tan\frac{\mu_\alpha}{2}=\tan\frac{\pi^2}{4\alpha},\\
\eta_2(\alpha)&:=a_1(\alpha)+a_2(\alpha)=\cot\frac{\mu_\alpha}{2}=\cot\frac{\pi^2}{4\alpha}=\frac{1}{\eta_1(\alpha)},
\end{split}
\end{equation}
and the corresponding eigenvectors do not depend on $\alpha$, see Remark \ref{rem:Xareeigenvectors}.
\end{itemize}
\end{remark}

Given a side of length  $\ell$, define the {\it side transfer matrix}
\begin{equation}\label{eq:Bdef1}
\mathtt{B}(\ell,\sigma):=\begin{pmatrix}
\exp(\ir\ell\sigma)&0\\
0&\exp(-\ir\ell\sigma)
\end{pmatrix},
\end{equation}
where $\sigma$ is a real parameter.

\begin{remark}\label{rem:Bproperties} Similarly to Remark \ref{rem:Aproperties}(b), we have, for any $\ell>0$ and $\sigma\in\mathbb{R}$,    $\det\mathtt{B}(\ell,\sigma)=1$, and $(\mathtt{B}(\ell,\sigma))^{-1}=\overline{\mathtt{B}(\ell,\sigma)}$.
\end{remark}

Set 
\begin{equation}\label{eq:Sdef1}
\mathtt{C}(\alpha,\ell,\sigma):=\mathtt{A}(\alpha)\mathtt{B}(\ell,\sigma)=
\begin{pmatrix}
\csc\left(\frac{\pi^2}{2\alpha}\right)\exp(\ir\ell\sigma)&-\ir\cot\left(\frac{\pi^2}{2\alpha}\right)\exp(-\ir\ell\sigma)\\
\ir\cot\left(\frac{\pi^2}{2\alpha}\right)\exp(\ir\ell\sigma)&\csc\left(\frac{\pi^2}{2\alpha}\right)\exp(-\ir\ell\sigma)
\end{pmatrix}.
\end{equation}

Given a non-exceptional polygon $\mathcal{P}(\balpha,\bell)$, we construct the matrix
\begin{equation}\label{eq:Tdef1}
\mathtt{T}(\balpha,\bell,\sigma):=\mathtt{C}(\alpha_n,\ell_n,\sigma)\mathtt{C}(\alpha_{n-1},\ell_{n-1},\sigma)\cdots\mathtt{C}(\alpha_1,\ell_1,\sigma).
\end{equation}

\subsection{Quasi-eigenvalues, non-exceptional polygons}
\label{sec:quasi}
\begin{definition}\label{def:quasi}
Let $\mathcal{P}=\mathcal{P}(\balpha,\bell)$ be a non-exceptional curvilinear polygon.
A non-negative number $\sigma$ is called a \emph{quasi-eigenvalue} of the Steklov problem on $\mathcal{P}$ if the matrix $\mathtt{T}(\balpha,\bell,\sigma)$ has an eigenvalue $1$. 
\end{definition}

\begin{remark}\label{rem:Tnotinvariant}
We note that although the matrix $\mathtt{T}(\balpha,\bell,\sigma)$ depends upon our choice of an enumeration of polygon vertices, it is easily checked that the definition of quasi-eigenvalues is invariant.
\end{remark}

The following result immediately follows from Remarks \ref{rem:Aproperties}(b) and \ref{rem:Bproperties}, and the equation \eqref{eq:Tdef1}.

\begin{lemma}\label{lem:Tproperties}\ 
\begin{itemize}
\item[(a)] The matrix  $\mathtt{T}=\mathtt{T}(\balpha,\bell,\sigma)$ has eigenvalue $1$ if and only if 
\begin{equation}\label{eq:Ttrace}
\Tr \mathtt{T}=2.
\end{equation} 
\item[(b)] The eigenvalue $1$ of $\mathtt{T}$ always has algebraic multiplicity two. It has geometric multiplicity two if and only if $\mathtt{T}=\Id$.
\item[(c)] The corresponding eigenvector(s) may be chosen from
\[
\Conj:=\left\{\begin{pmatrix}c\\\overline{c}\end{pmatrix}\,\middle|\,c\in\mathbb{C}\right\}.
\]
\end{itemize}
\end{lemma}

\begin{definition}\label{def:quasimult}
 In the absence of exceptional angles, the \emph{multiplicity} of a quasi-eigenvalue $\sigma>0$  is defined as the geometric multiplicity of the eigenvalue $1$ of the matrix  $\mathtt{T}(\balpha,\bell,\sigma)$. If $\sigma=0$ is a quasi-eigenvalue, its \emph{multiplicity} is defined to be one.
\end{definition}

\begin{remark} It follows immediately from Lemma \ref{lem:Tproperties}  that a quasi-eigenvalue of a non-exceptional curvilinear polygon has multiplicity at most two.
\end{remark}

\subsection{Quasi-eigenvalues, exceptional  polygons}
\label{sec:exceptional}
For curvilinear polygons having exceptional angles the definition of quasi-eigenvalues is more involved. Let $\mathcal{P}$ be a curvilinear $n$-gon with  $K$  exceptional angles $\alpha^\mathcal{E}_1=\alpha_{E_1}=\frac{\pi}{2k_1},\dots, \alpha^\mathcal{E}_K=\alpha_{E_K}=\frac{\pi}{2k_K}$, where $1\le K\le n$, and $1\le E_1<E_2<\dots<E_K\le n$. Without loss of generality we can take $E_K=n$ and identify $E_0$ with $E_K$, and $E_{K+1}$ with $E_1$.  These exceptional angles split the boundary of the polygon into $K$ parts, which we will call \emph{exceptional (boundary) components}, each consisting of either one smooth side or more smooth sides joined at non-exceptional angles. 

Let $n_\kappa=E_\kappa-E_{\kappa-1}$, $\kappa=1,\dots,K$, denote the number of smooth boundary pieces between two consecutive exceptional angles. Obviously, $n_1+n_2+\dots+n_K=n$. Re-label the full sequence of angles $\alpha_1,\dots,\alpha_n$ as
\[
\begin{split} 
&\alpha^{(1)}_1, \dots, \alpha^{(1)}_{n_1-1}, \alpha^{(1)}_{n_1}=:\alpha^\mathcal{E}_1,\\
&\alpha^{(2)}_1, \dots, \alpha^{(2)}_{n_2-1},\alpha^{(2)}_{n_2}=:\alpha^\mathcal{E}_2,\\
&\qquad\qquad\dots\\
&\alpha^{(K-1)}_{1},\dots,\alpha^{(K-1)}_{n_{K-1}-1},\alpha^{(K-1)}_{n_{K-1}}=:\alpha^\mathcal{E}_{K-1},\\
&\alpha^{(K)}_1, \dots,  \alpha^{(K)}_{n_K-1},\alpha^{(K)}_{n_K}=:\alpha^\mathcal{E}_K.
\end{split}
\]

The vertices of the polygon will be re-labeled in the same manner. 
We also re-label the full sequence of side lengths $\ell_1,\dots,\ell_n$ (recall that the side $I_j$ of length $\ell_j$ joins the vertices $V_{j-1}$ and $V_{j}$) as
\[
\ell^{(1)}_1, \dots, \ell^{(1)}_{n_1}, \ell^{(2)}_1, \dots, \ell^{(2)}_{n_2},\dots,\ell^{(K)}_1, \dots, \ell^{(K)}_{n_K},
\]
so that the exceptional vertex $V^\mathcal{E}_\kappa$ has adjoint sides of lengths $\ell^{(\kappa)}_{n_\kappa}$ and $\ell^{(\kappa+1)}_{1}$, see Figure \ref{fig:except} for an example.

\begin{figure}[htb]
\begin{center}
\includegraphics{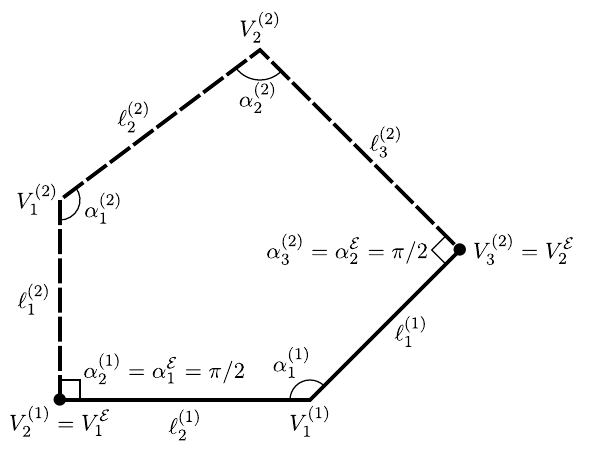}
\end{center}
\caption{An example of re-labelling for a pentagon with two exceptional angles and therefore two exceptional boundary components, one exceptional component (solid lines) consisting of  two pieces, and the other (dashed lines) consisting of three pieces\label{fig:except}}
\end{figure}

Denote also, for $\kappa=1,\dots,K$,
\[
\begin{split}
{\balpha'}^{(\kappa)}&=\left(\alpha^{(\kappa)}_1, \dots, \alpha^{(\kappa)}_{n_\kappa-1}\right), \\
\balpha^{(\kappa)}&=\left(\alpha^{(\kappa)}_1, \dots, \alpha^{(\kappa)}_{n_\kappa-1}, \alpha_\kappa^\mathcal{E}\right), \\
\bell^{(\kappa)}&=\left(\ell^{(\kappa)}_1, \dots, \ell^{(\kappa)}_{n_\kappa-1},\ell^{(\kappa)}_{n_\kappa}\right).
\end{split}
\]
We will be denoting an exceptional boundary component joining exceptional vertices $V^\Eangles_{\kappa-1}$  and $V^\Eangles_{\kappa}$ by $\mathcal{Y}_\kappa= \mathcal{Y}(\balpha^{(\kappa)},\bell^{(\kappa)})$.

Set
\begin{equation}\label{eq:Udef}
\mathtt{U}\left({\balpha'}^{(\kappa)}, \bell^{(\kappa)},\sigma\right)=\mathtt{B}\left(\ell^{(\kappa)}_{n_\kappa},\sigma\right)\mathtt{A}\left(\alpha^{(\kappa)}_{n_\kappa-1}\right)\mathtt{B}\left(\ell^{(\kappa)}_{n_\kappa-1},\sigma\right)\cdots \mathtt{A}\left(\alpha^{(\kappa)}_{1}\right)\mathtt{B}\left(\ell^{(\kappa)}_{1},\sigma\right).
\end{equation}
By \eqref{eq:Tdef1} and \eqref{eq:Sdef1},
\begin{equation}\label{eq:UviaBT}
\mathtt{U}\left({\balpha'}^{(\kappa)}, \bell^{(\kappa)},\sigma\right):=\mathtt{B}\left(\ell^{(\kappa)}_{n_\kappa},\sigma\right) \mathtt{T}\left({\balpha'}^{(\kappa)}, {\bell'}^{(\kappa)},\sigma\right),
\end{equation}
where ${\bell'}^{(\kappa)}=\left(\ell^{(\kappa)}_1, \dots, \ell^{(\kappa)}_{n_\kappa-1}\right)$.

Set also
\begin{equation}\label{eq:orthogspecial1}
\Cvect_\mathrm{even} = \frac{1}{\sqrt{2}}\begin{pmatrix}  \er^{-\ir \pi/4}  \\ \er^{\ir \pi/4} \end{pmatrix},\quad \Cvect_\mathrm{odd} = \frac{1}{\sqrt{2}}\begin{pmatrix}  \er^{\ir \pi/4}  \\ \er^{-\ir \pi/4} \end{pmatrix},
\end{equation}
and, for an exceptional angle $\alpha\in\Eangles$,
\begin{equation}\label{eq:orthogspecial2}
\Cvect(\alpha):=\begin{cases} 
\Cvect_\mathrm{even}\qquad&\text{if }\Odd(\alpha)=1,\\
\Cvect_\mathrm{odd}\qquad&\text{if }\Odd(\alpha)=-1.
\end{cases}
\end{equation}
 
\begin{remark}\label{rem:Cevenoddorthog}  We note that 
\begin{equation}\label{eq:Ceveoddorthog}
\overline{\Cvect_\mathrm{even}}=\Cvect_\mathrm{odd},\qquad\text{and}\qquad \Cvect_\mathrm{even}\cdot \Cvect_\mathrm{odd}=0,
\end{equation}
and we therefore set
\begin{equation}\label{eq:Cevenoddperp}
\Cvect_\mathrm{even}^\perp:=\Cvect_\mathrm{odd},\qquad \Cvect_\mathrm{odd}^\perp:=\Cvect_\mathrm{even},\qquad \Cvect^\perp(\alpha):=\left( \Cvect(\alpha)\right)^\perp=\begin{cases} 
\Cvect_\mathrm{odd}\qquad&\text{if }\Odd(\alpha)=1,\\
\Cvect_\mathrm{even}\qquad&\text{if }\Odd(\alpha)=-1.
\end{cases}.
\end{equation}
In \eqref{eq:Ceveoddorthog}, and throughout this paper, the dot product in $\mathbb{C}^2$ is understood in the usual sense:
\[
\begin{pmatrix}u_1\\u_2\end{pmatrix}\cdot \begin{pmatrix}v_1\\v_2\end{pmatrix}=u_1\overline{v_1}+u_2\overline{v_2}.
\]
\end{remark} 
 
\begin{remark}\label{rem:Xareeigenvectors} It is easily checked that $\Cvect_\mathrm{odd}$ and $\Cvect_\mathrm{even}$ are eigenvectors of the matrix $\mathtt{A}(\alpha)$ corresponding to the eigenvalues $\eta_1(\alpha)$ and $\eta_2(\alpha)$, respectively, for any $\alpha\not\in\Eangles$.
\end{remark} 
 
\begin{definition} 
\label{def:quasiexc}
Let $\mathcal{P}$ be a curvilinear polygon with exceptional angles $\alpha_{E_1}=\frac{\pi}{2k_1}$, \dots, $\alpha_{E_K}=\frac{\pi}{2k_K}$ as defined above.
We say that $\sigma\ge 0$ is a {\it quasi-eigenvalue} of $\mathcal{P}$ if there exists $1 \le \kappa\le K$,  such that $\sigma$ is a solution of the equation
\begin{equation}
\label{excepeq}
\mathtt{U}\left({\balpha'}^{(\kappa)}, \bell^{(\kappa)},\sigma\right) \Cvect\left(\alpha_{E_{\kappa-1}}\right) \cdot \Cvect\left(\alpha_{E_\kappa}\right)
=0.
\end{equation}
\end{definition}

\begin{remark}\label{rem:quasiexparallel}
Condition \eqref{excepeq} can be equivalently restated as 
\begin{equation}\label{eq:excepeqparallel}
\mathtt{U}\left({\balpha'}^{(\kappa)}, \bell^{(\kappa)},\sigma\right) \Cvect\left(\alpha_{E_{\kappa-1}}\right) \text{ is proportional  to } \Cvect^\perp\left(\alpha_{E_\kappa}\right).
\end{equation}
\end{remark}

\begin{definition}\label{def:exceptionalcomponentsoddeven} We will call an exceptional boundary component $\mathcal{Y}_\kappa$ which joins two exceptional angles $\alpha_{E_{\kappa-1}}$ and $\alpha_{E_\kappa}$ an \emph{even exceptional component} if the parities $\Odd\left(\alpha_{E_{\kappa-1}}\right)$ and $\Odd\left(\alpha_{E_\kappa}\right)$ are equal, and an \emph{odd exceptional component} if these parities differ.
\end{definition}

\begin{definition}\label{def:quasiexcmult}
 In the presence of exceptional angles, the \emph{multiplicity} of a quasi-eigenvalue $\sigma> 0$ is defined as the number of distinct values $\kappa$ for which $\sigma$ is a solution of \eqref{excepeq}. The \emph{multiplicity} of  quasi-eigenvalue $\sigma=0$ is defined as half the number of sign changes in the cyclic  sequence of exceptional angle parities $\Odd\left(\alpha_{E_1}\right),\dots$, $\Odd\left(\alpha_{E_K}\right)$, $\Odd\left(\alpha_{E_1}\right)$, or equivalently as half the number of \emph{odd} exceptional boundary components (see Definition \ref{def:exceptionalcomponentsoddeven}) joining the exceptional vertices. 
\end{definition}

\begin{remark} It is easy to see that the definition of multiplicity of a quasi-eigenvalue $\sigma=0$ in the exceptional case is consistent -- it always produces an integer as there is always an even number of odd  exceptional boundary components.
\end{remark}

\begin{remark} Let us  compare Definitions \ref{def:quasi} and \ref{def:quasiexc}. In the former, the quasi-eigenvalues are defined in the terms of the \emph{whole} boundary $\partial\mathcal{P}$. In the latter, the exceptional angles split the boundary into a number of exceptional boundary components, each producing its own independent sequence of quasi-eigenvalues. 
\end{remark}

\subsection{Quasi-eigenvalues as roots of trigonometric polynomials} We can re-formulate the quasi-eigenvalue equations \eqref{eq:Ttrace} and \eqref{excepeq} as the conditions that $\sigma$ is a root of some explicit trigonometric polynomials.  To define  these polynomials, we need to introduce some combinatorial notation. Let
\[
\pmset{n}=\{\pm1\}^n,
\]
and for a vector $\bzeta=(\zeta_1,\dots,\zeta_n)\in \pmset{n}$ with cyclic identification $\zeta_{n+1}\equiv\zeta_1$, let
\begin{equation}\label{eq:Cintrodef}
\changes(\bzeta):=\{j \in\{1,\dots,n\}\mid \zeta_j\ne \zeta_{j+1} \} 
\end{equation}
denote the set of indices of sign change in $\bzeta$, e.g.
\[
\changes((1, 1, 1)) = \varnothing; \quad \changes((-1, -1, 1, 1)) = \{2, 4\}.
\]

Given a curvilinear polygon $\mathcal{P}(\balpha,\bell)$, we now define the following trigonometric polynomials in real variable $\sigma$: firstly, we set 
\begin{equation}\label{eq:Fevendefn}
F_\mathrm{even}(\balpha,\bell,\sigma):=\sum_{\substack{\bzeta\in\pmset{n}\\\zeta_1=1}} \mathfrak{p}_{\bzeta} \cos(\bell\cdot\bzeta\sigma),
\end{equation}
where
\begin{equation}\label{eq:pfrakdefn}
\mathfrak{p}_{\bzeta}=\mathfrak{p}_{\bzeta}(\balpha):=\prod_{j\in\changes(\bzeta)}\cos\left(\frac{\pi^2}{2\alpha_j}\right),
\end{equation}
and we assume the convention $\prod\limits_\emptyset=1$. 

We further define
\begin{equation}\label{eq:quasieq1}
F^\mathcal{P}(\balpha,\bell,\sigma):=F_\mathrm{even}(\balpha,\bell,\sigma)-\prod_{j=1}^n\sin\left(\frac{\pi^2}{2\alpha_j}\right),
\end{equation}
which differs from \eqref{eq:Fevendefn} only in the constant term.

In either exceptional or non-exceptional  case, we have the following
\begin{theorem}\label{thm:polygoneqn0} 
 Let $\mathcal{P}(\balpha,\bell)$ be a curvilinear polygon. Then  $\sigma\ge 0$ is a quasi-eigenvalue if and only if it is a root of the trigonometric polynomial $F^\mathcal{P}(\balpha,\bell,\sigma)$. The multiplicity of a quasi-eigenvalue $\sigma>0$ coincides with its multiplicity as a root of  \eqref{eq:quasieq1}, and the multiplicity of the quasi-eigenvalue $\sigma=0$
 is half its multiplicity as a root of  \eqref{eq:quasieq1}.
 \end{theorem}
 
The following result is more convenient for the actual computation of quasi-eigenvalues in the exceptional case, and also simplifies the calculation of multiplicities. 
 
\begin{theorem}\label{thm:polygoneqn1} \
\begin{itemize}
\item[(a)] Let $\mathcal{P}(\balpha,\bell)$ be a non-exceptional curvilinear polygon. Then  a root $\sigma>0$ of \eqref{eq:quasieq1} is a quasi-eigenvalue of multiplicity two if additionally $\sigma$ is a root of 
\begin{equation}\label{eq:Fodddefn}
F_\mathrm{odd}(\balpha,\bell,\sigma):=\sum_{\substack{\bzeta\in\pmset{n}\\\zeta_1=1}} \mathfrak{p}_{\bzeta}\sin(\bell\cdot\bzeta\sigma),
\end{equation}
otherwise it has multiplicity one. 
\item[(b)] Let $\mathcal{P}(\balpha,\bell)$ be a curvilinear polygon with exceptional angles $\alpha_{E_1}=\frac{\pi}{2k_1}$, \dots, $\alpha_{E_K}=\frac{\pi}{2k_K}$. 
Then  $\sigma\ge 0$ is a quasi-eigenvalue if and only if it is  a root of one of the trigonometric polynomials
\begin{equation}\label{eq:quasieq2}
F_\mathrm{even/odd}\left(\balpha^{(\kappa)},\bell^{(\kappa)},\sigma\right),\qquad \kappa=1,\dots,K,
\end{equation}    
corresponding to an exceptional boundary component $\mathcal{Y}\left(\balpha^{(\kappa)},\bell^{(\kappa)}\right)$. Here, $F_\mathrm{even/odd}$ stands for $F_\mathrm{even}$ if the exceptional boundary component $\mathcal{Y}\left(\balpha^{(\kappa)},\bell^{(\kappa)}\right)$ is even (or equivalently if $\Odd\left(\alpha_{E_{\kappa-1}}\right)=\Odd\left(\alpha_{E_\kappa}\right)$), and for $F_\mathrm{odd}$ if $\mathcal{Y}\left(\balpha^{(\kappa)},\bell^{(\kappa)}\right)$ is odd (or equivalently if $\Odd\left(\alpha_{E_{\kappa-1}}\right)=-\Odd\left(\alpha_{E_\kappa}\right)$), cf. Definition \ref{def:exceptionalcomponentsoddeven}.

The multiplicity of $\sigma>0$ is equal to the number of trigonometric polynomials \eqref{eq:quasieq2} for which it is a root, and the multiplicity of $\sigma=0$ is equal to  half the number of times  $F_\mathrm{odd}$  is chosen in \eqref{eq:quasieq2}.
\end{itemize}
\end{theorem} 

We prove Theorem \ref{thm:polygoneqn1} directly from the definitions of quasi-eigenvalues in Section \ref{sec:trigpolynoms}; the proof of Theorem \ref{thm:polygoneqn0}, which uses the quantum graph analogy discussed below in subsection \ref{subsec:quantum}, is given in subsection \ref{subs:quantumgrapheigenvalues}.

\begin{remark} According to the definition, the quasi-eigenvalue $\sigma=0$ in the non-exceptional case always has multiplicity one if present. Moreover, as will be seen from the proof of Theorem  \ref{thm:quantum} in subsection \ref{subsec:proofquantum}, $\sigma=0$ is a quasi-eigenvalue in the non-exceptional case if and only if
\[
\prod_{j=1}^n \tan\frac{\pi^2}{4\alpha_j}= \prod_{j=1}^n \cot\frac{\pi^2}{4\alpha_j}=1.
\]
\end{remark}

\begin{remark}\label{rem:multiplicityexc} In the exceptional case, the set of roots of equations \eqref{eq:quasieq2} can be equivalently re-written as a set of roots of a single trigonometric equation
\begin{equation}\label{eq:quasieq2bis}
\prod_{\kappa\in\mathfrak{K}_\mathrm{even}} F_\mathrm{even}(\balpha^{(\kappa)},\bell^{(\kappa)};\sigma)\times 
\prod_{\kappa\in\mathfrak{K}_\mathrm{odd}} F_\mathrm{odd}(\balpha^{(\kappa)},\bell^{(\kappa)};\sigma)=0,
\end{equation}
where 
\[
\begin{split}
\mathfrak{K}_\mathrm{odd}&:=\left\{\kappa\in\{1,\dots,K\}: \Odd\left(\alpha^\mathcal{E}_\kappa\right)=-\Odd\left(\alpha^\mathcal{E}_{\kappa-1}\right)\right\},\\
\mathfrak{K}_\mathrm{even}&:=\left\{\kappa\in\{1,\dots,K\}: \Odd\left(\alpha^\mathcal{E}_\kappa\right)=\Odd\left(\alpha^\mathcal{E}_{\kappa-1}\right)\right\}
\end{split}
\]
The multiplicity of a positive quasi-eigenvalue is then equal to an algebraic multiplicity of it as a root of  \eqref{eq:quasieq2bis}, and the multiplicity of $\sigma=0$ is $\frac{\#\mathfrak{K}_\mathrm{odd}}{2}$.
\end{remark}

Since the  multiplicities of quasi-eigenvalues are finite,  Theorem \ref{thm:polygoneqn0}  immediately implies
the following 
\begin{proposition}
\label{prop:discrete}
 The quasi-eigenvalues of a curvilinear polygon form a discrete set with  the accumulation points only at $+\infty$. 
\end{proposition}
Indeed,  \eqref{eq:quasieq1} is an analytic functions of a real variable $\sigma$, and zeros  of analytic functions are isolated.

\begin{remark} 
\label{rem:symmetric}
It is easily seen that the real roots $\sigma$ of \eqref{eq:quasieq1}  are symmetric with respect to $\sigma=0$, and therefore the algebraic multiplicity of $\sigma=0$ is always even.
This, in principle, would also allow us to consider \emph{all real} quasi-eigenvalues in both non-exceptional and exceptional cases, and not just  the non-negative ones as in Definitions \ref{def:quasi} and \ref{def:quasiexc}, cf.\ also Remark \ref{rem:posnonpos}. Such an approach will be sometimes advantageous, and we will make clear when we  use it.
\end{remark}

\subsection{An eigenvalue problem on a quantum graph}
\label{subsec:quantum}
Consider the boundary of the polygon $\mathcal{P}(\balpha,\bell)$ as  a cyclic metric graph  $\mathcal{G}(\bell)$ with $n$ vertices $V_1,\dots, V_n$ and $n$ edges $I_j$ (joining $V_{j-1}$ and $V_j$, with $V_0$ identified with $V_n$) of length $\ell_j$, $j=1,\dots,n$.
Let $s$ be the arc-length parameter on $\mathcal{G}(\bell)$  starting at $V_1$ and going in the clockwise direction, see Figure \ref{fig:graph}.

\begin{figure}[htb]
\begin{center}
\includegraphics{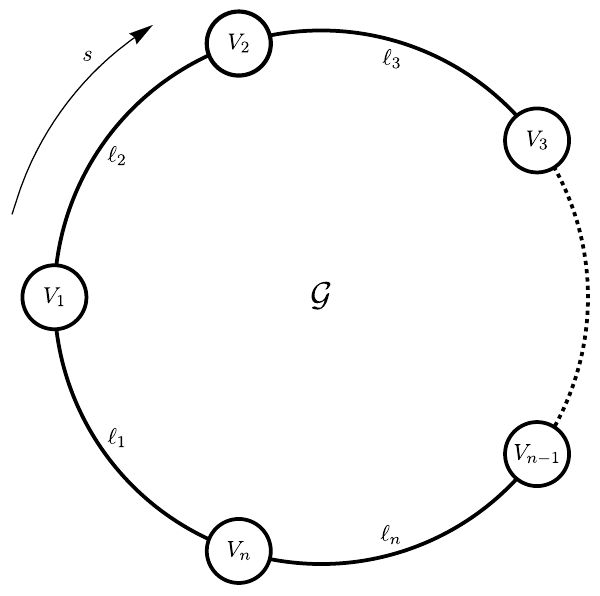}
\end{center}
\caption{A quantum graph.\label{fig:graph}}
\end{figure} 

Consider the spectral problem for a quantum graph Laplacian on $\mathcal{G}$ (see \cite{BK13} and references therein),
\[
-\frac{\mathrm{d}^2 f}{\mathrm{d}s^2}=\nu f,
\]
with matching  conditions 
\begin{equation}
\label{matchcondboth}
\begin{split}
\sin\left(\frac{\pi^2}{4\alpha_j}\right)f|_{V_j+0}&=\cos\left(\frac{\pi^2}{4\alpha_j}\right)\,f|_{V_j-0},\\
\cos\left(\frac{\pi^2}{4\alpha_j}\right)f'|_{V_j+0}&=\sin\left(\frac{\pi^2}{4\alpha_j}\right)\,f'|_{V_j-0}.
\end{split}
\end{equation}
Hereinafter at each vertex $V_j$, $j=1,\dots,n$, $g|_{V_j-0}$ and $g|_{V_j+0}$ denote the limiting values of a quantity $g(s)$ as $s$ approaches the vertex $V_j$ from the left and from the right, respectively, in the direction of $s$. 
 
 \begin{remark} For  $\alpha_j \notin \mathcal{E}$, we can re-write the matching conditions as
\begin{equation}
\label{matchcond}
\begin{split}
f|_{V_j+0}&=\cot\left(\frac{\pi^2}{4\alpha_j}\right)\,f|_{V_j-0},\\
f'|_{V_j+0}&=\tan\left(\frac{\pi^2}{4\alpha_j}\right)\,f'|_{V_j-0},
\end{split}
\end{equation}
For $\alpha_j \in \mathcal {E}$, the matching conditions are given by
\begin{equation}
\label{matchcond:exc}
\begin{cases}
f|_{V_j-0}=f'|_{V_j+0}=0\quad &\text{if }\ \Odd(\alpha_j)=1,\\
f|_{V_j+0}=f'|_{V_j-0}=0\quad &\text{if }\ \Odd(\alpha_j)=-1.
\end{cases}
\end{equation}
\end{remark}

We will denote the operator $f\mapsto -\frac{\mathrm{d}^2f}{\mathrm{d}s^2}$ subject to matching conditions  \eqref{matchcondboth} by $\Delta_\mathcal{G}$. It is easy to check that $\Delta_\mathcal{G}$ is self-adjoint and non-negative.  Therefore, its spectrum is given by
a sequence of non-negative real eigenvalues 
\[
0\le \nu_1 \le \nu_2  \le \dots \nu_m  \le \dots \nearrow +\infty,
\]
listed  with the account of multiplicities. 
\begin{remark}\label{rem:varprinciple}
The eigenvalues $\nu_m$ also satisfy a standard variational principle: if
\[
\operatorname{Dom}(Q_\mathcal{G}):=\left\{f\in\bigoplus_{j=1}^n H^1(I_j): 
\sin\left(\frac{\pi^2}{4\alpha_j}\right)f|_{V_j+0}=\cos\left(\frac{\pi^2}{4\alpha_j}\right)f|_{V_j-0}\right\}
\]
denotes the domain of the quadratic form 
\[
Q_\mathcal{G}[f]:=\sum\limits_{j=1}^n \int_{I_j} (f'(s))^2\,\dr s
\]
of  $\Delta_\mathcal{G}$, then
\[
\nu_m=\inf_{\substack{S\subset\operatorname{Dom}(Q_\mathcal{G})\\\dim S=m}}\sup_{0\ne f\in S} \frac{Q_\mathcal{G}[f]}{\sum\limits_{j=1}^n \int_{I_j} (f(s))^2\,\dr s}.
\]
\end{remark}

 It turns out the eigenvalues $\nu_m$ are precisely the squares of the quasi-eigenvalues of the Steklov problem on the polygon $\mathcal{P}(\balpha,\bell)$ as defined by Definitions \ref{def:quasi} and \ref{def:quasiexc}.
\begin{theorem}
\label{thm:quantum}
Let $\sigma_m$, $m\ge 1$, be the Steklov quasi-eigenvalues of a curvilinear polygon $\mathcal{P}(\balpha,\bell)$, and let $\nu_m$, $m\ge 1$, be the eigenvalues of $\Delta_\mathcal{G}$, in both cases ordered non-decreasingly with account of multiplicities. Then $\sigma_m^2=\nu_m$ for all $m\ge 1$.
\end{theorem}

\begin{remark}\label{rem:secular} Theorem \ref{thm:polygoneqn0} will be derived from Theorem \ref{thm:quantum}: we will demonstrate  in subsection  \ref{subs:quantumgrapheigenvalues} that the quantum graph eigenvalues $\nu=\nu_m$ are the roots of the graph \emph{secular equation} \eqref{eq:secular}, which is equivalent to $F^\mathcal{P}(\balpha,\bell,\sqrt{\nu})=0$. 
\end{remark}

\begin{remark} 
\label{rem:notquant}
We would like to emphasise that the eigenfunctions of $\Delta_\mathcal{G}$ are \emph{not} the quasimodes of  the Dirichlet-to-Neumann map $\mathcal{D}_\Omega$; moreover, they do not even belong to the domain of $\mathcal{D}_\Omega$. What rather happens is that each eigenfunction of $\mathcal{D}_\Omega$ carries enough information to \emph{construct} a corresponding proper  Dirichlet-to-Neumann quasimode. Note also that we cannot deduce the completeness of the set of eigenfunctions of  
$\mathcal{D}_\Omega$ corresponding to quasimodes directly from the completeness of the set of eigenfunctions of $\Delta_\mathcal{G}$. Indeed, while the eigenfunctions of $\mathcal{D}_\Omega$ could in principle be viewed as perturbations of the eigenfunctions of $\Delta_\mathcal{G}$, the error is too big to guarantee the completeness of the perturbed set via  the standard Bary-Krein lemma \cite[Lemma 4.8]{sloshing}. 
\end{remark}

The proof of Theorem \ref{thm:quantum} is postponed until Section \ref{sec:quantum}. It uses an alternative formulation of the quantum graph problem which, although more complicated to state, is more closely related to the Steklov problem. We consider the eigenvalue problem for the 
following Dirac-type operator on $\mathcal{G}(\bell)$:
\begin{equation}\label{eq:defofDirac}
\Dir=\begin{pmatrix} -\ir\frac{\mathrm{d}}{\mathrm{d}s} &0 \\0 & \ir\frac{\mathrm{d}}{\mathrm{d}s}
\end{pmatrix},
\end{equation}
acting on vector functions $\mathbf{f}(s)=\begin{pmatrix}f_1(s) \\ f_2(s)\end{pmatrix}$; here $s$ is the arc-length coordinate on $\mathcal{G}(\bell)$, see Figure \ref{fig:graph}.  For $\alpha_j \notin \mathcal{E}$, we impose matching conditions at $V_j$ given by
\begin{equation}
\label{matchVj}
\mathbf{f}|_{V_j+0} = {\mathtt A}(\alpha_j)\mathbf{f}|_{V_j-0},
\end{equation}
where ${\mathtt A}(\alpha_j)$ is the vertex transfer matrix defined by \eqref{eq:Adef1}. If $\alpha_j \in \mathcal{E}$ we set 
\begin{equation}\label{matchVjexc}
\begin{split}
\mathbf{f}|_{V_j-0}&\text{ is proportional to } \Cvect(\alpha_j)^\perp\\
\mathbf{f}|_{V_j+0}&\text{ is proportional to }  \Cvect(\alpha_j),
\end{split}
\end{equation}
where $ \Cvect(\alpha_j)$ is defined by \eqref{eq:orthogspecial2}. 

We have the following
\begin{proposition}\label{prop:Dirac} The operator $\Dir$, with the domain consisting of vector-functions $\mathbf{f}(s)$ such that their restrictions to the edge $I_j$ are in $(H^1(I_j))^2$  and they satisfy the matching conditions above, is self-adjoint in $(L^2(\mathcal{G}))^2$.  Moreover, with multiplicity, its eigenvalues are the real solutions of equation \eqref{eq:Ttrace} (provided $\alpha_j \notin \mathcal{E}$, $j=1,\dots, n$), or of equation \eqref{excepeq} if there exists $\alpha_j\in \mathcal{E}$.
\end{proposition}
Proposition \ref{prop:Dirac} will be proved in Section \ref{sec:quantum}. Along with Theorem \ref{thm:quantum} it shows that the squares of the eigenvalues of $\Dir$ are precisely the eigenvalues $\nu_m$ of our quantum graph Laplacian.

\begin{remark}\label{rem:quantum} Note that in the case of a graph Dirac operator, we need to consider  \emph{all} solutions of the characterstic equations, not just non-negative ones as in Definitions \ref{def:quasi} and \ref{def:quasiexc}. Moreover, in view of Remark \ref{rem:symmetric} and Definitions \ref{def:quasimult} and \ref{def:quasiexcmult}, the spectrum of $\Dir$ may be represented as $\{\pm \sigma_m\}$, with the same multiplicities as for quasi-eigenvalues of $\mathcal{P}$ if $\sigma_m>0$ and \emph{twice} the multiplicity of an eigenvalue $\sigma_m=0$. In other words, the multiplicity of $\sigma^2$  in the spectrum of $\Dir^2$ coincides with \emph{twice} the multiplicity of $\sigma$ as a quasi-eigenvalue of $\mathcal{P}$.
\end{remark}

\subsection{Riesz mean and heat trace asymptotics} 
Let $\{s_m\}$, $ m=1,2,\dots$, be a non-decreasing sequence of nonnegative real numbers. 
\begin{definition}
The function  $\mathcal{N}(\{s_m\}; z):=\#\{m\in \mathbb{N}\,| \, s_m \le  z \}$ is called the {\it counting function} for the sequence $\{s_m\}$, and the function
\begin{equation}
\label{eq:rieszdef}
\mathcal{R}(\{s_m\}; z)=\mathcal{R}_1(\{s_m\}; z):=\int_0^z \mathcal{N}(\{s_m\}; t)\, \dr t=\sum_{m=1}^\infty (z-s_m)_+
\end{equation}
is called the {\it first Riesz mean} (or simply the Riesz mean) of $\{s_m\}$. Here $z_+=\max(z,0)$. 
\end{definition}
The asymptotics of the Riesz mean often captures more refined features of the distribution of the sequence $\{ s_n \}$  than the asymptotics of the counting function. In particular, it is a standard tool to study eigenvalue asymptotics, see, for instance, \cite{Sa, LH}.

Let $\NP(\lambda):=\mathcal{N}(\{\lambda_m\};\lambda)$ and $\mathcal{R}_\mathcal{P}(\lambda):=\mathcal{R}(\{\lambda_m\}; \lambda)$  be, respectively,  the eigenvalue counting function and the Riesz mean for the Steklov eigenvalues on a curvilinear polygon 
$\mathcal{P}$. We first prove a basic Weyl law. Observe that due to Theorem \ref{thm:quantum}, if $\mathcal{N}(\{\sigma_m\};\sigma)$ is the counting function for the \emph{quasi-eigenvalues} $\sigma_m$, we have by \cite[Lemma 3.7.4]{BK13} that
\begin{equation}\label{eq:Weylawquasi}
\mathcal{N}(\{\sigma_m\};\sigma)=\frac{|\partial \mathcal{P}|}{\pi}\sigma+O(1).
\end{equation}
This can be easily combined with Theorem \ref{thm:main} to yield the following Weyl law, which was proved in \cite[Corollary 1.11]{sloshing} but only for straight polygons.
\begin{proposition} For \emph{any} curvilinear polygon $\mathcal P$ with angles less than $\pi$,
\begin{equation}
\label{eq:Weylaw}
\NP(\lambda)=\frac{|\partial \mathcal{P}|}{\pi}\lambda + O(1)\qquad \text{ as }\lambda\to+\infty.
\end{equation}
\end{proposition}

As a consequence, one expects (see \cite{Sa}) that
\begin{equation}\label{eq:riesz}
\Riesz(\lambda)=\frac{|\partial \mathcal{P}|}{2 \pi}\lambda^2 + c_1\lambda + o(\lambda) 
\end{equation}
for some constant coefficient $c_1$. 
\begin{theorem}
\label{thm:riesz}
Let  $\mathcal{P}$ be a  curvilinear polygon with $n$ sides of  lengths $\ell_1, \dots \ell_n$. Let $\tilde\varepsilon\in(0,\varepsilon_0)\cap\left(0,\frac{1}{2n+1}\right]$, with $\varepsilon_0$ as in Theorem \ref{thm:main}. Then the Riesz mean for the Steklov eigenvalues of $\mathcal{P}$ satisfies the asymptotics
\begin{equation}\label{eq:rieszcorr}
\Riesz(\lambda)=\frac{|\partial \mathcal{P}|}{2 \pi}\lambda^2 + O(\lambda^{1-\tilde\varepsilon})\quad\text{ as }\lambda\to+\infty.
\end{equation}
In particular, the formula  \eqref{eq:riesz} holds with  the coefficient $c_1=0$. 
\end{theorem}

Theorem \ref{thm:riesz} immediately implies
\begin{corollary}\label{cor:heat} The Steklov heat trace on a curvilinear polygon $\mathcal{P}$ satisfies an asymptotic formula
\begin{equation}
\label{heatrace}
\sum_{k=1}^\infty \er^{-t\lambda_k}=\frac{|\partial \mathcal{P}|}{\pi\,t}+O(t^{\tilde\varepsilon})\,\,\, \text{as } \,\, t\to 0^+.
\end{equation}
\end{corollary}
Indeed, this follows by a direct computation from a well-known relation between the heat trace and the Riesz mean.
\[
\sum_{k=1}^\infty \er^{-t\lambda_k}=t^2\int_0^\infty \mathcal{R}(\lambda_k; z)\,e^{-zt}\,\dr z.
\]

\begin{remark}
 It would be interesting to establish the existence of a complete asymptotic expansion for the Steklov heat trace
on a curvilinear polygon, similarly to the smooth case, see \cite[formula (1.2.2)]{PS}. Formula \eqref{heatrace} implies that the first heat invariant is zero, since the constant term on the right-hand side of \eqref{heatrace} vanishes. Note that the same result holds for smooth planar domains, see \cite[Remark 1.4.5]{PS}). The fact that the constant term in the Steklov heat trace is the same for polygons and for smooth domains is somewhat surprising, as it is not the case for the heat invariants arising from the boundary value problems for the Laplacian, see \cite{MR, NRS}.
\end{remark}
\begin{remark}
In view of Theorem \ref{thm:quantum}, one could also deduce the expansion \eqref{heatrace} from the heat asymptotics for the eigenvalues of a quantum graph \cite{Rueck}
using the standard results relating the heat traces of an operator and of its power via the zeta function (see \cite{Gil, Gru}). 
\end{remark}

\subsection{Zigzags} 
The notation and results of this section may seem rather esoteric. Although they are auxiliary, they are absolutely essential for proving the main Theorems of the paper.

\begin{definition}\label{def:zigzag} 
Let $n\in\mathbb{N}$, $\bell=(\ell_1,\dots,\ell_n)\in\mathbb{R}_+^n$, and $\balpha=(\alpha_1,\dots,\alpha_{n-1})\in\Pi^{n-1}$. 
A curvilinear $n$ piece \emph{zigzag} $\mathcal{Z}=\mathcal{Z}(\balpha,\bell)$ is a piecewise smooth continuous non-self-intersecting curve in $\mathbb{R}^2$ with  vertices $V_0$,\dots,$V_n$ and  smooth arcs $I_j$ of length $\ell_j$ joining $V_{j-1}$ and $V_j$, $j=1,\dots,n$. The arcs $I_j$ and $I_{j+1}$ meet at $V_j$ at an angle $\alpha_j$ (measured  from $I_j$ to $I_{j+1}$ counterclockwise), $j=1,\dots,n-1$, see Figure \ref{fig:zigzag}. The vertices $V_0$ and $V_n$ will be called the \emph{start} and \emph{end points} of $\mathcal{Z}$, respectively (or just endpoints if we do not need to distinguish them).

\begin{figure}[htb]
\begin{center}
\includegraphics{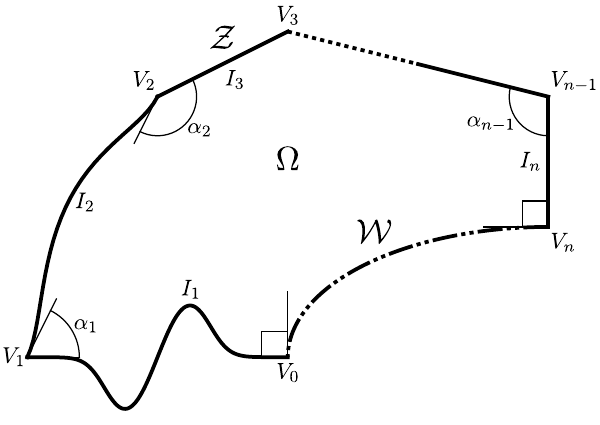}
\end{center}
\caption{A zigzag $\mathcal{Z}$ and a zigzag domain $\Omega$\label{fig:zigzag}}
\end{figure}

We will call a zigzag \emph{straight} if its arcs $I_1,\dots,I_n$, are straight-line intervals, and \emph{partially curvilinear} if the arcs  are straight in a neighbourhood of each vertex. 

We will call a zigzag \emph{non-exceptional} or \emph{exceptional} if $\balpha\in(\Pi\setminus\Eangles)^{n-1}$ or if there exists $\alpha_j\in\Eangles$, respectively.
\end{definition} 

\begin{definition}\label{def:zigzagdomain} Let $\mathcal{Z}$ be a zigzag. A \emph{$\mathcal Z$-zigzag domain} $\Omega\subset\mathbb{R}^2$ (or just a \emph{zigzag domain}) is an open bounded simply connected set whose boundary $\partial\Omega=\mathcal{Z}\cup\mathcal{W}$, where a piecewise smooth non-self-ntersecting curve $\mathcal{W}$ meets $\mathcal{Z}$ only at the start and end points  of $\mathcal{Z}$ forming interior angles $\frac{\pi}{2}$.
\end{definition}

Let $\Omega$ be a zigzag domain with boundary $\partial\Omega=\mathcal{Z}\cup\mathcal{W}$. We consider in $\Omega$ generalised mixed Dirichlet-Neumann-sloshing  eigenvalue problems of the type
\begin{equation}\label{eq:gensloshing}
\Delta u =0\quad \text{in }\Omega,\qquad \frac{\partial u}{\partial n}=\lambda u\quad \text{on }\mathcal{Z},\qquad u\frac{\partial u}{\partial n}=0\quad \text{on }\mathcal{W},
\end{equation}
The last condition is understood in the following sense: we represent $\mathcal{W}$ as a closure of a finite union of non-intersecting open arcs, and impose either Dirichlet or Neumann condition on each arc.  We will write 
\begin{equation}\label{eq:gensloshingDN}
\mathcal{D}_{\Omega,\mathcal{Z}}: u |_\mathcal{Z} \left.\mapsto \frac{\partial u}{\partial n}\right|_\mathcal{Z}\quad\text{subject to }\Delta u=0\text{ in }\Omega,\ u\frac{\partial u}{\partial n}=0 \text{ on }\mathcal{W}
\end{equation}
for the corresponding (partial) Dirichlet-to-Neumann map on $\mathcal{Z}$.

Each such generalised mixed Dirichlet-Neumann-sloshing problem has a discrete spectrum of eigenvalues $\lambda_1<\lambda_2\le \dots$ accumulating to $+\infty$.

We will term \eqref{eq:gensloshing} a \emph{Dirichlet--Dirichlet zigzag problem} (or \emph{$DD$-zigzag} for short) and refer to it as \eqref{eq:gensloshing}$_{DD}$ if the Dirichlet boundary condition is imposed on $\mathcal{W}$ in neighbourhoods of both start and end points of $\mathcal{Z}$, independently of the boundary conditions on the rest of $\mathcal{W}$. Similarly, we will term  \eqref{eq:gensloshing} a \emph{Neumann--Dirichlet zigzag problem} (or \emph{$ND$-zigzag} for short)  and refer to it as \eqref{eq:gensloshing}$_{ND}$ if the Neumann boundary condition is imposed on $\mathcal{W}$ in a neighbourhood of the start point of $\mathcal{Z}$, and the Dirichlet boundary condition in a neighbourhood of the end point of $\mathcal{Z}$. The \emph{$DN$-zigzags} and \emph{$NN$-zigzags} are defined analogously. In general, we will write $\aleph\beth$-zigzag, or $\mathcal{Z}^{(\aleph\beth)}$, with $\aleph, \beth\in\{D,N\}$, and refer to \eqref{eq:gensloshing}
as \eqref{eq:gensloshing}$_{\aleph\beth}$ to indicate the boundary conditions imposed on $\mathcal{W}$ near the start and end point of $\mathcal{Z}^{(\aleph\beth)}$.

Define the vectors 
\begin{equation}\label{eq:ND}
\mathbf{N}:=\begin{pmatrix}1\\1\end{pmatrix},\qquad \mathbf{D}:=\begin{pmatrix}\ir\\-\ir\end{pmatrix}.
\end{equation}

Note that the vectors $\mathbf{N}, \mathbf{D}$ are orthogonal, and we will set $\mathbf{N}^\perp:=\mathbf{D}$ and $\mathbf{D}^\perp:=\mathbf{N}$. We will write $\baleph, \bbeth$ to indicate any of the vectors $\mathbf{N}, \mathbf{D}$.

\begin{definition}\label{def:quasizigzag} 
Let $\mathcal{Z}=\mathcal{Z}(\balpha,\bell)$ be a non-exceptional zigzag. Let $\aleph, \beth\in\{D,N\}$. A real number $\sigma$ is called a \emph{quasi-eigenvalue of the $\aleph\beth$-zigzag} $\mathcal{Z}$ if $\sigma$ is a solution of the equation
\begin{equation}\label{eq:quasizigzag}
\mathtt{U}\left(\balpha, \bell,\sigma\right) \baleph \cdot \bbeth^\perp=0,
\end{equation}
where $\mathtt{U}$ is defined in \eqref{eq:Udef}.
\end{definition}

\begin{remark}\label{rem:parallel} Condition \eqref{eq:quasizigzag} can be equivalently restated as
\begin{equation}\label{eq:quasizigzagparallel}
\mathtt{U}\left(\balpha, \bell,\sigma\right) \baleph \text{ is proportional to } \bbeth,
\end{equation}
cf. Remark \ref{rem:quasiexparallel}.
\end{remark}

The enumeration of zigzag quasi-eigenvalues is much more delicate than in the Steklov problem, but with an appropriate choice of the so-called \emph{natural enumeration}, see Section \ref{sec:completeness}, we have
\begin{theorem}\label{thm:quasizigzags} Let $\mathcal{Z}$ be a partially curvilinear zigzag  with all non-exceptional angles $\alpha_1$,\dots,$\alpha_{n-1}$, and let $\Omega$ be any $\mathcal{Z}$-zigzag domain. For $\aleph, \beth\in\{D,N\}$, let $\lambda_m^{(\aleph\beth)}$ denote the eigenvalues of \eqref{eq:gensloshing}$_{\aleph\beth}$ enumerated in increasing order with account of multiplicities, and let $\sigma_m^{(\aleph\beth)}$ denote the quasi-eigenvalues of the $\aleph\beth$-zigzag $\mathcal{Z}$ in the natural enumeration. Then
\[
\lambda_m^{(\aleph\beth)}=\sigma_m^{(\aleph\beth)}+o(1)\qquad\text{as }m\to\infty.
\]
\end{theorem}

\begin{remark}\label{rem:posnonpos} There is a distinction between the quasi-eigenvalue Definitions \ref{def:quasi} and \ref{def:quasiexc} for polygons and Definition \ref{def:quasizigzag} for zigzag domains --- the former include only \emph{non-negative} quasi-eigenvalues, whereas the latter allow for all the \emph{real} ones, cf.\ also Remark \ref{rem:symmetric}. This is not an oversight but a deliberate choice, although a forced one. The reason for that is that the natural enumeration for zigzag domains mentioned above sometimes takes into account \emph{some} negative quasi-eigenvalues.
\end{remark}

An analog of Theorem \ref{thm:quasizigzags} exists for exceptional zigzags, but we postpone the statement until Section \ref{sec:completeness}.

There is also a quantum graph analogy of Proposition \ref{prop:Dirac}  for an $\aleph\beth$-zigzag problem. Let us associate with a non-exceptional zigzag $\mathcal{Z}(\balpha,\bell)$ a path $\mathcal{L}$ joining the vertex $V_0$ to the vertex $V_n$  through $V_1,\dots,V_{n-1}$, see Figure \ref{fig:graphL}. The length of each edge $I_j$ joining $V_{j-1}$ to $V_j$, $j=1,\dots,n$, is taken to be $\ell_j$, and let $s$ be  the coordinate on $\mathcal{L}$.

\begin{figure}[htb]
\begin{center}
\includegraphics{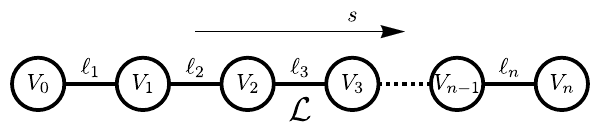}
\end{center}
\caption{A path $\mathcal{L}$ \label{fig:graphL}}
\end{figure} 

Consider the Dirac operator \eqref{eq:defofDirac} on $\mathcal{L}$ acting on vector functions $\mathbf{f}(s)$ with the matching conditions \eqref{matchVj} at internal vertices $V_1,\dots,V_{n-1}$ and with the boundary conditions 
\begin{equation}\label{eq:Diracalephbeth}
\mathbf{f}|_{V_0+}\cdot\aleph^\perp=\mathbf{f}|_{V_n-}\cdot\beth^\perp=0.
\end{equation}

We have the following 
\begin{proposition}\label{prop:DiracL} The operator $\Dir$ on the path $\mathcal{L}$, with the domain consisting of vector-functions $\mathbf{f}(s)$ such that their restrictions to the edge $I_j$ are in $(H^1(I_j))^2$  and they satisfy the matching and boundary conditions above, is self-adjoint in $(L^2(\mathcal{L}))^2$.  Moreover, with multiplicity, its eigenvalues are the real solutions of equation \eqref{eq:quasizigzag}.
\end{proposition}

The proof of Proposition \ref{prop:DiracL}  is almost identical to that of Proposition \ref{prop:Dirac} and is omitted.
\clearpage\section{Auxiliary problems in a sector. Peters solutions}\label{sec:sector}
\subsection{Plane wave solutions in a sector}
Let $(x,y)$ be Cartesian coordinates in $\mathbb{R}^2$, let $z=x+\ir y\in\mathbb C$, and let $(\rho,\theta)$ denote polar coordinates so that $z=\rho\er^{\ir\theta}$. Consider the sector $\Sct{\alpha}=\{-\alpha<\theta<0\}$, where $0<\alpha\le\pi$, and denote its boundary components by $I_{\inn}=\{\theta=-\alpha\}$ and $I_{\out}=\{\theta=0\}$. Let $I=\{\theta=-\alpha/2\}$ denote its bisector. Let us additionally introduce the natural coordinate $s$ on $I_{\inn}\cup I_{\out}$ so that $s$ is zero at the vertex, negative on $I_{\inn}$ and positive on $I_{\out}$, see Figure \ref{fig:sector}.

\begin{figure}[htb]
\begin{center}
\includegraphics{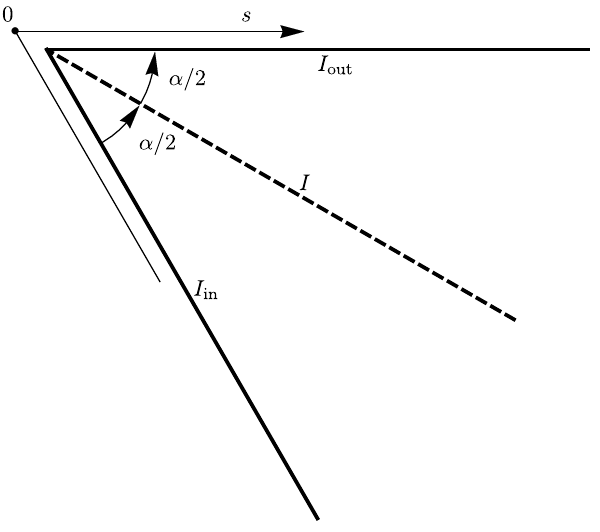}
\end{center}
\caption{Sectors $\Sct{\alpha}$ and $\Sct{\alpha/2}$.\label{fig:sector}}
\end{figure}

Let,  for $t\in\mathbb{R}$,
\begin{equation}\label{eq:edef}
\mathbf{e}(t):=\begin{pmatrix} \er^{-\ir t}\\\er^{\ir t}\end{pmatrix}.
\end{equation}
For any fixed vector $\mathbf h=\begin{pmatrix}h_1\\h_2\end{pmatrix}\in\mathbb C^2$, define the harmonic plane waves (which we will call  \emph{outgoing and incoming} plane waves)
\[
W^{\mathbf h}_{\out,\alpha}(z):=\er^y(\mathbf h\cdot \mathbf{e}(x))=\er^y\left(h_1\er^{\ir x}+h_2\er^{-\ir x}\right),\qquad W^{\mathbf h}_{\inn,\alpha}(z):=W_{\out,\alpha}^{\mathbf{h}'}(\mathcal M_{\alpha}(z)),
\]
where ${\mathbf h}'=\begin{pmatrix}h_2\\h_1\end{pmatrix}$ and $\mathcal M_{\alpha}:(\rho,\theta)\mapsto (\rho, \alpha-\theta)$ is the operator of reflection across the bisector $I$.  It is important to observe that
\begin{equation}\label{eq:Wboundary}
\left.W^{\mathbf h}_{\out,\alpha}(z)\right|_{I_{\out}}=\mathbf h\cdot\mathbf{e}(s),\qquad \left.W^{\mathbf h}_{\inn,\alpha}(z)\right|_{I_{\inn}}=\mathbf h\cdot\mathbf e(s),
\end{equation}
and that $W^{\mathbf h}_{\out,\alpha}(z)$ and $W^{\mathbf h}_{\inn,\alpha}(z)$ are bounded inside the sector.

Consider the Robin boundary value problem 
\begin{equation}\label{eq:Robin1}
\Delta \Phi = 0\quad\text{ in }\Sct{\alpha},\qquad \frac{\partial\Phi}{\partial n}=\Phi\quad\text{ on }\partial\Sct{\alpha}
\end{equation}
in the sector $\Sct{\alpha}$, cf.\ \cite{Kh,KOBP}. We are interested in solutions of \eqref{eq:Robin1} which approximately behave as a combination of an incoming and an outgoing plane wave,
that is, as 
\begin{equation}\label{eq:definitionofphi}
\Phi(z)= \Phi_{\alpha}^{(\mathbf h_{\inn},\mathbf h_{\out})}(z):=W^{\mathbf h_{\out}}_{\out,\alpha}(z)+W^{\mathbf h_{\inn}}_{\inn,\alpha}(z)+R^{\mathbf h_{\inn},\mathbf h_{\out}}_{\alpha}(z),
\end{equation}
with some vectors $\mathbf h_{\inn}$ and $\mathbf h_{\out}\in\mathbb{C}^2$, where the remainder $R=R^{\mathbf{h}_{\inn},\mathbf{h}_{\out}}_{\alpha}(z)$ is decreasing, together with its gradient, away from the corner, in the sense that 
\begin{equation}\label{eq:universalremainder}
|R(z)|+\|\rho\nabla_{(x,y)}R(z)\|\leq C\, \rho^{-r}
\end{equation}
for all $z\in\Sct{\alpha}$ with $|z|$ sufficiently large, with some constant $r>0$ depending on the angle $\alpha$, and some constant $C>0$ which may additionally depend on $\|\mathbf h_{\inn}\|$ and $\|\mathbf h_{\out}\|$. In particular, we are interested in sufficient conditions on $\mathbf h_{\inn}$ and $\mathbf h_{\out}$ for the existence of a solution \eqref{eq:definitionofphi}. The next result, which is the main statement of this section, shows that these sufficient conditions differ depending upon exceptionality of the  angle $\alpha$. 

Throughout the rest of this section, let
\begin{equation}\label{eq:chi}
\begin{aligned}
\mu&=\mu_{\alpha/2}=\frac{\pi}{\alpha},\\
\chi_N&=\chi_{\alpha/2,N}=\frac{\pi}{4}(1-\mu)=\frac{\pi}{4}-\frac{\pi^2}{4\alpha},\\
\chi_D&=\chi_{\alpha/2,D}=\frac{\pi}{4}(1+\mu)=\frac{\pi}{4}+\frac{\pi^2}{4\alpha}.
\end{aligned}
\end{equation}
This notation is chosen to match \cite{sloshing}.

\begin{theorem}\label{thm:mainthmssector} 
(a) Let $\alpha$ be non-exceptional, i.e.\ $\alpha\notin\Eangles$. Then for any vector $\mathbf h_{\inn}\in\mathbb C^2$, there exists a vector $\mathbf h_{\out}\in\mathbb C^2$ and a solution \eqref{eq:definitionofphi} of \eqref{eq:Robin1} satisfying \eqref{eq:universalremainder} with $r=\mu_{\alpha/2}$ and $C=C_\alpha \|\mathbf h_{\inn}\|$, where $C_\alpha>0$ is some constant depending only on $\alpha$.

Moreover, in this case 
\begin{equation}\label{eq:hinoutA}
\mathbf h_{\out}=\mathtt{A}(\alpha)\mathbf h_{\inn},
\end{equation} 
where $\mathtt{A}(\alpha)$ is the matrix defined in \eqref{eq:Adef1}.

(b) If $\alpha=\alpha^\Eangles$ is exceptional,  $\alpha=\dfrac{\pi}{2k}\in\Eangles$, $k\in\mathbb{N}$, then for any two vectors $\mathbf h_{\inn}$ and $\mathbf h_{\out}\in\mathbb C^2$ additionally satisfying
\begin{equation}\label{eq:exceptionalcondition}
\mathbf{h}_{\inn}\cdot\Cvect(\alpha)= \mathbf h_{\out}\cdot\Cvect^\perp(\alpha)=0
\end{equation}
(with $\Cvect(\alpha)$ defined by \eqref{eq:orthogspecial2}, see also \eqref{eq:orthogspecial1} and \eqref{eq:Cevenoddperp}), there exists a solution \eqref{eq:definitionofphi} of \eqref{eq:Robin1} again satisfying \eqref{eq:universalremainder} with $r=\mu_{\alpha/2}$ and $C=C_\alpha (\|\mathbf h_{\inn}\|+\|\mathbf h_{\out}\|)$, where $C_\alpha>0$ is some constant depending only on $\alpha$.
\end{theorem}

\begin{remark} In both the non-exceptional and exceptional angle cases, we obtain the existence of a solution $\Phi_{\alpha}^{(\mathbf h_{\inn},\mathbf h_{\out})}$ by fixing \emph{two} out of the four components of the vectors $\mathbf h_{\inn}$ and $\mathbf h_{\out}$. The difference is that in the non-exceptional case we fix the two components of the same vector and find the other vector from \eqref{eq:hinoutA} (it does not in fact matter whether we fix either of the two vectors as $\mathtt{A}(\alpha)$ is invertible), whereas in the exceptional case we fix exactly one component of each of $\mathbf h_{\inn}$ and $\mathbf h_{\out}$, and recover the other ones from \eqref{eq:exceptionalcondition}.
\end{remark}

\begin{remark}\label{rem:Cparallel} Conditions \eqref{eq:exceptionalcondition} can be equivalently rewritten as
\[
\mathbf{h}_{\inn}\in\operatorname{Span}\left\{\Cvect^\perp(\alpha)\right\},\qquad \mathbf h_{\out}\in\operatorname{Span}\left\{\Cvect(\alpha)\right\}.
\]
\end{remark}

\begin{remark}\label{rem:alphabiggerpi} Note that our proof of Theorem \ref{thm:mainthmssector} does not work for $\alpha\ge\pi$ for reasons explained in \cite[Remark 2.4]{sloshing}.
\end{remark}

\subsection{Sloping beach problems and Peters solutions} 
Consider, in the half sector $\Sct{\alpha/2}$, a mixed Robin-Neumann problem 
\begin{equation}\label{eq:slopingN}
\Delta \Phi=0\quad\text{in }\Sct{\alpha/2},\qquad \left(\left.\frac{\partial\Phi}{\partial y}-\Phi\right)\right|_{I_\out}=0,\qquad \left.\frac{\partial\Phi}{\partial n}\right|_{I}=0,
\end{equation}
and a similar mixed  Robin-Dirichlet problem 
\begin{equation}\label{eq:slopingD}
\Delta \Phi=0\quad\text{in }\Sct{\alpha/2},\qquad \left(\left.\frac{\partial\Phi}{\partial y}-\Phi\right)\right|_{I_\out}=0,\qquad \left.\Phi\right|_{I}=0,
\end{equation}
These two problems, called the \emph{sloping beach} problems and arising in hydrodynamics, have  special solutions, originally due to Peters \cite{Pet50} in the Neumann case, are written down, with some improvements on the remainder terms, in \cite[Theorem 2.1]{sloshing}. We now define two specific solutions $\Phi_{\alpha,N}$ and $\Phi_{\alpha,D}$ of the problem \eqref{eq:Robin1} in the full sector $\Sct{\alpha}$, which we call the \emph{symmetric/anti-symmetric Peters solutions} in $\Sct{\alpha}$. To obtain $\Phi_{\alpha,N}$, we  take the even (with respect to $I$) extension of Peters sloping beach solution of \eqref{eq:slopingN}.  To obtain $\Phi_{\alpha,D}$, we  take the odd (with respect to $I$) extension  of Peters sloping beach solution of \eqref{eq:slopingD}. 

The key properties of $\Phi_{\alpha,N}$ and $\Phi_{\alpha,D}$ now follow quickly from \cite[Theorem 2.1]{sloshing}:
\begin{lemma}\label{lem:propsofsymantisym} We have, for $\aleph\in\{N,D\}$,
\[\Phi_{\alpha,\aleph}(z)=W^{\mathbf g_{\out,\aleph}}_{\out,\alpha}(z)+W^{\mathbf g_{\inn,\aleph}}_{\inn,\alpha}(z)+\widetilde R_{\alpha,\aleph}(z),\]
with
\begin{equation}\label{eq:gdef}
\begin{split}
\mathbf g_{\out,N}&=\frac 12\begin{pmatrix}\er^{-\ir\chi_{N}}\\ \er^{\ir\chi_{N}}\end{pmatrix},\ \mathbf g_{\inn,N}=\overline{\mathbf g_{\out,N}},\\ 
\mathbf g_{\out,D}&=\frac 12\begin{pmatrix}\er^{-\ir\chi_{D}}\\ \er^{\ir\chi_{D}}\end{pmatrix},\ \mathbf g_{\inn,D}=-\overline{\mathbf g_{\out,D}},
\end{split}
\end{equation}
and the remainder terms $R=\widetilde R_{\alpha,\aleph}(z)$ satisfy \eqref{eq:universalremainder} with some constants $C>0$  depending only on $\alpha$, and with $r=\mu$ in the case $\aleph=N$ and  $r=2\mu$ in the case $\aleph=D$.
\end{lemma}

\begin{proof} We prove this for the Neumann solution and for $-\alpha/2\leq\theta\leq 0$. By \cite[Theorem 2.1]{sloshing}, the Peters solution for \eqref{eq:slopingN} in $\Sct{\alpha/2}$ is equal to
\[
\er^y\cos(x-\chi_{N})+R_{N}(x,y),
\]
where $R=R_{N}$ satisfies \eqref{eq:universalremainder} with $r=\mu$. 
Converting the cosine term to a complex exponential, we obtain, with account of \eqref{eq:gdef},
\[
\er^y\cos(x-\chi_{N})=\er^y\,\frac{\er^{-\ir\chi_{N}}\er^{\ir x}+\er^{\ir\chi_{N}}\er^{-\ir x}}{2}=\er^y\mathbf{g}_{\out,N}\cdot\mathbf{e}(x),
\]
which is precisely $W^{\mathbf g_{\out},N}_{\out,\alpha}(z)$. The other term, $W^{\mathbf g_{\inn,N}}_{\inn,\alpha}(z)$, decays exponentially in the distance from $z$ to $I_{\inn}$. In $\{-\alpha\leq\theta\leq 0\}$, this distance is bounded below by a positive multiple of $\rho$, so this term decays exponentially in $\rho$ and may therefore be absorbed into the remainder.

The case where $-\alpha\leq\theta\leq -\alpha/2$ follows by symmetry, and the Dirichlet case is similar.
\end{proof}

\subsection{Proof of Theorem \ref{thm:mainthmssector}}
Now we consider arbitrary linear combinations of the symmetric and anti-symmetric solutions.

\begin{proposition}\label{prop:planewaveapprox} Consider, for $\mathbf{F}=\begin{pmatrix} F_{N}\\F_{D} \end{pmatrix}\in \mathbb{C}^2$, a linear combination $\Phi(z) = F_N \Phi_{\alpha,N}(z)+ F_D \Phi_{\alpha,D}(z)$. Let
\[
\mathtt G_{\out}(\alpha):=\frac 12\begin{pmatrix} \er^{-\ir \chi_{N}} & \er^{-\ir \chi_{D}}\\ \er^{\ir \chi_{N}} & \er^{\ir \chi_{D}}\end{pmatrix},
\qquad
\mathtt G_{\inn}(\alpha):=\frac 12\begin{pmatrix} \er^{\ir \chi_{N}} & -\er^{\ir \chi_{D}}\\ \er^{-\ir \chi_{N}} &- \er^{-\ir \chi_{D}}\end{pmatrix},
\]
\begin{equation}\label{eq:fginout}
\mathbf h_{\out}=\mathtt G_{\out}(\alpha)\mathbf{F},\qquad \mathbf h_{\inn}=\mathtt G_{\inn}(\alpha)\mathbf{F}.
\end{equation}
Then we have
\[
\Phi(z)=W^{\mathbf h_{\out,\alpha}}_{\out,\alpha}(z)+W^{\mathbf h_{\inn,\alpha}}_{\inn,\alpha}(z)+ R_{\alpha,\mathbf{F}}(z),
\]
where $R=R_{\alpha,\mathbf{F}}$ satisfies \eqref{eq:universalremainder} with $r=\mu$ and $C=C_\alpha \|\mathbf{F}\|$ with some constant $C_\alpha$ depending only on $\alpha$. 
\end{proposition}

\begin{proof} Since $W_{\out,\alpha}^{\mathbf h}$ and $W_{\inn,\alpha}^{\mathbf h}$ are linear in $\mathbf{h}$, the proof follows instantaneously from Lemma \ref{lem:propsofsymantisym} and linear algebra. Note that in the remainder estimate we obtain the weaker, Neumann, exponent for an arbitrary linear combination.
\end{proof}

We proceed to the proof of Theorem \ref{thm:mainthmssector}. At least in the case $\alpha\notin\Eangles$, we would like to start with an arbitrary $\mathbf h_{\inn}\in\mathbb C^2$ and apply Proposition \ref{prop:planewaveapprox} with
\[
\mathbf{F}=(\mathtt G_{\inn}(\alpha))^{-1}\mathbf h_{\inn},\qquad \mathbf h_{\out}=\mathtt G_{\out}(\alpha)(\mathtt G_{\inn}(\alpha))^{-1}\mathbf h_{\inn}.
\]
Indeed, this gives us everything we want, including the remainder estimate, as long as $\mathtt G_{\inn}(\alpha)$ is invertible. By a direct computation, we find
\[\det\mathtt G_{\out}(\alpha)=\det\mathtt G_{\inn}(\alpha)=\frac{\ir}{2}\sin\left(\frac{\pi^2}{2\alpha}\right).\]
Therefore $\mathtt G_{\out}(\alpha)$ and $\mathtt G_{\inn}(\alpha)$ are invertible if and only if $\alpha$ is not exceptional. Observing that, again by a direct calculation,
\[
\mathtt{A}(\alpha)=\mathtt{G}_{\out}(\alpha)(\mathtt{G}_{\inn}(\alpha))^{-1},
\]
leads to \eqref{eq:hinoutA}.

Now suppose $\alpha\in\Eangles$. In this case, given $\mathbf h_{\inn}$ and $\mathbf h_{\out}$ satisfying \eqref{eq:exceptionalcondition}, we want to find $\mathbf{F}$ such that we have \eqref{eq:fginout}. We will use 

\begin{lemma}\label{lem:rangeGinout} Let $\alpha=\dfrac{\pi}{2k}\in\Eangles$. Consider $\mathtt G_{\out}(\alpha)$ and $\mathtt G_{\inn}(\alpha)$ as linear mappings $\mathbb{C}^2\to\mathbb{C}^2$. Then
\[
\operatorname{Range} \mathtt G_{\out}(\alpha) = \operatorname{Span}_\mathbb{C}\left\{\Cvect(\alpha)\right\},\qquad
\operatorname{Range} \mathtt G_{\inn}(\alpha) = \operatorname{Span}_\mathbb{C}\left\{\Cvect^\perp(\alpha)\right\}.
\]
and
\[
\operatorname{Ker} \mathtt G_{\out}(\alpha) = \operatorname{Span}_\mathbb{C}\left\{\mathbf{K}(\alpha)\right\},\qquad
\operatorname{Ker} \mathtt G_{\inn}(\alpha) = \operatorname{Span}_\mathbb{C}\left\{\mathbf{K}^\perp(\alpha)\right\},
\]
where 
\[
\mathbf{K}(\alpha):=\dfrac{1}{\sqrt{2}}\begin{pmatrix}1\\\mathcal{O}(\alpha)\end{pmatrix}=\dfrac{1}{\sqrt{2}}\begin{pmatrix}1\\(-1)^k\end{pmatrix},\qquad 
\mathbf{K}^\perp(\alpha):=\dfrac{1}{\sqrt{2}}\begin{pmatrix}1\\-\mathcal{O}(\alpha)\end{pmatrix}=\dfrac{1}{\sqrt{2}}\begin{pmatrix}1\\(-1)^{k+1}\end{pmatrix}.
\]
\end{lemma}

\begin{proof}[Proof of Lemma \ref{lem:rangeGinout}]  We have in this case
\[ 
\mathtt G_{\out}(\alpha)=\er^{-\ir \pi/4}\er^{\ir \pi k/2}\begin{pmatrix} 1 & (-1)^k\\ \ir(-1)^k & \ir\end{pmatrix},\qquad  \mathtt G_{\inn}(\alpha)=\er^{-\ir \pi/4}\er^{\ir \pi k/2}\begin{pmatrix}  \ir(-1)^k & -\ir\\ 1 & -(-1)^k\end{pmatrix},
\]
and the statement follows by a direct computation and comparison with  \eqref{eq:orthogspecial2}, \eqref{eq:orthogspecial1} and \eqref{eq:Cevenoddperp}.
\end{proof}

By Lemma \ref{lem:rangeGinout}, the conditions \eqref{eq:exceptionalcondition} (or their equivalent form, see Remark \ref{rem:Cparallel}), are the necessary conditions for the solvability of \eqref{eq:fginout}. We can now assume $\mathbf h_{\inn}=h_{\inn} \Cvect^\perp(\alpha)$, $\mathbf h_{\out}=h_{\out} \Cvect(\alpha)$ with some constants 
$h_{\inn}, h_\out\in\mathbb{C}$. Taking now
\[
\mathbf{F}=\frac{h_\inn}{\mathtt G_{\inn}(\alpha)\mathbf{K}(\alpha)\cdot \Cvect^\perp(\alpha)}\mathbf{K}(\alpha)+
\frac{h_\out}{\mathtt G_{\out}(\alpha)\mathbf{K}^\perp(\alpha)\cdot \Cvect(\alpha)} \mathbf{K}^\perp(\alpha)
\]
gives the desired result. Indeed, applying Lemma \ref{lem:rangeGinout} again, we obtain
\[
G_{\inn}(\alpha)\mathbf{F}\cdot \Cvect^\perp(\alpha)=h_\inn= \mathbf h_{\inn}\cdot\Cvect^\perp(\alpha),\qquad
G_{\out}(\alpha)\mathbf{F}\cdot \Cvect(\alpha)=h_\out= \mathbf h_{\out}\cdot\Cvect(\alpha),
\]
and therefore \eqref{eq:fginout}.

We have now found a vector $\mathbf{F}$ with the desired properties, and moreover $\|\mathbf{F}\|\leq C_\alpha(\|\mathbf h_{\inn}\|+\|\mathbf h_{\out}\|)$ for some constant $C_\alpha$ depending only on $\alpha$. Applying Proposition \ref{prop:planewaveapprox} with this vector $\mathbf{F}$ completes the proof of Theorem \ref{thm:mainthmssector}.

\begin{remark} Conditions \eqref{eq:hinoutA} or \eqref{eq:exceptionalcondition} are not just sufficient but also \emph{necessary} for the existence of a solution \eqref{eq:definitionofphi} of \eqref{eq:Robin1}, see \cite[Remark 2.2]{sloshing}.
\end{remark}

\begin{remark} The effects observed in Theorem \ref{thm:mainthmssector} are similar to \emph{scattering}. In fact, if we define
\begin{equation}\label{eq:scbeach}
\mathtt{Sc}(\alpha):=\begin{pmatrix} \ir\cos\frac{\pi^2}{2\alpha}&\sin\frac{\pi^2}{2\alpha}\\\sin\frac{\pi^2}{2\alpha}&\ir\cos\frac{\pi^2}{2\alpha}\end{pmatrix},
\end{equation}
then $\mathtt{Sc}(\alpha)$ can be thought of as the \emph{scattering matrix} for the Peters solutions: in the sense of Theorem \ref{thm:mainthmssector}, we have
\[
\begin{pmatrix} h_{\inn,1}\\h_{\out,2}\end{pmatrix}=\mathtt{Sc}(\alpha) \begin{pmatrix} h_{\inn,2}\\h_{\out,1}\end{pmatrix}.
\]
Then at exceptional angles the scattering is fully reflective and at special angles there is no reflection at all. Therefore, we will from now on call the solutions \eqref{eq:definitionofphi} the \emph{scattering Peters solutions}.
\end{remark}
\clearpage\section{Construction of quasimodes}\label{sec:quasieigenvalues}

\subsection{General approach}\label{subsec:genapproach}  The main purpose of this section is to prove the following theorem, which establishes that the quasi-eigenvalues introduced in Definitions
\ref{def:quasi} and \ref{def:quasiexc} are indeed approximate eigenvalues of the Steklov problem \eqref{eq:steklovproblem}.

\begin{theorem}\label{thm:approx} Let $\mathcal{P}$ be a curvilinear polygon, and let $\{\sigma_m\}$ be its sequence of quasi-eigenvalues. Then there exists a non-decreasing sequence $\{i_m\}$ and a sequence of positive real numbers $\{\varepsilon_m\}$ approaching zero such that
\[
|\sigma_m-\lambda_{i_m}|\le\varepsilon_m\qquad\text{for all }m.
\]
\end{theorem}
We will prove Theorem \ref{thm:approx} by constructing an appropriate sequence of \emph{quasimodes} which we first define in a very general setting. 

Let $\mathcal{P}$ be a curvilinear polygon with all angles in $\Pi$. Suppose that $\partial\mathcal{P}$ is decomposed in the union $\partial_S\mathcal{P}\sqcup\partial_D\mathcal{P}\sqcup\partial_N\mathcal{P}$, where each of  $\partial_S\mathcal{P}$, $\partial_D\mathcal{P}$ and $\partial_N\mathcal{P}$ are unions of the boundary arcs (and thus meet only at the vertices), with $\partial_S\mathcal{P}$ being non-empty. 

Consider in $\mathcal{P}$ a mixed Steklov-Dirichlet-Neumann eigenvalue problem
\begin{equation}\label{eq:SDNproblem}
\Delta u =0\quad \text{in }\mathcal{P},\qquad \frac{\partial u}{\partial n}=\lambda u\quad \text{on }\partial_S \mathcal{P},\qquad u=0\quad \text{on }\partial_D \mathcal{P},
\qquad \frac{\partial u}{\partial n}=0\quad \text{on }\partial_N\mathcal{P},
\end{equation}
and denote its eigenvalues and the corresponding eigenvectors by $\lambda_m$, $u_m$, where $m=1,2,\dots$, and 
\[
\|u_m\|_{L^2(\partial_S\mathcal{P})}=1.
\]

\begin{remark} Zigzag problems \eqref{eq:gensloshing} are just special cases of \eqref{eq:SDNproblem}, and we can define the partial Dirichlet-to-Neumann map
$\mathcal{D}_{\mathcal{P},\partial_S \mathcal{P}}$ for \eqref{eq:SDNproblem} analogously to \eqref{eq:gensloshingDN}.
\end{remark}

\begin{definition}\label{def:quasimodes} A sequence of functions $\{v_m\}\subset H^2(\mathcal{P})$ with $\|v_m\|_{L^2(\partial_S \mathcal{P})}=1$,  is called a sequence of \emph{quasimodes} corresponding to a monotonically converging to $+\infty$ sequence of quasi-eigenvalues $\{\sigma_m\}$ for the problem \eqref{eq:SDNproblem} if the three non-negative number sequences $\varepsilon^{(j)}_m$, $j=1,2,3$, defined by 
\[
\begin{split}
\varepsilon^{(1)}_m&:=\left\|\frac{\partial v_m}{\partial n}-\sigma_m v_m\right\|_{L^2(\partial_S \mathcal{P})},\\ 
\varepsilon^{(2)}_m&:=\left\|\frac{\partial v_m}{\partial n}\right\|_{L^2(\partial_N \mathcal{P})}+\|v_m\|_{H^1(\partial_D \mathcal{P})},\\ 
\varepsilon^{(3)}_m&:=(\sigma_m+1)\|\Delta v_m\|_{L^2(\mathcal{P})},
\end{split}
\]
all converge to zero as $m\to\infty$. Moreover, for $\delta_m$ a given sequence converging to zero, we say $\{v_m\}$ are quasimodes of \emph{order} $\delta_m$ if there exists a constant $C>0$ independent of $m$ such that
\[\varepsilon_m^{(1)}+\varepsilon_m^{(2)}+\varepsilon_m^{(3)}\leq C\delta_m.\]
\end{definition}

The point of this definition is the following approximation result:
\begin{theorem}\label{thm:approx1} Suppose that there exist sequences of quasi-eigenvalues $\{\sigma_m\}$ and quasimodes $\{v_m\}$, of order $\delta_m$, for the problem \eqref{eq:SDNproblem}. Then there exist a sequence $\{i_m\}$ of non-negative integers and a sequence of functions $\{\tilde u_{m}\}$ such that
\[
|\sigma_m-\lambda_{i_m}|\le C\delta_m\qquad\text{and}\qquad \|v_m-\tilde u_m\|_{L^2(\partial_S\mathcal{P})}\le C\sqrt{\delta_m},
\]
with $C>0$ a constant independent of $m$, and with each $\tilde u_m$ being a linear combination of eigenfunctions of \eqref{eq:SDNproblem} with eigenvalues in the interval $[\sigma_m - C\sqrt{\delta_m},\sigma_m + C\sqrt{\delta_m}]$.
\end{theorem}

\begin{remark} Later on in sections \ref{sec:completeness} and \ref{sec:layer} we will prove, for Steklov curvilinear polygons and zigzag domains, that $i_m=m$ under an appropriate choice of enumeration. 
\end{remark}

Assuming that quasimodes have been constructed, Theorem \ref{thm:approx1} almost immediately implies Theorem \ref{thm:approx}. The only detail that remains is to show that the sequence $\{i_m\}$ may be chosen non-decreasing. This can be done via the following manoeuvre: let
\[
\varepsilon_m:=\sup_{k\geq m}C\delta_k,
\]
with $C$ as in Theorem \ref{thm:approx1}. Observe that $\varepsilon_m$ is now a decreasing sequence converging to zero. Now, for each $m$, the interval $(\sigma_m-\varepsilon_m,\sigma_m+\varepsilon_m)$ must contain at least one $\lambda_i$, and we redefine $i_m$ by letting $i_m$ be the minimal index among such $\lambda_i$. We claim that $\{i_m\}$, defined in this way, is a non-decreasing sequence. Indeed since $\{\sigma_m\}$ is increasing then $\sigma_{m-1}-\varepsilon_{m-1}\le\sigma_m-\varepsilon_m$. Therefore, the interval $(\sigma_m-\varepsilon_m,\sigma_m+\varepsilon_m)$ cannot contain any $\lambda_i$ which both fails to be an element of $(\sigma_{m-1}-\varepsilon_{m-1},\sigma_{m-1}+\varepsilon_{m-1})$ and which is smaller than all $\lambda_i$ in the latter interval. Thus $i_m\ge i_{m-1}$, so $\{i_m\}$ is non-decreasing.

The proof of Theorem \ref{thm:approx} has thus been reduced to the proof of Theorem \ref{thm:approx1} and the construction of quasimodes for $\mathcal{P}$ satisfying Definition \ref{def:quasimodes}.

\subsection{Boundary quasimodes. Justification of quasi-eigenvalue definitions}\label{subsec:bqm}  Before proceeding to the proof of Theorem \ref{thm:approx}, we give a semi-informal justification of the quasi-eigenvalue Definitions \ref{def:quasi} and \ref{def:quasiexc}, in the case of a purely Steklov polygon $P=P(\balpha,\bell)$ with $\partial_D P=\partial_N P=\varnothing$.

We introduce on $\partial P$ near each vertex $V_j$ the local coordinate $s_j$ such that $s_j$ is zero at $V_j$, negative on the side $I_{j}$, and positive on the side $I_{j+1}$. Note that on each side $I_j$ joining $V_{j-1}$ and $V_j$ we have effectively two coordinates: the coordinate $s_j$ running from $-\ell_j$ to $0$, and the coordinate $s_{j-1}$ running from $0$ to $\ell_j$, related as
\begin{equation}\label{eq:sjjm1}
s_j=s_{j-1}-\ell_j.
\end{equation}

\begin{figure}[htb]
\begin{center}
\includegraphics{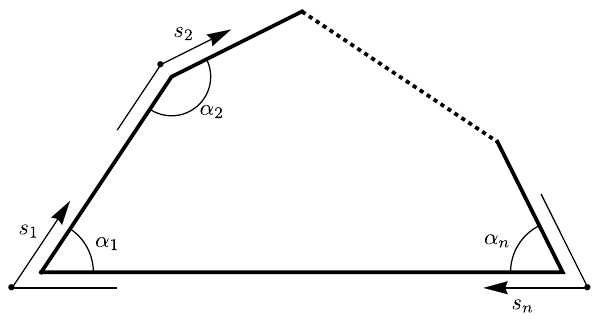}
\end{center}
\caption{A straight polygon with local coordinates\label{fig:polygon_coords}}
\end{figure}

Let $\mathcal{V}_j$ be the orientation-preserving isometry of the complex plane which maps the sector  $V_{j-1}V_jV_{j+1}$ into the sector $\Sct{\alpha_j}$ with the vertex at the origin. We will seek the quasimode waves $v_{\sigma}(z)$ of our problem  \eqref{eq:steklovproblem}.

We first consider the situation when all angles are non-exceptional. Near each vertex $V_j$, $v_{\sigma}(z)$ will be closely approximated by a specific scattering Peters solution constructed in Section \ref{sec:sector}. Specifically we will have, for $z$ in a neighbourhood of $V_j$,
\begin{equation}\label{eq:Phi}
v_{\sigma}(z)= \Phi_{\alpha_j}^{(\mathbf c_{j,\inn},\mathbf c_{j,\out})}(\sigma\mathcal{V}_j z)+o(1)\text{ as }\sigma\to\infty,
\end{equation}
where suitable values of the quasi-eigenvalues $\sigma$ and the coefficients $\mathbf{c}_{j,\inn}, \mathbf{c}_{j,\out}\in\mathbb{C}^2$ are to be determined. By  Theorem \ref{thm:mainthmssector}(a), the vectors $\mathbf{c}_{j,\inn}, \mathbf{c}_{j,\out}$ should be related by 
\begin{equation}\label{eq:coutAcin}
\mathbf c_{j,\out}:=\mathtt{A}(\alpha_j)\mathbf c_{j,\inn}
\end{equation}
to ensure the existence of the scattering Peters solutions. Note that these rescaled scattering Peters solutions satisfy Steklov boundary conditions on the sides $I_j$ and $I_{j+1}$ with parameter $\sigma$. 

As a consequence of \eqref{eq:Wboundary} and Theorem \ref{thm:mainthmssector},
\[
\left.v_{\sigma}\right|_{\partial\Omega}=\Psi+o(1)\quad\text{ as }\sigma\to\infty,
\]
where
\begin{equation}\label{eq:Psijsj}
\left.\Psi\right|_{I_j}(s_j)=: \Psi_{j}(s_j)= \mathbf{c}_{j,\inn}\cdot\mathbf{e}(\sigma s_j),
\end{equation}
or, alternatively, using the coordinate $s_{j-1}$,
\begin{equation}\label{eq:Psijsj1}
\Psi_{j}(s_{j-1})=\mathbf{c}_{j-1,\out} \cdot\mathbf{e}(\sigma s_{j-1}).
\end{equation}
If we want \eqref{eq:Psijsj} and \eqref{eq:Psijsj1} to match, we must have, with account of \eqref{eq:sjjm1},
\begin{equation}\label{eq:cpm}
\mathbf{c}_{j,\inn}=\mathtt{B}(\ell_j,\sigma) \mathbf{c}_{j-1,\out},
\end{equation}

We call \eqref{eq:Psijsj}, or equivalently the vector $\mathbf{c}_{j,\inn}$, the \emph{boundary quasi-wave incoming into} $V_j$ (from $V_{j-1}$) and  \eqref{eq:Psijsj1}, or equivalently the vector $\mathbf{c}_{j-1,\out}$, the \emph{boundary quasi-wave outgoing from} $V_{j-1}$ (towards $V_j$). In order for our scattering Peters solutions on $I_j$ to match, these must be related by \eqref{eq:cpm}. 

This formulation allows us to think of our problem as a transfer problem. Consider a boundary quasi-wave $\mathbf{b}:=\mathbf{c}_{n,\out}$ outgoing from the vertex $V_n$  towards $V_1$. It arrives at the vertex $V_1$ as an incoming quasi-wave $\mathbf{c}_{1,\inn}=\mathtt{B}(\ell_1,\sigma)\mathbf{b}$ and, according to Section \ref{sec:sector} and \eqref{eq:Phi}, leaves $V_1$ towards $V_2$ as an outgoing quasi-wave 
\[
\mathbf{c}_{1,\out}=\mathtt{A}(\alpha_1)\mathbf{c}_{1,\inn}=\mathtt{A}(\alpha_1)\mathtt{B}(\ell_1,\sigma)\mathbf{b}.
\] 
It then arrives at $V_2$ as an incoming quasi-wave 
\[
\mathbf{c}_{2,\inn}=\mathtt{B}(\ell_2,\sigma)\mathtt{A}(\alpha_1)\mathtt{B}(\ell_1,\sigma)\mathbf{b},
\] 
and leaves $V_2$ towards $V_3$ as an outgoing quasi-wave 
\[
\mathbf{c}_{2,\out}=\mathtt{A}(\alpha_2)\mathtt{B}(\ell_2,\sigma)\mathtt{A}(\alpha_1)\mathtt{B}(\ell_1,\sigma)\mathbf{b}.
\] 
Continuing the process, we conclude that it arrives at $V_n$ as an incoming quasi-wave \[\mathbf{c}_{n,\inn}=\mathtt{B}(\ell_n,\sigma)\prod\limits_{j=n-1}^1\mathtt{A}(\alpha_j)\mathtt{B}(\ell_j,\sigma)\mathbf{b}\] and leaves $V_n$ towards $V_1$ as an outgoing quasi-wave \[\mathbf{c}_{n,\out}=\prod\limits_{j=n}^1\mathtt{A}(\alpha_j)\mathtt{B}(\ell_j,\sigma)\mathbf{b}.\] This must match the original outgoing quasi-wave $\mathbf b$, which imposes a quantisation condition on $\sigma$:
\[
\mathtt{T}(\balpha,\bell,\sigma)\mathbf{b}=\mathbf{b},
\]
thus justifying Definition \ref{def:quasi}.

Let us now deal with the situation when there are $K$ exceptional angles $\alpha_{E_1}$,\dots, $\alpha_{E_K}=\alpha_n$. We will seek the quasimodes again in the form \eqref{eq:Phi}. 
At  an exceptional vertex $V_{E_\kappa}$, $\kappa=1,\dots,K$ the incoming and outgoing boundary quasi-waves
must satisfy, according  to Theorem \ref{thm:mainthmssector}, the
conditions \eqref{eq:exceptionalcondition}, see also Remark \ref{rem:Cparallel}, which take the form
\begin{equation}\label{eq:exceptionalconditionmark2}
\mathbf{c}_{E_\kappa,\inn} \perp \operatorname{Span}_\mathbb{C}\left\{\Cvect\left(\alpha_{E_\kappa}\right)\right\},\qquad
\mathbf{c}_{E_\kappa,\out} \in \operatorname{Span}_\mathbb{C}\left\{\Cvect\left(\alpha_{E_\kappa}\right)\right\}.
\end{equation}
Noting that the transfer along each exceptional boundary component joining $V_{E_{\kappa-1}}$ and $V_{E_{\kappa}}$ leads, with account of re-labelling as in Section \ref{sec:statements}, to
\[
\mathbf c_{E_\kappa,\inn}=\mathtt{U}({\balpha'}^{(\kappa)},\bell^{(\kappa)},\sigma)\mathbf c_{E_{\kappa-1},\out},
\]
we arrive at \eqref{excepeq}, thus justifying Definition \ref{def:quasiexc}.

Although this justification was done in case of an exact polygon $P$, it also gives the correct heuristics for a curvilinear polygon. The construction of quasimodes is more difficult, but as we see in the next few sections, it can be done. 

In order to assist in this construction, we define some new notation. Observe that for each $m\in\mathbb N$, the quantisation condition gives a quasi-eigenvalue $\sigma_m$, and also a corresponding collection of vectors $\mathbf{c}_{j,\inn}$ and $\mathbf{c}_{j,\out}$ (which also depend on $m$) which satisfy the transfer conditions \eqref{eq:cpm} along each side and either \eqref{eq:coutAcin} or \eqref{eq:exceptionalconditionmark2} at each corner depending on whether or not the angle is exceptional. These vectors $\mathbf{c}_{j,\inn}$ and $\mathbf{c}_{j,\out}$ are the solutions of a system of linear equations, and if the multiplicity of $\sigma_m$ is one then they are determined up to an overall multiplicative constant. These define a boundary quasi-wave.
\begin{definition}\label{def:bqw} 
For each $m\in\mathbb N$, let $\Psi^{(m)}$ be a boundary quasi-wave, defined by \eqref{eq:Psijsj}, associated to the quasi-eigenvalue $\sigma_m$, normalised so that $\|\Psi^{(m)}\|_{\partial P}=1$. The restrictions to $I_j$ of each $\Psi^{(m)}$ are denoted by $\Psi^{(m)}_j$.
\end{definition}
Observe this definition may also be applied if the multiplicity of $\sigma_m$ is greater than one, as then the quasi-waves form a linear space  of dimension greater than one. In this situation we simply pick $\Psi^{(m)}$ to be any boundary quasi-wave which is in that space, is normalised, and is orthogonal to all previous choices of boundary quasi-waves for the same quasi-eigenvalue.

\begin{remark}\label{rem:evfromConj} We note that all the vectors  $\mathbf{c}_{j,\inn}$ and $\mathbf{c}_{j,\out}$ may be chosen from $\Conj$. In the non-exceptional case, this is true for $\mathbf{b}:=\mathbf{c}_{n,\out}$ by Lemma \ref{lem:Tproperties}(c), and for the rest of the vectors by the fact that matrices $\mathtt{A}(\alpha)$ and $\mathtt{B}(\ell)$ preserve $\Conj$. The exceptional case
is similar. 
\end{remark}

\begin{remark} Using a similar scheme, we may also define quasi-frequencies $\sigma_m$ and boundary quasi-waves $\Psi^{(m)}$ for the mixed problem \eqref{eq:SDNproblem}. The quasi-waves $\Psi^{(m)}$ are supported on $\partial_S\mathcal{P}\cup\partial_N\mathcal{P}$ and vanish on  $\partial_D\mathcal{P}$. Suppose for the moment that $\partial_S\mathcal{P}$ is a single connected component, without loss of generality beginning at vertex $V_1$ and ending at vertex $V_k$. We define $\Psi^{(m)}$ by specifying collections $\mathbf{c}_{j,\inn}$ and $\mathbf{c}_{j,\out}$ as before and then using \eqref{eq:Psijsj}. Along each side, we have the transfer conditions \eqref{eq:cpm}, and at each non-endpoint vertex we have either \eqref{eq:coutAcin} or \eqref{eq:exceptionalconditionmark2} depending on whether or not the angle is exceptional. 

However, at the endpoint vertices $V_1$ and $V_k$, something different happens: our $\Psi^{(m)}$ must be chosen to match the appropriate Peters \emph{sloping beach} solution, either Dirichlet or Neumann. Note that these are sloping beach solutions in a sector $\Sct{\alpha}$ rather than $\Sct{\alpha/2}$, so the terminology in Lemma \ref{lem:propsofsymantisym} needs to be adjusted. Specifically, we consider the vectors obtained by taking \eqref{eq:gdef} and replacing $\alpha$ with $2\alpha$ throughout. These are
\begin{equation}\label{eq:rem48}
\mathbf{g}_{\out,N,2\alpha}=\begin{pmatrix}\er^{-\ir(\frac{\pi}{4}-\frac{\pi^2}{8\alpha})}\\ \er^{\ir(\frac{\pi}{4}-\frac{\pi^2}{8\alpha})}\end{pmatrix},\qquad 
\mathbf{g}_{\out,D,2\alpha}=\begin{pmatrix}\er^{-\ir(\frac{\pi}{4}+\frac{\pi^2}{8\alpha})}\\ \er^{\ir(\frac{\pi}{4}+\frac{\pi^2}{8\alpha})}\end{pmatrix},
\end{equation}
with, as in \eqref{eq:gdef},
\[\mathbf{g}_{\inn,N,2\alpha}=\overline{\mathbf{g}_{\out,N,2\alpha}},\qquad \mathbf{g}_{\inn,D,2\alpha}=-\overline{\mathbf{g}_{\out,N,2\alpha}}.\]
Now suppose that $V_1\in\partial_{\aleph}\mathcal{P}$ and $V_k\in\partial_{\beth}\mathcal{P}$, with $\aleph,\beth\in\{N,D\}$. We require
\begin{equation}\label{eq:imposeendpointconditions}\mathbf{c}_{1,\out}\in\Span_{\mathbb R}(\mathbf{g}_{\out,\aleph,2\alpha}),\qquad \mathbf{c}_{k,\inn}\in\Span_{\mathbb R}(\mathbf{g}_{\inn,\beth,2\alpha}).
\end{equation}
Imposing the conditions \eqref{eq:imposeendpointconditions}, in addition to all the conditions previously discussed for the sides and the non-endpoint vertices, leads to a quantization condition for $\sigma$, which yields a sequence of quasi-eigenvalues $\sigma_m$, each with an accompanying collection of $\mathbf{c}_{j,\inn}$ and $\mathbf{c}_{j,\out}$. These may then be used to define $\Psi^{(m)}$ as before. In the event that $\partial_S\mathcal{P}$ consists of multiple connected components, each component is treated separately and independently.
\end{remark}

\subsection{Proof of Theorem \ref{thm:approx1}}\label{subsec:proofapprox1}

\begin{lemma}\label{lem:improvement} Let $\mathcal{P}$ be a curvilinear polygon with all  angles in $\Pi$ and the same conditions as before: namely, that $\partial\mathcal{P}=\partial_S\mathcal{P}\sqcup\partial_D\mathcal{P}\sqcup\partial_N\mathcal{P}$, where each of  $\partial_S\mathcal{P}$, $\partial_D\mathcal{P}$ and $\partial_N\mathcal{P}$ are the unions of the boundary arcs (and thus meet only at the vertices), with $\partial_S\mathcal{P}$ being non-empty.

Then the under-determined problem
\begin{equation}\label{eq:adjunctproblemgeneral}
\begin{dcases} \Delta w=f_1(z)&\quad\textrm{ on }\mathcal{P},\\ 
\frac{\partial w}{\partial n}=f_2(z)&\quad\textrm{ on }\partial_N\mathcal{P},\\ 
w=f_3(z)&\quad\text{ on }\partial_D\mathcal{P},
\end{dcases}
\end{equation}
where $f_1\in L^2(\mathcal{P})$, $f_2\in L^2(\partial_N\mathcal{P})$, and $f_3\in H^1(\partial_D\mathcal{P})$, has a solution $w(z)$ which satisfies the following estimates on $\partial_S\mathcal{P}$,
\begin{align}
\label{eq:estimate1forlemma}
\|w\|_{H^1(\partial_S\mathcal{P})}&\leq C\|f_1\|_{L^2(\mathcal{P})},\\
\label{eq:estimate2forlemma}
\left\|\frac{\partial w}{\partial n}\right\|_{L^2(\partial_S\mathcal{P})}&\leq C\left(\|f_1\|_{L^2(\mathcal{P})}+\|f_2\|_{L^2(\partial_N\mathcal{P})}+\|f_3\|_{H^1(\partial_D\mathcal{P})}\right),
\end{align}
where $C$ are constants depending only on $\mathcal{P}$, $\partial_S\mathcal{P}$, $\partial_N\mathcal{P}$, and $\partial_D\mathcal{P}$. In fact the constant in \eqref{eq:estimate1forlemma} depends only on the diameter of $\mathcal{P}$.
\end{lemma}

%

\begin{proof} Throughout we let $C$ be various constants depending only on $\mathcal{P}$, $\partial_S\mathcal{P}$, $\partial_N\mathcal{P}$, and $\partial_D\mathcal{P}$. The proof proceeds by first dealing with $f_1(z)$, then with $f_2(z)$ and $f_3(z)$.

Let $B$ be a large disk compactly containing $\mathcal P$, and let $\tilde f_1(z)$ be the extension by zero of $f_1(z)$ to a function on $B$. Then certainly $\|\tilde f_1\|_{L^2(B)}=\|f_1\|_{L^2(\mathcal P)}$. By the usual elliptic estimate for the solution of the Poisson problem on a disk with Dirichlet boundary conditions, there exists a function $w_1(z)\in H^2(B)$  vanishing on the boundary $\partial B$, with $\Delta w_1(z)=\tilde f_1(z)$ and
\[
\|w_1\|_{H^2(B)}\leq C \|f_1\|_{L^2(\mathcal P)}.
\]
Now let $G$ be any smooth non-self-intersecting arc in the interior of $B$. Let $\mathbf{n}_G$ be a unit normal vector field along $G$. Then let $\mathbf{V}$ be a vector field on $B$ whose restriction to $G$ is $\mathbf{n}_G$ and which is bounded in the sense that $\mathbf{V}:H^2(B)\to H^1(B)$ is bounded. By the trace theorem,
\[
\|w_1\|_{H^1(G)}+\left\|\frac{\partial w_1}{\partial n}\right\|_{L^2(G)}\leq C\left (\|w_1\|_{H^{3/2}(B)}+\left\|\mathbf{V}w_1\right\|_{H^{1/2}(B)}\right )\leq C\left (\|w_1\|_{H^{2}(B)}+\left\|\mathbf{V}w_1\right\|_{H^{1}(B)}\right ).
\]
Therefore, by the definition of Sobolev norms and the elliptic estimate above,
\begin{equation}\label{eq:interiorarc}
\|w_1\|_{H^1(G)}+\left\|\frac{\partial w_1}{\partial n}\right\|_{L^2(G)}\leq C\|f_1\|_{L^2(\mathcal{P})}.
\end{equation}
We note that \eqref{eq:interiorarc} still holds if $G$ is a finite disjoint union of smooth non-self-intersecting arcs $G_1,\dots, G_i$, and we interpret $H^1(G)$ as the direct sum $H^1(G_1)\oplus \dots\oplus H^1(G_i)$. 
This applies, in particular, with $G=\partial_S\mathcal{P}$, $\partial_N\mathcal{P}$, and $\partial_D\mathcal{P}$, so by combining all three estimates we certainly have
\begin{equation}\label{eq:u1estimate}
\|w_1\|_{H^1(\partial\mathcal{P})}+\left\|\frac{\partial w_1}{\partial n}\right\|_{L^2(\partial\mathcal{P})}\leq C\|f_1\|_{L^2(\mathcal{P})}.
\end{equation}

Now we would like to find a function $w_2$ on $\mathcal{P}$ satisfying
\begin{equation}\label{eq:adjunctproblemgeneralu2}
\begin{dcases} \Delta w_2=0&\quad\textrm{ on }\mathcal{P},\\ 
\frac{\partial w_2}{\partial n}=f_4(z):=f_2(z)-\frac{\partial w_1}{\partial n}(z)&\quad\textrm{ on }\partial_N\mathcal{P},\\ 
w_2=f_5(z):=f_3(z)-w_1(z)&\quad\text{ on }\partial_D\mathcal{P},\\
w_2=0 &\quad\text{ on }\partial_S\mathcal{P}.
\end{dcases}
\end{equation}
Assuming such a function exists, $w=w_1+w_2$ solves \eqref{eq:adjunctproblemgeneral} and $w=w_1$ on $S$, and we will see that this $w$ satisfies \eqref{eq:estimate1forlemma} and \eqref{eq:estimate2forlemma}. 

To construct $w_2$, we can use the theory of boundary value problems on Lipschitz domains developed by Brown \cite{brown}. We have Dirichlet data on $\partial_D\mathcal{P}\cup \partial_S\mathcal{P}$ and Neumann datum on $\partial_N\mathcal{P}$, so our assumption on vertex angles tells us that the angles between components with Dirichlet data and components with Neumann data are less than $\pi$. This is precisely what is needed for the estimates in \cite{brown} to hold. Specifically, since $\partial_S\mathcal{P}\cup\partial_D\mathcal{P}$, the Dirichlet portion of the boundary, is nonempty, the problem \eqref{eq:adjunctproblemgeneralu2}
has a unique solution \cite[Theorem 2.1]{brown}. Moreover, by the same theorem and the discussion after  \cite[formula (2.12)]{brown}, we have the estimate
\[
\|\nabla w_2\|^2_{L^2(\partial\mathcal{P})}\leq C\left(\|f_4\|^2_{L^2(\partial_N\mathcal{P})}+\|\nabla_{\tan}f_5\|^2_{L^2(\partial_D\mathcal{P}))}+\|f_5\|^2_{L^2(\partial_D\mathcal{P})}\right),
\]
where $\nabla_{\tan}$ denotes a tangential derivative along $\partial_N\mathcal{P}$. Note that we have omitted the portion of this estimate involving $\partial_S\mathcal{P}$ because our Dirichlet datum on $\partial_S\mathcal{P}$ is trivial. Using the definition of the $H^1$ norm, and restricting to $\partial_S\mathcal{P}\subset\partial\mathcal P$ on the left hand side, we obtain
\begin{equation}\label{eq:u2estimate}
\left\|\frac{\partial w_2}{\partial n}\right\|^2_{L^2(\partial_S\mathcal{P})}\leq C\left(\|f_4\|^2_{L^2(\partial_N\mathcal{P})}+\|f_5\|^2_{H^1(\partial_D\mathcal{P})}\right).
\end{equation}
This translates to
\begin{equation}\label{eq:u2estimate2}
\left\|\frac{\partial w_2}{\partial n}\right\|^2_{L^2(\partial_S\mathcal{P})}\leq C\left(\left\|f_2-\frac{\partial w_1}{\partial n}\right\|^2_{L^2(\partial_N\mathcal{P})}+\|f_3-w_1\|^2_{H^1(\partial_D\mathcal{P})}\right).
\end{equation}
Removing the squares, using $\sqrt{a^2+b^2}\leq a+b$, gives
\[\left\|\frac{\partial w_2}{\partial n}\right\|_{L^2(\partial_S\mathcal{P})}\leq C\left(\|f_2\|_{L^2(\partial_N\mathcal{P})}+\left\|\frac{\partial w_1}{\partial n}\right\|_{L^2(\partial_N\mathcal{P})}+\|f_3\|_{H^1(\partial_D\mathcal{P})}+\|w_1(z)\|_{H^1(\partial_D\mathcal{P})}\right).\]
Applying \eqref{eq:interiorarc} with $G=\partial_N\mathcal{P}$ and $G=\partial_D\mathcal{P}$ gives
\begin{equation}\label{eq:u2estimate3}
\left\|\frac{\partial w_2}{\partial n}\right\|_{L^2(\partial_S\mathcal{P})}\leq C\left(\|f_2\|_{L^2(\partial_N\mathcal{P})}+\|f_3\|_{H^1(\partial_D\mathcal{P})}+\|f_1\|_{L^2(\mathcal{P})}\right).
\end{equation}
Since $w=w_1$ on $\partial_S\mathcal{P}$, \eqref{eq:estimate1forlemma} is instantaneous from \eqref{eq:u1estimate}. And \eqref{eq:estimate2forlemma} follows from $w=w_1+w_2$, \eqref{eq:u1estimate}, and \eqref{eq:u2estimate3}. This completes the proof of Lemma \ref{lem:improvement}.
\end{proof}

Now we prove Theorem \ref{thm:approx1}. 

\begin{proof}[Proof of Theorem  \ref{thm:approx1}] The idea is to take a sequence of quasimodes, use Lemma \ref{lem:improvement} to correct them to harmonic functions, then use some general linear algebra for the mixed Dirichlet-to-Neumann operator developed in \cite{sloshing} to prove existence of nearby eigenvalues and eigenfunctions of \eqref{eq:SDNproblem}. 

So let $\{v_m\}$ be a sequence of quasimodes of order $\delta_m$ for the problem \eqref{eq:SDNproblem}, with the usual assumptions on $\mathcal{P}$. For each $m$ we use Lemma \ref{lem:improvement} to produce a function $w_{m}$ with $\Delta w_{m}=-\Delta v_{m}$ on $\mathcal{P}$, $\frac{\partial}{\partial n} w_{m}=-\frac{\partial}{\partial n}  v_{m}$ on $\partial_N\mathcal{P}$, and $w_{m}=-v_{m}$ on $\partial_D\mathcal{P}$. As a result, the functions
\[
\tilde v_{m}:=v_m+w_m
\]
are harmonic on $\mathcal{P}$, satisfy Neumann boundary conditions on $\partial_N\mathcal{P}$, and satisfy Dirichlet boundary conditions on $\partial_D\mathcal{P}$. Therefore, if we let $\phi_m=\tilde v_{m}|_{\partial_S\mathcal{P}}$, we have
\[
\mathcal{D}_{\mathcal{P},\partial_S\mathcal{P}}\phi_m=\frac{\partial\tilde v_m}{\partial n}.
\]
By direct computation, on $\partial_S\mathcal{P}$,
\[
\mathcal{D}_{\mathcal{P},\partial_S\mathcal{P}}\phi_{m}-\sigma_m \phi_{m}=\frac{\partial \tilde v_{m}}{\partial n}-\sigma_m \phi_{m}=\left(\frac{\partial v_m}{\partial n}-\sigma_m v_{m}\right)+\frac{\partial w_m}{\partial n}-\sigma_mw_m.
\]
Using the triangle inequality,
\[
\|\mathcal{D}_{\mathcal{P},\partial_S\mathcal{P}}\phi_{m}-\sigma_m \phi_{m}\|_{L^2(\partial_S\mathcal{P})}\leq\left\|\frac{\partial v_m}{\partial n}-\sigma_m v_{m}\right\|_{L^2(\partial_S\mathcal{P})}+\|\frac{\partial w_m}{\partial n}\|_{L^2(\partial_S\mathcal{P})}+\|\sigma_mw_m\|_{L^2(\partial_S\mathcal{P})}.
\]
We can apply the estimates of Lemma \ref{lem:improvement} to bound the second and third terms on the right-hand side. Using \eqref{eq:estimate2forlemma} and \eqref{eq:estimate1forlemma} respectively, along with the definition of $w_m$, yields
\begin{equation}\label{eq:abigimprovementprelim}
\begin{split}
\|\mathcal{D}_{\mathcal{P},\partial_S\mathcal{P}}\phi_{m}-\sigma_m \phi_{m}\|_{L^2(\partial_S\mathcal{P})}\leq&\left\|\frac{\partial v_m}{\partial n}-\sigma_m v_{m}\right\|_{L^2(\partial_S\mathcal{P})}
\\+&C\left(||\Delta v_m||_{L^2(\mathcal P)}+\left\|\frac{\partial v_m}{\partial n}\right\|_{L^2(\partial_N\mathcal{P})}+\|v_m\|_{H^1(\partial_D\mathcal{P})}\right)\\
+&C\sigma_m\|\Delta v_m\|_{L^2(\mathcal{P})}.
\end{split}
\end{equation}
But $\{v_m\}$ are quasimodes of order $\delta_m$. So, using the terminology of Definition \ref{def:quasimodes}, we have
\[
\|\mathcal{D}_{\mathcal{P},\partial_S\mathcal{P}}\phi_{m}-\sigma_m \phi_{m}\|_{L^2(\partial_S\mathcal{P})}\leq \varepsilon_m^{(1)}+C\left(\varepsilon_m^{(3)}+\varepsilon_m^{(2)}\right)+C\varepsilon_m^{(3)}\leq C\delta_m.
\]
Since $\delta_m$ approaches zero as $m\to\infty$, we may apply \cite[Theorem 4.1]{sloshing}, which gives the existence of sequences $\{i_m\}$ and $\{\tilde u_m\}$, with satisfying the eigenvalue bounds in Theorem \ref{thm:approx1}. As we are applying that theorem to $\phi_m$ rather than to $v_m$ directly, the obtained eigenfunction bound appears slightly different from the one we want. We know
\[\|\phi_m-\tilde u_m\|_{L^2(\partial_S\mathcal{P})}\leq C\sqrt{\delta_m}.\]
We want to replace $\phi_m$ with $v_m$. However, \eqref{eq:estimate1forlemma} in Lemma \ref{lem:improvement} shows that 
\[\|v_m-\phi_m\|_{L^2(\partial_S\mathcal{P})}=\|w_m\|_{L^2(\partial_S\mathcal{P})}\leq \|w_m\|_{H^1(\partial_S\mathcal{P})}\leq C\|\Delta v_m\|_{L^2(\mathcal P)}\leq C\varepsilon^{(3)}_m\leq C\delta_m\leq C\sqrt{\delta_m}.\] 
Combining these last two equations using the triangle inequality yields the eigenfunction bound, and with it Theorem \ref{thm:approx1}.
\end{proof}

\begin{remark}\label{rem:keyhole1}We can also consider ``keyhole domains'' $\mathcal P$, that is, domains which have all the same requirements on the angles but for which some components of $\partial\mathcal{P}$ coincide with each other, with opposite orientations. An example is an annulus with a single, straight cut from the outside to the inside. For these domains we may define quasimodes as in Definition \ref{def:quasimodes}. We claim that Theorem \ref{thm:approx1} and Theorem \ref{thm:approx} still hold in this setting. Indeed, the only difficulty is in the proof of Lemma \ref{lem:improvement}, which cannot be repeated verbatim as the argument with the larger disk does not make sense. However, a keyhole domain $\mathcal P$ may be conformally mapped to a non-keyhole domain, with a conformal factor bounded above and below on $\mathcal P$. In the case of an annulus with a single straight cut along the negative real axis, this conformal map can be chosen to be the inverse of the map $z\mapsto z^2$. Applying Lemma \ref{lem:improvement} to the conformally related problem there and then pulling the solution back to $\mathcal P$, absorbing the various conformal factors by possibly increasing the constant $C$, yields the result. The remainder of the arguments in this section, and in the proof of Theorem \ref{thm:approx}, apply to keyhole and ordinary domains alike.
\end{remark}

\subsection{Quasimodes near a curved boundary}

In this section we construct functions that will be used to approximate our Steklov eigenfunctions away from the corners but near a curved boundary. 

Consider a domain $\Omega\subset\mathbb R^2$ with a boundary parametrised by arc length. Assume for the moment that the boundary is smooth. Consider a patch $(s,t)$ of boundary orthogonal coordinates, where $s$ is the coordinate along the boundary and $t$ is the normal coordinate, positive into the interior. Ideally, we would like to find functions $w_{\sigma}(s,t)$ that are harmonic,  satisfy the Steklov boundary condition with parameter $\sigma$, and for which $w_{\sigma}(s,0)=\er^{\sigma\ir s}$ (see Figure \ref{fig:orthog}). 
Our ansatz will be of the form
\[
w_{\sigma}(s,t)=\er^{\sigma \omega(s,t)},
\]
where $\omega(s,t)$ is a complex-valued function with $\omega(s,0)=\ir s$. By immediate computation, under these assumptions, the Steklov boundary condition with parameter $\sigma$ is precisely 
\begin{equation}\label{eq:stekbc}
\frac{\partial\omega}{\partial t}(s,0)=-1.
\end{equation}

\begin{figure}[htb]
\begin{center}
\includegraphics{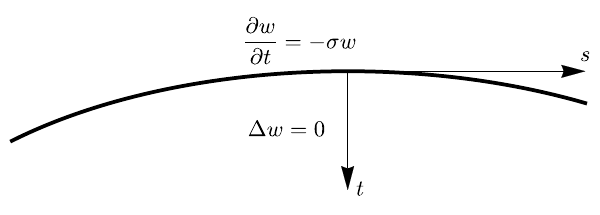}
\end{center}
\caption{Boundary orthogonal coordinates for a curvilinear boundary.\label{fig:orthog}}
\end{figure}

Now we write out the Laplacian. Let $\gamma(s)$ be the signed curvature of the boundary. Set \[\Gamma(s,t):=1+t\gamma(s).\]
Then from \cite{DES}, the expression for the Laplacian in our orthogonal coordinates is
\[
\Delta:= \Gamma^{-1/2}\left(-\frac{\partial}{\partial s} \Gamma^{-2}\frac{\partial}{\partial s}-\frac{\partial^2}{\partial t^2}\right)\Gamma^{1/2}-\frac{\gamma^2}{4\Gamma^2}+\frac{t\gamma''}{2\Gamma^3}-\frac{5t^2(\gamma')^2}{4\Gamma^4}.
\]
By direct computation, collecting powers of $\sigma$,
\[
\begin{split}
\Delta w_{\sigma}=-&\sigma^2\left(
		\left(\frac{\partial\omega}{\partial t}\right)^2
		+\Gamma^{-2}\left(\frac{\partial\omega}{\partial s}\right)^2
		\right)w_{\sigma}\\
-&\sigma\left(
\frac{\partial^2\omega}{\partial t^2}+\Gamma^{-1}\frac{\partial\omega}{\partial t}\frac{\partial\Gamma}{\partial t}
+\Gamma^{-2}\frac{\partial^2\omega}{\partial s^2}
-\Gamma^{-3}\frac{\partial\omega}{\partial s}\frac{\partial\Gamma}{\partial s}
\right)w_{\sigma}.
\end{split}
\]
This may be rewritten in a ``factorised'' form:
\begin{equation}\label{eq:factoredlaplacian}
\begin{split}
\Delta w_{\sigma}=-&\sigma^2\left(\frac{\partial\omega}{\partial t}+\ir\Gamma^{-1}\frac{\partial\omega}{\partial s}\right)\left(\frac{\partial\omega}{\partial t}-\ir \Gamma^{-1}\frac{\partial\omega}{\partial s}\right)w_{\sigma}\\
&-\sigma\left(\frac{\partial}{\partial t}+\ir\Gamma^{-1}\frac{\partial}{\partial s}+\Gamma^{-1}\frac{\partial\Gamma}{\partial t}\right)\left(\frac{\partial\omega}{\partial t}-\ir \Gamma^{-1}\frac{\partial\omega}{\partial s}\right)w_{\sigma}.
\end{split}
\end{equation}
Therefore, $w_{\sigma}(s,t)$ is harmonic for \emph{all} $\sigma$ if and only if
\begin{equation}\label{eq:twoparts}
\begin{dcases}
\displaystyle\left(\frac{\partial\omega}{\partial t}+\ir \Gamma^{-1}\frac{\partial\omega}{\partial s}\right)\left(\frac{\partial\omega}{\partial t}-\ir \Gamma^{-1}\frac{\partial\omega}{\partial s}\right)&=0;\\ 
\displaystyle\left(\frac{\partial}{\partial t}+\ir\Gamma^{-1}\frac{\partial}{\partial s}+\Gamma^{-1}\frac{\partial\Gamma}{\partial t}\right)\left(\frac{\partial\omega}{\partial t}-\ir \Gamma^{-1}\frac{\partial\omega}{\partial s}\right)&=0.
\end{dcases}
\end{equation}
\begin{proposition}Suppose that $\omega(s,t)$ satisfies the initial value problem
\begin{equation}\label{eq:complextransport}
\begin{dcases}
\frac{\partial\omega}{\partial t}=\ir\Gamma^{-1} \frac{\partial\omega}{\partial s},\\ 
\omega(s,0)=\ir s,
\end{dcases} 
\end{equation}
in a patch with coordinates $(s,t)$.
Then $w_{\sigma}(s,t)=\er^{\sigma\ir \omega(s,t)}$ is harmonic for all $\sigma$ and satisfies a Steklov boundary condition with parameter $\sigma$.
\end{proposition}

\begin{proof}The Steklov boundary condition \eqref{eq:stekbc} is automatic from \eqref{eq:complextransport} and the fact that $\Gamma(s,0)=1$, and \eqref{eq:twoparts} follows instantly from \eqref{eq:complextransport} as well. \end{proof}

Now consider the problem \eqref{eq:complextransport}. In the setting where the boundary is analytic, then the curvature $\gamma(s)$ is analytic, and hence $\Gamma(s,t)$ is analytic. By the Cauchy--Kovalevskaya theorem the problem \eqref{eq:complextransport} has a unique solution in that setting, and the coefficients of its power series in $t$ may be determined recursively. However, if the curvature $\gamma(s)$ is only assumed $C^{\infty}$ or less, the problem may have no solution; there is always a power series expansion, but it may not converge for any positive $t$. 

Nevertheless, approximate solutions may be constructed by truncation of the formal power series expansion. Define a set of functions $\{\tilde\omega_j(s)\}$ recursively by
\begin{equation}\label{eq:recursionrel}
\tilde\omega_0(s):=is; \quad \tilde\omega_{j}(s):=\frac{\ir}{j}\tilde\omega_{j-1}'(s)-\frac{j-1}{j}\gamma(s)\tilde\omega_{j-1}(s)\mbox{ for }j\geq 1.
\end{equation}
Note that this process works to produce $\tilde\omega_M(s)$ for an integer $M\geq 0$ as long as $\gamma(s)$ is $(M-2)$ times differentiable, that is, as long as the boundary is $C^{M}$. Then set
\begin{equation}\label{eq:defofomegan}
\omega_M(s,t):=\sum_{j=0}^{M}\tilde\omega_j(s)t^j\quad \textrm{ and }\quad \hat\omega_M(s):=M\tilde\omega_M(s)-\ir\tilde\omega_M'(s),
\end{equation}
with the latter definition making sense as long as $\gamma(s)$ is $(M-1)$ times differentiable. If additionally $\gamma(s)$ is $C^{M-1}$, that is, the boundary is $C^{M+1}$, then $\hat\omega_M(s)$ is a continuous function of $s$.
\begin{proposition}\label{prop:prop177} Suppose the curvature $\gamma(s)$ is $(M-1)$ times differentiable. Then the function $\omega_M(s,t)$ defined by \eqref{eq:defofomegan} satisfies
\begin{equation*}
\left(\frac{\partial\omega}{\partial t}-\ir \Gamma^{-1}\frac{\partial\omega}{\partial s}\right)\omega_M(s,t)= \Gamma^{-1}\hat\omega_M(s)t^M
\end{equation*}
and thus may be interpreted as an order-$M$ approximate solution of \eqref{eq:complextransport}.
\end{proposition}

\begin{proof} We compute, multiplying by $\Gamma$ (bounded above and below in a neighbourhood of the boundary) for convenience:
\[
\Gamma\left(\frac{\partial}{\partial t}-\ir \Gamma^{-1}\frac{\partial}{\partial s}\right)\sum_{j=0}^{M}\tilde\omega_j(s)t^j=\sum_{j=1}^{M}j\tilde\omega_j(s)(1+t\gamma(s))t^{j-1}-\sum_{j=0}^M\ir\tilde\omega_j'(s)t^{j}.
\]
Rearranging and re-labeling,
\[
\Gamma\left(\frac{\partial}{\partial t}-\ir \Gamma^{-1}\frac{\partial}{\partial s}\right)\sum_{j=0}^{M}\tilde\omega_j(s)t^j=\sum_{j=0}^{M-1}(j+1)\tilde\omega_{j+1}(s)t^j+\sum_{j=1}^Mj\tilde\omega_j(s)\gamma(s)t^j-\sum_{j=0}^M\ir\tilde\omega_j'(s)t^j.
\]
We may as well add $j=0$ to the second sum, since it is zero. Rearranging yet again, we see that the recursion relation causes most of the terms to cancel, yielding
\begin{equation}
\Gamma\left(\frac{\partial\omega}{\partial t}-\ir \Gamma^{-1}\frac{\partial\omega}{\partial s}\right)\omega_M(s,t)= M\tilde\omega_M(s)\gamma(s)t^M-\ir\tilde\omega'_M(s)t^M.\end{equation}
The proposition then follows from the definition of $\hat\omega_M(s)$.
\end{proof}

Now, formally, assuming that $\gamma(s)$ is $(M-2)$ times differentiable, we set
\begin{equation}\label{eq:curvilinearansatz}
w_{\sigma,M}(s,t)=\er^{\sigma \omega_M(s,t)}.
\end{equation}
This function immediately satisfies a Steklov boundary condition with parameter $\sigma$, and further:

\begin{proposition}\label{prop:ansatzestimates} There exist constants $C<\infty$ and $c>0$, and a sufficiently small $t_0$, such that in our patch $(s,t)$ with $t<t_0$,
\begin{equation}
|w_{\sigma,M}(s,t)|\leq C\er^{-\sigma ct}.
\end{equation}
If $\gamma(s)$ is $(M-1)$ times differentiable, we have similar constants such that
\begin{equation}
 |\nabla w_{\sigma,M}(s,t)|\leq C\sigma \er^{-\sigma ct}.
\end{equation}
Finally if $\gamma(s)$ is $M$ times differentiable with $M\geq 1$, then also $w_{\sigma,M}$ is approximately harmonic in the sense that
\begin{equation}
|\Delta w_{\sigma,M}(s,t)|\leq C\left(\sigma^2t^{M}+\sigma t^{M-1}\right)\er^{-\sigma c t}.
\end{equation}
\end{proposition}
\begin{proof} The estimate on $|w_{\sigma,M}|$ is immediate since $\omega_M$ is simply a polynomial in $t$ with the leading terms $(\ir s-t)$. Indeed, with a sufficiently small neighbourhood, $c$ and $C$ may be chosen arbitrarily close to one. 

Taking a derivative in $s$ or $t$ multiplies $w_{\sigma,M}$ by $\sigma$ times the appropriate derivative of $\omega_{M}(s,t)$. That derivative is again a polynomial in $t$ with coefficients that may depend on, now, $(M-1)$ derivatives of curvature. Its leading term is either $\ir$ for an $s$-derivative or $-1$ for a $t$-derivative. The result follows.

To compute $\Delta w_{\sigma,M}(s,t)$, we use \eqref{eq:factoredlaplacian} and Proposition \ref{prop:prop177} to obtain
\begin{equation}
\begin{split}
\frac{\Delta w_{\sigma,M}(s,t)}{w_{\sigma,M}(s,t)}=-&\sigma^2\left(\frac{\partial\omega}{\partial t}+\ir\Gamma^{-1}\frac{\partial\omega}{\partial s}\right)\Gamma^{-1}\hat\omega_M(s)t^M\\
-&\sigma\left(\frac{\partial}{\partial t}+\ir\Gamma^{-1}\frac{\partial}{\partial s}+\Gamma^{-1}\frac{\partial\Gamma}{\partial t}\right)\Gamma^{-1}\hat\omega_M(s)t^M.
\end{split}
\end{equation}
When computed, each term is a fraction with denominator some power of $m$ and numerator a polynomial in $t$ with coefficients depending on curvature. The number of derivatives of curvature that appear is at most $\max\{1,M\}$, the $M$ from taking $\hat\omega_M'(s)$ and the $1$ from taking $\frac{\partial\Gamma}{\partial s}$. The leading order terms are $\sigma^2t^M$ and $\sigma t^{M-1}$, yielding the result.
\end{proof}

\subsection{Quasimodes for a partially curvilinear polygon or zigzag}

We recall that a polygon or a zigzag is called partially curvilinear if all the sides are straight in some neighbourhoods of the vertices.  In this subsection we construct quasimodes for partially curvilinear polygons, proving the following theorem:
\begin{theorem}\label{thm:quasimodesstraightnearcorners} 
Let $\mathcal P$ be a partially curvilinear polygon, and consider the mixed Steklov-Dirichlet-Neumann problem  \eqref{eq:SDNproblem}. Assume additionally that $\partial_S\mathcal{P}\ne\varnothing$, and that the curvature of each side in $\partial_S\mathcal{P}$ is $M$ times differentiable with $M\geq 3$. Finally, assume $\delta>1$, where
\begin{equation}\label{eq:defofdelta}
\begin{split}
\delta=\min&\left(\left\{\frac{\pi}{\alpha_k}: V_k\notin(\partial_D\mathcal{P}\cup \partial_N\mathcal{P})\right\}\right.\\
\cup&\left\{\frac{\pi}{2\alpha_k}:V_k\in (\partial_S\mathcal{P})\cap(\partial_D\mathcal{P}\cup \partial_N\mathcal{P}),\alpha_k\neq \pi/2\right\}
\cup\left.\left\{M-\frac 32\right\}\right).
\end{split}
\end{equation}
Then there exists a sequence of quasimodes $\{v_m\}$ for the problem \eqref{eq:SDNproblem}, of order $\sigma_m^{-\delta+1}$.
\end{theorem}

\begin{remark} The condition that $\delta>1$ is implied by the following:  all Steklov-Dirichlet and Steklov-Neumann angles are less than or equal to $\pi/2$, and each side with a Steklov boundary condition has at least $C^5$ regularity. Sides in $\partial_D\mathcal{P}\cup \partial_N\mathcal{P}$ need only be differentiable.
\end{remark}
\begin{remark} This theorem covers the case of zigzag domains.
\end{remark}

Using Theorem \ref{thm:approx1}, and applying the same argument as in the proof of Theorem \ref{thm:approx}, see the paragraph after Theorem \ref{thm:approx1}, we immediately obtain as a consequence:
\begin{corollary}\label{cor:quasimodesstraightnearcorners} 
For a partially curvilinear polygon $\mathcal{P}$ that satisfies the assumptions of Theorem \ref{thm:quasimodesstraightnearcorners}, there is a non-decreasing sequence $\{i_m\}$ and a constant $C>0$ such that
\[
|\sigma_m-\lambda_{i_m}|\leq C\sigma_m^{-\delta+1}\qquad\text{for all }m\in\mathbb N,
\]
where the $\lambda_{i_m}$ are eigenvalues of the problem \eqref{eq:SDNproblem}. There is also a sequence of $\tilde u_m$ as in Theorem \ref{thm:approx1}, with $\delta_m=\sigma_m^{-\delta+1}$.
\end{corollary}
To begin the quasimode construction for $\mathcal{P}$, first define a partition of unity
\[
1=\chi_V(z)+\chi_I(z)+\chi_{0}(z),
\]
such that (see Figure \ref{fig:dpartition})

\begin{itemize}
\item $\chi_V$ is supported in a union of pairwise disjoint neighbourhoods of each vertex in which $\mathcal{P}$ is isometric to an exact wedge;
\item $\chi_I$ is supported in a union of pairwise disjoint neighbourhoods of the portion of each edge  away from the vertices;
\item $\chi_0$ has support compactly contained in the interior of $\mathcal P$;
\item $\nabla\chi_V$ (and therefore $\nabla\chi_I$) is perpendicular to the normal vector $\mathbf{n}$ on the boundary $\partial\mathcal{P}$, that is, each of our cut-off functions has zero normal derivative on $\partial\mathcal{P}$.
\end{itemize}

{\color{red}  
\begin{figure}[htb]
\begin{center}
\includegraphics{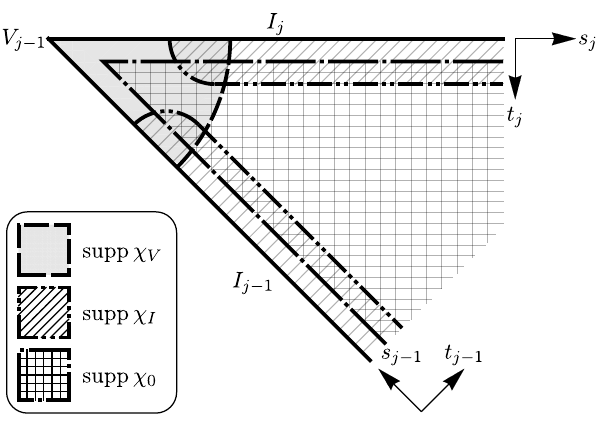}
\end{center}
\caption{Partition of unity and boundary coordinates in a neighbourhood of $V_j$ \label{fig:dpartition}}
\end{figure} 
}

A partition of unity with the required properties can be constructed, for example, in the following manner. First, let $\chi(x)$ be a standard smooth nonnegative cut-off function defined on $[0,+\infty)$ such that 
\begin{equation}\label{eq:cutofffunction}
\chi(x)=1\text{ or }x\in[0,1], \qquad\chi(x)=0\text{ for }x\in[2,\infty).
\end{equation}
Then, working in local polar coordinates $(\rho_j,\theta_j)$ in the vicinity of each vertex $V_j$, set $\chi_{V_j}=\chi(\rho_j/\varepsilon_j)$, choosing the parameters $\varepsilon_j>0$ in such a way that  $\supp \chi_{V_j}$ does not intersect the curved part of the boundary, and that $\supp \chi_{V_j}\cap \supp \chi_{V_k}=\emptyset$ for $k\ne j$, and define 
\[
\chi_V:=\sum_{j=1}^n \chi_{V_j}.
\]
Further, working in the vicinity of each side $I_j$ in local coordinates $(s_j, t_j)$, as shown in Figure \ref{fig:dpartition}, set $\chi_{I_j}=\chi(t_j/\delta_j)(1-\chi_V)$ and 
\[
\chi_I:=\sum_{j=1}^n \chi_{I_j},
\]
again choosing the parameters $\delta_j>0$ in order to make sure that $\supp \chi_{I_j}\cap \supp \chi_{I_k}=\emptyset$ for $k\ne j$. Finally, set 
\[
\chi_0:=1-\chi_V-\chi_I.
\]
        
Now, for each $m\in\mathbb N$, recall that we have an eigenvalue $\sigma_m$ and a boundary quasi-wave $\Psi^{(m)}(s)$. We use these data to define two functions $v_{m,V}(z)$ and $v_{m,I}(z)$ which are supported on the supports of $\chi_V$ and $\chi_I$ respectively. To define $v_{m,V}(z)$, we need to prescribe its value for $z$ in a neighbourhood of each vertex $V_j$. So fix a $j$ and suppose that $z$ is in a small neighbourhood of $V_j$. The boundary quasi-wave $\Psi^{(m)}(s)$ gives coefficients $\mathbf{c}_{j,\inn}$ and $\mathbf{c}_{j,\out}$ which satisfy the appropriate transfer conditions at $V_j$. We may therefore let, as in \eqref{eq:Phi}, 
\begin{equation}\label{eq:defofvmV}
v_{m,V}(z):=\Phi_{\alpha_j}^{(\mathbf{c}_{j,\inn},\mathbf{c}_{j,\out})}(\sigma_m\mathcal{V}_jz).
\end{equation}
Putting these together for each $j$ gives a full definition of $v_{m,V}(z)$.

To define $v_{m,I}(z)$, we localise to the edge $I_j$. We assume without loss of generality that the boundary orthogonal coordinates $(s_j,t_j)$ are valid on the connected component of the support of $\chi_I$ which intersects $I_j$ (if not, take $\chi_I$ supported closer to the boundary). In this neighbourhood we expect $v_{m,I}(s_j,0)$ to equal $\Psi^{(m)}_j(s_j)$.
So we use a solution of the form \eqref{eq:curvilinearansatz}, namely, for some $M$ to be chosen later,
\[
v_{m,I}(z):=\Re(w_{\sigma_m,M}(s_j+\beta,t_j)),
\]
where the shift $\beta$ is chosen so that
\[
\Re(w_{\sigma_m,M}(s_j+\beta,0))=\Psi^{(m)}_j(s_j).
\]

Our overall quasimode is obtained by gluing these together in the obvious way:
\[v_m(z):=\chi_V(z)v_{m,V}(z)+\chi_I(z)v_{m,I}(z).\]
We now claim that $\{v_m\}$ is a sequence of quasimodes for \eqref{eq:SDNproblem}, of order $\sigma_m^{-\delta+1}$. Indeed, using the terminology of Definition \ref{def:quasimodes}, we see that $\varepsilon_m^{(2)}=0$, as because $\nabla\chi_{V}$ and $\nabla\chi_I$ are perpendicular to the normal to the boundary, the functions $\{v_m\}$ satisfy all Dirichlet and Neumann conditions of \eqref{eq:SDNproblem}. Moreover, for the same reason, together with the fact that $v_{m,V}(z)$ and $v_{m,I}(z)$ both satisfy the Steklov conditions of \eqref{eq:SDNproblem} on $\partial_S\mathcal{P}$, with frequency $\sigma_m$, we have $\varepsilon_m^{(1)}(z)=0$. Thus the only issue is $\varepsilon_m^{(3)}$ and indeed this is nonzero, as $v_m$ may not be harmonic. We may compute its Laplacian and use the fact that $v_{m,V}(z)$ is harmonic to obtain
\begin{equation}\label{eq:computedlaplacian}
\begin{split}
\Delta v_m(z)=2\nabla v_{m,V}(z)\cdot\nabla\chi_V(z)&+v_{m,V}(z)\Delta\chi_V(z)
+\Delta v_{m,I}(z)\chi_I(z)\\&+2\nabla v_{m,I}(z)\cdot\nabla\chi_I(z)+v_{m,I}(z)\Delta\chi_I(z).
\end{split}
\end{equation}
The third term of \eqref{eq:computedlaplacian} is nonzero on the support of $\chi_I(z)$. However, by Proposition \ref{prop:ansatzestimates}, we have
\[|\Delta v_{m,I}(z)|\leq C(\sigma_m^2t^M+\sigma_m t^{M-1})\er^{-\sigma_m ct}.\]
By a direct calculation, the $L^2$ norm of this term over the support of $\chi_I(z)$, indeed all the way out to $t=\infty$, is bounded by a universal constant $C$ times $\sigma_m^{3/2-M}$, and thus by $C\sigma_m^{-\delta}$.

Thus we may turn our attention to estimating the other four terms of \eqref{eq:computedlaplacian}, which are only nonzero on the transition regions where the \emph{gradients} of some elements of the partition of unity are nonzero. These regions have two types: the ones contained in the support of $\chi_{0}(z)$, and the ones where only $\chi_V(z)$ and $\chi_I(z)$ are nonzero. We consider each in turn.

First consider the support of $\chi_{0}(z)$, which is compactly contained in the interior of $\mathcal{P}$. By Proposition \ref{prop:ansatzestimates}, the fourth and fifth terms of \eqref{eq:computedlaplacian} decay uniformly exponentially in $m$ on this region. As for the first two terms, recall \eqref{eq:defofvmV}, which identifies $v_m(z)$ with a function $\Phi$. The function $\Phi$ is a linear combination of plane waves with frequency $\sigma_m$ and remainder terms $R(\sigma_m\mathcal{V}_jz)$, with $R(z)$ satisfying \eqref{eq:universalremainder} for various values of $r$ depending on the boundary conditions. On the support of $\chi_{0}(z)$, the plane waves decay uniformly exponentially in $m$ as well (and are zero away from neighbourhoods of each vertex, as there $\chi_{V}(z)$ is zero). The decay of the remainder terms can be computed using the chain rule, scaling, and \eqref{eq:universalremainder}. Using these and the fact that $d(z,V_j)$ is bounded uniformly above and below on these regions, we have that on the support of $\chi_0(z)$,
\begin{equation}\label{eq:universalremainder2}
|R(\sigma_m\mathcal{V}_jz)|+|\nabla_{z}R(\sigma_m\mathcal{V}_jz)|\leq C\sigma_m^{-r}.
\end{equation}
The same estimate therefore applies to the $L^2$ norm of each term. Not that the constant $C$ \emph{a priori} may depend on $\Psi^{(m)}$, in particular on the norms of various $\Psi^{(m)}_j$. However, the normalisation condition on $\Psi^{(m)}$ implies that these norms are universally bounded independent of $m$, and thus $C$ may be taken independent of $m$.

As for the exponent $r$, it depends on the angle. If $V_j\notin(\partial_D\mathcal{P}\cup\partial_N\mathcal{P})$, we have a Steklov-Steklov corner and we extract $r=\mu_{\alpha_j/2}=\pi/\alpha_j$. If $V_j\in\partial_S\mathcal{P}\cap\partial_D\mathcal{P}$ or $V_j\in\partial_S\mathcal{P}\cap\partial_N\mathcal{P}$, we get $r=\mu_{\alpha_j}=\pi/(2\alpha_j)$, unless $\alpha_j=\pi/2$, in which case the remainder term vanishes, since a sloping beach Peters solution in this case is a pure plane wave. If $V_j\notin\partial_S\mathcal{P}$, then $v_m(z)$ is identically zero and we do not care about it. Overall, from our observations and \eqref{eq:universalremainder2}, we obtain precisely that
\[
\|\Delta v_m(z)\|_{L^2(\supp\chi_{0})}\leq C\sigma_m^{-\delta},
\]
where $\delta$ is given by \eqref{eq:defofdelta}.

Finally, consider the first, second, fourth, and fifth terms of \eqref{eq:computedlaplacian} in a region where $\chi_V(z)$ and $\chi_I(z)$ are nonzero --- specifically assume we are along some edge $I_j$, without loss of generality near a vertex $V_j$ rather than $V_{j-1}$. By our geometric assumptions, in this region, the boundary is a straight line. Therefore, $v_{m,I}(z)$ is equal to a plane wave with frequency $\sigma_m$ and boundary phase $\Psi^{(m)}_j$. Moreover, by \eqref{eq:defofvmV}, $v_{m,V}(z)$ is a solution $\Phi$ which equals a plane wave along $I_j$ with phase $\mathbf{c}_{j,\inn}$, plus a plane wave along $I_{j+1}$ with phase $\mathbf{c}_{j,\out}$, plus a remainder term $R(\sigma_m\mathcal{V}_j(z))$ for some $j$. By definition of $\Psi^{(m)}_j$, the plane wave along $I_j$ is exactly $v_{m,I}(z)$. We also observe that $\nabla\chi_I(z)=-\nabla\chi_V(z)$ and $\Delta\chi_I(z)=-\Delta\chi_V(z)$. This allows us to combine the first, second, fourth, and fifth terms of \eqref{eq:computedlaplacian} as
\begin{equation}
2\nabla R(\sigma_m\mathcal{V}_j(z))\cdot\nabla\chi_I(z)+ R(\sigma_m\mathcal{V}_j(z))\Delta\chi_V(z),
\end{equation}
plus two further terms from the plane wave along $I_{j+1}$, which both decay exponentially, uniformly on our region. Estimating the $R$ terms may now be handled precisely as it was on the support of $\chi_0$, and the $L^2$ norm here is no worse than $C\sigma_m^{-\delta}$. 

Overall, putting everything together, we have proven that for some constant $C$ independent of $m$,
\[\|\Delta v_m(z)\|_{L^2(\mathcal{P})}\leq C\sigma_m^{-\delta}.\] 
This shows that we may take $\varepsilon^{(m)}_3=C(\sigma_m+1)\sigma_m^{-\delta}\leq C\sigma_m^{-\delta+1}$ in Definition \ref{def:quasimodes}, which completes the proof of the results in this section.

There is an important special case in which we get enumeration as well.
\begin{corollary}\label{cor:curvsloshing} Suppose that $\mathcal P$ is a partially curvilinear polygon satisfying the conditions of Theorem \ref{thm:quasimodesstraightnearcorners}, with $M\geq 4$ (so that the boundary is $C^6$). Suppose further that $\partial_S\mathcal{P}$ is a single boundary arc, and that the angle at each end is $\pi/2$. Then there exists a constant $C>0$ such that for all $m$,
\[
|\sigma_m-\lambda_m|\leq C\sigma_m^{-M+\frac 52}.
\]
\end{corollary}
\begin{proof} First observe that Theorem \ref{thm:quasimodesstraightnearcorners} applies with $\delta=M-\frac 32$, in particular with $\delta>2$. Consider our quasimodes $v_m$. Since the two angles at each end of $\partial_S\mathcal{P}$ are $\pi/2$, the sloping beach Peters solutions are exact plane waves and the remainders $R(z)$ are all zero. By construction, in this case, the restrictions $\phi_m=v_m|_{\partial_S\mathcal{P}}$ are exact trigonometric functions, of frequency $\sigma_m$, satisfying Dirichlet or Neumann conditions at each end. Moreover, again by direct calculation, $\{\sigma_m\}$ are precisely the eigenvalues, and $\{\phi_m\}$ the eigenfunctions, of the one-dimensional Laplacian $\Delta_{\partial_S\mathcal{P}}$ with Dirichlet or Neumann boundary conditions at each end as appropriate. Each $\sigma_m$ is simple and a (half-)integer multiple of $\pi/L$, and they form an arithmetic progression.

Consider the set $\{\tilde u_m\}$ given by Corollary \ref{cor:quasimodesstraightnearcorners}. Each $\tilde u_m$ is a linear combination of true eigenfunctions of \eqref{eq:SDNproblem}, with eigenvalues within $C\sigma_m^{(-M+\frac 52)/2}$ of $\sigma_m$, and we also have
\[
\|\tilde u_m-\phi_m\|_{L^2(\partial_S\mathcal{P})}\leq C\sigma_m^{(-M+\frac 52)/2}.
\]
Since $M\geq 4$, the sequence $\|\tilde u_m-\phi_m\|_{L^2(\partial_S\mathcal{P})}$ is square summable. Hence, there is an $M_0$ such that
\[
\sum_{k=M_0+1}^{\infty}\|\tilde u_k-\phi_k\|_{L^2(\partial_S\mathcal{P})}^{2}<1.
\]
Additionally, since $\sigma_{m+1}-\sigma_m$ is bounded away from zero, at some point the intervals
\[
E_m:=\left[\sigma_m-C\sigma_m^{-\delta+1},\sigma_m+C\sigma_m^{-\delta+1}\right],
\] 
where $C$ and $\delta$ are as in Corollary \ref{cor:quasimodesstraightnearcorners}, become disjoint from all preceding intervals. 
In other words, there exists $M_1\in\mathbb{N}$ such that $m\ge M_1$ implies
\begin{equation}\label{eq:nonintersect}
E_m\cap E_k=\emptyset
\end{equation}
for all $k\in\mathbb{N}$, $k\ne m$.
For that $M_1$, all the functions $\tilde u_m$, $m\geq M_1$, are  linear combinations of eigenfunctions of \eqref{eq:SDNproblem} corresponding to non-intersecting spectral windows, and therefore are mutually orthogonal.
Now let $M_2=\max\{M_0,M_1\}$, pick any $m$ such that $m\geq M_2$, and consider the two subspaces
\[
\Span\{\phi_{m+1},\phi_{m+2},\ldots\}\textrm{ and }\Span\{\tilde u_{m+1},\tilde u_{m+2},\ldots\}.
\]
By the version of the Bary--Krein lemma given in \cite[Lemma 4.8]{sloshing}, and the fact that $m\geq M_0$, these two subspaces have the same codimension. Since $\{\phi_m\}$ are the eigenfunctions of a one-dimensional problem, they form a complete orthonormal basis of $L^2(\partial_S\mathcal{P})$, and hence
\[
\codim(\Span\{\tilde u_{m+1},\tilde u_{m+2},\ldots\})=\codim(\Span\{\phi_{m+1},\phi_{m+2},\ldots\})=m.
\]
This means in particular that at most finitely many $\tilde u_j$ can be linear combinations of more than one eigenfunction of \eqref{eq:SDNproblem}, otherwise the codimension would be infinite. So there exists $M_3\ge M_2$ such that $j\ge M_3$ implies that $\tilde u_j$ is a pure eigenfunction of \eqref{eq:SDNproblem}, that is, $\tilde u_j=u_{\Lambda(j)}$ for some function $\Lambda: \{M_3+1,M_3+2,\ldots\}\to\mathbb N$. As a consequence of \eqref{eq:nonintersect}, this function is strictly increasing, and the complement of its range has $M_3$ elements. So the complement of its range has a largest element, and beyond that we must have $\Lambda(j)=j$. Therefore, for sufficiently large $m$, $\tilde u_m=u_m$. 

Thus the eigenvalue $\lambda_m$ is within $C\sqrt{\sigma_m}$ of $\sigma_m$ for sufficiently large $m$, and is a bounded distance away from each other $\sigma_m$. Therefore in Corollary \ref{cor:quasimodesstraightnearcorners} we have to have $i_m=m$ for large enough $m$. Since any finite set of indices is irrelevant, Corollary \ref{cor:curvsloshing} follows.
\end{proof}

\subsection{Quasimodes for a fully curvilinear polygon}

Here we generalise and construct quasimodes for a curvilinear polygon, not necessarily straight near the corners. However, we are now only interested in the fully Steklov problem rather than the mixed problem.
\begin{theorem}\label{thm:fullcurvquasi} Let $\mathcal{P}$ be a curvilinear polygon which is piecewise $C^5$. Let
\begin{equation}\label{eq:defofdeltafull}
\delta=\min\left(\left\{\frac{\pi}{\alpha_k}: k\in\{1,\ldots,n\}\right\}\cup\left\{\frac 32\right\}\right),
\end{equation}
and observe that $\delta>1$. Then there is a sequence $\{v_m\}$ of quasimodes for the Steklov problem on $\mathcal{P}$, corresponding to the previously constructed sequence of quasi-eigenvalues $\{\sigma_m\}$, such that they are of order $\sigma_m^{-\tilde\delta+1}$ for any $\tilde\delta<\delta$.
\end{theorem}
\begin{remark} In this case, as opposed to the partially curvilinear case, it is not possible to increase the $\frac 32$ in \eqref{eq:defofdeltafull} by increasing the smoothness of the boundary arcs. This term is due to the influence of the curvature at the corners. 
\end{remark}
\begin{remark}\label{rem:defofepsilon} With $\delta$ as defined here, Theorem \ref{thm:main} holds with $\varepsilon_0:=\frac 12(\delta-1)$.
\end{remark}
As in the previous subsection, there is the usual corollary, with an identical proof:
\begin{corollary}\label{cor:quasimodesfullycurv} For a curvilinear polygon $\mathcal{P}$ with any $\tilde\delta<\delta$, there is an non-decreasing sequence $\{i_m\}$ and a constant $C>0$ such that
\[
|\sigma_m-\lambda_{i_m}|\leq C\sigma_m^{-\tilde\delta+1}\qquad\text{for all }m\in\mathbb N,
\]
where the $\lambda_{i_m}$ are Steklov eigenvalues of $\mathcal{P}$. There is also a sequence of $\tilde u_m$ as in Theorem \ref{thm:approx1}, with $\delta_m=\sigma_m^{-\tilde\delta+1}$.
\end{corollary}

Our quasimode construction in this section will proceed by ``straightening out" a neighbourhood of each corner with a conformal map, then applying a partition of unity argument as in the previous subsection. One subtlety is that we only use the conformal map to modify the \emph{remainders} in the scattering Peters solutions, rather than the solutions themselves. So, rather than the sum of two plane waves and a remainder, our models near each vertex will be the sum of the curved boundary models along each adjacent side, plus a conformally mapped remainder. 

For each $j$, we use the Riemann mapping theorem to define a conformal map $\Theta_j$ from a small neighbourhood $U_j$ of $V_j$ into a small neighbourhood of the origin in $\Sct{\alpha_j}$, with $V_j$ mapped to the origin, and with $|D\Theta_j(V_j)|=1$. By \cite{pp14}, no matter what our choice of $\Theta_j$, the map $\Theta_j$ is in the H\"older class $C^{1,\gamma}$ for any $\gamma<1$. So $D\Theta_j$ is a continuous function, and thus in a sufficiently small neighbourhood of $V_j$ (without loss of generality, $U_j$), we have $|D\Theta_j|\geq 2/3$ and in the image of that neighbourhood we have $|D\Theta_j^{-1}|\geq 1/2$.

Now we define a partition of unity $1=\chi_V+\chi_I+\chi_{0}$ precisely as in the previous section, with the property that each $\chi_{V_j}(z)$ is supported in a compact subset of $U_j$, and with the gradient of each cutoff function perpendicular to $\mathbf{n}$ at every point of $\partial\mathcal{P}$. We will define two functions $v_{m,V}(z)$ and $v_{m,I}(z)$. In fact, the definition of $v_{m,I}(z)$ is identical to the one in the previous subsection: near the edge $I_j$ we have
\[v_{m,I_j}(z):=\Re(w_{\sigma_m,M}(s_j+\beta,t_j)),\]
where the shift $\beta$ is chosen so that the restriction to the boundary of $v_{m,I}(z)$ is $\Psi^{(m)}_j(s_j)$. Then $v_{m,I}(z)$ is the sum of these over $j$. To define $v_{m,V}(z)$ we have to work a little harder. We would like to use \eqref{eq:defofvmV} but cannot because the sector is not straight. Instead, we use Proposition \ref{prop:planewaveapprox} to write
\begin{equation}\label{eq:straightcomparisonorig}
\Phi_{\alpha_j}^{(\mathbf{c}_{j,\inn},\mathbf{c}_{j,\out})}(z)=W^{\mathbf c_{j,\inn}}_{\out,\alpha_j}(z)+W^{\mathbf c_{j,\inn}}_{\inn,\alpha_j}(z)+ R_{j}(z),
\end{equation}
where the $W$ terms are pure plane waves and $R_j(z)$ is the remainder. We define
\begin{equation}\label{eq:defofvmVcurv}
v_{m,V_j}(z):=v_{m,I_j}(z)+v_{m,I_{j+1}}(z)+R_j(\sigma_m\Theta_j(z)).
\end{equation}
The first two terms here are the curvilinear, approximately harmonic functions along the incoming and outgoing edges from $V_j$ respectively (again, without loss of generality we assume that $U_j$ is a small enough neighbourhood so these are defined). The third term is the remainder term in the appropriate scattering Peters solution, pulled back.
As before, we let $v_{m,V}(z)$ be the sum of these over all $j$, and then set, in full,
\[v_{m}(z):=v_{m,V}(z)\chi_V(z)+v_{m,I}(z)\chi_I(z).\]

We need to prove that these are quasimodes. This requires estimating $\varepsilon_m^{(i)}$, $i=1,2,3$, in Definition \ref{def:quasimodes}. Since $\partial_{N}\mathcal{P}=\partial_{D}\mathcal{P}=\emptyset$, we may take $\varepsilon_m^{(2)}=0$. To address $\varepsilon_m^{(3)}$, we use a very similar argument to that in the previous subsection. The formula \eqref{eq:computedlaplacian} still holds, and the analysis proceeds analogously, with the plane waves replaced by $v_{m,I_j}(z)$ and $v_{m,I_{j-1}}(z)$ and the remainders $R(\sigma_m\mathcal{V}_j(z))$ replaced by $R_j(\sigma_m\Theta_j(z))$. Since $\Theta_j(z)$ has derivative bounded away from zero on $U_j$, the analogue of the decay estimate \eqref{eq:universalremainder2}, with the above replacement, still holds away from the vertices. Moreover, the functions $v_{m,I_j}(z)$ decay exponentially in $\sigma_m$ away from $I_j$. This allows the argument to proceed unchanged, and we may as before take $\varepsilon_3^{(m)}\leq C\sigma_m^{-\delta+1}$, which is enough. 

It remains only to estimate $\varepsilon_m^{(1)}$. To do this we introduce a new piece of terminology: for a family of functions $\tilde v_m(z)$ on $\mathcal{P}$, we define their ``Steklov defect" to be the functions on the boundary given by
\[
\texttt{SD}(v_m)(z):=\frac{\partial v_m}{\partial n}(z)-\sigma_m v_m(z).
\]
Our goal is to find an upper bound for the $L^2$ norms of $\texttt{SD}(v_m)(z)$. Since all cutoff functions have zero normal derivative at the boundary, and the functions $v_{m,I}(z)$ satisfy an exact Steklov boundary condition,
\[\texttt{SD}(v_m)(z)=\texttt{SD}(v_{m,V})(z)\chi_V(z)+\texttt{SD}(v_{m,I})(z)\chi_I(z)=\texttt{SD}(v_{m,V})(z)\chi_V(z).\]
This is supported in a union of $2n$ regions, one on each side of each vertex, and we can consider the $L^2$ norm over each separately. So fix $j$, and consider a short segment along $I_{j+1}$ in which $\chi_{V,j}(z)|_{I_{j+1}}$ is supported. Note that $s_{j+1}$ is the coordinate here, and this segment is contained in $[0,\varepsilon]$ for some $\varepsilon>0$. So we need to bound 
\[\|\texttt{SD}(v_{m,V_j})\|_{L^2[0,\varepsilon]}.\]

We will do this via comparison with a non-curvilinear case. Our function $v_{m,V_j}(z)$ is given by \eqref{eq:defofvmVcurv}. If the sides were straight near the corner, then we would instead have the function $\tilde v_{m,V_j}(z)$ given by:
\begin{equation}\label{eq:straightcomparison}
\tilde v_{m,V_j}(z)=W^{\mathbf c_{j,\inn}}_{\inn,\alpha_j}(\sigma_m z)+W^{\mathbf c_{j,\out}}_{\out,\alpha_j}(\sigma_m z)+ R_{j}(\sigma_m z).
\end{equation}
Of course, the values of $z$ in \eqref{eq:defofvmVcurv} and \eqref{eq:straightcomparison} do not have the same domain. Nevertheless, we can compare the Steklov defects of these two functions, as $s_{j+1}$ is a legitimate coordinate along the \emph{boundary} of each. Moreover, the Steklov defect of $\tilde v_{m,V_j}(z)$ is zero since it is an exact (scaled) scattering Peters solution. Thus it suffices to bound
\[
\|\texttt{SD}(v_{m,V_j})-\texttt{SD}(\tilde v_{m,V_j})\|_{L^2[0,\varepsilon]}.
\]
This involves comparing \eqref{eq:defofvmVcurv} and \eqref{eq:straightcomparison} term by term. Observe that since we are along $I_{j+1}$ rather than $I_j$, the Steklov defect of $v_{m,I_{j+1}}(z)$ is zero by the observation before Proposition \ref{prop:ansatzestimates}. The Steklov defect of the outgoing plane wave $W^{\mathbf c_{j,\out}}_{\out,\alpha_j}(\sigma_m z)$ is also zero. So we just need to compare the first and third terms, and it suffices to bound
\begin{equation}\label{eq:firstandthird}
\|\texttt{SD}(v_{m,I_j})-\texttt{SD}(W^{\mathbf c_{j,\inn}}_{\inn,\alpha_j}(\sigma_m z))\|_{L^2[0,\varepsilon]} + 
\|\texttt{SD}(R_j(\sigma_m\Theta_j(z)))-\texttt{SD}(R_j(\sigma_m z))\|_{L^2[0,\varepsilon]}.
\end{equation}
Both of these can be handled with direct calculations. In all we claim that the expression \eqref{eq:firstandthird} is bounded by $C\sigma^{1-\tilde\delta}$. Assuming this claim, we can take $\varepsilon^{(1)}_m=\varepsilon^{(3)}_m=C\sigma_m^{1-\tilde\delta}$ in Definition \ref{def:quasimodes}, which proves Theorem \ref{thm:fullcurvquasi}. It therefore remains only to prove the needed bounds on \eqref{eq:firstandthird}.

We begin by analysing the remainder term of \eqref{eq:firstandthird}. Recall that $R_j(z)$ is a remainder term in a scattering Peters solution, defined on an infinite sector. For $x\in[0,\infty)$, define 
\[
f(x):=R_j(x,0);\quad g(x):=\frac{\partial R_j}{\partial y}(x,0).
\]
Bounds on these functions and on their $x$-derivatives, for both large $x$ and small $x$, may be extracted from \cite[Theorem 2.1]{sloshing} and the usual angle-doubling reflection argument. Note that our normalisation condition ensures that the constants $C$ may be chosen independent of $m$. The bounds we obtain, with $\mu=\pi/\alpha_j$ (note that $\mu\geq\delta >1$), are:
\begin{equation}\label{eq:petersrembounds}
\begin{aligned}{2}
|f(x)|&\leq C(1+x)^{-\mu};&\quad |f'(x)|&\leq C(1+x)^{-\mu-1};\\ 
|g(x)|&\leq C(1+x)^{-\mu-1};&\quad |g'(x)|&\leq C\min\left\{x^{\mu-2},x^{-\mu-2}\right\}.
\end{aligned}
\end{equation}

Now let $\theta$ be the restriction of $\Theta$ to the edge $I_{j+1}$, in the coordinate $s_{j+1}$. Observe that $\theta'(0)=1$ and both $\theta$ and $\theta^{-1}$ have derivatives bounded below by $1/2$ and above by $2$ on $[0,\varepsilon]$. Moreover by \cite{pp14}, $\theta$ is $C^{1,\gamma}$ and $\theta'$ is $C^{0,\gamma}$ for every $\gamma<1$. The remainder term of \eqref{eq:firstandthird} is, with all this terminology,
\begin{equation}\label{eq:thingtoboundgandf}
\|(\sigma_m|\theta'(s_{j+1})|g(\sigma_m\theta(s_{j+1}))-\sigma_m f(\sigma_m\theta(s_{j+1})))-(\sigma_m g(\sigma_m s_{j+1})-\sigma_m f(\sigma_m s_{j+1}))\|_{L^2[0,\varepsilon]}.
\end{equation}

We bound the differences of the $g$ terms and the $f$ terms separately. For the difference of $f$ terms, we write
\[
\begin{split}
|\sigma_m(f(\sigma_m s_{j+1})-f(\sigma_m\theta(s_{j+1}))|&\leq\sigma_m\max_{x\in[\sigma_m s_{j+1},\sigma_m\theta(s_{j+1})]}|f'(x)|\cdot|\sigma_m s_{j+1}-\sigma_m\theta(s_{j+1})|\\
&\leq\sigma_m^2|s_{j+1}-\theta(s_{j+1})|\max_{x\in[\sigma_m s_{j+1}/2,2\sigma_m s_{j+1}]}|f'(x)|.
\end{split}
\]
The function $s_{j+1}-\theta(s_{j+1})$ is $C^{1,\gamma}$ for all $\gamma<1$, and both the function and its derivative are zero at $s_{j+1}=0$, so in fact $|s_{j+1}-\theta(s_{j+1})|\leq Cs_{j+1}^{1+\gamma}$ for all $\gamma<1$. As for $|f'(x)|$, since $x\in\sigma_m s_{j+1}[1/2,2]$, we can bound it using \eqref{eq:petersrembounds}. In all we conclude
\[
|\sigma_m(f(\sigma_m s_{j+1})-f(\sigma_m\theta(s_{j+1})))|\leq C\sigma_m^2s_{j+1}^{1+\gamma}(1+\sigma_m s_{j+1})^{-\mu-1}.
\]
The square of the $L^2$ norm of the right-hand side, using the substitution $w=\sigma_m s_{j+1}$, is
\begin{equation}\label{eq:thing833}
C\sigma_m^{1-2\gamma}\int_0^{\sigma_m\varepsilon}w^{2+2\gamma}(1+w)^{-2\mu-2}\, \dr w\leq C\sigma_m^{1-2\gamma}\int_0^{\sigma_m\varepsilon}(1+w)^{-2\mu+2\gamma}\, \dr w.
\end{equation}
Using $\mu\geq\delta$, as long as we avoid choosing $\gamma=\mu-\frac 12$, this is bounded by
\[C\sigma_m^{1-2\gamma}|\sigma_m^{-2\delta+2\gamma+1}+1|\leq C\sigma_m^{1-2\gamma}+C\sigma_m^{2-2\delta}.\]
Taking square roots to get the $L^2$ norm and using $\sqrt{a+b}\leq\sqrt{a}+\sqrt{b}$, we have
\[\||\sigma_m(f(\sigma_m s_{j+1})-f(\sigma_m\theta(s_{j+1}))\|_{L^2[0,\varepsilon]}\leq C\sigma_m^{\frac 12-\gamma}+C\sigma_m^{1-\delta}.\]
Since $\tilde\delta<\delta\leq 3/2$, we can choose $\gamma$ sufficiently close to $1$ so that both terms are bounded by $C\sigma_m^{1-\tilde\delta}$, as desired.

Now for the difference of $g$ terms. By adding and subtracting $\sigma_m g(\sigma_m\theta(s_{j+1}))$ we can write it as
\begin{equation}
\sigma_m(|\theta'(s_{j+1})|-1)g(\sigma_m\theta(s_{j+1}))+\sigma_m (g(\sigma_m\theta(s_{j+1}))- g(\sigma_m s_{j+1})).
\end{equation}
The first of these two terms, again using H\"older continuity of $\theta$, \eqref{eq:petersrembounds}, and $s_{j+1}/2\leq\theta(s_{j+1})\leq 2s_{j+1}$, is bounded by 
\[C\sigma_m s_{j+1}^{\gamma}(1+\sigma_m s_{j+1})^{-\mu-1}.\]
Using the exact same argument as for the $f$ terms we can show that the square of the $L^2$ norm of this quantity is bounded by \eqref{eq:thing833}, except with $(1+w)^{-2\mu+2\gamma-2}$ instead of $(1+w)^{-2\mu+2\gamma}$. This makes the integral smaller rather than larger, so the same $C\sigma_m^{1-\tilde\delta}$ bound holds. As for the second of the two terms, by similar arguments as before, it is less than
\[
\sigma_m\max_{x\in[\sigma_m s_{j+1}/2,2\sigma_m s_{j+1}]}|g'(x)|\cdot |\sigma_m\theta(s_{j+1})-\sigma_m s_{j+1}|\leq C\sigma_m^2s_{j+1}^{1+\gamma}\min\left\{(\sigma_m s_{j+1})^{-\mu-2},(\sigma_m s_{j+1})^{\mu-2})\right\}.
\]
The $L^2$ norm squared is thus bounded, using the substitution $w=\sigma_m s_{j+1}$ again, along with $\mu\geq\delta$, by
\[
C\sigma_m^{1-2\gamma}\left(\int_0^1 w^{2+2\gamma}w^{2\mu-4}\, \dr w+\int_1^{\sigma_m\varepsilon} w^{2+2\gamma}w^{-2\delta-4}\, \dr w\right).
\]
Since $\mu>1$, for $\gamma$ sufficiently close to 1, the exponent in the first term is positive and the integral is bounded by 1. For the second term, we have (without the pre-factor) a bound of $C\sigma_m^{2\gamma-2\delta-1}$. After incorporating the prefactor and taking square roots, the $L^2$ norm is bounded by
\[C\sigma_m^{\frac 12-\gamma}+C\sigma_m^{-\delta},\]
which as before is bounded by $C\sigma_m^{1-\tilde\delta}$ for $\gamma$ sufficiently close to $1$. This proves the necessary bound for the remainder term of \eqref{eq:firstandthird}.

It remains only to analyse the first term in \eqref{eq:firstandthird}. To do this, let $(s,t)=(s_{j+1},t)$ be the curvilinear coordinates along $I_{j+1}$, and let $(\tilde s,\tilde t)$ be the curvilinear coordinates along the adjacent edge $I_j$, so that $\tilde s=L_j-s_j$. 
For the exact sector of angle $\alpha=\alpha_j$, we have, for some shift $\beta$ depending on $\mathbf{c}_{j,\inn}$,
\[
W^{\mathbf c_{j,\inn}}_{\inn,\alpha_j}(\sigma(\tilde s,\tilde t))=\Re(\er^{\sigma(\ir(\tilde s+\beta)-\tilde t))}),
\]
and for a curvilinear sector, we have, for the same $\beta$,
\[
v_{m,I_j}(\tilde s,\tilde t)=\Re(\er^{\sigma f_M(\tilde s,\tilde t)})=\Re(\er^{\sigma(\ir(\tilde s+\beta)-\tilde t+O(\tilde t^2))}).
\]
We need to compute the Steklov defects of each of these functions in the coordinates along $I_{j+1}$. To do this, first compute the gradients of each:
\begin{equation}\label{eq:computedthegradients}
\nabla v_{m,I_j}(\tilde s,\tilde t)=\sigma\Re\begin{pmatrix}(\ir+O(\tilde t^2)) \er^{\sigma(\ir(\tilde s+\beta)-\tilde t+O(\tilde t^2))}\\ (-1+O(\tilde t))\er^{\sigma(\ir(\tilde s+\beta)-\tilde t+O(\tilde t^2))}\end{pmatrix},
\end{equation}
with the same expression, without the error terms, for $\nabla W$.

We will need to take $L^2$ norms in $s$, so we need to discuss how the coordinates are related. For an exact sector, we have $\begin{pmatrix}\tilde s\\ \tilde t\end{pmatrix}=\begin{pmatrix}s\cos\alpha\\ s\sin\alpha\end{pmatrix}$, and the normal vector $\mathbf{n}$ to the opposite edge $I_{j}$, as a function of $s$, is $\mathbf{n}(s)=\begin{pmatrix}\sin\alpha\\ -\cos\alpha\end{pmatrix}$. For the \emph{curvilinear} sector, we have errors of the following types:
\[
\begin{pmatrix}\tilde s\\ \tilde t\end{pmatrix}=\begin{pmatrix}s\cos\alpha\\ s\sin\alpha\end{pmatrix}+O(s^2);\quad 
\mathbf{n}(s)=\begin{pmatrix}\sin\alpha\\ -\cos\alpha\end{pmatrix}+O(s).
\]
So the difference of Steklov defects we need to consider is
\begin{equation}\label{eq:thing834}
\begin{split}
&\sigma\Re\left((\er^{\sigma(\ir(s\cos\alpha+\beta)-s\sin\alpha))}\right)-\Re\left(\er^{\sigma(\ir(s\cos\alpha+\beta)-s\sin\alpha+O(s^2)))})\right)\\
-&\left(\begin{pmatrix}\sin\alpha\\ -\cos\alpha\end{pmatrix}\cdot \nabla W^{\mathbf c_{j,\inn}}_{\inn,\alpha_j}(\sigma(\tilde s(s),\tilde t(s))\right)\\
-&\left(\begin{pmatrix}\sin\alpha\\ -\cos\alpha\end{pmatrix}+O(s))\cdot\nabla v_{m,I_j}(\tilde s(s),\tilde t(s))\right).
\end{split}
\end{equation}

Consider first the difference of terms without gradients in \eqref{eq:thing834} and take its absolute value. It is
\begin{equation}\label{eq:thing835}
\left|-\sigma\Re\left(\er^{\sigma(\ir(s\cos\alpha+\beta)-s\sin\alpha)}\right)\left(\er^{\sigma O(s^2)}-1\right)\right|\leq \sigma \er^{\sigma(-s\sin\alpha)}\left(\er^{C\sigma s^2}-1\right).
\end{equation}
Use the fact that $\er^x-1\leq x\er^x$, then use the fact that $\sin\alpha>0$, to see that there exists $c>0$ such that for $s$ sufficiently small, this is bounded by
\[
C\sigma^2s^2\er^{-c\sigma s}.
\]
The $L^2$ norm of this function can be computed directly via the usual change of variables $w=\sigma s$, and is bounded by $C\sigma^{-1/2}$. Since $\tilde\delta<3/2$ this is bounded by $C\sigma^{1-\tilde\delta}$ as desired.

Now analyse the gradient terms in \eqref{eq:thing834}. First examine the $O(s)$ term in \eqref{eq:thing834}. Using \eqref{eq:computedthegradients}, and the equivalence of the $\ell_2$ and $\ell_1$ norms on $\mathbb R^2$, a bound for the absolute value of this term is
\[
C\sigma s\left((1+O(s^2))\er^{-\sigma s\sin\alpha+O(s^2)}+(1+O(s))\er^{-\sigma s\sin\alpha+O(s^2)}\right).
\]
As with the terms without gradients, there exists $c>0$ such that this is bounded by $C\sigma s \er^{-c\sigma s}$, and then the same $L^2$ norm computation gives the same bound. The remainder of the terms are given by
\[
-\begin{pmatrix}\sin\alpha\\ -\cos\alpha\end{pmatrix}\cdot\left(\nabla W^{\mathbf c_{j,\inn}}_{\inn,\alpha_j}(\sigma\tilde s(s),\sigma\tilde t(s))-\nabla v_{m,I_j}(\tilde s(s),\tilde t(s))\right).
\]
Using \eqref{eq:computedthegradients} again and the usual adding/subtracting trick, we have a bound of
\[
C\sigma\left(O(s)\er^{-\sigma s\sin\alpha+O(s^2)}\right)+C\sigma\left(\er^{-\sigma s\sin\alpha+O(s^2)}-\er^{-\sigma s\sin\alpha}\right).
\]
The first term is again bounded by $C\sigma s \er^{-c\sigma s}$ and may be taken care of as before. The second term satisfies the bound \eqref{eq:thing835} and therefore can be treated the same way as well. This shows that all terms are bounded by $C\sigma_m^{1-\tilde\delta}$, completing the proof of Theorem \ref{thm:fullcurvquasi}.

For later use we also record a corollary:
\begin{corollary}\label{cor:approxofquasibypsi} As $m\to\infty$, the quasimodes $v_m$ on the boundary, as well as the corresponding linear combinations $\tilde u_m$ of Steklov eigenfunctions, get closer to $\Psi^{(m)}$ in the sense that
\begin{equation}\label{eq:almostorthogonal1} \|v_m-\Psi^{(m)}\|_{L^2(\partial\Omega)}=O(m^{-1/2}),\quad \|\tilde u_m-\Psi^{(m)}\|_{L^2(\partial\Omega)}=O(m^{\frac 12(1-\tilde\delta)}).
\end{equation}
\end{corollary}

\begin{proof} First consider $v_m$, for which the difference is zero off the support of $\chi_V$. On the support of $\chi_V$, near a vertex $V_j$, the restriction of the term $v_{m,I_{j+1}}(z)$ to the edge $I_{j+1}$ is precisely $\Psi^{(m)}$. The other two terms in \eqref{eq:defofvmVcurv} both have $L^2$ norms which go to zero as $m\to\infty$ and do so, via a scaling argument and direct integration, at order $\sigma_m^{-1/2}$. An analogous argument works along the edge $I_j$, and adding up the contributions from the finitely many vertices completes the proof for $v_m$. The statement for $\tilde u_m$ follows immediately from the statement for $v_m$ and Corollary \ref{cor:quasimodesfullycurv}.
\end{proof}

\subsection{Almost orthogonality and its consequences}\label{subsec:orthogcons}

In statements such as Theorem \ref{thm:approx} we showed that near each quasi-eigenvalue $\sigma_m$ there exists a true eigenvalue $\lambda_{i_m}$. We did not, however, show that there is a \emph{distinct} true eigenvalue near each quasi-eigenvalue. That is, we did not show the map $m\mapsto i_m$ is injective. We now remedy this, but at a small cost.
\begin{theorem}\label{thm:injectivity} Let $\mathcal P$ be a curvilinear polygon with all angles less than $\pi$. Let $1<\tilde\delta<\delta$, where $\delta$ is defined by \eqref{eq:defofdeltafull}. Then there exists a map $j:\mathbb N\to\mathbb N$ which is \emph{strictly} increasing for large arguments, and there exists a constant $C$, such that
\[|
\sigma_m-\lambda_{j(m)}|\leq Cm^{\frac 12(1-\tilde\delta)}.
\]
\end{theorem}
\begin{remark} The cost is simply the extra factor of $\frac{1}{2}$ in the exponent. This shows up as a consequence of an abstract linear algebra result \cite[Theorem 4.1]{sloshing}. It may be able to be removed in our setting, but for simplicity we have not done so.
\end{remark}

To begin the proof, first we show that the boundary plane waves $\Psi^{(m)}$ are nearly orthogonal. We do this by taking advantage of the relationship between $\Psi^{(m)}$ and the eigenfunctions of the quantum graph Dirac operator $\Dir$ defined in \eqref{eq:defofDirac}. Although $\Psi^{(m)}$ turn out \emph{not} to be orthonormal in $L^2(\mathcal G)$, inner products of distinct boundary plane waves are nevertheless small.
\begin{proposition}\label{prop:almostorthogonality} There exists a constant $C$ such that for all $m,l\in\mathbb N$ with $m\neq l$,
\[
|\langle\Psi^{(m)},\Psi^{(l)}\rangle|\leq C(\sigma_m+\sigma_l)^{-1}.
\]
\end{proposition}
\begin{proof} Recall from Proposition \ref{prop:Dirac} that $\Dir$, with the matching conditions \eqref{matchcondboth}, is self-adjoint and that its set of non-negative eigenvalues is precisely $\{\sigma_m\}\setminus\{0\}$. As a consequence of self-adjointness, its basis eigenfunctions $\mathbf{f}_{m,\pm}=\begin{pmatrix}f_{m,\pm,1}\\f_{m,\pm,2}\end{pmatrix}$ corresponding to eigenvalues $\pm \sigma_m$ can be chosen to be orthonormal in $(L^2(\mathcal{G}))^2$. We note that the eigenfunctions can be chosen in the form 
\[
\mathbf{f}_{m,\pm}|_{\mathcal I_j} =\begin{pmatrix} d_{m,j,\pm,1}\er^{\pm\ir\sigma_m s} \\ d_{m,j,\mp,2}\er^{\mp\ir\sigma_m s} \end{pmatrix},
\]
with some constants $d_{m,j,\pm,p}\in\mathbb{C}$, $p=1,2$. Moreover, by the same reasoning as in Remark \ref{rem:evfromConj} we can choose 
\[
d_{m,j,\pm,2}=\overline{d_{m,j,\pm,1}}.
\]

We will be mostly interested in eigenfunctions $\mathbf{f}_{m,\pm}$ from now now. Comparing these eigenfunctions with the boundary quasi-waves $\Psi^{(m)}$,of the Steklov problem and their restrictions $\Psi^{(m)}_j$ on $I_j$, see Definition \ref{def:bqw} and equation \eqref{eq:Psijsj}, we immediately conclude that up to a scaling factor 
\[
\Psi^{(m)}_j=\left.{f}_{m,+,1}+{f}_{m,+,2}\right|_{I_j}=\left.2\Re\left({f}_{m,+,1}\right)\right|_{I_j},
\]
and therefore we may write
\[
\langle\Psi^{(m)},\Psi^{(l)}\rangle=\langle{f}_{m,+,1},{f}_{l,+,1}\rangle+\langle{f}_{m,+,2},{f}_{l,+,2}\rangle+\langle{f}_{m,+,1},{f}_{l,+,2}\rangle+\langle{f}_{m,+,2},{f}_{l,+,1}\rangle.
\]
Since we have $m\neq l$, and the basis eigenfunctions $\mathbf{f}_{m,+}$ are orthonormal in $(L^2(\mathcal{G}))^2$, the first two terms of this sum add to zero, leaving only the last two terms to estimate.

Thus, remembering that we have a complex conjugate on the second entry in our inner product, and setting $d_{m,j}:=d_{m,j,+,1}$, we have
\[
\langle{f}_{m,+,1},{f}_{l,+,2}\rangle=\sum_{j=1}^n\int_{I_j}d_{m,j}\er^{\ir\sigma_m s}d_{l,j}\er^{\ir\sigma_l s}\, \dr s=\sum_{j=1}^n d_{m,j}d_{l,j}\int_{I_j}\er^{\ir(\sigma_m+\sigma_l)s}\, \dr s,
\]
where $n$ is the number of vertices. By explicit integration by parts, each such integral, in absolute value, is bounded by $2(\sigma_m+\sigma_l)^{-1}$. Thus,
\[
|\langle{f}_{m,+,1},{f}_{l,+,2}\rangle|\leq 2(\sigma_m+\sigma_l)^{-1}\sum_{j=1}^{n}\left|d_{m,j}d_{l,j}\right|\leq  2(\sigma_m+\sigma_l)^{-1}\left(\sum_{j=1}^{n}\left|d_{m,j}\right|^2\right)^{1/2}\left(\sum_{j=1}^{n}|d_{l,j}|^2\right)^{1/2},
\]
which by normalisation of the Dirac eigenfunctions is $2(\sigma_m+\sigma_l)^{-1}$.  A similar analysis works for the other nonzero term, showing that in fact
\[
|\langle\Psi^{(m)},\Psi^{(l)}\rangle|\leq 4(\sigma_m+\sigma_l)^{-1},
\]
proving Proposition \ref{prop:almostorthogonality}. 
\end{proof}

\begin{corollary}\label{cor:almostorthogonality} For any $\varepsilon>0$ there exists $M\in\mathbb N$ such that $m,l\geq M$ with $m\neq l$ implies $|\langle\tilde u_m,\tilde u_l\rangle|<\varepsilon$.
\end{corollary}

\begin{proof} This is an immediate consequence of Proposition \ref{prop:almostorthogonality} and Corollary \ref{cor:approxofquasibypsi}.
\end{proof}

\begin{corollary}\label{cor:linind} Pick any $N\in\mathbb N$. There exists $M\in\mathbb N$ such that any $N$ of the functions $\{\tilde u_m\}_{m\geq M}$ form a linearly independent set.
\end{corollary}

\begin{proof} As $m\to\infty$, $\|\tilde u_m\|\to 1$ by Corollary \ref{cor:approxofquasibypsi}, and all inner products go to zero by Corollary \ref{cor:almostorthogonality}. So there is an $M$ such that if $m,l\geq M$ with $m\neq l$, $\|\tilde u_m\|\geq \sqrt{\frac{1}{2}}$ and $|\langle\tilde u_m,\tilde u_l\rangle|\leq \frac{1}{2N}$. Now select $N$ of the $\{\tilde u_m\}$ --- call them $\tilde u_{m_1},\ldots,\tilde u_{m_N}$ --- with all $m_j\geq M$. Suppose for contradiction they are \emph{not} linearly independent; then there exists a nontrivial relation among them, which without loss of generality may be written
\[
\tilde u_{m_1}=a_2\tilde u_{m_2}+\ldots+a_N\tilde u_{m_N},
\]
with all $|a_i|\leq 1$. Now take inner products with $\tilde u_{m_1}$ and use the triangle inequality, obtaining
\[
\|\tilde u_{m_1}\|^2\leq\sum_{i=2}^{N}|\langle\tilde u_{m_i},\tilde u_{m_1}\rangle|.
\]
But this means $\frac{1}{2}\leq  \frac{N-1}{2N}$, a contradiction which completes the proof.
\end{proof}

Now we complete the argument. By Corollary \ref{cor:quasimodesfullycurv}, for each $m$, $\tilde u_m$ is a linear combination of eigenfunctions with eigenvalues in the interval
\[
\mathcal I_m:=\left(\sigma_m-Cm^{\frac 12(1-\tilde\delta)}, \sigma_m+Cm^{\frac 12(1-\tilde\delta)}\right).
\]

\begin{proposition}\label{prop:nomoreN} There exists an $N>0$ such that no more than $N$ of the intervals $\mathcal I_m$ overlap (that is, have a connected union).
\end{proposition}
\begin{proof} By Theorem \ref{thm:quantum}, the quasi-eigenvalues $\sigma_m$ are the square roots of the eigenvalues of a quantum graph Laplacian with non-Robin boundary conditions \cite{BK13}, and as such, obey a Weyl law with bounded remainder \cite[Lemma 3.7.4]{BK13}. This means that there exists an $N_0>0$ such that the number of $\sigma_m$ in any interval of length $1$ is less than $N_0$. As a result, the sequence $\{\sigma_m\}$ cannot go $N_0$ terms without a gap of size at least $1/N_0$. Now the length $\left|\mathcal{I}_m\right|\to 0$ as $m\to\infty$, so for sufficiently large $m$, each gap of size $1/N_0$ will cause $\bigcup\mathcal I_m$ to be disconnected. Thus for sufficiently large $m$, at most $N_0$ intervals $\mathcal I_m$ may overlap. 
\end{proof}

Now let $\mathcal C_k$ be the $k$-th connected component of $\bigcup\mathcal I_m$, let $N_k$ be the number of quasi-eigenvalues it contains, and let $\sigma_{m_k}$ be the smallest quasi-eigenvalue it contains. By our work to this point we have $1\leq N_k\leq N$ and thus $k\leq m_k\leq Nk$. Further, the length of $\mathcal C_k$ is at most $N\left|\mathcal{I}_{m_k}\right|=2CNm_{k}^{\frac 12(1-\tilde\delta)}$. For sufficiently large $k$, by Corollary \ref{cor:linind}, the functions $\tilde u_m$ associated to each of the $N_k$ quasi-eigenvalues in $\mathcal C_k$ are linearly independent. Since each of them is a linear combination of eigenfunctions with eigenvalues in $\mathcal C_k$, we see that $\mathcal C_k$ must contain (at least) $N_k$ eigenvalues. We construct $j(m)$ for $m\in\mathcal C_k$ by listing these in increasing order. As a result the map $m\mapsto j(m)$ is injective on $\mathcal C_k$. Further, if $\sigma_m\in\mathcal C_k$,
\[|\sigma_m-\lambda_{j(m)}|\leq\ell(\mathcal C_k)=2CNm_k^{\frac 12(1-\tilde\delta)}\leq 4CNm^{\frac 12(1-\tilde\delta)},\]
where we have used that $m_k^{\frac 12(1-\tilde\delta)}\leq 2m^{\frac 12(1-\tilde\delta)}$ for large enough $k$ and $m\in\mathcal C_k$. This proves Theorem \ref{thm:injectivity}. \qed

Additionally, we observe that the natural analogue of Theorem \ref{thm:injectivity} holds for a zigzag as well, by precisely the same arguments, see \eqref{eq:defofdelta}.

\begin{theorem}\label{thm:injectivityzigzag} Let $\mathcal Z$ be a partially curvilinear zigzag domain such that $\delta>1$, where $\delta$ is defined by \eqref{eq:defofdelta}. Then there exists a map $j:\mathbb N\to\mathbb N$ which is \emph{strictly} increasing beyond a certain point (i.e. for $m$ greater than some threshold value), and there exists a constant $C$, such that
\[|\sigma_m-\lambda_{j(m)}|\leq Cm^{\frac 12(1-\delta)}.\]
\end{theorem}

This is proved in an almost identical way, using Corollary \ref{cor:quasimodesstraightnearcorners} as the replacement for Corollary \ref{cor:quasimodesfullycurv}.  There is still a self-adjoint Dirac operator on a quantum path $\mathcal{L}$ whose eigenvalues coincide with the quasi-eigenvalues of $\mathcal{Z}$, see Proposition \ref{prop:DiracL}. We cannot use Proposition \ref{prop:nomoreN} at this stage because the square of this Dirac operator cannot be decomposed as a direct sum of two quantum graph (non-Robin) Laplacians, but we can use Proposition \ref{prop:nomoreNexc} instead.  In addition, the analogue of Corollary \ref{cor:approxofquasibypsi} still holds for a zigzag, by an even easier argument.

\subsection{Asymptotics of eigenfunctions}\label{subsec:asymptef} We note that we have not yet proved Theorem \ref{thm:main} stating in particular that each true eigenvalue corresponds to a quasi-eigenvalue. Assume, however, for the rest of this section that  it holds.  We can deduce the following generalisation of Theorem \ref{thm:maineigenfunction}.

\begin{theorem}\label{thm:maineigenfunctiongen} For any curvilinear polygon $\mathcal P$, there exists $C>0$ such that the restrictions of the eigenfunctions $u_m$ to the boundary $\partial{\mathcal P}$ satisfy
\begin{equation}\label{eq:objective1}
\|u_m|_{\partial\mathcal{P}}-\tilde\Psi^{(m)}\|=O(m^{-\varepsilon}),
\end{equation}
where $\tilde\Psi^{(m)}$ is a linear combination of the functions $\Psi^{(l)}$ corresponding to quasi-eigenvalues $\sigma_l$ in the interval $[\sigma_m-Cm^{-\varepsilon},\sigma_m+Cm^{-\varepsilon}]$.
\end{theorem}
Note that Theorem \ref{thm:maineigenfunctiongen} immediately implies Theorem \ref{thm:maineigenfunction}, since under the  assumptions of the latter we must have $\tilde\Psi^{(m)}$ equal to a multiple of $\Psi^{(m)}$ itself, and $\Psi^{(m)}$ is a trigonometric polynomial of frequency $\sigma_m$ along each edge. It also implies the following slight variation: 

\begin{corollary}\label{cor:split} Suppose that a curvilinear polygon $\mathcal P$ has $K\ge 2$ exceptional angles (and therefore $K$ exceptional boundary components $\mathcal{Y}_\kappa$, $\kappa=1,\dots,K$). Let $C$ be as in   Theorem \ref{thm:maineigenfunctiongen}, let $\sigma_m$ be a quasi-eigenvalue, and let 
\[
\mathfrak{K}_m:=\bigcup_\kappa \mathcal{Y}_\kappa,
\]
where the union is taken over all $\kappa$ such that \eqref{eq:quasieq2} has a root $\sigma$ in the interval $[\sigma_m-Cm^{-\varepsilon},\sigma_m+Cm^{-\varepsilon}]$. Then
\[
\|u_m|_{\partial\mathcal{P}\setminus\mathfrak{K}_m}\|=O(m^{-\varepsilon}).
\]
\end{corollary}

In other words, Corollary \ref{cor:split} states that in the exceptional case the boundary values of Steklov eigenfunctions are asymptotically concentrated on the unions of exceptional components contributing to the particular clusters of eigenvalues. 

\begin{proof}[Proof of Theorem \ref{thm:maineigenfunctiongen}]
By the clustering argument above, the quasi-eigenvalues $\{\sigma_m\}$ separate into clusters of width bounded by $Cm^{-\varepsilon}$. By Theorem \ref{thm:main}, the same is true for the eigenvalues $\{\lambda_m\}$. At the cost of possibly increasing $C$, we may assume the clusters for both the eigenvalues and quasi-eigenvalues are the same. 

Now pick a cluster $\mathcal C_k$ with $N_k$ eigenvalues and $N_k$ quasi-eigenvalues, with indices from $m=a$ to $m=b=a+N_k-1$. For each $m$ with $\sigma_m\in\mathcal C_k$, by Corollary \ref{cor:approxofquasibypsi}, $\Psi^{(m)}$ is within $Cm^{-\varepsilon}$ of a linear combination $\tilde u_m$ of eigenfunctions with eigenvalues in $\mathcal C_k$. Note that $\Psi^{(m)}$ has norm $1$ and therefore each $\tilde u_m$ has norm within $Cm^{-\varepsilon}$ of $1$ (and is therefore within $Cm^{-\varepsilon}$ of its normalised version). This shows that there exist two $N_k$-by-$N_k$ matrices $G_k$ and $H_k$ for which
\[
\begin{pmatrix}\tilde u_a\\ \vdots\\ \tilde u_b\end{pmatrix}=G_k\begin{pmatrix}\Psi^{(a)}\\ \vdots\\ \Psi^{(b)}\end{pmatrix}+O(m^{-\varepsilon}),\quad \begin{pmatrix}\tilde u_a\\ \vdots\\ \tilde u_b\end{pmatrix}=H_k\begin{pmatrix}u_a\\ \vdots\\ u_b\end{pmatrix}+O(m^{-\varepsilon}).
\]
However, by Corollary \ref{cor:linind}, $H_k$ is invertible for sufficiently large $k$, and by Corollary \ref{cor:almostorthogonality} its inverse is uniformly bounded. We deduce
\[
\begin{pmatrix} u_a\\ \vdots\\ u_b\end{pmatrix}=H_k^{-1}G_k\begin{pmatrix}\Psi^{(a)}\\ \vdots\\ \Psi^{(b)}\end{pmatrix}+O(m^{-\varepsilon}),
\]
which is precisely Theorem \ref{thm:maineigenfunctiongen}.
\end{proof}
We also prove Proposition \ref{prop:equidistrib}, which now becomes very simple.
 
\begin{proof}[Proof of Proposition \ref{prop:equidistrib}] If all angles are special,  the quasi-eigenvalues are given by \eqref{eq:allspecial}, with each nonzero eigenvalue having multiplicity two. The corresponding $\Psi^{(m)}$ may be taken to be
\[
\sqrt{\frac{2}{|\partial\mathcal{P}|}}\sin(\sigma_m s),\quad \sqrt{\frac{2}{|\partial\mathcal{P}|}}\cos(\sigma_m s),
\]
where $s$ is an arc length coordinate along the boundary; note that these functions are orthogonal for each $m$. Therefore the functions $\tilde\Psi^{(m)}$ are linear combinations of $\Psi^{(m)}$, and each has norm $1+O(m^{-\varepsilon})$ by \eqref{eq:objective1}. An immediate calculation shows that these are equidistributed and in fact that the error is $O(m^{-\varepsilon})$. On the other hand, if all angles are exceptional, for all $m$ we may choose $\Psi^{(m)}$ to have support on just \emph{one} side. If $\sigma_m$ is isolated in the sense of Proposition \ref{prop:equidistrib}, then for sufficiently large $m$ we have $\tilde\Psi^{(m)}=\Psi^{(m)}+O(m^{-\varepsilon})$, thus $u_m=\Psi^{(m)}+O(m^{-\varepsilon})$, from which the result follows. 
\end{proof}
\clearpage\section{Enumeration of quasi-eigenvalues}\label{sec:completeness}

\subsection{Matrix groups}\label{subsec:propscat}
Let $\alpha$ be a non-exceptional angle and let $\mathtt{A}:=\mathtt{A}(\alpha)$ and $\mathtt{B}:=\mathtt{B}(\ell,\sigma)$ be the vertex and side transfer matrices defined by equations \eqref{eq:Adef1} and \eqref{eq:Bdef1}, see also \eqref{eq:mualphadef}, \eqref{eq:a1a2defn} and Remarks  \ref{rem:Aproperties},  \ref{rem:Bproperties}, and \ref{rem:Xareeigenvectors}.  

Define
\begin{align*}
\mathcal{M}_1&=\left\{\mathtt{M}=\begin{pmatrix}p&q\\\overline{q}&\overline{p}\end{pmatrix}\,\middle|\, p,q\in\mathbb{C},\ \det \mathtt{M}=|p|^2-|q|^2=1\right\},\\
\mathcal{M}_{1,\mathtt{A}}&=\left\{\mathtt{M}=\begin{pmatrix}p&-\ir r\\\ir r&p\end{pmatrix}\,\middle|\, p,r\in\mathbb{R},\ \det \mathtt{M}=p^2-r^2=1\right\}\subset \mathcal{M}_1,\\
\mathcal{M}_{1,\mathtt{B}}&=\left\{\mathtt{M}=\begin{pmatrix}p&0\\0&\overline{p}\end{pmatrix}\,\middle|\, p\in\mathbb{C},\ \det \mathtt{M}=|p|^2=1\right\}\subset \mathcal{M}_1,
\end{align*}
It is easy to check that $\mathcal{M}_1$ is a group with respect to matrix multiplication, $\mathcal{M}_{1,\mathtt{A}}$ and $\mathcal{M}_{1,\mathtt{B}}$ are subgroups of $\mathcal{M}_{1}$, and that $\mathtt{A}(\alpha)\in\mathcal{M}_{1,\mathtt{A}}$ for any $\alpha\not\in\Eangles$, and  $\mathtt{B}(\ell,\sigma)\in\mathcal{M}_{1,\mathtt{B}}$ for any real $\ell$ and $\sigma$. It is also easy to check that any matrix from $\mathcal{M}_1$ maps the (real) linear subspace $\Conj$ of $\mathbb{C}^2$ onto itself. Therefore, this is also true for the matrices  $\mathtt{T}$ and $\mathtt{U}$ defined by \eqref{eq:Tdef1} and \eqref{eq:Udef} respectively.

\subsection{Representation of vectors and matrices on the universal cover} %
\label{subsec:rep}
We can naturally identify vectors $\mathbf{b}=\begin{pmatrix}b\\\overline{b}\end{pmatrix}\in\Conj$ with vectors $\mathbf{b}^\sha=\begin{pmatrix}\Re b\\\Im b\end{pmatrix}$
considered as elements of $\mathbb{R}^2$ (or just elements $b$ of $\mathbb{C}$). As an illustration, the vectors $\mathbf{N}$ and $\mathbf{D}$ defined in \eqref{eq:ND} give rise to
\begin{equation}\label{eq:NDshadef}
\mathbf{N}^{\sha}=\begin{pmatrix}1\\ 0\end{pmatrix},\qquad \mathbf{D}^{\sha}=\begin{pmatrix}0\\ 1\end{pmatrix}.
\end{equation}
The mapping $\mathbb{R}^2\to\Conj$ is defined by $\mathbf{b}=\mathtt{J}\mathbf{b}^\sha$, with 
\begin{equation}\label{eq:Jdef}
\mathtt{J}= \begin{pmatrix} 1&\ir\\1&-\ir\end{pmatrix},
\end{equation}
Matrices $\mathtt{M}\in\mathcal{M}_1$ therefore act on $\mathbb{R}^2$ as
\[
(\mathtt{M}\mathbf{b})^\sha=\mathtt{J}^{-1}\mathtt{M}\mathbf{b}=\mathtt{J}^{-1}\mathtt{M}\mathtt{J}\mathbf{b}^\sha,
\]
where 
\[
\mathtt{J}^{-1}= \frac{1}{2}\begin{pmatrix} 1&1\\-\ir&\ir\end{pmatrix},
\]
and we set 
\[
\mathtt{M}^\sha:=\mathtt{J}^{-1}\mathtt{M}\mathtt{J}.
\]

It is straightforward to check  that the mapping $\mathtt{M}\mapsto \mathtt{M}^\sha$ sends $\mathcal{M}_1$ into the space $\mathcal{M}^\sha_1$ of all real $2\times2$ matrices with determinant one. Moreover, it maps the subgroup $\mathcal{M}_{1,\mathtt{A}}\subset \mathcal{M}_1$ into the subgroup $\mathcal{S}^\sha\subset \mathcal{M}_1^\sha$ of all symmetric real $2\times2$ matrices with determinant one and equal diagonal entries, and the subgroup $\mathcal{M}_{1,\mathtt{B}}\subset \mathcal{M}_1$ into the subgroup $\mathcal{R}^\sha\subset \mathcal{M}_1^\sha$ of all  real $2\times2$ rotation matrices.

The matrices in $\mathcal{R}^\sha$ are characterised by a single parameter, the angle of rotation, and we will denote them by
\[
\mathtt{R}^\sha(\psi):=\begin{pmatrix} \cos(\psi) & -\sin(\psi)\\ \sin(\psi) & \cos(\psi)\end{pmatrix}\in\mathcal{R}^\sha.
\]
In particular, we have
\[
(\mathtt{B}(\ell,\sigma))^\sha=\mathtt{R}^\sha(\sigma\ell).
\]

The matrices in $\mathcal{S}^\sha$ always have normalised eigenvectors 
\[
\Cvect^\sha_\text{even}=\frac{1}{\sqrt{2}}\begin{pmatrix} 1\\-1\end{pmatrix},\qquad \Cvect^\sha_\text{odd}=\frac{1}{\sqrt{2}}\begin{pmatrix} 1\\1\end{pmatrix}.
\]
Thus, any $\mathtt{S}^\sha\in \mathcal{S}^\sha$ may be characterised  by one eigenvalue $\tau$ (the other eigenvalue being $\frac{1}{\tau}$) and a corresponding normalised eigenvector $\mathbf{w}^\sha\in\left\{\pm \Cvect^\sha_\text{even}, \pm \Cvect^\sha_\text{odd}\right\}$,  and we will write 
\[
\mathtt{S}^\sha=\mathtt{S}^\sha(\tau,\mathbf{w}^\sha).
\]
This representation is not unique:
\[
\mathtt{S}^\sha(\tau,\mathbf{w}^\sha)=\mathtt{S}^\sha(\tau,-\mathbf{w}^\sha)=\mathtt{S}^\sha(1/\tau,(\mathbf{w}^\sha)^\perp)=\mathtt{S}^\sha(1/\tau,-(\mathbf{w}^\sha)^\perp),
\]
where $(\mathbf{w}^\sha)^\perp$ is a normalised eigenvector perpendicular to $\mathbf{w}^\sha$. 

In particular, for $\mathtt{A}(\alpha)=\begin{pmatrix}a_1(\alpha) & -\ir a_2(\alpha) \\ \ir a_2(\alpha) & a_1(\alpha) \end{pmatrix}$ with eigenvalues $\eta_j(\alpha)$ defined in \eqref{eq:eta12}, we have
\begin{equation}\label{eq:Ahat}
\begin{split}
\mathtt{A}^\sha(\alpha):=(\mathtt{A}(\alpha))^\sha=\begin{pmatrix} a_1(\alpha) & -a_2(\alpha)\\ -a_2(\alpha) & a_1(\alpha)\end{pmatrix}
&=\mathtt{S}^\sha\left(\eta_1(\alpha),\Cvect^\sha_\text{odd}\right)=\mathtt{S}^\sha\left(\eta_1(\alpha),-\Cvect^\sha_\text{odd}\right)\\
&=\mathtt{S}^\sha\left(\eta_{2}(\alpha),\Cvect^\sha_\text{even}\right)=\mathtt{S}^\sha\left(\eta_2(\alpha),-\Cvect^\sha_\text{even}\right),
\end{split}
\end{equation}
The matrix $\mathtt{A}^\sha(\alpha)$ is positive or negative depending on the sign of $\sin(\mu_\alpha)$.

Throughout this section it will be useful to deal, instead of vectors $\mathbf{b}\in\Conj\setminus\{\mathbf{0}\}$ or $\mathbf{b}^\sha\in\mathbb{R}^2\setminus\{\mathbf{0}\}$, with vectors $\hat{\mathbf{b}}$ on the \emph{universal cover} $\UC$ of the punctured complex plane, that is, of the logarithmic surface. The elements $\hat{\mathbf{b}}\in\UC$ have positive moduli and arguments $\arg \hat{\mathbf{b}}\in(-\infty,+\infty)$.  Let $\Pro:\UC\to\mathbb{R}^2\setminus\{\mathbf{0}\}$ be the projection which preserves the modulus but takes argument modulo $2\pi$ in such a way that $\arg(\Pro\hat{\mathbf{b}})\in(-\pi,\pi]$. Any element $\hat{\mathbf{b}}\in\UC$ such that $\Pro\hat{\mathbf{b}}=\mathbf{b}^\sha\in\mathbb{R}^2\setminus\{\mathbf{0}\}$ will be called a \emph{lift} of $\mathbf{b}^\sha$ onto $\UC$.  We will distinguish the \emph{principal lift} $\Pro^{-1}: \mathbb{R}^2\setminus\{\mathbf{0}\}\to\UC$ such that $\arg(\Pro^{-1}\mathbf{b}^\sha)\in(-\pi,\pi]$. This allows us to lift previously defined vectors to the universal cover:
\begin{equation}\label{eq:special_args}
\hat{\mathbf N}=\Pro^{-1}\mathbf{N}^{\sha},\quad \hat{\mathbf{D}}=\Pro^{-1}\mathbf{D}^{\sha},\quad \hat{\Cvect}_\text{even}=\Pro^{-1}\Cvect^{\sha}_\text{even}, \quad \hat{\Cvect}_\text{odd}=\Pro^{-1}\Cvect^{\sha}_\text{odd}.
\end{equation}

We now need to define the analogues of matrices $\mathtt{M}^\sha\in\mathcal{R}^\sha$ and  $\mathtt{M}^\sha\in\mathcal{S}^\sha$ acting on the universal cover. They will be maps $\hat{\mathtt{M}}$ from $\UC$ to $\UC$, which we will call \emph{lifted matrices}, defined in the following way. Firstly, we require, for any $\hat{\mathbf{b}}\in\UC$,
\begin{equation}\label{eq:Prohat}
\Pro\left(\hat{\mathtt{M}}\hat{\mathbf{b}}\right):=\mathtt{M}^\sha\Pro\hat{\mathbf{b}}=\mathtt{M}^\sha\mathbf{b}^\sha.
\end{equation}
The relation \eqref{eq:Prohat} defines the modulus of $\hat{\mathtt{M}}\hat{\mathbf{b}}$ uniquely, and its argument modulo $2\pi$. 
We prescribe the exact value of $\arg\hat{\mathtt{M}}\hat{\mathbf{b}}$ in two distinct ways depending on whether $\mathtt{M}^\sha\in\mathcal{R}^\sha$ or  $\mathtt{M}^\sha\in\mathcal{S}^\sha$. 

In the former case $\mathtt{M}^\sha=\mathtt{R}^\sha(\psi)$,  we set
\[
\arg\left(\hat{\mathtt{R}}(\psi)\hat{\mathbf{b}}\right):=\arg\hat{\mathbf{b}}+\psi.
\]
We remark that although the matrix-valued function $\mathtt{R}^\sha(\psi)$ is $2\pi$-periodic in $\psi\in\mathbb{R}$, the maps $\hat{\mathtt{R}}(\psi)$ and $\hat{\mathtt{R}}(\psi+2\pi)$ are different, and therefore $\hat{\mathtt{R}}(\psi)$ should not be viewed as a ``lift'' of $\mathtt{R}^\sha(\psi)$ onto the universal cover, but rather as an independent object depending on the parameter $\psi$.

In the latter case, we note that any $\mathtt{M}^\sha=\mathtt{S}^\sha(\tau, \mathbf{w}^\sha)\in \mathcal{S}^\sha$ is either positive definite or negative definite, and we deal first with the positive ones, requesting that,  for any $\hat{\mathbf{b}}\in\UC$,
\begin{equation}\label{eq:hatbound}
\left|\arg\left(\hat{\mathtt{S}}\hat{\mathbf{b}}\right)-\arg\hat{\mathbf{b}}\right|<\frac{\pi}{2}.
\end{equation}
The conditions \eqref{eq:Prohat} and \eqref{eq:hatbound} define $\hat{\mathtt{S}}\hat{\mathbf{b}}$ uniquely. A more explicit formula for $\arg\left(\hat{\mathtt{S}}\hat{\mathbf{b}}\right)$ is given below in Lemma \ref{lem:explicitarg}. 

If $\mathtt{S}^\sha$ is negative, we choose
\begin{equation}
\label{arg:choiceneg}
\arg\left(\hat{\mathtt{S}}\hat{\mathbf{b}}\right)=\arg\left(\hat{(-\mathtt{S})}\hat{\mathbf{b}}\right)+\pi.
\end{equation}

If $\mathtt{S}^\sha=\mathtt{S}^\sha(\tau, \mathbf{w}^\sha)\in \mathcal{S}^\sha$, and $\hat{\mathbf{w}}$ is any lift of $\mathbf{w}^\sha$, we will denote the corresponding map on the universal cover  as $\hat{\mathtt{S}}=\hat{\mathtt{S}}(\tau, \hat{\mathbf{w}})$ and call $\hat{\mathbf{w}}$ an \emph{eigenvector} of $\hat{\mathtt{S}}$ corresponding to the eigenvalue $\tau$. Of course such a representation is not unique. We will say that $\hat{\mathtt{S}}$ is \emph{positive} or \emph{negative} if the corresponding $\mathtt{S}^\sha$ (or $\tau$) is positive or negative, respectively.

Let us introduce the set
\begin{equation}\label{eq:Xcalhat}
\hat{\mathcal{X}}:=\left\{\hat{\mathbf{w}}\in\UC:\ \Pro\hat{\mathbf{w}} \in\left\{\pm \Cvect^\sha_\text{even}, \pm \Cvect^\sha_\text{odd}\right\}\right\}
=\left\{\hat{\mathbf{w}}\in\UC:\ |\hat{\mathbf{w}}|=1, \arg\hat{\mathbf{w}}=\frac{\pi}{4}\pmod{\frac{\pi}{2}}\right\}.
\end{equation}
The set $\hat{\mathcal{X}}$ consists of all the lifts onto the universal cover of all normalised eigenvectors of matrices $\mathtt{S}^\sha$.  The elements of $\hat{\mathcal{X}}$ divide $\UC$ into \emph{quadrants} of argument width $\frac{\pi}{2}$: for any $\hat{\mathbf{b}}\in\UC$ there exist the elements $\hat{\mathbf{w}}_1,\hat{\mathbf{w}}_2\in \hat{\mathcal{X}}$ (which depend on $\hat{\mathbf{b}}$) such that
\[
\arg\hat{\mathbf{w}}_1\le \arg\hat{\mathbf{b}}<\arg\hat{\mathbf{w}}_2=\arg\hat{\mathbf{w}}_1+\pi/2.
\]
The following lemma gives an explicit expression for $\arg\hat{\mathtt{S}}\hat{\mathbf{b}}$ in terms of $\arg\hat{\mathbf{b}}$ and $\arg\hat{\mathbf{w}}_1$.

\begin{lemma}\label{lem:explicitarg} Let $\mathtt{S}^\sha\in\mathcal{S}^\sha$ be positive, let $\hat{\mathbf{b}}\in\UC$, and let $\hat{\mathbf{w}}_j\in\hat{\mathcal{X}}$, $j=1,2$ and $\hat{\mathtt{S}}=\hat{\mathtt{S}}(\tau,\hat{\mathbf{w}}_1)$ as above. Then
\begin{equation}\label{arg:choice}
\arg\left(\hat{\mathtt{S}}\hat{\mathbf{b}}\right)=\arg \hat{\mathbf{w}}_1+\operatorname{arctan}\left(\frac{1}{\tau^2}\tan\left(\arg \hat{\mathbf{b}}-\arg\hat{\mathbf{w}}_1\right)\right)
\end{equation}
\end{lemma}
\begin{proof} Let $\mathbf{b}^\sha=\Pro\hat{\mathbf{b}}$ and $\mathbf{w}^\sha_j=\Pro\hat{\mathbf{w}}_j$, $j=1,2$. Write $\mathbf{b}^\sha$ in the basis $\mathbf{w}^\sha_1$, $\mathbf{w}^\sha_2$:  $\mathbf{b}^\sha=c_1\mathbf{w}^\sha_1+c_2\mathbf{w}^\sha_2$. Then $\mathtt{S}^\sha\mathbf{b}^\sha=c_1\tau\mathbf{w}^\sha_1+c_2\tau^{-1}\mathbf{w}^\sha_2$. The result then follows by some elementary trigonometry and by lifting $\mathtt{S}^\sha\mathbf{b}^\sha$ back to $\hat{\mathtt{S}}\hat{\mathbf{b}}$ with account of \eqref{eq:hatbound}.
\end{proof}

We have the following important monotonicity result.

\begin{lemma}\label{lem:monotone} Let $\hat{\bm{\xi}}_1, \hat{\bm{\xi}}_2\in\UC$ with $\arg\hat{\bm{\xi}}_1<\arg\hat{\bm{\xi}}_2$. Then for any $\hat{\mathtt{S}}=\hat{\mathtt{S}}(\tau,\hat{\mathbf{w}})\in\hat{\mathcal{S}}$ we have 
\begin{equation}\label{eq:monotone}
\arg\left(\hat{\mathtt{S}}\hat{\bm{\xi}}_1\right)<\arg\left(\hat{\mathtt{S}}\hat{\bm{\xi}}_2\right).
\end{equation}
\end{lemma}

\begin{proof} Without loss of generality we can assume $\hat{\mathtt{S}}$ to be positive (that is, take $\tau>0$) , otherwise we just consider $-\hat{\mathtt{S}}$ and add $\pi$ to both arguments in \eqref{eq:monotone}.   If $\hat{\bm{\xi}}_1, \hat{\bm{\xi}}_2$ lie in different quadrants with respect to eigenvectors $\pm\hat{\mathbf{w}}$, $\pm\hat{\mathbf{w}}^\perp$ of $\hat{\mathtt{S}}$, the result is immediate by our definition of the action of $\hat{\mathtt{S}}$ on $\UC$. Suppose they lie in the same quadrant
\[
\arg\hat{\mathbf{w}}_1\le\arg\hat{\bm{\xi}}_1<\arg\hat{\bm{\xi}}_2\le\arg\hat{\mathbf{w}}_2=\arg\hat{\mathbf{w}}_1+\frac{\pi}{2},
\]
where $\hat{\mathbf{w}}_1, \hat{\mathbf{w}}_2$ are two orthogonal eigenvectors of $\hat{\mathtt{S}}$ corresponding to eigenvalues $\tau>0$ and $1/\tau$ respectively.
Then the result follows from \eqref{arg:choice} applied to $\hat{\mathbf{b}}=\hat{\bm{\xi}}_i$ as both $\tan$ and $\operatorname{arctan}$ are monotone increasing on $(0,\pi/2)$, $(0,+\infty)$, respectively.
\end{proof} 

Let $n\ge 1$, $\balpha=(\alpha_1,\dots,\alpha_n)$, $\balpha'=(\alpha_1,\dots,\alpha_{n-1})$, $\bell=(\ell_1,\dots,\ell_n)$, and consider now the matrices $\mathtt{T}=\mathtt{T}(\balpha,\bell,\sigma)$ defined by \eqref{eq:Tdef1} and $\mathtt{U}=\mathtt{U}(\balpha',\bell,\sigma)$ defined by \eqref{eq:Udef}. When acting on $\UC$ they become
\[
\hat{\mathtt{T}}=\hat{\mathtt{T}}(\balpha,\bell,\sigma)=\hat{\mathtt{A}}(\alpha_n)\hat{\mathtt{R}}(\sigma\ell_n)\cdots\hat{\mathtt{A}}(\alpha_1)\hat{\mathtt{R}}(\sigma\ell_1),
\]
and
\[
\hat{\mathtt{U}}=\hat{\mathtt{U}}(\balpha',\bell,\sigma)=\hat{\mathtt{R}}(\sigma\ell_{n})\hat{\mathtt{A}}(\alpha_{n-1})\hat{\mathtt{R}}(\sigma\ell_{n-1})\cdots\hat{\mathtt{A}}(\alpha_1)\hat{\mathtt{R}}(\sigma\ell_1),
\]
and we have
\begin{lemma}\label{lem:monotoneT} Let $\hat{\bm{\xi}}_1, \hat{\bm{\xi}}_2\in\UC$ with $\arg\hat{\bm{\xi}}_1<\arg\hat{\bm{\xi}}_2$. Then 
\begin{equation}\label{eq:monotoneTxi}
\arg\left(\hat{\mathtt{T}}(\balpha,\bell,\sigma)\hat{\bm{\xi}}_1\right)<\arg\left(\hat{\mathtt{T}}(\balpha,\bell,\sigma)\hat{\bm{\xi}}_2\right)\quad\text{for any }\sigma,
\end{equation}
and for any $\sigma_1<\sigma_2$, 
\begin{equation}\label{eq:monotoneTsigma}
\arg\left(\hat{\mathtt{T}}(\balpha,\bell,\sigma_1)\hat{\bm{\xi}}\right)<\arg\left(\hat{\mathtt{T}}(\balpha,\bell,\sigma_2)\hat{\bm{\xi}}\right)\quad\text{for any }\bm{\xi}\in\UC.
\end{equation}
Moreover, $\arg\left(\hat{\mathtt{T}}(\balpha,\bell,\sigma)\hat{\bm{\xi}}\right)$ is continuous in $\sigma$ for any $\bm{\xi}\in\UC$.

All these statements remain true if $\hat{\mathtt{T}}(\balpha,\bell,\sigma)$ is replaced by $\hat{\mathtt{U}}(\balpha',\bell,\sigma)$.
\end{lemma}

\begin{proof}  To prove \eqref{eq:monotoneTxi} we just notice that any rotation matrix increases the argument by the same amount, and apply Lemma \ref{lem:monotone}
when acting by each matrix $\hat{\mathtt{A}}(\alpha_j)$. To prove \eqref{eq:monotoneTsigma} we notice that the rotation matrices increase the arguments of vectors they act upon monotonically in $\sigma$, and the matrices $\hat{\mathtt{A}}$ are $\sigma$-independent. The continuity statement is obvious.
\end{proof}

\subsection{Enumeration of quasi-eigenvalues for non-exceptional zigzags}
\label{sec:zigzags}
Let $\mathcal{Z}=\mathcal{Z}^{(\aleph\beth)}=\mathcal{Z}^{(\aleph\beth)}(\bm{\alpha},\bm{\ell})$ be a non-exceptional zigzag and let $\hat{\mathtt{U}}_\mathcal{Z}(\sigma):=\hat{\mathtt{U}}(\bm{\alpha},\bm{\ell}, \sigma)$ be the corresponding zigzag matrix acting on the universal cover $\UC$. Recall that the quasi-eigenvalues of $\mathcal{Z}$  are defined by  equation \eqref{eq:quasizigzag}, see also \eqref{eq:quasizigzagparallel}. 
Let us give an equivalent definition in terms of the action of the matrix $\hat{\mathtt{U}}_\mathcal{Z}$ on the vectors $\hat{\baleph},\hat{\bbeth}\in\{\hat{\mathbf{N}},\hat{\mathbf{D}}\}$.

\begin{remark}\label{rem:U0pos} For the rest of this section we will assume that all the matrices $\hat{\mathtt{U}}_\mathcal{Z}(0)$ (which are  products of symmetric matrices) are positive. If this is not the case, we can just formally replace  $\hat{\mathtt{U}}_\mathcal{Z}(\sigma)$ by $-\hat{\mathtt{U}}_\mathcal{Z}(\sigma)$, and the vector $\hat{\bbeth}$ by 
$-\hat{\bbeth}$ throughout. (Similarly to \eqref{arg:choiceneg}, we understand $-\hat{\bbeth}$ as $\hat{\mathtt{R}}(\pi)\hat{\bbeth}$, so that $-(-\hat{\bbeth})$ is $\hat{\bbeth}$ rotated by $2\pi$ rather than $\hat{\bbeth}$.)
\end{remark}

A real number $\sigma$ is a quasi-egenvalue of a $\aleph \beth$-zigzag $\mathcal{Z}$ if and only if 
\begin{equation}\label{eq:hatzigzag}
\arg \left(\hat{\mathtt{U}}_\mathcal{Z}(\sigma)\, \hat{\mathbf\baleph}\right) =\arg\hat{\bbeth}  \pmod{\pi},
\end{equation}
which should be used together with \eqref{eq:special_args}. Equivalently, \eqref{eq:hatzigzag} may be re-stated as
\begin{equation}\label{eq:hatzigzaginv}
-\arg \left(\hat{\mathtt{U}}^{-1}_\mathcal{Z}(\sigma)\, \hat{\mathbf\bbeth}\right) =-\arg{\hat{\baleph}} \pmod{\pi},
\end{equation}
or, if we set
\begin{equation}\label{eq:phialephbethdef}
\begin{split}
\varphi_{\mathcal{Z}}(\sigma)=\varphi_{\mathcal{Z}^{(\aleph\beth)}}(\sigma)&:=\frac{\arg \left(\hat{\mathtt{U}}_\mathcal{Z}(\sigma)\, \hat{\baleph}\right)-\arg\hat{\bbeth}}{\pi},\\
\tilde{\varphi}_{\mathcal{Z}}(\sigma)=\tilde{\varphi}_{\mathcal{Z}^{(\aleph\beth)}}(\sigma)&:=\frac{-\arg \left(\hat{\mathtt{U}}^{-1}_\mathcal{Z}(\sigma)\, \hat{\bbeth}\right)+\arg\hat{\baleph}}{\pi},
\end{split}
\end{equation}
as
\begin{equation}\label{eq:phicond}
\varphi_{\mathcal{Z}}(\sigma)\in\mathbb{Z}\qquad\text{or}\qquad\tilde{\varphi}_{\mathcal{Z}}(\sigma)\in\mathbb{Z}.
\end{equation}

Similarly to Proposition \ref{prop:discrete},  one can show that  the solutions of any of the equations \eqref{eq:hatzigzag}  form a discrete set.  
Therefore, the set of such solutions could be viewed as a  sequence of real numbers $\{\sigma_m^{(\aleph \beth)}\}_{m=-\infty}^{+\infty}$ which is monotone increasing with $m$.
In order to fix enumeration of this sequence we need to specify the element  $\sigma_1^{(\aleph \beth)}$. Alternatively, we can prescribe the definitions of \emph{zigzag quasi-eigenvalue  counting functions}
\[
\mathcal{N}^q_{\mathcal{Z}^{(\aleph\beth)}}(\sigma)=\#\{m\in\mathbb{N}\mid \sigma_m^{(\aleph \beth)}\le \sigma\},
\]
which are only defined a priori modulo addition of an integer. This is done according to the following
 
\begin{definition}\label{defn:naturalenumeration} The {\em natural enumeration} of the quasi-eigenvalues of a zigzag  $\mathcal{Z}^{(\aleph\beth)}$ is defined by setting
\[
\mathcal{N}^q_{\mathcal{Z}^{(\aleph\beth)}}(\sigma):=
\begin{cases}
\left[\varphi_{\mathcal{Z}^{(\aleph\beth)}}(\sigma)\right]+1,\quad&\text{if\ \ }\aleph=N,\\
\left[\varphi_{\mathcal{Z}^{(\aleph\beth)}}(\sigma)\right],\quad&\text{if\ \ }\aleph=D,\\
\end{cases}
\]
where $[\cdot]$ denotes the integer part. 
\end{definition}

In order to reformulate Definition \ref{defn:naturalenumeration} in terms of specifying the element  $\sigma_1^{(\aleph \beth)}$, we need to look at the values of $\varphi_{\mathcal{Z}^{(\aleph\beth)}}(0)$. We recall that the corresponding zigzag matrix $\hat{\mathtt{U}}(0)$ is  just  a product of symmetric matrices $\hat{\mathtt{A}}(\alpha_{n-1})\cdots\hat{\mathtt{A}}(\alpha_{1})$ and therefore has eigenvectors $\pm\hat{\Cvect}_\mathrm{even}$ and $\pm\hat{\Cvect}_\mathrm{odd}$ whose arguments 
are odd multiples of $\frac{\pi}{4}$. Thus, applying definition \eqref{eq:hatbound} and \eqref{arg:choiceneg}, and recalling Remark \ref{rem:U0pos} we deduce that
\[
\arg \left(\hat{\mathtt{U}}_\mathcal{Z}(0)\, \hat{\mathbf{N}}\right)\in
\left(-\frac{\pi}{4}, \frac{\pi}{4}\right)
\qquad
\text{and}
\qquad
\arg \left(\hat{\mathtt{U}}_\mathcal{Z}(0)\, \hat{\mathbf{D}}\right)\in
\left(\frac{\pi}{4}, \frac{3\pi}{4}\right).
\]

We now consider four zigzag problems separately.
\begin{proposition}\label{prop:startcount}\ 
\begin{itemize}
\item[(i)] For an $NN$-zigzag, if $\arg\left(\hat{\mathtt{U}}_\mathcal{Z}(0)\, \hat{\mathbf{N}}\right)\ge 0$, then $\sigma_1^{(NN)}$ is  the \emph{first non-positive} quasi-eigenvalue
(i.e., the non-positive quasi-eigenvalue with the smallest absolute value), otherwise $\sigma_1^{(NN)}$ is the \emph{first positive} quasi-eigenvalue.
\item[(ii)] For an $ND$-zigzag, $\sigma_1^{(ND)}$ is the \emph{first positive} quasi-eigenvalue.
\item[(iii)] For a $DN$-zigzag, $\sigma_1^{(DN)}$ is the \emph{first positive} quasi-eigenvalue.
\item[(iv)] For a $DD$-zigzag, if $\arg\left(\hat{\mathtt{U}}_\mathcal{Z}(0)\, \hat{\mathbf{D}}\right)\ge \frac{\pi}{2}$, then $\sigma_1^{(DD)}$ is  the \emph{first positive} quasi-eigenvalue, otherwise  $\sigma_1^{(DD)}$ is the \emph{second positive} quasi-eigenvalue.
\end{itemize}
\end{proposition}

\begin{proof} It is sufficient to check that the counting functions induced by the choice of $\sigma_1$ in the statements matches Definition \ref{defn:naturalenumeration} at one value of $\sigma$, say, $\sigma=0$.
\begin{itemize}
\item[(i)] For an $NN$-zigzag,  the formula \eqref{eq:phialephbethdef} and Definition \ref{defn:naturalenumeration} yield 
\[
\mathcal{N}^q_{\mathcal{Z}^{(NN)}}(0)=\begin{cases}
1,\quad&\text{if\ \ }\arg \left(\hat{\mathtt{U}}_\mathcal{Z}(0)\, \hat{\mathbf{N}}\right)\ge 0,\\
0,\quad&\text{if\ \ }\arg \left(\hat{\mathtt{U}}_\mathcal{Z}(0)\, \hat{\mathbf{N}}\right)<0,
\end{cases}
\]
implying the result.
\item[(ii)] For an $ND$-zigzag,  by \eqref{eq:phialephbethdef} and Definition \ref{defn:naturalenumeration}, 
\[
\mathcal{N}^q_{\mathcal{Z}^{(ND)}}(0)=0,
\]
hence the result.
\item[(iii)] Similarly, for a $DN$-zigzag, 
\[
\mathcal{N}^q_{\mathcal{Z}^{(DN)}}(0)=0.
\]
\item[(iv)] For a $DD$-zigzag, again by formula \eqref{eq:phialephbethdef} and Definition \ref{defn:naturalenumeration},
\[
\mathcal{N}^q_{\mathcal{Z}^{(DD)}}(0)=\begin{cases}
0,\quad&\text{if\ \ }\arg \left(\hat{\mathtt{U}}_\mathcal{Z}(0)\, \hat{\mathbf{D}}\right)\ge \frac{\pi}{2},\\
-1,\quad&\text{if\ \ }\arg \left(\hat{\mathtt{U}}_\mathcal{Z}(0)\, \hat{\mathbf{D}}\right)<\frac{\pi}{2},
\end{cases}
\]
implying the result.
\end{itemize}
\end{proof}

The following result will be useful for expressing the quasi-eigenvalues counting functions in terms of $\varphi_{\mathcal{Z}}(\sigma)$. 

\begin{lemma}\label{lem:phiandtildephi} Consider a zigzag  $\mathcal{Z}^{(\aleph\beth)}$. Then for all $\sigma\in\mathbb{R}$,
\[
\left[\tilde{\varphi}_{\mathcal{Z}^{(\aleph\beth)}}(\sigma)\right]=\left[\varphi_{\mathcal{Z}^{(\aleph\beth)}}(\sigma)\right].
\]
\end{lemma}

\begin{proof} By \eqref{eq:phicond}, the two expressions may differ only by an integer as they have jumps at the same points. Therefore it is enough to check the equality for $\sigma=0$. This is done exactly in the same manner as in the  proof of  Proposition \ref{prop:startcount}.
\end{proof}
 
In general, the functions $\varphi_{\mathcal{Z}^{(\aleph\beth)}}(\sigma)$ and $\tilde{\varphi}_{\mathcal{Z}^{(\aleph\beth)}}(\sigma)$ are not the same, although their integer parts coincide. It is easy to check that both  functions  are smooth.  Moreover, they are strictly monotone with the derivatives bounded away from zero. Namely, we have the following result which strengthens  Lemma \ref{lem:monotone}.
\begin{lemma}\label{lem:phi}
There exist constants $C_1,C_2>0$ such that $C_1\le \varphi'_{\mathcal{Z}^{(\aleph\beth)}}(\sigma)<C_2$ and  $C_1<\tilde{\varphi}'_{\mathcal{Z}^{(\aleph\beth)}}(\sigma)<C_2$ for all $\sigma>0$.
\end{lemma} 
\begin{proof} We will work with the function $\varphi_{\mathcal{Z}^{(\aleph\beth)}}(\sigma)$; the reasoning for  $\tilde{\varphi}_{\mathcal{Z}^{(\aleph\beth)}}(\sigma)$ will be similar. In view of the definition of the matrix $\hat{\mathtt{U}}$, the function  $\varphi_{\mathcal{Z}^{(\aleph\beth)}}(\sigma)$  is equal to the cumulative changes of the argument under the action of  rotation matrices $\hat{\mathtt{R}}(\sigma \ell_i)$ and   symmetric matrices  belonging to $\hat{\mathcal{S}}$ which are independent of $\sigma$.  The rotation matrices increase the argument linearly in $\sigma$. Now apply formula \eqref{arg:choice} with $\hat{\mathtt{S}}=\hat{\mathtt{A}}(\alpha_1)$ and $\hat{\mathbf{b}}=\hat{\mathtt{R}}(\sigma \ell_1)\hat{\baleph}$ together with the chain rule. This leads to the bound
\[
C_{1,1}\le \frac{\dr }{\dr \sigma}\arg\left(\hat{\mathtt{A}}(\alpha_1)\hat{\mathtt{R}}(\sigma \ell_1)\hat{\baleph}\right)<C_{1,2}
\]
for some constants $0<C_{1,1}<C_{1,2}$. Applying this observation  iteratively to the matrices arising in the representation of   $\hat{\mathtt{U}}$ we obtain the desired inequalities.
\end{proof} 
 
We immediately have
\begin{corollary}\label{cor:diffbounded} The difference $\sigma_{m+1}-\sigma_m$ for a non-exceptional zigzag is bounded away from zero.
\end{corollary} 
 
Our next goal is to prove Theorem   \ref{thm:quasizigzags}.
The result follows from the following two propositions.
\begin{proposition} 
\label{lemma1}
Theorem \ref{thm:quasizigzags} holds for partially curvilinear zigzags with one side and for straight zigzags with two equal sides.
\end{proposition}
\begin{proposition}
\label{lemma2}
Let $\mathcal{Z}:=Z_{PQ}^{(\aleph\beth)}$ be a partially curvilinear zigzag with endpoints $P$ and $Q$, and let $W\in \mathcal{Z}$ be a point which is not a vertex and such that the zigzag $\mathcal{Z}$ is straight in some neighbourhood of $W$. The point $W$ splits $\mathcal{Z}$ into two partially curvilinear zigzags $\mathcal{Z}_{PW}$, starting at $P$ and ending at $W$, and $\mathcal{Z}_{WQ}$, starting at $W$ and ending at $Q$.  Impose a boundary condition $\daleth\in\{D,N\}$ at $W$. If Theorem  \ref{thm:quasizigzags} holds for both $\mathcal{Z}_\mathrm{I}:=\mathcal{Z}_{PW}^{(\aleph\daleth)}$ and $\mathcal{Z}_\mathrm{II}:=\mathcal{Z}_{WQ}^{(\daleth\beth)}$ then it also holds for $\mathcal{Z}_{PQ}^{(\aleph\beth)}$.
\end{proposition}

To prove Theorem  \ref{thm:quasizigzags} for an arbitrary partially curvilinear zigzag it remains simply to note that any partially curvilinear zigzag can be represented as a union of partially curvilinear zigzags with one side and straight zigzags with two equal sides, see Figure \ref{fig:zigzagdecomposition}.
\begin{figure}[htb]
\begin{center}
\includegraphics{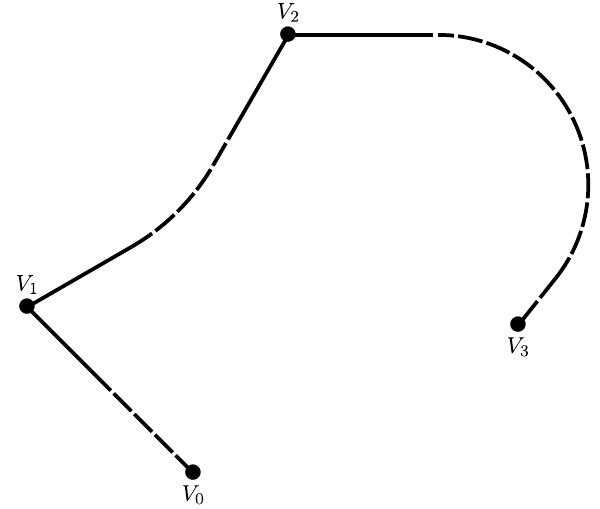}
\end{center}
\caption{Decomposition of a three-arc partially curvilinear zigzag into the union of two two-piece straight zigzags with equal sides (solid lines) and three partially  curvilinear one-piece zigzags (dashed lines)\label{fig:zigzagdecomposition}}
\end{figure}
 
\subsection{Proof of Proposition \ref{lemma1}}
Consider first  a zigzag $\mathcal{Z}_1$  consisting of one side of length $\ell$. The corresponding matrix is given by 
\[
\hat{\mathtt{U}}_{\mathcal{Z}_1}=\hat{\mathtt{R}}(\ell\sigma),
\]
and therefore
\[
\varphi_{\mathcal{Z}_1^{(\aleph\beth)}}(\sigma)=\ell\sigma+\arg\hat{\baleph}-\arg\hat{\bbeth},
\]
leading,  by Definition \ref{defn:naturalenumeration}, to
\begin{equation}\label{eq:NZ1}
\mathcal{N}^q_{\mathcal{Z}_1^{(NN)}}(\sigma)=[\ell\sigma]+1,\quad
\mathcal{N}^q_{\mathcal{Z}_1^{(DD)}}(\sigma)=[\ell\sigma],\quad
\mathcal{N}^q_{\mathcal{Z}_1^{(ND)}}(\sigma)=\mathcal{N}^q_{\mathcal{Z}_1^{(DN)}}(\sigma)=\left[\ell\sigma+\frac{1}{2}\right].
\end{equation}
At the same time, it  follows from Corollary \ref{cor:curvsloshing} (which is applicable since, according to Definition \ref{def:zigzagdomain}, zigzag domains always have angles $\pi/2$ at the ends of the corresponding zigzag, see Figure \ref{fig:zigzag}) that $\lambda^{(\aleph\beth)}_m-\sigma^{(\aleph\beth)}_m=o(1)$, where
\[
 \ell\sigma_m^{(NN)}=\pi(m-1),\quad
 \ell\sigma_m^{(DD)} = \pi m, \quad
 \ell\sigma_m^{(ND)}= \ell\sigma_m^{(DN)}=\pi\left(m-\frac{1}{2}\right),\qquad m\in\mathbb{N},
\]
which is in agreement with \eqref{eq:NZ1}. This proves Proposition \ref{lemma1} for a one-sided zigzag.

\smallskip

Consider now a zigzag $\mathcal{Z}_2:=\mathcal{Z}((\alpha),(\ell,\ell))$ with two equal straight sides of length $\ell$ and the angle $\alpha$ between them. The corresponding zigzag matrix is given by 
\[
\hat{\mathtt{U}}_{\mathcal{Z}_2}(\sigma)=\hat{\mathtt{R}}(\ell\sigma) \hat{\mathtt{A}}(\alpha)\hat{\mathtt{R}}(\ell\sigma),
\]
and a direct calculation gives
\begin{equation}\label{eq:U2N}
\hat{\mathtt{U}}_{\mathcal{Z}_2}(\sigma)\hat{\mathbf{N}}=\csc(\mu_\alpha)\begin{pmatrix}\cos(2\ell\sigma)\\-\cos(\mu_\alpha)+\sin(2\ell\sigma)\end{pmatrix}
\end{equation}
and
\begin{equation}\label{eq:U2D}
\hat{\mathtt{U}}_{\mathcal{Z}_2}(\sigma)\hat{\mathbf{D}}=\csc(\mu_\alpha)\begin{pmatrix}-\cos(\mu_\alpha)-\sin(2\ell\sigma)\\\cos(2\ell\sigma)\end{pmatrix}.
\end{equation}
We note additionally that
\[
\arg\left(\hat{\mathtt{U}}_{\mathcal{Z}_2}(0)\hat{\mathbf{N}}\right)=\arg\begin{pmatrix}\csc(\mu_\alpha)\\-\cot(\mu_\alpha)\end{pmatrix}\pmod\pi.
\]
and therefore
\begin{equation}\label{eq:U20Narg}
\arg\left(\hat{\mathtt{U}}_{\mathcal{Z}_2}(0)\hat{\mathbf{N}}\right)
\in\left[0,\frac{\pi}{4}\right)\iff \mu_\alpha\in\left[\frac{\pi}{2},\frac{3\pi}{2}\right]\;(\bmod\;2\pi)\iff\left\{\frac{\mu_\alpha}{2\pi}\right\}\in\left[\frac{1}{4},\frac{3}{4}\right],
\end{equation}
where $\{\cdot\}$ denotes the fractional part.  Similarly,
\[
\arg\left(\hat{\mathtt{U}}_{\mathcal{Z}_2}(0)\hat{\mathbf{D}}\right)=\arg\begin{pmatrix}-\cot(\mu_\alpha)\\\csc(\mu_\alpha)\end{pmatrix}\pmod\pi.
\]
and therefore
\begin{equation}\label{eq:U20Darg}
\arg\left(\hat{\mathtt{U}}_{\mathcal{Z}_2}(0)\hat{\mathbf{D}}\right)
\in\left(\frac{\pi}{4},\frac{\pi}{2}\right)\iff \mu_\alpha\in\left(\frac{\pi}{2},\frac{3\pi}{2}\right)\;(\bmod\;2\pi)\iff\left\{\frac{\mu_\alpha}{2\pi}\right\}\in\left(\frac{1}{4},\frac{3}{4}\right).
\end{equation}
 
Consider first the Neumann--Neumann case. By \eqref{eq:U2N}, a real $\sigma$ is a quasi-eigenvalue whenever 
\[
-\cos(\mu_\alpha)+\sin(2\ell\sigma)=0,
\] 
that is, when
\begin{equation}\label{eq:Z2NN}
2\sigma\ell\in\left\{2\pi m-\frac{3\pi}{2}\pm\mu_\alpha\mid m\in\mathbb{Z}\right\}.
\end{equation}

At the same time, symmetrising the zigzag $\mathcal{Z}_2$ along the bisector, one can  represent the eigenvalue problem on a corresponding zigzag domain as the union of two mixed Steklov--Neumann and Steklov--Neumann--Dirichlet eigenvalue problems (with either Neumann or Dirichlet condition imposed on the bisector, see Figure \ref{fig:zigzag-Neumann}). The eigenvalue asymptotics for these problems are known due to the results of \cite[Propositions 1.3 and 1.13]{sloshing}:  the quasi-eigenvalues are given by 
\begin{equation}\label{eq:Z2NNsloshing}
2\sigma\ell\in\left\{2\pi m-\frac{3\pi}{2}\pm\mu_\alpha\mid m\in\mathbb{N}\right\}.
\end{equation}

\begin{figure}[htb]
\begin{center}
\includegraphics{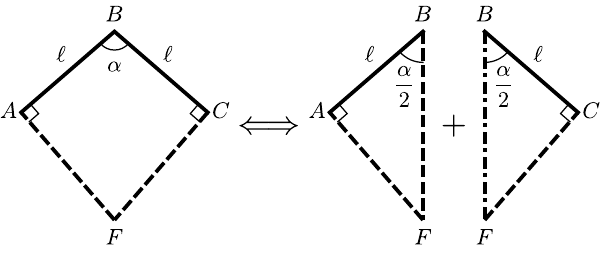}
\end{center}
\caption{Symmetric decomposition of Steklov--Neumann symmetric zigzag domain. Solid lines denote Steklov conditions, dashed lines --- Neumann conditions, and dot-dashed lines --- Dirichlet conditions\label{fig:zigzag-Neumann}}
\end{figure}

We need to show that for sufficiently large $\sigma$ the enumeration defined by \eqref{eq:Z2NNsloshing} and the natural enumeration of \eqref{eq:Z2NN} are the same.  The natural enumeration of \eqref{eq:Z2NN}  means starting counting from $m_\pm=\left[\frac{3}{4}\mp\frac{\mu_\alpha}{2\pi}\right]+1$ instead of starting counting from $1$, giving the \emph{total loss} of 
\begin{equation}\label{eq:gainloss}
\left[\frac{3}{4}+\frac{\mu_\alpha}{2\pi}\right]+\left[\frac{3}{4}-\frac{\mu_\alpha}{2\pi}\right]
=\begin{cases}
1,\quad&\text{if }\left\{\frac{\mu_\alpha}{2\pi}\right\}\in\left[\frac{1}{4},\frac{3}{4}\right],\\
0,\quad&\text{otherwise}.
\end{cases}
\end{equation}
Therefore, if the condition $\left\{\frac{\mu_\alpha}{2\pi}\right\}\in\left[\frac{1}{4},\frac{3}{4}\right]$ is satisfied, we must start counting from the first non-positive quasi-eigenvalue to ensure correct enumeration. But this is exactly the condition \eqref{eq:U20Narg}, which with account of Proposition \ref{prop:startcount}(i) guarantees that the enumeration imposed by Definition \ref{defn:naturalenumeration} is correct, thus proving Theorem \ref{thm:quasizigzags} for a symmetric straight $NN$-zigzag with two sides.

Consider now the case of the Dirichlet-Dirichlet  boundary conditions. By \eqref{eq:U2D}, a real $\sigma$ is a quasi-eigenvalue whenever 
\[
-\cos(\mu_\alpha)-\sin(2\ell\sigma)=0,
\] 
that is, when
\begin{equation}\label{eq:Z2DD}
2\sigma\ell\in\left\{2\pi m- \frac{\pi}{2} \pm \mu_\alpha\mid m\in\mathbb{Z}\right\}.
\end{equation}
Symmetrising as above (see Figure \ref{fig:zigzag-Dirichlet}) and using \cite[Propositions 1.8 and 1.13]{sloshing} we know that the quasi-eigenvalues will be correctly enumerated if we count over $m\in\mathbb{N}$ in \eqref{eq:Z2DD}. 
\begin{figure}[htb]
\begin{center}
\includegraphics{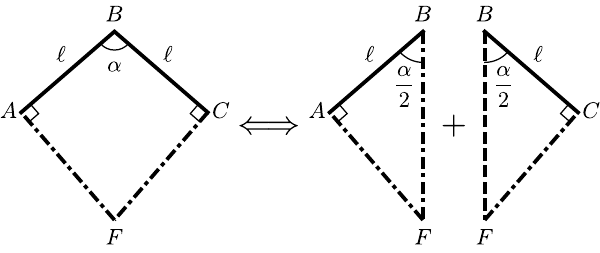}
\end{center}
\caption{Symmetric decomposition of Steklov--Dirichlet symmetric zigzag domain. Solid lines denote Steklov conditions, dashed lines --- Neumann conditions, and dot-dashed lines --- Dirichlet conditions\label{fig:zigzag-Dirichlet}}
\end{figure}

Similarly to the Neumann--Neumann case we compare this to counting only \emph{positive} quasi-eigenvalues in \eqref{eq:Z2DD}. In the latter case, the total loss
is now given (after some simplifications) by 
\begin{equation}
\left[\frac{1}{4}+\frac{\mu_\alpha}{2\pi}\right]+\left[\frac{1}{4}-\frac{\mu_\alpha}{2\pi}\right]
=\begin{cases}
-1,\quad&\text{if }\left\{\frac{\mu_\alpha}{2\pi}\right\}\in\left(\frac{1}{4},\frac{3}{4}\right),\\
0,\quad&\text{otherwise}.
\end{cases}
\end{equation}
Comparing with  \eqref{eq:U20Darg} and using Proposition \ref{prop:startcount}(iv) guarantees that the enumeration imposed by Definition \ref{defn:naturalenumeration} is correct, thus proving Theorem \ref{thm:quasizigzags} for a symmetric straight $DD$-zigzag with two sides.

Finally, consider the Neumann--Dirichlet or Dirichlet--Neumann boundary conditions on the zigzag $\mathcal{Z}_2$. In either case, the set of real quasi-eigenvalues is given by
\[
2\sigma\ell\in\left\{-\frac{\pi}{2}+\pi m\mid m\in\mathbb{Z}\right\},
\]
see \eqref{eq:U2N} and \eqref{eq:U2D}. 
However, the boundary conditions are no longer symmetric with respect to the bisector,  therefore a direct comparison to a sloshing problem is impossible, and a different approach is needed. We will use the following isospectrality result. Let $ABC=\mathcal{Z}_2=\mathcal{Z}((\alpha),(\ell,\ell))$ be a zigzag with two equal straight sides $AB$ and $BC$ of length $\ell$ joined at an angle $\alpha$, and let $ABCF$ be a  $ND$-zigzag domain, with the straight line intervals $FA$ and $FC$ being orthogonal to $AB$ and $BC$, respectively. Also, let $A'C'F'$ be an isosceles triangle with the base $A'C'$ of length $2\ell$ and angles $\alpha/2$ between the base and the sides, see Figure \ref{fig:zigzag-isospectral}. 
\begin{figure}[htb]
\begin{center}
\includegraphics{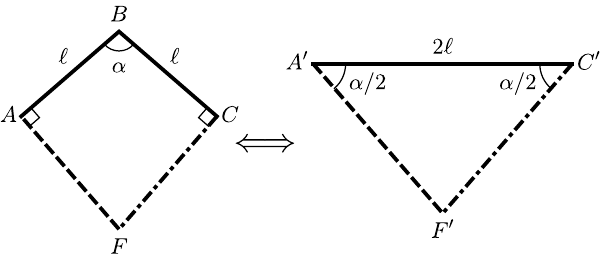}
\end{center}
\caption{Two isospectral Steklov--Neumann--Dirichlet problems. Solid lines denote Steklov conditions, dashed lines --- Neumann conditions, and dot-dashed lines --- Dirichlet conditions\label{fig:zigzag-isospectral}}
\end{figure}

\begin{lemma}
\label{lemma:transpl}
The Steklov--Neumann--Dirichlet eigenvalue problem 
\[
-\Delta u=0\text{ in }ABCF,\quad\left.\left(\frac{\partial u}{\partial n}-\lambda u\right)\right|_{ABC}=0,\quad \left.\frac{\partial u}{\partial n}\right|_{AF}=0,\quad u|_{CF}=0
\]
is isospectral to the Steklov--Neumann--Dirichlet eigenvalue problem 
\[
-\Delta u=0\text{ in }A'C'F',\quad\left.\left(\frac{\partial u}{\partial n}-\lambda u\right)\right|_{A'C'}=0,\quad \left.\frac{\partial u}{\partial n}\right|_{A'F'}=0,\quad u|_{C'F'}=0.
\]
\end{lemma}
\begin{proof}[Proof of Lemma  \ref{lemma:transpl}]
The lemma follows from a direct application of the transplantation argument of \cite[Theorem 3.1]{LPP}. In our case, the construction block $K$ is the triangle $ABF$, the line $a$ is the side $AF$, and the line $b$ is $BF$. Note that although \cite[Theorem 3.1]{LPP} is stated for the Laplacian with mixed Dirichlet-Neumann boundary conditions, its proof applies verbatim in our case, see also \cite{GHW}.
\end{proof}
Using Lemma \ref{lemma:transpl} and applying \cite[Proposition 1.13]{sloshing} to the isosceles triangle constructed in the lemma, we immediately obtain that
\[
2\ell \sigma_m^{(ND)}=-\frac{\pi}{2}+\pi k, m=1,2,\dots,
\]
and therefore $\sigma_1^{(ND)}$ is the first positive quasi-eigenvalue as prescribed by Definition  \ref{defn:naturalenumeration}. Exactly the same argument works for $\sigma_1^{(DN)}$. This completes the proof of Proposition \ref{lemma1}.

\subsection{Proof of Proposition \ref{lemma2}} 
The proof  is  based on the Dirichlet--Neumann bracketing. Given some boundary conditions $\aleph$ and $\beth$ at the points $P$ and $Q$, respectively, we  impose the Dirichlet or Neumann boundary condition $\daleth$ at the point $W$ and use the assumption that Theorem \ref{thm:quasizigzags} holds for two parts $\mathcal{Z}_\text{I}:=\mathcal{Z}_{PW}^{(\aleph\daleth)}$ and $\mathcal{Z}_\text{II}:=\mathcal{Z}_{WQ}^{(\daleth\beth)}$ of $\mathcal{Z}$ with $\daleth\in\{D,N\}$. We  then show that the only enumeration of eigenvalues on  the big zigzag $\mathcal{Z}^{(\aleph\beth)}$ that agrees with the Dirichlet--Neumann bracketing is the one given by Definition  \ref{defn:naturalenumeration}.

Indeed, let $\mathcal{N}_{\mathcal{Z}}(\sigma)$, $\mathcal{N}_{\mathcal{Z}_{\tI}}(\sigma)$, $\mathcal{N}_{\mathcal{Z}_{\tII}}(\sigma)$ be the eigenvalue counting functions (i.e. the number of eigenvalues less or equal than $\sigma$) for the zigzags $\mathcal{Z}$, $\mathcal{Z}_{\tI}$ and 
$\mathcal{Z}_{\tII}$ with given boundary conditions at the end points. Similarly, let   
$\mathcal{N}^q_{\mathcal{Z}}(\sigma)$, $\mathcal{N}^q_{\mathcal{Z}_{\tI}}(\sigma)$, $\mathcal{N}^q_{\mathcal{Z}_{\tII}}(\sigma)$ be the corresponding quasi-eigenvalue counting functions, where the quasi-eigenvalues are enumerated according to Definition \ref{defn:naturalenumeration}. The following key lemma holds.
\begin{lemma}
\label{lemma2key}
Fix $\aleph,\beth,\daleth\in\{D,N\}$. There exists $\delta>0$ such that for any $M>0$ there exists an interval $\mathcal{I}_{M}h\subset(M,+\infty)$ of length $\delta$ such that 
\begin{equation}
\label{equal:lemma2key}
\mathcal{N}^q_{\mathcal{Z}_{\tI}^{(\aleph\daleth)}}(\sigma)+\mathcal{N}^q_{\mathcal{Z}_{\tII}^{(\daleth\beth)}}(\sigma)=\mathcal{N}^q_{\mathcal{Z}^{(\aleph\beth)}}(\sigma)\qquad\text{for any }\sigma \in \mathcal{I}_{M}.
\end{equation}
\end{lemma}
Before proving  Lemma \ref{lemma2key} let us  show first how it  implies Proposition \ref{lemma2}. 

\begin{proof}[Proof of Proposition \ref{lemma2}] Consider a zigzag  domain corresponding to the zigzag $\mathcal{Z}$.  It can be represented as a union of two zigzag domains corresponding to the zigzags  $\mathcal{Z}_{\tI}$, $\mathcal{Z}_{\tII}$. By Dirichlet-Neumann bracketing, for all $\sigma>0$ we have
\begin{equation}
\label{twosided}
\mathcal{N}_{\mathcal{Z}_{\tI}^{(\aleph D)}}(\sigma)+\mathcal{N}_{\mathcal{Z}_{\tII}^{(D\beth)}}(\sigma) \le 
\mathcal{N}_{\mathcal{Z}^{(\aleph\beth)}}(\sigma)\le 
\mathcal{N}_{\mathcal{Z}_{\tI}^{(\aleph N)}}(\sigma)+\mathcal{N}_{\mathcal{Z}_{\tII}^{(N\beth)}}(\sigma).
\end{equation}
At the same time,  by our assumption,  the natural enumeration holds for the 
zigzags $\mathcal{Z}_\tI$ and $\mathcal{Z}_\tII$.  Therefore,  the eigenvalue counting functions and the corresponding quasimode counting functions of these zigzags coincide away from a union of intervals of lengths tending to zero. Let us combine this observation with Lemma \ref{lemma2key} and formula  \eqref{twosided}.  We deduce that 
there exists a positive number $\delta'$ such that for any $M>0$ there exist intervals $\mathcal{I}_M^N, \mathcal{I}_M^D\subset(M,+\infty)$ of length $\delta'$ on which the following inequalities hold:
\begin{equation}
\label{deduced}
\begin{split}
\mathcal{N}_{\mathcal{Z}^{(\aleph\beth)}}(\sigma) \le \mathcal{N}^q_{\mathcal{Z}^{(\aleph\beth)}}(\sigma),\qquad  &\sigma \in \mathcal{I}_M^N;\\
\mathcal{N}_{\mathcal{Z}^{(\aleph\beth)}}(\sigma) \ge \mathcal{N}^q_{\mathcal{Z}^{(\aleph\beth)}}(\sigma),\qquad &\sigma \in \mathcal{I}_{M}^D.
\end{split}
\end{equation} 
At the same time, it follows from Theorem \ref{thm:injectivityzigzag} that  there exists a limit (possibly equal to $+\infty$)
\begin{equation}
\label{llimit}
\lim_{\substack{\sigma\to \infty \\  \sigma \notin S}} \left(\mathcal{N}_{\mathcal{Z}^{(\aleph\beth)}}(\sigma) - \mathcal{N}^q_{\mathcal{Z}^{(\aleph\beth)}}(\sigma)\right),
\end{equation}
 where $S$ is a union of intervals of lengths tending to zero. In fact this also follows directly from Corollaries \ref{cor:quasimodesstraightnearcorners} and \ref{cor:diffbounded}. Clearly, \eqref{llimit} implies that both inequalities in \eqref{deduced}
are equalities.  Therefore,  the natural enumeration holds for the zigzag $\mathcal{Z}^{(\aleph\beth)}$ which completes the proof of Proposition \ref{lemma2}. 
\end{proof}

It remains to prove  Lemma \ref{lemma2key}.  The following abstract proposition will be used in the proof of the lemma. 
\begin{proposition}
\label{prop:integparts}
Let $\varphi_1, \varphi_2 \in C^1(\mathbb{R})$ be two monotone increasing functions such that $0<C_1<\varphi_1',\varphi_2' <C_2$ for some constants $C_1,C_2>0$.
Then there exists $\delta>0$ such that for any $M\in \mathbb{N}$ there exist  intervals $\mathcal{I},\mathcal{I}' \subset (M,+\infty)$ of length $\delta$ such that 
\begin{equation}
\label{integparts1}
[\varphi_1(\sigma)]+[\varphi_2(\sigma)]+1=[\varphi_1(\sigma)+\varphi_2(\sigma)], \,\,\,  \sigma \in \mathcal{I}.
\end{equation}
\begin{equation}
\label{integparts2}
[\varphi_1(\sigma)]+[\varphi_2(\sigma)]=[\varphi_1(\sigma)+\varphi_2(\sigma)], \,\,\,  \sigma \in \mathcal{I}'.
\end{equation}
\end{proposition}
We postpone the proof of Proposition \ref{prop:integparts} and proceed with the proof of Lemma \ref{lemma2key}.

\begin{proof}[Proof of Lemma \ref{lemma2key}]  We start by making the following observation: $\sigma$ is a quasi-eigenvalue of $\mathcal{Z}^{(\aleph\beth)}$ if and only if
\begin{equation}\label{eq:NZassum}
\varphi_{\mathcal{Z}_{\tI}^{(\aleph\daleth)}}(\sigma)+\tilde{\varphi}_{\mathcal{Z}_\tII^{(\daleth\beth)}}(\sigma)\in\mathbb{Z},\quad\daleth\in\{D,N\},
\end{equation}
Indeed, for $\sigma$ to be a quasi-eigenvalue we must have 
\[
\begin{split}
\mathtt{U}_{\mathcal{Z}_{\tII}}(\sigma)\,\mathtt{U}_{\mathcal{Z}_{\tI}}(\sigma)\,\baleph&\text{ is proportional to }\bbeth\\
&\Updownarrow\\
\arg\left(\mathtt{U}_{\mathcal{Z}_{\tI}}(\sigma)\,\baleph\right)&=\arg\left(\mathtt{U}^{-1}_{\mathcal{Z}_{\tII}}(\sigma)\,\bbeth\right)\pmod\pi\\
&\Updownarrow\\
\arg\left(\mathtt{U}_{\mathcal{Z}_{\tI}}(\sigma)\,\baleph\right)-\arg(\bm{\daleth})&+\arg(\bm{\daleth})-\arg\left(\mathtt{U}^{-1}_{\mathcal{Z}_{\tII}}(\sigma)\,\bbeth\right)=0\pmod\pi,
\end{split}
\] 
and then recall the definitions \eqref{eq:phialephbethdef} giving us \eqref{eq:NZassum}.

Therefore the quasi-eigenvalue counting function $\mathcal{N}^q_{\mathcal{Z}^{(\aleph\beth)}}(\sigma)$ may only differ from the integer part of the left-hand side of \eqref{eq:NZassum} by addition of an integer $m_0$ independent of $\sigma$. To find $m_0$ it is enough to consider $\sigma=0$.

We further assert that for any $\aleph,\beth,\daleth\in\{D,N\}$ we have
\begin{equation}\label{eq:phiZZIZII}
\left[\varphi_{\mathcal{Z}^{(\aleph\beth)}}(0)\right] =
\left[\varphi_{\mathcal{Z}_{\tI}^{(\aleph\daleth)}}(0)+\tilde{\varphi}_{\mathcal{Z}_\tII^{(\daleth\beth)}}(0)\right].
\end{equation}
We prove \eqref{eq:phiZZIZII} in the case $\aleph=\beth=\daleth=N$. Recall that all matrices $\mathtt{U}(0)$ are symmetric and we can therefore write $\mathtt{U}_{\mathcal{Z}_{\tI}}(0)=\hat{\mathtt{S}}\left(1/\tau_\tI, \hat{\Cvect}_\mathrm{even}\right)$ and $\mathtt{U}_{\mathcal{Z}_{\tII}}(0)=\hat{\mathtt{S}}\left(1/\tau_\tII, \hat{\Cvect}_\mathrm{even}\right)$ with some $\tau_\tI, \tau_\tII\in\mathbb{R}$.  Using \eqref{eq:phialephbethdef} and \eqref{arg:choice} we obtain
\begin{equation}\label{eq:phiZNNN0}
\left[\varphi_{\mathcal{Z}^{(NN)}}(0)\right]=\left[\frac{\arctan\left(\tau_\tI^2\tau_\tII^2\right)}{\pi}-\frac{1}{4}\right]\\
=\begin{cases}
-1,\quad&\text{if }\tau_\tI^2\tau_\tII^2<1,\\
0,\quad&\text{if }\tau_\tI^2\tau_\tII^2\ge 1.
\end{cases}
\end{equation}
On the other hand,
\[
\left[\varphi_{\mathcal{Z}_{\tI}^{(NN)}}(0)+\tilde{\varphi}_{\mathcal{Z}_\tII^{(NN)}}(0)\right]=\left[\frac{\arctan\left(\tau_\tI^2\right)-\arctan\left(\tau_\tII^{-2}\right)}{\pi}\right],
\]
thus coinciding with the right-hand side of \eqref{eq:phiZNNN0} and proving \eqref{eq:phiZZIZII} in the case $\aleph=\beth=\daleth=N$. The other cases of \eqref{eq:phiZZIZII}  are treated similarly.

But now the combination of \eqref{eq:phiZZIZII}, Definition \ref{defn:naturalenumeration}, and Lemma \ref{lem:phiandtildephi} allows us to find the integer $m_0$ in each case. We arrive at
the following table:
\renewcommand*{\arraystretch}{1.5}
\[
\begin{matrix}
\aleph&\beth&\daleth&\mathcal{N}^q_{\mathcal{Z}^{(\aleph\beth)}}(\sigma)&\mathcal{N}^q_{\mathcal{Z}_\tI^{(\aleph\daleth)}}(\sigma)&\mathcal{N}^q_{\mathcal{Z}_\tII^{(\daleth\beth)}}(\sigma)\\
\hline
N&N&N&\left[\varphi_{\mathcal{Z}_{\tI}^{(NN)}}(\sigma)+\tilde{\varphi}_{\mathcal{Z}_\tII^{(NN)}}(\sigma)\right]+1&\left[\varphi_{\mathcal{Z}_{\tI}^{(NN)}}(\sigma)\right]+1&\left[\tilde{\varphi}_{\mathcal{Z}_\tII^{(NN)}}(\sigma)\right]+1\\
N&N&D&\left[\varphi_{\mathcal{Z}_{\tI}^{(NN)}}(\sigma)+\tilde{\varphi}_{\mathcal{Z}_\tII^{(DN)}}(\sigma)\right]+1&\left[\varphi_{\mathcal{Z}_{\tI}^{(ND)}}(\sigma)\right]+1&\left[\tilde{\varphi}_{\mathcal{Z}_\tII^{(DN)}}(\sigma)\right]\\
N&D&N&\left[\varphi_{\mathcal{Z}_{\tI}^{(NN)}}(\sigma)+\tilde{\varphi}_{\mathcal{Z}_\tII^{(ND)}}(\sigma)\right]+1&\left[\varphi_{\mathcal{Z}_{\tI}^{(NN)}}(\sigma)\right]+1&\left[\tilde{\varphi}_{\mathcal{Z}_\tII^{(ND)}}(\sigma)\right]+1\\
N&D&D&\left[\varphi_{\mathcal{Z}_{\tI}^{(ND)}}(\sigma)+\tilde{\varphi}_{\mathcal{Z}_\tII^{(DD)}}(\sigma)\right]+1&\left[\varphi_{\mathcal{Z}_{\tI}^{(ND)}}(\sigma)\right]+1&\left[\tilde{\varphi}_{\mathcal{Z}_\tII^{(DD)}}(\sigma)\right]\\
D&N&N&\left[\varphi_{\mathcal{Z}_{\tI}^{(DN)}}(\sigma)+\tilde{\varphi}_{\mathcal{Z}_\tII^{(NN)}}(\sigma)\right]&\left[\varphi_{\mathcal{Z}_{\tI}^{(DN)}}(\sigma)\right]&\left[\tilde{\varphi}_{\mathcal{Z}_\tII^{(NN)}}(\sigma)\right]+1\\
D&N&D&\left[\varphi_{\mathcal{Z}_{\tI}^{(DD)}}(\sigma)+\tilde{\varphi}_{\mathcal{Z}_\tII^{(DN)}}(\sigma)\right]&\left[\varphi_{\mathcal{Z}_{\tI}^{(DD)}}(\sigma)\right]&\left[\tilde{\varphi}_{\mathcal{Z}_\tII^{(DN)}}(\sigma)\right]\\
D&D&N&\left[\varphi_{\mathcal{Z}_{\tI}^{(DN)}}(\sigma)+\tilde{\varphi}_{\mathcal{Z}_\tII^{(ND)}}(\sigma)\right]&\left[\varphi_{\mathcal{Z}_{\tI}^{(DN)}}(\sigma)\right]&\left[\tilde{\varphi}_{\mathcal{Z}_\tII^{(ND)}}(\sigma)\right]+1\\
D&D&D&\left[\varphi_{\mathcal{Z}_{\tI}^{(DD)}}(\sigma)+\tilde{\varphi}_{\mathcal{Z}_\tII^{(DD)}}(\sigma)\right]&\left[\varphi_{\mathcal{Z}_{\tI}^{(DD)}}(\sigma)\right]&\left[\tilde{\varphi}_{\mathcal{Z}_\tII^{(DD)}}(\sigma)\right]\\
\hline
\end{matrix}
\]
\renewcommand*{\arraystretch}{1.25}
Recalling Lemma \ref{lem:phi}, the proof of Lemma  \ref{lemma2key} now follows by the application of Proposition \ref{prop:integparts}, which applies in all eight of these cases.
\end{proof}

We conclude this subsection by the proof of Proposition \ref{prop:integparts}.

\begin{proof}[Proof of Proposition \ref{prop:integparts}]  Assume without loss of generality that $C_2=C$ and $C_1=1/C$, for some $C>1$. Let us first prove the assertion \eqref{integparts1}. Let $\omega_1(\sigma)=\left\{\varphi_1(\sigma)\right\}, \omega_2(\sigma)=\left\{\varphi_2(\sigma)\right\}$ denote the fractional parts of $\varphi_1(\sigma)$,
$\varphi_2(\sigma)$, respectively.  Note that the equality in  \eqref{integparts1} is equivalent to the inequality
\begin{equation}
\label{greater1}
\omega_1(\sigma)+\omega_2(\sigma)\ge 1.
\end{equation}

Choose an integer number $N>M$ and let $s$ be the value for which $\varphi_1(s)+\varphi_2(s)=N$. If  $\omega_1(s)=\omega_2(s)=0$ then the result trivially follows for the interval $(s-\delta, s)$ and $\delta=\frac{1}{2C}$.
Therefore, we may  suppose that $\omega_1(s)+\omega_2(s)=1$, and   assume without loss of generality that $\omega_2(s)\ge \frac{1}{2} \ge \omega_1(s)$. There are two cases.

Suppose first that $\omega_1(s)\leq\frac{1}{3C^2}$. Then since $\varphi_1'>\frac 1C$, there exists (precisely one) $s'\in(s-\frac{1}{3C},s)$ for which $\omega_1(s')=0$. On the other hand, since $\varphi_2'<C$, $\varphi_2(s')\geq\varphi_2(s)-\frac 13$, and since $\omega_2(s)\geq 1-\frac{1}{3C^2}\geq \frac 23$, we must have $\omega_2(s')\geq\frac 13$. Therefore $(\omega_1+\omega_2)(s')\geq\frac 13$. Our inequality \eqref{greater1} then holds on the interval $(s'-\frac 1{6C},s')$, which is a subset of $(s-\frac{2}{3C},s)$.

On the other hand suppose that $\omega_1(s)>\frac{1}{3C^2}$. Then $\omega_2(s)<1-\frac{1}{3C^2}$, and so on the interval $(s,s+\frac{1}{3C^3})$, $\omega_2$ remains less than 1, as does $\omega_1$. Since both are still increasing, \eqref{greater1} holds on the interval $(s,s+\frac{1}{3C^3})$.

In either case, the interval $(s-\frac{2}{3C},s+\frac{2}{3C})$ contains an interval of length at least $\frac{1}{6C^3}$ where \eqref{greater1} holds. Since there are infinitely many values of $s$, \eqref{integparts1} follows.

The relation \eqref{integparts2} is proved in a similar manner. 
\end{proof}

\subsection{Enumeration of quasi-eigenvalues for non-exceptional polygons}\label{subs:noexc}
Let $\mathcal{P}:=\mathcal{P}(\balpha,\bell)$ be a partially curvilinear non-exceptional polygon and let $\T(\sigma):=\hat{\mathtt{T}}(\balpha,\bell,\sigma)$ be the lifted corresponding matrix defined in subsection \ref{subsec:propscat} acting on the universal cover $\UC$. Recall that a real number $\sigma \ge 0$ is  a quasi-eigenvalue of the polygon 
$\mathcal{P}$
 if the matrix  $\mathtt{T}(\balpha,\bell,\sigma)$ has eigenvalue one. Equivalently, this means  that there exists a vector $ 0\neq \hat{\bm{\xi}} \in \UC$ such that $|\hat{\bm{\xi}}|=|\hat{\mathtt{T}}(\sigma) \hat{\bm{\xi}}|$ and 
\begin{equation}
\label{condperiod}
\arg\left(\hat{\mathtt{T}}(\sigma)\hat{\bm{\xi}}\right)=\arg\hat{\bm{\xi}} \pmod{2\pi}.
\end{equation}

Let us for the moment switch back to the representation of vectors and matrices on $\mathbb{R}^2$. Given that $\det(\mathtt{T}^\sha(\sigma))=1$, for each $\sigma$ there exist two linearly independent vectors $\mathfrak{t}^\sha_1=\mathfrak{t}^\sha_1(\mathtt{T}^\sha(\sigma))$ and $\mathfrak{t}^\sha_2=\mathfrak{t}^\sha_2(\mathtt{T}^\sha(\sigma))$ such that 
\[
\left|\mathtt{T}^\sha(\sigma)\mathfrak{t}^\sha_j\right|=\left|\mathfrak{t}^\sha_j\right|,\qquad j=1,2.
\]  
Indeed, by polar decomposition the matrix $\mathtt{T}^\sha(\sigma)$ could be represented as a product of a symmetric matrix and a rotation; the latter does not change length, and for a symmetric matrix the statement is easy to check. (Interestingly, the problem of finding the vectors whose length is preserved under the action of a given matrix has other unexpected applications, see \cite{Tol21}.) Moreover, the vectors $\mathfrak{t}^\sha_1$ and $\mathfrak{t}^\sha_2$ are uniquely defined up to multiplication by a constant or up to a swap, unless $\mathtt{T}^\sha(\sigma)$ is a pure rotation, in which case one could take any pair of linearly independent vectors. To fix the argument, we shall assume that 
\begin{equation}
\label{argassump}
0\le \arg \mathfrak{t}^\sha_j <\pi, \qquad j=1,2.
\end{equation}
Set now
\[
\hbe_j:=\Pro^{-1}\mathfrak{t}^\sha_j , \qquad j=1,2,
\]
and
\begin{equation}
\label{def:delta}
\mathfrak{d}_j(\hat{\mathtt{T}}(\sigma)):=\arg\left(\hat{\mathtt{T}}(\sigma) \hbe_j\right)-\arg\left(\hbe_j\right), \qquad j=1,2.
\end{equation}
Note that $\mathfrak{d}_j(\T(\sigma))$ is always well-defined: if $\T(\sigma)$ is a rotation by an angle $\psi$, then $\mathfrak{d}_j(\T(\sigma))=\psi$ for any choice of $\hbe_j$.
The following proposition is an immediate consequence of Definition \ref{def:quasi}. 
\begin{proposition}
\label{prop:deltazeros}
A number $\sigma \ge 0$ is a quasi-eigenvalue of the polygon $\mathcal{P}$ if and only if $\mathfrak{d}_j(\T(\sigma))=0\pmod{2\pi}$ for either $j=1$ or $j=2$.
If $\sigma>0$ and $\mathfrak{d}_j(\T(\sigma))=0 \pmod{2\pi}$ for both $j=1,2$, then $\sigma$ is a quasi-eigenvalue of multiplicity two.
\end{proposition}
\begin{remark}
\label{rem:deltachoice}
The matrix $\T(\sigma)$ corresponding to a polygon $\mathcal{P}$ is defined up to similarity: it depends on the choice of enumeration of vertices of the polygon $\mathcal{P}$. 
As a consequence,  the vectors $\hbe_j$ and the functions $\mathfrak{d}_j(\sigma)$, $j=1,2$, depend on this choice as well.  To simplify  notation, in what follows we write 
$\hbe_j(\sigma):=\hbe_j(\T(\sigma))$ and $\mathfrak{d}_j(\sigma):=\mathfrak{d}_j(\T(\sigma))$, when the choice of the matrix $\T$ is clear from the context.  Note also that by Proposition \ref{prop:deltazeros},  the values of $\sigma$ such that $\mathfrak{d}_j(\T(\sigma))=0 \pmod{2\pi}$  depend only on  the polygon $\mathcal{P}$ but not on the choice of the matrix $\T$, cf. Remark  \ref{rem:Tnotinvariant}.
\end{remark}
 The following regularity properties of the functions $\mathfrak{d}_j(\T(\sigma))$, $j=1,2$ may be deduced from the structure of the matrices $\T(\sigma)$.
\begin{lemma}
\label{lemma:diff}
The function $\mathfrak{d}_j(\T(\sigma))$ is a continuous function in $\sigma$, $j=1,2$.  Moreover, if $\T(\sigma_0)$ is not a rotation, then $\mathfrak{d}_j(\T(\sigma))$ is differentiable at $\sigma=\sigma_0$;
otherwise, left and right derivatives at $\sigma=\sigma_0$ exist.
\end{lemma}
\begin{proof}
Let us write down the polar decomposition for $\T(\sigma)$ explicitly.  First, observe by a direct computation that  $\Sh(\tau,\w)\R(\psi)=\R(\psi)\Sh(\tau,\R(-\psi)\w)$, where $\Sh$ is a symmetric matrix and $\R$ is a rotation as defined in subsection \ref{subsec:rep}. Iterating this relation  and taking into account \eqref{eq:Ahat} we obtain
\begin{equation}
\label{polardec}
\T(\sigma)=\T(\balpha, \bell, \sigma)= \R(L\sigma) \Sh_n(\sigma)\Sh_{n-1}(\sigma)\cdots \Sh_1(\sigma),
\end{equation}
where $L_j=\sum_{k=1}^j \ell_k$, $L=L_n$, and 
\[
\Sh_j(\sigma)=\Sh(a_1(\alpha_j)-a_2(\alpha_j),\R(-L_j\sigma)\Codd).
\]
It follows from \eqref{polardec} that $\T(\sigma)$ is a rotation if and only if 
\begin{equation}
\label{rotcond}
\Sh_n(\sigma)\Sh_{n-1}(\sigma)\cdots \Sh_1(\sigma) =\pm \hat{\Id}.
\end{equation}
Moreover, note that the entries of $\T(\sigma)$ are real analytic functions of $\sigma$. Hence,  for a given $\sigma=\sigma_0$  there are three possibilities: 
\begin{itemize}
\item[(i)] $\T(\sigma)$ is not a rotation  in some neighbourhood of $\sigma_0$. Then the vectors $\hbe_j(\sigma)$, $j=1,2$, are uniquely defined for each $\sigma$ in this neighbourhood and 
$\mathfrak{d}_j(\sigma)$ depends smoothly on $\sigma$.
\item[(ii)] $\T(\sigma)$ is a rotation in some neighbourhood of $\sigma_0$. Then $\mathfrak{d}_j(\sigma)$ is a linear function in $\sigma$ in this neighbourhood.
\item[(iii)] $\T(\sigma_0)$ is a rotation, but $T(\sigma)$ is not a rotation in some punctured neighbourhood of $\sigma_0$. In this case $\sigma=\sigma_0$ corresponds to a double eigenvalue. 
However, we claim the left and right derivatives of  $\mathfrak{d}_j(\sigma)$, which are defined a priori only for $\sigma\neq\sigma_0$, in fact exist at $\sigma=\sigma_0$. Indeed, this follows from a standard perturbation theory result of Rellich \cite{Rel}. The matrix $\Sh_n(\sigma)\dots\Sh_1(\sigma)$ is a symmetric matrix, all of whose coefficients have analytic dependence on $\sigma$. By \cite[Theorem 1, p. 42]{Rel}, its eigenvalues and eigenvectors may be chosen to have analytic dependence on $\sigma$ in a neighbourhood of $\sigma=\sigma_0$, with $\sigma=\sigma_0$ corresponding to an intersection of analytic branches. By a direct calculation, the unit vectors whose length is preserved by $\Sh_n(\sigma)\dots\Sh_1(\sigma)$ have analytic dependence on the eigenvalues and eigenvectors, and hence themselves depend analytically on $\sigma$. However, by \eqref{polardec}, these vectors are precisely $\hbe_j(\sigma)$, $j=1,2$. The result follows.
\end{itemize}
This completes the proof of the lemma.
\end{proof}

The next proposition is important for our analysis.
\begin{proposition}
\label{delta:monotone}
The functions $\mathfrak{d}_j(\T(\sigma))$ are monotone increasing in $\sigma$ and
$$
0<C_1 \le \frac{\dr^\pm\,\mathfrak{d}_j(\T(\sigma))}{\dr\sigma} \le C_2, 
$$
where $\frac{d^\pm}{d\sigma}$ denotes  one-sided derivatives.
\end{proposition}
\begin{proof}
We consider separately the cases (i)-(iii) above. Consider first case (i). Then $\mathfrak{d}_j(\T\sigma)$ depends smoothly on $\sigma$ and
\[
\begin{split}
 \left.\frac{\dr\,\mathfrak{d}_j(\T(\sigma))}{\dr\sigma}\right|_{\sigma=\sigma_0}&= \left.\frac{\dr}{\dr\sigma}\right|_{\sigma=\sigma_0}\arg(\T(\sigma))\hbe_j(\sigma)-
\left.\frac{\dr}{\dr\sigma}\right|_{\sigma=\sigma_0} \arg\hbe_j(\sigma)\\
&= \left.\frac{\dr}{\dr\sigma}\right|_{\sigma=\sigma_0}\arg(\T(\sigma))\hbe_j(\sigma_0)\\
&\quad+
\left( \left.\frac{\dr}{\dr\sigma}\right|_{\sigma=\sigma_0}\arg(\T(\sigma_0))\hbe_j(\sigma)-\left.\frac{\dr}{\dr\sigma}\right|_{\sigma=\sigma_0} \arg\hbe_j(\sigma)\right).
\end{split}
\]
Set 
\[
D_1(\sigma)= \left.\frac{\dr}{\dr\sigma}\right|_{\sigma=\sigma_0}\arg(\T(\sigma))\hbe_j(\sigma_0)
\]
and 
\[
D_2(\sigma)= \left.\frac{\dr}{\dr\sigma}\right|_{\sigma=\sigma_0}\arg(\T(\sigma_0))\hbe_j(\sigma)-\left.\frac{\dr}{\dr\sigma}\right|_{\sigma=\sigma_0} \arg\hbe_j(\sigma)
\]
Arguing in the same way as in the proof of of Lemma \ref{lem:phi} one can check that there exist constants $C_1,C_2>0$ such that $C_1\le D_1(\sigma)\le C_2$. 
The proof of Proposition \ref{delta:monotone} in case (i) then follows from the following claim:
\begin{equation}
\label{claimI2}
D_2(\sigma)=0.
\end{equation}
To prove \eqref{claimI2}, let us assume without loss of generality that $j=1$ and $|\hbe_1(\sigma_0)|=1$.  For $\sigma$ close to $\sigma_0$, let  
\newcommand{\myparallel}{{\mkern3mu\vphantom{\perp}\vrule depth 0pt\mkern2mu\vrule depth 0pt\mkern3mu}}
\[
\hbe_1(\sigma)=\hbe_1(\sigma_0)+(\sigma-\sigma_0)  \hbv_1^\myparallel+ (\sigma-\sigma_0)  \hbv_1^\perp + o(\sigma-\sigma_0),
\]
where $\hbv_1^\myparallel$ is a vector in the same direction as $\hbe_1(\sigma_0)$ and the angle between $\hbv_1^\myparallel$ and $\hbv_1^\perp$ is equal to $\pi/2$. It is easy to see that 
\begin{equation}
\label{hbv}
\left.\frac{\dr}{\dr\sigma}\right|_{\sigma=\sigma_0} \arg\hbe_1(\sigma)=\left|\hbv_1^\perp\right|.
\end{equation}
 Let $\hbw_1'=\T(\sigma_0)\hbv_1^\myparallel$ and $\hbw_1''=\T(\sigma_0)\hbv_1^\perp$.
By definition of $\hbe_1(\sigma_0)$ we have $|\T(\sigma_0)\hbe_1(\sigma_0)|=|\hbe_1(\sigma_0)|$ and therefore $\left|\hbw_1'\right|=\left|\hbv_1^\myparallel\right|$.   At the same time, $\det \T(\sigma_0)=1$,
and therefore the areas of the parallelograms generated by the pairs of vectors $(\hbv_1^\myparallel,\hbv_1^\perp)$ and $(\hbw_1', \hbw_1'')$ are the same. Hence the  projection of $\hbw_1''$
on $(\hbw_1')^\perp$ has the same length as $\hbv_1^\perp$.  One can check that the length of this projection is equal to 
$\left.\frac{\dr}{\dr\sigma}\right|_{\sigma=\sigma_0}\arg(\T(\sigma_0))\hbe_1(\sigma)$.  Hence,  taking into account \eqref{hbv}, one obtains \eqref{claimI2}.

This completes the proof of Proposition \ref{delta:monotone} in case (i). In case (iii) the argument is exactly the same with the derivative replaced by one-sided derivatives.
Consider now the remaining case  (ii). Then, as follows from \eqref{polardec} and \eqref{rotcond},  the function $\mathfrak{d}_j(\T\sigma)$ is linear in $\sigma$ and its derivative is equal to $L$,
which immediately implies the proposition.
\end{proof}

Let us now define the natural enumeration for polygons. Let now $\{\sigma_m(\mathcal{P})\}, m\in \mathbb{Z}$, be the sequence of \emph{all real} quasi-eigenvalues of the polygon
$\mathcal{P}$ repeated with multiplicities, which is monotone increasing with $m$, see Remark \ref{rem:symmetric}. Recall that a quasi-eigenvalue $\sigma$ has  multiplicity two if 
$\mathfrak{d}_j(\T(\sigma))=0\,\, (\operatorname{mod}\,\,  2\pi$) for both $j~=~1,2$.
\begin{definition}
The first quasi-eigenvalue $\sigma_1(\mathcal{P})$  is defined as the first non-negative element the sequence  $\{\sigma_m(\mathcal{P})\}$. Moreover, if $\sigma_1(\mathcal{P})=0$ then
$\sigma_2(\mathcal{P})>0$, i.e., a zero quasi-eigenvalue is counted only once.
\end{definition}

We can now state the main result of this subsection that the natural enumeration of quasi-eigenvalues  yields the correct enumeration of Steklov eigenvalues for partially curvilinear polygons without exceptional angles.
\begin{theorem}
\label{thm:mainpolygons}
Theorem \ref{thm:main} holds for partially curvilinear non-exceptional polygons. 
\end{theorem}
\begin{proof}
Let $\mathcal{P}$ be a partially curvilinear polygon. Take any point $V_0$ on a straight piece of the boundary and make a straight cut perpendicular to $\partial\mathcal{P}$ at this point into the interior of $\mathcal{P}$; at the top of the cut we add another small circular cut, see Figure \ref{fig:poly-cut}.

\begin{figure}[htb]
\begin{center}
\includegraphics{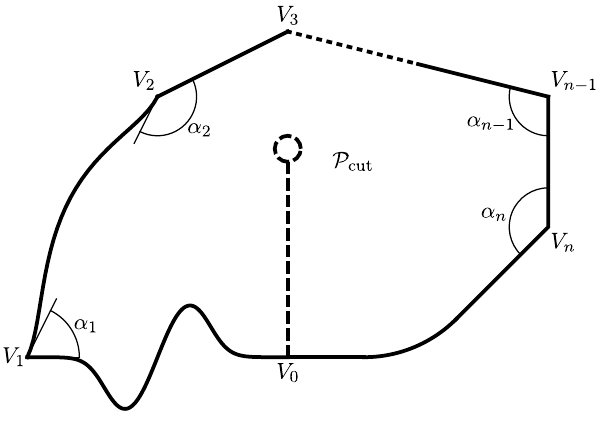}
\end{center}
\caption{A polygon with a cut\label{fig:poly-cut}}
\end{figure}

Imposing Neumann or Dirichlet conditions on the cut we may consider the polygon with a cut as a zigzag domain 
$\mathcal{P}_{\text{cut}}$
corresponding to the zigzag $\mathcal{Z}=\mathcal{Z}_\mathcal{P}=\partial\mathcal{P}$ with the same start and end point $V_0$. The assumptions of Lemma \ref{lem:improvement} are satisfied for this zigzag domain (note that the circular cut was added precisely to avoid having an angle greater than 
$\pi$ at the top of the vertical cut), and therefore  the quasimode construction applies. Denote by $\mathcal{N}_\mathcal{P}(\sigma)$  the eigenvalue counting function on a polygon, and by $\mathcal{N}_{\mathcal{Z}^{(DD)}}(\sigma)$ and $\mathcal{N}_{\mathcal{Z}^{(NN)}}(\sigma)$ the eigenvalue counting functions of zigzag $\mathcal{Z}$ with, respectively, the Dirichlet and Neumann condition on the cut.
Denote also by $\mathcal{N}^q_\mathcal{P}(\sigma)$, $\mathcal{N}_{\mathcal{Z}^{(DD)}}^q(\sigma)$ and $\mathcal{N}^q_{\mathcal{Z}^{(NN)}}(\sigma)$  the corresponding quasi-eigenvalue counting functions. By the Dirichlet--Neumann bracketing we have for all positive $\sigma$
\begin{equation}
\label{eq:DNbrack}
\mathcal{N}_{\mathcal{Z}^{(DD)}}(\sigma)\le \NP(\sigma) \le \mathcal{N}_{\mathcal{Z}^{(NN)}}(\sigma).
\end{equation}
In view of Remark \ref{rem:deltachoice}, we need to fix the choice of the matrix $\T$ corresponding to the polygon $\mathcal{P}$.
 From now on,  we choose it to be the matrix corresponding to the zigzag  $\mathcal{Z}_\mathcal{P}$.  This could be done by introducing an auxiliary vertex at $V_0$ with the angle equal to $\pi$ , and use this vertex as the starting point for enumeration of the vertices of $\mathcal{P}$. Note that the vertex transfer matrix at  $V_0$  is equal to identity and therefore does not affect $\T$.

Consider the functions
\begin{equation}\label{fourfunctions}
\psi_j(\sigma)=\frac{\mathfrak{d}_j(\sigma)}{2\pi}, \, j=1,2.
\end{equation}

\begin{remark} 
\label{rem:pmatrix}
In what follows we shall assume that $\T(0)$ is a positive matrix. If $\T(0)$ is negative, the proof follows along the same lines with minor modifications which will be indicated later.  
\end{remark}
\begin{lemma}
\label{lem:countfunctions}
The following formula holds:
\[
\NPq(\sigma)=[\psi_1(\sigma)]+[\psi_2(\sigma)]+1.
\]
\end{lemma}
\begin{proof}

We first note that the right-hand side is a step function that has discontinuities precisely at the quasi-eigenvalues. Moreover,  the jump at a discontinuity is equal to one if the corresponding quasi-eigenvalue is simple, and is equal to two if the corresponding quasi-eigenvalue is double. Therefore, it remains to check the equality for $\sigma=0$.

 Note that the vectors $\hbe_1(0)$ and $\hbe_2(0)$ can be chosen to be symmetric reflections of each other about one of the eigenvectors of the matrix $\T(0)$.
 Moreover, one can easily check that $\T(0)\hbe_1(0)$ and $\T(0)\hbe_2(0)$ remain symmetric with respect to the same vector. Therefore, either $\mathfrak{d}_j(0)$, $j=1,2$, have opposite signs, or $\mathfrak{d}_1(0)=\mathfrak{d}_2(0)=0$. In both cases the equality
\[
\NPq(0)=[\psi_1(0)]+[\psi_2(0)]+1
\]
follows from Definitions \ref{def:quasi} and \ref{def:quasimult}.
This completes the proof of the lemma.
\end{proof}

Let us go back to the proof of Theorem \ref{thm:mainpolygons}.  We recall that by Definition \ref{defn:naturalenumeration},
\[
\mathcal{N}_{\mathcal{Z}^{(NN)}}^q(\sigma)=\left[\varphi_{\mathcal{Z}^{(NN)}}(\sigma)\right]+1,
\qquad
\mathcal{N}_{\mathcal{Z}^{(DD)}}^q(\sigma)=\left[\varphi_{\mathcal{Z}^{(DD)}}(\sigma)\right].
\]
where by \eqref{eq:phialephbethdef}
\begin{equation}\label{eq:phicut}
\varphi_{\mathcal{Z}^{(NN)}}(\sigma)=\frac{\arg\left(\T(\sigma)\hat{\mathbf N}\right)}{\pi},
\qquad
\varphi_{\mathcal{Z}^{(DD)}}(\sigma)=\frac{\arg\left(\T(\sigma)\hat{\mathbf D}\right)}{\pi}-\frac{1}{2}.
\end{equation}

By Theorem \ref{thm:quasizigzags} applied to the 
$\mathcal{Z}$, we have that 
\[
\mathcal{N}_{\mathcal{Z}^{(NN)}}(\sigma)=\mathcal{N}_{\mathcal{Z}^{(NN)}}^q(\sigma)\quad\text{and}\quad \mathcal{N}_{\mathcal{Z}^{(DD)}}(\sigma)=\mathcal{N}_{\mathcal{Z}^{(DD)}}^q(\sigma)
\] 
for all $\sigma$ except some intervals of lengths tending to zero as $\sigma \to \infty$. Using Theorem \ref{thm:injectivity} to obtain an analogue of \eqref{llimit}, we may argue as in the proof of Proposition \ref{lemma2}.  We need to show that  for $\sigma$ belonging to some intervals of lengths bounded below and located arbitrarily far away on the real line,
\begin{equation}
\label{count:eqpol}
\NPq(\sigma)=\mathcal{N}_{\mathcal{Z}^{(NN)}}^q(\sigma),
\end{equation}
and for another collection of intervals with the same properties, 
\begin{equation}
\label{count:eqpol1}
\NPq(\sigma)= \mathcal{N}_{\mathcal{Z}^{(DD)}}^q(\sigma).
\end{equation}

In order to prove \eqref{count:eqpol} we will need the following proposition.

\begin{proposition}
\label{prop:auxdelta}
Let $\aleph\in\{D,N\}$ and $j\in\{1,2\}$. There exists  $\varepsilon>0$ such that for all $\sigma>0$,
\begin{equation}
\label{eq:auxdelta}  
\left|2\psi_j(\sigma)-\varphi_{\mathcal{Z}^{(\aleph\aleph)}}(\sigma)\right|\le 1-\varepsilon.
\end{equation}
\end{proposition}

Let us postpone the proof of the proposition and proceed with the proof of \eqref{count:eqpol} for $\sigma$ belonging to intervals of length bounded below located arbitrary far on the real line.
In view of Lemma \ref{lem:countfunctions}  we need to show  that for such $\sigma$
\begin{equation}
\label{equalityNeumann}
[\psi_1(\sigma)]+[\psi_2(\sigma)]+1=\left[\varphi_{\mathcal{Z}^{(NN)}}(\sigma)\right];
\end{equation}
similarly, \eqref{count:eqpol1} is equivalent to
\begin{equation}
\label{equalityDirichlet}
[\psi_1(\sigma)]+[\psi_2(\sigma)]=\left[\varphi_{\mathcal{Z}^{(DD)}}(\sigma)\right].
\end{equation}

Let us prove \eqref{equalityNeumann} first. For any $k\in \mathbb{N}$, choose $\tilde \sigma_k$ so that $\varphi_{\mathcal{Z}^{(NN)}}(\tilde\sigma_k)=2k+1$. This is possible to achieve  since $\varphi_{\mathcal{Z}^{(NN)}}$
has a positive derivative bounded away from zero, see Lemma \ref{lem:phi}. By Proposition \ref{prop:auxdelta} we have
\[
\left|2\psi_j(\tilde\sigma_k)-\varphi_{\mathcal{Z}^{(NN)}}(\tilde\sigma_k)\right|\le 1-\varepsilon, \qquad j=1,2.
\]
Therefore,
\begin{equation}
\label{auxx} 
k+\frac{\varepsilon}{2}\le\psi_j(\tilde\sigma_k)\le k+1-\frac{\varepsilon}{2}, \qquad j=1,2.
\end{equation}
Since the derivatives of $\psi_j(\sigma)$ and $\varphi_{\mathcal{Z}^{(NN)}}(\sigma)$ are uniformly bounded, there exists an interval $\mathcal{I}_k^{(NN)}=(\tilde\sigma_k,\tilde\sigma_k')$ of length uniformly bounded away from zero such that 
$[\psi_j(\sigma)]=k$ and $\left[\varphi_{\mathcal{Z}^{(NN)}}((\sigma)\right]=2k+1$  for all $\sigma \in \mathcal{I}_k^{(NN)}$ and  $j=1,2$. This proves \eqref{equalityNeumann}.

Equality \eqref{equalityDirichlet} is obtained using a similar argument. As above, choose $\tilde\sigma_k$ such that $\varphi_{\mathcal{Z}^{(DD)}}(\tilde\sigma_k)=2k+1$ and use again Proposition \ref{prop:auxdelta} to obtain \eqref{auxx}.
In view of the uniform boundedness of the derivatives of $\psi_j(\sigma)$ and $\varphi_{\mathcal{Z}^{(DD)}}(\sigma)$ we deduce that there exist intervals $\mathcal{I}_k^{(DD)}=(\tilde\sigma_k'', \tilde\sigma_k)$ of length uniformly bounded below such that for any $\sigma \in \mathcal{I}_k^{(DD)}$, $[\psi_j(\sigma)]=k$, $j=1,2$, and $[\varphi_{\mathcal{Z}^{(DD)}}(\sigma)]=2k$. This implies \eqref{equalityDirichlet}, completing the proof of Theorem \ref{thm:mainpolygons} modulo the proof of Proposition \ref{prop:auxdelta}.

Let us now  prove Proposition \ref{prop:auxdelta}. We will need the following elementary linear algebra lemma.
\begin{lemma}
\label{linalgebra}
There exists a constant $C>0$ such that 
\[
|\arg (\T(\sigma)\hbv_1)-\arg (\T(\sigma)\hbv_2)|<C|\arg\hbv_1-\arg \hbv_2|
\]
for any $\hbv_1, \hbv_2 \in \UC$ and any $\sigma\ge 0$.
\end{lemma}
\begin{proof}
The matrices $\T(\sigma)$ are products of rotations depending on $\sigma$ and symmetric matrices independent of $\sigma$. The rotations preserve the angles and could be therefore ignored. It is therefore sufficient to verify the statement of the lemma for a single symmetric matrix of determinant one. Changing coordinates, we may assume that the matrix is symmetric with eigenvalues $\tau$ and $1/\tau$. The result then follows from an explicit computation that is left to the reader. 
\end{proof}

It remains to prove Proposition \ref{prop:auxdelta}.

\begin{proof}[Proof of Proposition \ref{prop:auxdelta}] It suffices to prove the inequality \eqref{eq:auxdelta} for $j=1$ and $\aleph=N$, all other cases are proved similarly.
Choose $\varepsilon>0$ small enough so that 
\begin{equation}\label{eq:epsbound}
\varepsilon<\frac{1}{C+2},
\end{equation} 
where  $C$ is  from Lemma \ref{linalgebra}. For brevity we will denote in this proof 
\[
\alpha(\sigma):=\arg(\hbe_1(\sigma)),\quad 
\beta(\sigma):=\arg(\T(\sigma)\hbe_1(\sigma)),\quad 
\gamma(\sigma):=\arg(\T(\sigma)\N);
\]
then
\begin{equation}\label{eq:phipsialphabetagamma}
\psi(\sigma)=\frac{\beta(\sigma)-\alpha(\sigma)}{2\pi},\qquad
\varphi_{\mathcal{Z}^{(NN)}}(\sigma)=\frac{\gamma(\sigma)}{\pi}.
\end{equation}
We recall also that by assumption \eqref{argassump}, 
\begin{equation}\label{eq:alphabounds}
0\le \alpha(\sigma)<\pi.
\end{equation}

By Lemma \ref{lem:monotone}, the  matrix
$\T(\sigma)$ preserves the order of vectors in terms of their arguments. Re-write \eqref{eq:alphabounds} as
\[
\arg(\N)\le \alpha(\sigma)<\arg(-\N),
\]
then by this monotonicity
\[
\gamma(\sigma)\le \beta(\sigma)<\gamma(\sigma)+\pi.
\]
Subtracting $\alpha(\sigma)+\gamma(\sigma)$ from these inequalities, dividing by $\pi$ and re-arranging with account of \eqref{eq:phipsialphabetagamma} yields
\[
-\frac{\alpha(\sigma)}{\pi}\le 2\psi(\sigma)-\varphi_{\mathcal{Z}^{(NN)}}(\sigma)<1-\frac{\alpha(\sigma)}{\pi},
\]
which implies \eqref{eq:auxdelta} assuming
\begin{equation}\label{eq:alphapirestriction1}
\varepsilon\le \frac{\alpha(\sigma)}{\pi}\le1-\varepsilon.
\end{equation}

To finish the proof we need to consider the situation when \eqref{eq:alphapirestriction1} is not satisfied.  Suppose that $0\le \alpha(\sigma)<\pi\varepsilon$.
Applying Lemma \ref{linalgebra} with $\hbv_1=\hbe_1(\sigma)$ and $\hbv_2=\N$, we obtain
\[
-C\varepsilon\pi<\beta(\sigma)-\gamma(\sigma)<C\varepsilon\pi,
\]
or, equivalently, subtracting  $\alpha(\sigma)$, dividing by $\pi$, and using  \eqref{eq:phipsialphabetagamma},
\[
-(C+1)\varepsilon < -C\varepsilon-\frac{\alpha(\sigma)}{\pi}<2\psi(\sigma)-\varphi_{\mathcal{Z}^{(NN)}}(\sigma)<C\varepsilon-\frac{\alpha(\sigma)}{\pi}\le C\varepsilon,
\]
and  \eqref{eq:auxdelta}  then follows since we have chosen $\varepsilon$ satisfying \eqref{eq:epsbound}.  

The case $1-\pi\varepsilon< \alpha(\sigma)<\pi$ is dealt with in the same way, the only difference being that $\hbv_2=-\N$ is used when applying Lemma  \ref{linalgebra}.
\end{proof}

We have therefore proved Theorem \ref{thm:mainpolygons} under the assumption that the matrix $\T(0)$ is positive, see Remark \ref{rem:pmatrix}. 
If $\T(0)$ is negative the argument is analogous. Indeed, as follows from  Remark \ref{rem:U0pos}, we need to account for an additional rotation by the angle $\pi$ and thus 
subtract $-1/2$ from each of  the two functions $\psi_1(\sigma)$, $\psi_2(\sigma)$, and $-1$ from $\varphi_{\mathcal{Z}^{(\aleph\aleph)}}(\sigma)$ in order to get the analogue of Lemma  \ref{lem:countfunctions}. The rest of the argument remains the same. This completes the proof of Theorem \ref{thm:mainpolygons}. 
\end{proof}

\smallskip

We conclude this subsection with two corollaries of the results obtained above. We present their proofs assuming that  $\mathcal{P}$ and $\mathcal{P}'$ are non-exceptional.
The proof for exceptional polygons could be obtained by a simple modification of this argument using the results of  subsection \ref{subs:exepenum} and is left to the reader.

The first corollary provides a way to  control the Steklov quasi-eigenvalues under perturbations of side lengths, provided all the angles remain the same. Note that this result is used in the proof of Theorem \ref{thm:riesz}.
\begin{corollary}
\label{lemma:pertside}
Let $\mathcal{P}(\balpha, \bell)$ and $\mathcal{P}'(\balpha,\bell')$  be two curvilinear $n$-gons with the same respective angles  and side lengths satisfying $|\ell_i-\ell_i'|\le\varepsilon$ for all
$i=1,\dots,n$ and some $\varepsilon>0$. Let $\sigma_m$ and $\sigma'_m$, $m=1,2,\dots$,  be the quasi-eigenvalues of $\mathcal{P}$ and $\mathcal{P}'$, respectively.
There exists a constant $C>0$ depending only on $\balpha$ such that 
for all $\sigma_m<\frac{1}{\varepsilon}$,
\begin{equation}
\label{sigapprox}
|\sigma_m-\sigma_m'| \le C\sigma_m\varepsilon.
\end{equation}
\end{corollary} 
\begin{proof}
Assume first  that $|\ell_1-\ell_1'|\le \varepsilon$ and $\ell_i=\ell_i'$, $ i=2,\dots, n$.  Without loss of generality we may also assume that $\ell_1'\ge \ell_1$.
Let $V'$ be a point on the side $I_1'$  of the curvilinear polygon $\mathcal{P}' $ which is at the distance $\ell_1$ from $V_1$. Let $\T(\sigma)$ be the lifted matrix corresponding to the polygon $\mathcal{P}$ with the starting point at $V_1$, and  $\T'(\sigma)$ be the similar matrix for $\mathcal{P}'$ with the starting point at $V'$ (which could be viewed as an auxiliary vertex with angle $\pi$). Then it is immediate that $\T'(\sigma)=\R((l_1'-l_1)\sigma)\, \T(\sigma)$.
Therefore, one may choose  the vectors $\hbe_j(\sigma)$ and $\hbe_j'(\sigma)$, $j=1,2$,  for the polygons   $\mathcal{P}$ and 
$\mathcal{P}'$ in such a way that  $\hbe_j(\sigma)=\hbe_j'(\sigma)$, $j=1,2$, for all $\sigma>0$. Moreover, for any
$\sigma_m< 1/\varepsilon$ we have:
\begin{equation}
\label{fff}
|\mathfrak{d}_j^{\mathcal{P}'}(\sigma_m) -2\pi k| =\sigma (l_1'-l_1) \le \sigma_m \varepsilon,
\end{equation}
for some $k \in \mathbb{N}$, where $\mathfrak{d}_j^{\mathcal{P}'}(\sigma)$ is the function  defined by  formula \eqref{def:delta} corresponding to the polygon $\mathcal{P}'$.
Therefore, applying Propositions \ref{prop:deltazeros} and \ref{delta:monotone} we conclude that there is a quasi-eigenvalue $\sigma_M'$ of $\mathcal{P}'$ such that
$\mathfrak{d}_j^{\mathcal{P}'}(\sigma_M')=2\pi k$ and $|\sigma_m-\sigma_M'|\le C\sigma_k\varepsilon$. At the same time, since $\sigma_m \varepsilon<1<\pi$, the index $M$ is uniquely defined, and there there is a natural one-to-one correspondence between the solutions of the equations  $\mathfrak{d}_j^{\mathcal{P}'}(\sigma)=0\, (\operatorname{mod}\,\,  2\pi) $  and  $\mathfrak{d}_j^{\mathcal{P}}(\sigma)=0 \, (\operatorname{mod}\,\,  2\pi)$. Therefore, $m=M$,  and we arrive at \eqref{sigapprox}. The fact that the constant $C$ on the right-hand side of \eqref{sigapprox}  depends only on $\balpha$ follows by inspection of the proofs of Proposition \ref{delta:monotone} and Lemma \ref{lem:phi}.

Consider now the general case, and choose a sequence of polygons $\mathcal{P}^{(k)}(\balpha, \bell^{(k)})$, $k=1, \dots, n$, such that $\bell^{(k)}=(\ell_1',\dots \ell'_k, \ell_{k+1}, \dots, \ell_n)$. Note that $\mathcal{P}^{(n)}=\mathcal{P'}$. The result then follows by induction in $k$.
Indeed, the argument above implies  \eqref{sigapprox} for $k=1$.  The inductive step from $k$ to $k+1$ follows from a simple observation that we may always reorder the vertices so that the $(k+1)$-st side is counted first, and choose the starting point $V'^{(k)}$ appropriately so that 
the corresponding matrices $\T^{(k)}(\sigma)$ and $\T^{(k+1)}(\sigma)$ differ by a composition with rotation as before. This completes the proof of the corollary.
\end{proof}

The second corollary could be viewed as domain monotonicity for Steklov quasi-eigenvalues  with respect to the side lengths of curvilinear polygons.
\begin{corollary} Let $\mathcal{P}(\balpha, \bell)$ and $\mathcal{P}'(\balpha,\bell')$  be two curvilinear polygons with the same respective angles  and side lengths satisfying 
$\ell_i \le \ell_i'$  for all $i=1,\dots,n$.  Then  $\sigma_m \ge \sigma'_m$, $m=1,2,\dots$,  where $\sigma_m$ and $\sigma'_m$ are the quasi-eigenvalues of $\mathcal{P}$ and $\mathcal{P}'$, respectively.
\end{corollary}
\begin{proof} Using the same inductive argument as in the proof of Corollary \ref{lemma:pertside}, it suffices to prove the result if $\ell_1 \le \ell_1'$ and $\ell_i=\ell_i'$, $i=2,\dots, n$. As above, we choose the matrices $\T(\sigma)$ and $\T'(\sigma)$ corresponding to the polygons $\mathcal{P}$ and $\mathcal{P}'$ in such a way that
 $\T'(\sigma)=\R((l_1'-l_1)\sigma)\, \T(\sigma)$. It then follows that $\mathfrak{d}_j^{\mathcal{P}'}(\sigma) \ge \mathfrak{d}_j^{\mathcal{P}}(\sigma)$ for all $\sigma>0$, $j=1,2$.
The result then immediately follows from Propositions \ref{prop:deltazeros} and \ref{delta:monotone}.
\end{proof}
\subsection{Enumeration of quasi-eigenvalues for  exceptional  polygons and zigzags} 
\label{subs:exepenum}
In this subsection we explain how to modify the arguments of subsection \ref{subs:noexc} to the case of polygons and zigzags with 
exceptional angles. We will follow the same outline: decompose a polygon into
zigzag domains, establish natural enumeration for ``basic'' zigzags and show that natural enumeration is preserved under gluing.

In order to proceed with this scheme, we first need  to define the quasi-eigenvalues of an exceptional zigzag. Let $\mathcal{Z}^{(\aleph\beth)}$ be a zigzag with endpoints $P$, $Q$ and exceptional angles  at the vertices $V^\Eangles_1=V_{E_1},\dots, V^\Eangles_K=V_{E_K}$. This zigzag can be represented as a union of exceptional components $\mathcal{Y}_\kappa=\mathcal{Y}_\kappa\left(\balpha^{(\kappa)},\bell^{(\kappa)}\right)$, $\kappa=2,\dots,K$, joining the exceptional vertices  $V^\Eangles_{\kappa-1}$ and $V^\Eangles_\kappa$  (see subsection \ref{sec:exceptional} for notation), and two \emph{endpoint exceptional components} $\mathcal{Y}_1^{(\aleph\Eangles)}=\mathcal{Y}_1^{(\aleph\Eangles)}(\balpha^{(1)},\bell^{(1)})$ and $\mathcal{Y}_{K+1}^{(\Eangles\beth)}=\mathcal{Y}_{K+1}^{(\Eangles\beth)}(\balpha^{(K+1)},\bell^{(K+1)})$, joining $P$ to $V^\Eangles_1$ and $V^\Eangles_K$ to $Q$, respectively, with the boundary condition $\aleph$, $\beth$ imposed at $P$, $Q$, respectively.  Here, $\bell^{(1)}=\left(\ell^{(1)}_1, \dots, \ell^{(1)}_{n_1}\right)$ is the vector of $n_1$ lengths of curvilinear pieces between $P$ and $V^\Eangles_1$, ${\balpha'}^{(1)}=\left(\alpha^{(1)}_1, \dots, \alpha^{(1)}_{n_1-1}\right)$ is the vector of $n_1-1$ non-exceptional angles between these pieces, and $\balpha^{(1)}=\left(\alpha^{(1)}_1,\dots,  \alpha^{(1)}_{n_1-1}, \alpha^\Eangles_1\right)$.   Similarly, $\bell^{(K+1)}=\left(\ell^{(K+1)}_1, \dots, \ell^{(K+1)}_{n_{K+1}}\right)$ are the $n_{K+1}$ lengths of curvilinear pieces between  $V^\Eangles_K$ and $Q$, ${\balpha'}^{(K+1)}=\left(\alpha^{(K+1)}_1, \dots, \alpha^{(K+1)}_{n_{K+1}-1}\right)$ is the vector of $n_{K+1}-1$ non-exceptional angles between these pieces, and $\balpha^{(K+1)}=\left(\alpha^\Eangles_K,\alpha^{(K+1)}_1,\dots,  \alpha^{(K+1)}_{n_{K+1}-1}\right)$.

We have already, in essence,  defined by equation  \eqref{excepeq}  the subsequences of quasi-eigenvalues ``generated'' by  the exceptional components $\mathcal{Y}_\kappa$,  $\kappa=2,\dots,K$. We need now to define the quasi-eigenvalues generated by endpoint exceptional components $\mathcal{Y}_1$ and $\mathcal{Y}_{K+1}$.

Consider the equations
\begin{equation}
\label{eq:endpoints1}
\mathtt{U}({\balpha'}^{(1)},\bell^{(1)},\sigma)\baleph   \cdot \Cvect(\alpha^\Eangles_1) =0,
\end{equation}
and
\begin{equation}
\label{eq:endpoints2}
\mathtt{U}({\balpha'}^{(K+1)},\bell^{(K+1)},\sigma)\Cvect(\alpha^\Eangles_K)  \cdot \bbeth^\perp =0,
\end{equation}
where $\Cvect(\alpha^\Eangles)$ depends on the parity of $\alpha^\Eangles$ and is defined by \eqref{eq:orthogspecial2}. 

\begin{definition}
\label{quasieigenxz}
A number $\sigma \ge 0$ is called a {\it quasi-eigenvalue} of an exceptional zigzag  if  $\sigma$ is a solution of  an equation \eqref{excepeq} with $\kappa=2,\dots,K$,  corresponding to one of the exceptional components, or the equation \eqref{eq:endpoints1} corresponding to the endpoint exceptional component $\mathcal{Y}_1$, or the equation \eqref{eq:endpoints2} corresponding to the endpoint exceptional component $\mathcal{Y}_{K+1}$.
\end{definition}

Let us rewrite equations \eqref{excepeq},  \eqref{eq:endpoints1}, and \eqref{eq:endpoints2} in terms of the matrices acting on the universal cover $\UC$. By analogy with \eqref{eq:hatzigzag}, $\sigma \ge 0$ is a quasi-eigenvalue of an exceptional zigzag $\mathcal{Z}^{(\aleph\beth)}$  if it is a solution of one of the equations
\begin{align}
\arg\left(\U({\balpha'}^{(1)},\bell^{(1)},\sigma) \hat{\baleph}\right)&=\arg(\hat{\Cvect}^\perp(\alpha^\Eangles_1)) \pmod{\pi}, \label{eq:quasizigex1}\\
\arg\left(\U({\balpha'}^{(\kappa)},\bell^{(\kappa)},\sigma) \hat{\Cvect}(\alpha^\Eangles_{\kappa-1})\right)&=\arg( \hat{\Cvect}^\perp(\alpha^\Eangles_\kappa)) \pmod{\pi},\quad \kappa=2,\dots,K,\label{eq:quasizigexkappa}\\
\arg\left(\U({\balpha'}^{(K+1)},\bell^{(K+1)},\sigma)\hat{\Cvect}(\alpha^\Eangles_K)\right)&=\arg( \hat{\beth}) \pmod{\pi}.\label{eq:quasizigexK1}
\end{align}

In order to define the natural enumeration for exceptional zigzags, let us introduce the functions
\begin{align}
\label{eq:varphi1}\varphi_{\mathcal{Y}_1^{(\aleph\Eangles)}}(\sigma)&:=\frac{\arg\left(\U({\balpha'}^{(1)},\bell^{(1)},\sigma) \hat{\baleph}\right)-\arg(\hat{\Cvect}^\perp(\alpha^\Eangles_1))}{\pi}, \\
\label{eq:varphikappa}\varphi_{\mathcal{Y}_\kappa}(\sigma)&:=\frac{\arg\left(\U({\balpha'}^{(\kappa)},\bell^{(\kappa)},\sigma) \hat{\Cvect}(\alpha^\Eangles_{\kappa-1})\right)-\arg( \hat{\Cvect}^\perp(\alpha^\Eangles_\kappa))}{\pi},\quad \kappa=2,\dots,K,\\
\label{eq:varphiK1}\varphi_{\mathcal{Y}_{K+1}^{(\Eangles\beth)}}(\sigma)&:=\frac{\arg\left(\U({\balpha'}^{(K+1)},\bell^{(K+1)},\sigma)\hat{\Cvect}(\alpha^\Eangles_K)\right)-\arg( \hat{\beth})}{\pi}.
\end{align}
Obviously, the functions \eqref{eq:varphi1}--\eqref{eq:varphiK1} experience jumps at those and only those real values of $\sigma$ which solve \eqref{eq:quasizigex1}--\eqref{eq:quasizigexK1}, respectively. In order to define the natural enumeration of quasi-eigenvalues for the whole zigzag, we first introduce below the natural enumeration of quasi-eigenvalues  for exceptional and endpoint exceptional components. We want to emphasise that this will be done for auxiliary purposes only. While the quasi-eigenvalues of an exceptional or an endpoint exceptional component are well-defined (they are the real solutions of one of the equations  \eqref{eq:quasizigex1}--\eqref{eq:quasizigexK1}), one cannot associate true eigenvalues to such components.  Indeed,  exceptional and endpoint components are not zigzags, as they do not correspond to any zigzag domain. 
 
\begin{definition}\label{defn:naturalenumerationcomponent}  The quasi-eigenvalue counting functions of exceptional and endpoint exceptional components are defined by setting
\begin{equation}
\label{eq:Nexceptkappa}
\mathcal{N}^q_{\mathcal{Y}_\kappa}(\sigma):=
\begin{cases}
[\varphi_{\mathcal{Y}_\kappa}(\sigma)],\quad&\text{if both }\alpha^\Eangles_{\kappa-1}\text{ and } \alpha^\Eangles_{\kappa} \text{ are odd},\\
[\varphi_{\mathcal{Y}_\kappa}(\sigma)]+\frac{1}{2},\quad&\text{if  }\alpha^\Eangles_{\kappa-1}\text{ and } \alpha^\Eangles_{\kappa} \text{ are of different parity},\\
[\varphi_{\mathcal{Y}_\kappa}(\sigma)]+1,\quad&\text{if both }\alpha^\Eangles_{\kappa-1}\text{ and } \alpha^\Eangles_{\kappa} \text{ are even},
\end{cases}
\end{equation}
for $\kappa=2,\dots, K$, 
\begin{equation}
\label{eq:Nexcept1}
\mathcal{N}^q_{\mathcal{Y}_1^{(\aleph\Eangles)}}(\sigma):=
\begin{cases}
[\varphi_{\mathcal{Y}_1^{(\aleph\Eangles)}}(\sigma)]-\frac{1}{2},\quad&\text{if }\aleph=D\text{ and }\alpha^\Eangles_{\kappa} \text{ is odd},\\
[\varphi_{\mathcal{Y}_1^{(\aleph\Eangles)}}(\sigma)],\quad&\text{if }\aleph=D\text{ and }\alpha^\Eangles_{\kappa} \text{ is even},\\
[\varphi_{\mathcal{Y}_1^{(\aleph\Eangles)}}(\sigma)]+\frac{1}{2},\quad&\text{if }\aleph=N\text{ and }\alpha^\Eangles_{\kappa} \text{ is odd},\\
[\varphi_{\mathcal{Y}_1^{(\aleph\Eangles)}}(\sigma)]+1,\quad&\text{if }\aleph=N\text{ and }\alpha^\Eangles_{\kappa} \text{ is even},
\end{cases}
\end{equation}
and
\begin{equation}
\label{eq:NexceptK1}
\mathcal{N}^q_{\mathcal{Y}_{K+1}^{(\Eangles\beth)}}(\sigma):=
\begin{cases}
[\varphi_{\mathcal{Y}_{K+1}^{(\Eangles\beth)}}(\sigma)]+\frac{1}{2},\quad&\text{if }\alpha^\Eangles_{\kappa} \text{ is odd},\\
[\varphi_{\mathcal{Y}_{K+1}^{(\Eangles\beth)}}(\sigma)]+1,\quad&\text{if }\alpha^\Eangles_{\kappa} \text{ is even},
\end{cases}
\end{equation}
(where in \eqref{eq:NexceptK1} the formulae are the same for $\aleph=D,N$).
\end{definition}

\begin{remark} In view of Definition \ref{defn:naturalenumerationcomponent},  the quasi-eigenvalue counting functions  could be interpreted similarly to Proposition \ref{prop:startcount} in the following way.  
For an exceptional component $\mathcal{Y}_\kappa$, we count all positive solutions of \eqref{eq:quasizigexkappa} if $\alpha^\Eangles_{\kappa-1}$ and $\alpha^\Eangles_{\kappa-1}$ are of the same parity, and all positive solutions plus a half if the parity is different. For the endpoint exceptional components $\mathcal{Y}_1^{(\aleph\Eangles)}$ and $\mathcal{Y}_{K+1}^{(\Eangles\beth)}$, we count all positive solutions of \eqref{eq:quasizigex1} and \eqref{eq:quasizigexK1}, respectively, if the exceptional vertex is even, all positive solutions plus a half  if the exceptional vertex is odd and the boundary condition at the other end is Neumann, and all positive solutions minus a half if the exceptional vertex is odd and the boundary condition at the other end is Dirichlet. This is checked directly by evaluating the quasi-eigenvalue counting functions at $\sigma=0$.
\end{remark}

We can now define the natural enumeration for an exceptional zigzag.

\begin{definition}
\label{def:natenumzigexc}
The natural enumeration of the quasi-eigenvalues of a zigzag $\mathcal{Z}^{(\aleph\beth)}$ with $K$ exceptional angles is defined by setting
\[
\mathcal{N}^q_{\mathcal{Z}^{(\aleph\beth)}}(\sigma):=\mathcal{N}^q_{\vphantom{\mathcal{Y}_{K+1}^{(\Eangles\beth)}}\mathcal{Y}_1^{(\aleph\Eangles)}}(\sigma)
+\sum_{\kappa=2}^{K} \mathcal{N}^q_{\vphantom{\mathcal{Y}_{K+1}^{(\Eangles\beth)}}\mathcal{Y}_\kappa}(\sigma)
+\mathcal{N}^q_{\mathcal{Y}_{K+1}^{(\Eangles\beth)}}(\sigma).
\]
\end{definition}

The following analogue of Theorem \ref{thm:quasizigzags} holds.
\begin{theorem}
\label{thm:quasizigzagexc}
Let $\mathcal{Z}$ be a partially curvilinear exceptional zigzag, and let $\Omega$ be any $\mathcal{Z}$-zigzag domain. For $\aleph, \beth\in\{D,N\}$, let $\lambda_m^{(\aleph\beth)}$ denote the eigenvalues of the mixed eigenvalue problem \eqref{eq:gensloshing}$_{\aleph\beth}$ enumerated in increasing order with account of multiplicities, and let $\sigma_m^{(\aleph\beth)}$ denote the quasi-eigenvalues of the $\aleph\beth$-zigzag $\mathcal{Z}$ in the natural enumeration given by Definition \ref{def:natenumzigexc}. Then
\[
\lambda_m^{(\aleph\beth)}=\sigma_m^{(\aleph\beth)}+o(1)\qquad\text{as }m\to\infty.
\]
\end{theorem}

\begin{proof}
Theorem \ref{thm:quasizigzagexc} is proved similarly to Theorem \ref{thm:quasizigzags}, see subsection \ref{sec:zigzags}. Below we outline the main steps of the argument and leave the details to the reader. 

We start with
\begin{proposition}
\label{lemma1exc}
Theorem \ref{thm:quasizigzagexc} holds  for zigzags consisting of two equal straight sides and an exceptional angle between them.
\end{proposition}
\begin{proof}
Proposition \ref{lemma1exc} is proved similarly to Proposition \ref{lemma1}. The problem is reduced to counting mixed Steklov-Neumann and Steklov-Dirichlet eigenvalues using either symmetry with respect to the bisector or the isospectral transformation described in Lemma \ref{lemma:transpl}. The result then follows by explicitly computing the total loss of quasi-eigenvalues as in the proof of  Proposition \ref{lemma1}  using the results of \cite{sloshing}. 
\end{proof}

The following two propositions are proved using a straightforward adaptation of the proof of Proposition~\ref{lemma2key}.

\begin{proposition}
\label{lemma2exc}
Let $\mathcal{Y}_\kappa$ be an exceptional component joining vertices $V^\Eangles_{\kappa-1}$  and $V^\Eangles_{\kappa}$  of a partially curvilinear zigzag $\mathcal{Z}$ with exceptional angles, and let $W\in \mathcal{Y}_\kappa$  be a point which is not a vertex and such that the zigzag $\mathcal{Z}$ is straight in some neighbourhood of $W$.  Let the boundary condition $\daleth\in\{D,N\}$ be imposed at $W$, which splits the exceptional component $\mathcal{Y}_\kappa$ into two endpoint exceptional components: 
$\mathcal{Y}_{\kappa, \mathrm{I}}^{(\Eangles\daleth)}$ starting at $V^\Eangles_{\kappa-1}$  and ending at $W$, and $\mathcal{Y}_{\kappa, \mathrm{II}}^{(\daleth\Eangles)}$ starting at $W$ and ending at  $V^\Eangles_{\kappa}$. 

Then  there exists $\delta>0$ such that for any $M>0$ there exists an interval $\mathcal{I}_M \subset (M,+\infty)$ of length $\delta$ such that
\[
\mathcal{N}^q_{\mathcal{Y}_{\kappa, \mathrm{I}}^{(\Eangles\daleth)}}(\sigma)
+\mathcal{N}^q_{\mathcal{Y}_{\kappa, \mathrm{II}}^{(\daleth\Eangles)}}(\sigma)
=\mathcal{N}^q_{\mathcal{Y}_\kappa}(\sigma)
\]
for any $\sigma\in \mathcal{I}_M$. 
\end{proposition}

\begin{proposition}
\label{lemma3exc}
Let $\mathcal{Y}^{(\aleph\Eangles)}_1$ be the endpoint exceptional component joining the vertices $P$ and $V^\Eangles_1$ of a partially curvilinear exceptional zigzag $\mathcal{Z}$, with the boundary condition $\aleph\in\{D,N\}$ imposed at its start point $P$. Let $W\in\mathcal{Y}^{(\aleph\Eangles)}_1$  be a point which is not a vertex and such that the zigzag  $\mathcal{Z}$ is straight in some neighbourhood of $X$. Let the boundary condition $\daleth\in\{D,N\}$ be imposed at $W$, which splits the endpoint exceptional component $\mathcal{Y}^{(\aleph\Eangles)}_1$ into the zigzag $\mathcal{Z}^{(\aleph\daleth)}_{1,\mathrm{I}}$ starting at $P$ and ending at $W$ and the endpoint exceptional component $\mathcal{Y}^{(\daleth\Eangles)}_{1,\mathrm{II}}$ starting at $W$ and ending at $V^\Eangles_1$. 

Then  there exists $\delta>0$ such that for any $M>0$ there exists an interval $\mathcal{I}_M^{\aleph,\daleth} \subset (M,+\infty)$ of length $\delta$ such that
\[
\mathcal{N}^q_{\mathcal{Z}^{(\aleph\daleth)}_{1,\mathrm{I}}}(\sigma)
+\mathcal{N}^q_{\mathcal{Y}^{(\daleth\Eangles)}_{1,\mathrm{II}}}(\sigma)
=\mathcal{N}^q_{\mathcal{Y}^{(\aleph\Eangles)}_1}(\sigma)
\]
for any $\sigma\in \mathcal{I}_M^{\aleph,\daleth}$. 
An analogous result holds  for the endpoint exceptional component $\mathcal{Y}^{(\Eangles\beth)}_{K+1}$.
\end{proposition}

Propositions \ref{lemma2exc} and \ref{lemma3exc} imply the analogue of Proposition \ref{lemma2} for partially curvilinear zigzags with exceptional angles. This result combined with 
Proposition \ref{lemma1exc} yields 
Theorem \ref{thm:quasizigzagexc} in the same way as Propositions \ref{lemma1} and \ref{lemma2}
yield Theorem \ref{thm:quasizigzags}.
\end{proof}
We can now prove the main result of this subsection.
\begin{theorem}
\label{thm:mainexc}
Theorem \ref{thm:main} holds for partially curvilinear exceptional polygons.
\end{theorem}
\begin{proof}
Let $\mathcal{P}$ be a partially curvilinear exceptional polygon as defined in subsection \ref{sec:exceptional}.  In view of Definitions \ref{def:quasiexc} and \ref{def:quasiexcmult}, the quasi-eigenvalue counting function for the polygon $\mathcal{P}$ is given by
\begin{equation}
\label{natenumpolexc}
\mathcal{N}_\mathcal{P}^q(\sigma)=\sum_{\kappa=1}^K \mathcal{N}_{\mathcal{Y}_\kappa}(\sigma),
\end{equation}
where $\mathcal{Y}_\kappa$ are the exceptional boundary components of $\mathcal{P}$.

Let us make a cut inside the polygon precisely as in the proof of Theorem \ref{thm:mainpolygons} (see Figure \ref{fig:poly-cut}). As before, the cut produces a zigzag domain with exceptional angles and identified endpoints at some point $V_0$ on a straight part of the boundary $\partial \mathcal{P}$. Imposing either Dirichlet or Neumann boundary condition at $V_0$  and   arguing in the same way as in the proof of Theorem \ref{thm:mainpolygons}, 
we observe that the result follows by Dirichlet-Neumann bracketing from an analogue of equalities \eqref{count:eqpol} relating the quasi-eigenvalue counting functions of $\mathcal{P}$ 
and of the corresponding zigzags. Such an analogue can be easily deduced from  Proposition \ref{lemma2exc}, Definition \ref{def:natenumzigexc} and formula \eqref{natenumpolexc}.
This completes the proof of the theorem.
\end{proof}
\begin{remark}
Recall that new tools were required  to deduce Theorem \ref{thm:main} from Theorem \ref{thm:quasizigzags}, see subsection \ref{subs:noexc}.  The reason is that the quasi-eigenvalue condition \eqref{condperiod} for non-exceptional polygons is a vector-valued condition, unlike the Dirichlet and Neumann boundary conditions for zigzags. On the other hand, the quasi-eigenvalue condition \eqref{natenumpolexc} for an exceptional polygon is scalar and very closely related to the condition \eqref{eq:quasizigexkappa} for an exceptional  zigzag.  This explains why Theorem~\ref{thm:mainexc} is much easier to prove, essentially a direct corollary of Theorem~\ref{thm:quasizigzagexc}.
\end{remark}

Results of this subsection, together with Corollary  \ref{cor:diffbounded}, also imply the following analogue of Proposition \ref{prop:nomoreN}. 

\begin{proposition}\label{prop:nomoreNexc} There exists a $d>0$ and an $N>0$ such that any interval of the real line of length $d$ contains not more than $N$ quasi-eigenvalues of a zigzag.
\end{proposition}
\clearpage\section{Quasi-eigenvalues as roots of trigonometric polynomials}\label{sec:trigpolynoms}

\subsection{Explicit expressions for the entries of \texorpdfstring{$\mathtt{T}(\balpha,\bell,\sigma)$}{T}}

In this section we will prove Theorem \ref{thm:polygoneqn1}; we start, however, by finding explicit expressions for the matrices $\mathtt{T}(\balpha,\bell,\sigma)$.

We recall that  for a binary vector $\bzeta=(\zeta_1,\dots,\zeta_n)\in \pmset{n}=\{\pm 1\}^n$ with cyclic identification $\zeta_{n+1}\equiv\zeta_1$, we let
\begin{equation}\label{eq:Chdef}
\changes(\bzeta):=\{j \in\{1,\dots,n\}\mid \zeta_j\ne \zeta_{j+1} \} 
\end{equation}
denote the set of indices of sign change in $\bzeta$.

Let additionally
\begin{equation}\label{eq:Chtildedef} 
\tilde{\changes}(\bm{\zeta}):=\changes((\zeta_1,\dots,\zeta_n,-1))\cap\{1,\dots,n\}.
\end{equation}
To clarify, in order to obtain $\tilde{\changes}(\bm{\zeta})$, we pad $\bm{\zeta}$ by adding an additional component $-1$, compute the set of sign changes for the resulting vector, and drop $n+1$ from the result if present (i.e. if $\zeta_1=1$).

Let $n\ge 1$, let $\balpha=(\alpha_1,\dots,\alpha_n)\in(\Pi\setminus\Eangles)^n$ be a vector of non-exceptional angles, and let $\bell=(\ell_1,\dots,\ell_n)\in\mathbb{R}_+^n$ be a vector of lengths. We have already established in Section \ref{subsec:propscat} that the matrix $\mathtt{T}_n:=\mathtt{T}(\balpha,\bell,\sigma)$ belongs to the class $\mathcal{M}_1$, and therefore we have
\begin{equation}\label{eq:Tviapandq}
\renewcommand*{\arraystretch}{1.5}
\mathtt{T}_n=\begin{pmatrix} p_n(\balpha,\bell,\sigma)&q_n(\balpha,\bell,\sigma)\\\overline{q_n(\balpha,\bell,\sigma)}&\overline{p_n(\balpha,\bell,\sigma)}\end{pmatrix}
\end{equation}
for some functions $p_n$ and $q_n$ such that $|p_n|^2-|q_n|^2=1$; we use subscript $n$ to emphasise the dependance upon the length of vectors $\balpha$ and $\bell$.   

\begin{theorem}\label{thm:pnqnformulae} We have
\begin{equation}\label{eq:pnformula}
p_n(\balpha,\bell,\sigma)=\frac{1}{\prod\limits_{j=1}^n \sin\left(\frac{\pi^2}{2\alpha_j}\right)}\sum_{\substack{\bm{\zeta}\in\pmset{n}\\\zeta_1=1}} \mathfrak{p}_{\bm{\zeta}} \exp(\ir\bm{\ell}\cdot\bm{\zeta}\sigma),
\end{equation}
and
\begin{equation}\label{eq:qnformula}
q_n(\balpha,\bell,\sigma)=\frac{-\ir}{\prod\limits_{j=1}^n \sin\left(\frac{\pi^2}{2\alpha_j}\right)}\sum_{\substack{\bm{\zeta}\in\pmset{n}\\\zeta_1=1}} \mathfrak{q}_{\bm{\zeta}} \exp(-\ir\bm{\ell}\cdot\bm{\zeta}\sigma),
\end{equation}
where 
\[
\mathfrak{p}_{\bm{\zeta}}=\mathfrak{p}_{\bm{\zeta}}(\balpha):=\prod_{j\in\changes(\bm{\zeta})}  \cos\left(\frac{\pi^2}{2\alpha_j}\right)
\]
 is already defined in  \eqref{eq:pfrakdefn}, and we additionally set
\begin{equation}\label{eq:qfrakdefn}
\mathfrak{q}_{\bm{\zeta}}=\mathfrak{q}_{\bm{\zeta}}(\balpha):=\prod_{j\in \tilde{\changes}(\bm{\zeta})}  \cos\left(\frac{\pi^2}{2\alpha_j}\right),
\end{equation}
assuming the convention $\prod\limits_\varnothing=1$.
\end{theorem}

\begin{proof} We remark, first of all, that the functions $p_n$ and $q_n$ obey the recurrence relations
\begin{align}
p_1&=\frac{1}{\sin\left(\frac{\pi^2}{2\alpha_1}\right)}\exp(\ir\ell_1\sigma),\qquad 
q_1=\frac{-\ir}{\sin\left(\frac{\pi^2}{2\alpha_1}\right)}\cos\left(\frac{\pi^2}{2\alpha_1}\right)\exp(-\ir\ell_1\sigma),\label{eq:tptq1}\\
p_{n+1} &=\frac{1}{\sin\left(\frac{\pi^2}{2\alpha_{n+1}}\right)} \left(\exp(\ir\ell_{n+1}\sigma)p_n-\ir \cos\left(\frac{\pi^2}{2\alpha_{n+1}}\right)\exp(-\ir\ell_{n+1}\sigma)\overline{q_n}\right),\label{eq:pn1}\\
q_{n+1} &= \frac{1}{\sin\left(\frac{\pi^2}{2\alpha_{n+1}}\right)} \left(\exp(\ir\ell_{n+1}\sigma)q_n-\ir \cos\left(\frac{\pi^2}{2\alpha_{n+1}}\right)\exp(-\ir\ell_{n+1}\sigma)\overline{p_n}\right),\label{eq:qn1}
\end{align}
which follows immediately by re-writing  \eqref{eq:Sdef1} and \eqref{eq:Tdef1} as
\[
\begin{split}
\renewcommand*{\arraystretch}{1.5}
&\begin{pmatrix} p_{n+1} & q_{n+1}\\\overline{q_{n+1}}&\overline{p_{n+1}}\end{pmatrix}\\
&=
\frac{1}{\sin\left(\frac{\pi^2}{2\alpha_1}\right)}
\begin{pmatrix}\exp(\ir\ell_{n+1}\sigma)&-\ir\cos\left(\frac{\pi^2}{2\alpha_{n+1}}\right)\exp(-\ir\ell_{n+1}\sigma)\\
\ir\cos\left(\frac{\pi^2}{2\alpha_{n+1}}\right)\exp(\ir\ell_{n+1}\sigma)&\exp(-\ir\ell_{n+1}\sigma)\end{pmatrix}
\begin{pmatrix} p_{n} & q_{n}\\\overline{q_{n}}&\overline{p_{n}}\end{pmatrix}.
\end{split}
\]

We now prove \eqref{eq:pnformula}--\eqref{eq:qnformula} by induction in $n$. 

For $n=1$, the only vector in $\pmset1$ with the first (and only) positive coordinate is $(1)$, with  $\changes((1))=\varnothing$ and 
$\tilde{\changes}((1))=\{1\}$. Thus $\mathfrak{p}_{(1)}=1$ and $\mathfrak{q}_{(1)}=\cos\left(\frac{\pi^2}{2\alpha_1}\right)$, and the statement of the Theorem matches \eqref{eq:tptq1}.

Assume that the statements hold for some $n\ge 1$. Denote $\bell^*=(\bell, \ell_{n+1})\in\mathbb{R}_+^{n+1}$ and $\bm{\zeta}^*=(\bm{\zeta},\zeta_{n+1})\in\pmset{n+1}$. Then by \eqref{eq:pn1},
\[
\begin{split}
&p_{n+1} \prod\limits_{j=1}^{n+1} \sin\left(\frac{\pi^2}{2\alpha_j}\right)\\
&=\sum_{\substack{\bm{\zeta}\in\pmset{n}\\\zeta_1=1}} \mathfrak{p}_{\bm{\zeta}} \exp(\ir\bm{\ell}\cdot\bm{\zeta}\sigma+\ir\ell_{n+1}\sigma)
+\cos\left(\frac{\pi^2}{2\alpha_{n+1}}\right)\sum_{\substack{\bm{\zeta}\in\pmset{n}\\\zeta_1=1}} \mathfrak{q}_{\bm{\zeta}} \exp(\ir\bm{\ell}\cdot\bm{\zeta}\sigma-\ir\ell_{n+1}\sigma)\\
&=\sum_{\substack{\bm{\zeta^*}\in\pmset{n+1}\\\zeta_1=\zeta_{n+1}=1}} \mathfrak{p}_{\bm{\zeta}} \exp(\ir\bm{\ell}^*\cdot\bm{\zeta}^*\sigma)
+\cos\left(\frac{\pi^2}{2\alpha_{n+1}}\right)\sum_{\substack{\bm{\zeta}\in\pmset{n}\\\zeta_1=-\zeta_{n+1}=1}} \mathfrak{q}_{\bm{\zeta}} \exp(\ir\bm{\ell}^*\cdot\bm{\zeta}^*\sigma).
\end{split}
\]
A careful analysis of definitions \eqref{eq:Chdef} and \eqref{eq:Chtildedef}  shows that we have
\[
\mathfrak{p}_{\bm{\zeta}^*} =\begin{cases} 
\mathfrak{p}_{\bm{\zeta}}\quad&\text{if }\zeta_1=\zeta_{n+1}=1,\\
\cos\left(\frac{\pi^2}{2\alpha_{n+1}}\right)\mathfrak{q}_{\bm{\zeta}}\quad&\text{if }\zeta_1=-\zeta_{n+1}=1,
\end{cases}
\]
and therefore
\[
p_{n+1} \prod\limits_{j=1}^{n+1} \sin\left(\frac{\pi^2}{2\alpha_j}\right)\\
=\sum_{\substack{\bm{\zeta}^*\in\pmset{n+1}\\\zeta_1=1}} \mathfrak{p}_{\bm{\zeta}^*} \exp(\ir\bm{\ell}^*\cdot\bm{\zeta}^*\sigma),
\]
thus proving \eqref{eq:pnformula}.

Similarly by \eqref{eq:qn1},
\[
\begin{split}
&q_{n+1} \prod\limits_{j=1}^{n+1} \sin\left(\frac{\pi^2}{2\alpha_j}\right)\\
&=-\ir\sum_{\substack{\bm{\zeta}\in\pmset{n}\\\zeta_1=1}} \mathfrak{q}_{\bm{\zeta}} \exp(-\ir\bm{\ell}\cdot\bm{\zeta}\sigma+\ir\ell_{n+1}\sigma)
-\ir\cos\left(\frac{\pi^2}{2\alpha_{n+1}}\right)\sum_{\substack{\bm{\zeta}\in\pmset{n}\\\zeta_1=1}} \mathfrak{p}_{\bm{\zeta}} \exp(-\ir\bm{\ell}\cdot\bm{\zeta}\sigma-\ir\ell_{n+1}\sigma)\\
&=-\ir\sum_{\substack{\bm{\zeta^*}\in\pmset{n+1}\\\zeta_1=-\zeta_{n+1}=1}} \mathfrak{q}_{\bm{\zeta}} \exp(-\ir\bm{\ell}^*\cdot\bm{\zeta}^*\sigma)
-\ir\cos\left(\frac{\pi^2}{2\alpha_{n+1}}\right)\sum_{\substack{\bm{\zeta}\in\pmset{n}\\\zeta_1=\zeta_{n+1}=1}} \mathfrak{p}_{\bm{\zeta}} \exp(-\ir\bm{\ell}^*\cdot\bm{\zeta}^*\sigma).
\end{split}
\]
Once more, an analysis of definitions \eqref{eq:Chdef} and \eqref{eq:Chtildedef} gives
\[
\mathfrak{	q}_{\bm{\zeta}^*} =\begin{cases} 
\cos\left(\frac{\pi^2}{2\alpha_{n+1}}\right)\mathfrak{p}_{\bm{\zeta}}\quad&\text{if }\zeta_1=\zeta_{n+1}=1,\\
\mathfrak{q}_{\bm{\zeta}}\quad&\text{if }\zeta_1=-\zeta_{n+1}=1,
\end{cases}
\]
and therefore
\[
q_{n+1} \prod\limits_{j=1}^{n+1} \sin\left(\frac{\pi^2}{2\alpha_j}\right)=-\ir\sum_{\substack{\bm{\zeta}^*\in\pmset{n+1}\\\zeta_1=1}} \mathfrak{q}_{\bm{\zeta}^*} \exp(-\ir\bm{\ell}^*\cdot\bm{\zeta}^*\sigma),
\]
thus proving \eqref{eq:qnformula}.
\end{proof}

\subsection{Proof of Theorem \ref{thm:polygoneqn1}(a)} We start by making the following simple observation, which follows immediately by comparing \eqref{eq:Fevendefn}, \eqref{eq:Fodddefn}, and \eqref{eq:pnformula}.

\begin{proposition}\label{prop:Fprelation} Let $n\ge 1$, let $\balpha=(\alpha_1,\dots,\alpha_n)\in(\Pi\setminus\Eangles)^n$ be a vector of non-exceptional angles, let $\bell=(\ell_1,\dots,\ell_n)\in\mathbb{R}_+^n$, and let the matrix $\mathtt{T}=\mathtt{T}_n:=\mathtt{T}(\balpha,\bell,\sigma)$ be written in the form \eqref{eq:Tviapandq}. Then
\[
p_n(\balpha,\bell,\sigma)\prod\limits_{j=1}^n \sin\left(\frac{\pi^2}{2\alpha_j}\right)=F_\mathrm{even}(\balpha,\bell,\sigma)+\ir  F_\mathrm{odd}(\balpha,\bell,\sigma).
\]
\end{proposition}

Theorem \ref{thm:polygoneqn1}(a) now follows easily. Indeed, Definition \ref{def:quasi} and Lemma \ref{lem:Tproperties}(a) imply that $\sigma$ is a quasi-eigenvalue if and only if  $\Tr \mathtt{T}_n=2\Re p_n=2$ which is equivalent to $\sigma$ being a root of \eqref{eq:quasieq1}. Moreover, in this case, by Definition \ref{def:quasimult} and Lemma  \ref{lem:Tproperties}(b),  $\sigma>0$ is a double quasi-eigenvalue if and only if $\Im p_n=0$, and therefore \eqref{eq:Fodddefn} holds.

\subsection{Proof of Theorem \ref{thm:polygoneqn1}(b)}  Before proceeding to the actual proof of Theorem \ref{thm:polygoneqn1}(b), we introduce some extra notation. 
Let $n\ge 1$, and let $\balpha\in\Pi^n$, $\bell\in\mathbb{R}_+^n$. We  set, using \eqref{eq:Fevendefn} and \eqref{eq:Fodddefn},
\begin{equation}\label{eq:Fndefn}
F_n(\balpha,\bell,\sigma):=F_\mathrm{even}(\balpha,\bell,\sigma)+\ir  F_\mathrm{odd}(\balpha,\bell,\sigma)=\sum_{\substack{\bm{\zeta}\in\pmset{n}\\\zeta_1=1}} \mathfrak{p}_{\bm{\zeta}}(\balpha) \exp(\ir\bm{\ell}\cdot\bm{\zeta}\sigma)
\end{equation}
using the subscript to emphasise the dependence upon the length $n$ of vectors $\balpha$, $\bell$. We also introduce, by analogy with  \eqref{eq:Fndefn}, the function
\begin{equation}\label{eq:tildeFndefn}
\tilde{F}_n(\balpha,\bell,\sigma):=-\ir\sum_{\substack{\bm{\zeta}\in\pmset{n}\\\zeta_1=1}} \mathfrak{q}_{\bm{\zeta}}(\balpha)  \exp(-\ir\bm{\ell}\cdot\bm{\zeta}\sigma),
\end{equation}
and set additionally
\[
F_0:=1,\qquad\tilde{F}_0:=0.
\]

We note that if $\balpha$ does not contain any exceptional angles, then by Theorem \ref{thm:pnqnformulae} and Proposition \ref{prop:Fprelation},
\begin{equation}\label{eq:Fpqrelations}
F_n(\balpha,\bell,\sigma)=p_{n}(\balpha,\bell,\sigma) \prod\limits_{j=1}^{n} \sin\left(\frac{\pi^2}{2\alpha_j}\right),\qquad \tilde{F}_n(\balpha,\bell,\sigma)=q_{n}(\balpha,\bell,\sigma) \prod\limits_{j=1}^{n} \sin\left(\frac{\pi^2}{2\alpha_j}\right).
\end{equation}
However, unlike $p_n$ and $q_n$, the functions $F_n$ and $\tilde{F}_n$ are defined in the presence of exceptional angles as well, and we have the following generalisation of
recurrence relations \eqref{eq:pn1} and \eqref{eq:qn1}, with the identical proof:

\begin{proposition}\label{prop:Frecurrence} Let $n\ge 1$, let $\balpha=(\alpha_1,\dots,\alpha_n)\in\Pi^n$, $\bell=(\ell_1,\dots,\ell_n)\in\mathbb{R}_+^n$, and let additionally $\balpha'=(\alpha_1,\dots,\alpha_{n-1})\in\Pi^{n-1}$, $\bell'=(\ell_1,\dots,\ell_{n-1})\in\mathbb{R}_+^{n-1}$ (or both empty if $n=1$). Then
\begin{align}
F_{n}(\balpha,\bell,\sigma) &= \exp(\ir\ell_{n}\sigma)F_{n-1}(\balpha',\bell',\sigma)-\ir \cos\left(\frac{\pi^2}{2\alpha_{n}}\right)\exp(-\ir\ell_{n}\sigma)\overline{\tilde{F}_{n-1}(\balpha',\bell',\sigma)},\label{eq:Fnrec}\\
\tilde{F}_{n}(\balpha,\bell,\sigma) &= \exp(\ir\ell_{n}\sigma)\tilde{F}_{n-1}(\balpha',\bell',\sigma)-\ir \cos\left(\frac{\pi^2}{2\alpha_{n}}\right)\exp(-\ir\ell_{n}\sigma)\overline{F_{n-1}(\balpha',\bell',\sigma)},\label{eq:tildeFnrec}
\end{align}
\end{proposition}  

We can now proceed with the proof of Theorem \ref{thm:polygoneqn1}(b) proper. Assume the notation of Proposition \ref{prop:Frecurrence}, and consider one exceptional boundary component consisting of $n\ge 1$ smooth pieces  joining exceptional angles $\alpha_0$ and $\alpha_n$; the $n-1$ non-exceptional angles between the pieces are collected in the vector $\balpha'$. We need to show that   
\begin{equation}
\label{eq:excepeqexplicit}
\mathtt{U}\left(\balpha', \bell,\sigma\right) \Cvect\left(\alpha_0\right) \cdot \Cvect\left(\alpha_n\right)=0\quad\iff\quad F_\mathrm{even/odd}\left(\balpha,\bell,\sigma\right)=0,
\end{equation}
where
\begin{equation}\label{eq:UviaBTprime}
\mathtt{U}\left(\balpha', \bell,\sigma\right):=\mathtt{B}\left(\ell_n,\sigma\right) \mathtt{T}\left(\balpha', \bell',\sigma\right),
\end{equation}
cf. \eqref{eq:UviaBT}, $\Cvect(\alpha)$ is defined by \eqref{eq:orthogspecial2} and \eqref{eq:orthogspecial1}, and 
\[
F_\mathrm{even/odd}\left(\balpha,\bell,\sigma\right)=
\begin{cases}
F_\mathrm{even}\left(\balpha,\bell,\sigma\right)=\Re\left(F_n\left(\balpha,\bell,\sigma\right)\right)\quad&\text{if }\Odd(\alpha_0)=\Odd(\alpha_n),\\
F_\mathrm{odd}\left(\balpha,\bell,\sigma\right)=\Im\left(F_n\left(\balpha,\bell,\sigma\right)\right)\quad&\text{if }\Odd(\alpha_0)\ne\Odd(\alpha_n).
\end{cases}
\]

Using \eqref{eq:UviaBTprime}, \eqref{eq:Bdef1}, \eqref{eq:Tviapandq}, and \eqref{eq:Fpqrelations}, we re-write the left equation in \eqref{eq:excepeqexplicit} as
\begin{equation}
\label{eq:excepeqexplicit1}
\renewcommand*{\arraystretch}{1.5}
\frac{1}{ \prod\limits_{j=1}^{n-1} \sin\left(\frac{\pi^2}{2\alpha_j}\right)}
\begin{pmatrix} \exp(\ir\ell_n\sigma)F_{n-1}(\balpha',\bell',\sigma)&\exp(\ir\ell_n\sigma)\tilde{F}_{n-1}(\balpha',\bell',\sigma)\\
\exp(-\ir\ell_n\sigma)\overline{\tilde{F}_{n-1}(\balpha',\bell',\sigma)}&\exp(-\ir\ell_n\sigma)\overline{F_{n-1}(\balpha',\bell',\sigma)}
\end{pmatrix} \Cvect\left(\alpha_0\right) \cdot \Cvect\left(\alpha_n\right)=0,
\end{equation}
and drop the non-zero product in the denominator from now on.

We now have to consider four cases:
\begin{itemize}
\item[(i)] $\alpha_0$ even, $\alpha_n$ even;
\item[(ii)] $\alpha_0$ even, $\alpha_n$ odd;
\item[(iii)] $\alpha_0$ odd, $\alpha_n$ even;
\item[(iv)] $\alpha_0$ odd, $\alpha_n$ odd.
\end{itemize}

In cases (i), (ii) we substitute $\Cvect(\alpha_0)=\Cvect_\mathrm{even} = \frac{ \er^{-\ir \pi/4}}{\sqrt{2}}\begin{pmatrix} 1  \\ \ir \end{pmatrix}$ into \eqref{eq:excepeqexplicit1} to get
\begin{equation}
\label{eq:excepeqexplicit2}
 \renewcommand*{\arraystretch}{1.5}
 \frac{ \er^{-\ir \pi/4}}{\sqrt{2}} \begin{pmatrix}
 \exp(\ir\ell_n\sigma)F_{n-1}(\balpha',\bell',\sigma)+\ir \exp(\ir\ell_n\sigma)\tilde{F}_{n-1}(\balpha',\bell',\sigma)\\
 \exp(-\ir\ell_n\sigma)\overline{\tilde{F}_{n-1}(\balpha',\bell',\sigma)}+\ir\exp(-\ir\ell_n\sigma)\overline{F_{n-1}(\balpha',\bell',\sigma)}
 \end{pmatrix} \cdot \Cvect\left(\alpha_n\right)=0.
\end{equation}

 In case (i), substituting further $\Cvect(\alpha_n)=\Cvect_\mathrm{even} = \frac{ \er^{-\ir \pi/4}}{\sqrt{2}}\begin{pmatrix} 1  \\ \ir \end{pmatrix}$ (and recalling that our definition of the dot product involves complex conjugation of the second argument) we obtain, after minimal simplifications,
 \[
\Re\left( \exp(\ir\ell_n\sigma)F_{n-1}(\balpha',\bell',\sigma)-\ir \exp(-\ir\ell_n\sigma)\overline{\tilde{F}_{n-1}(\balpha',\bell',\sigma)}\right)=0.
 \]
 Using now  \eqref{eq:Fnrec} with account of $\cos\left(\frac{\pi^2}{2\alpha_{n}}\right)=\Odd(\alpha_n)=1$, we arrive at the required equivalent equation 
 $\Re(F_n(\balpha,\bell,\sigma))=0$, thus proving \eqref{eq:excepeqexplicit} in case (i).
 
 Similarly, in case (ii),  substituting $\Cvect(\alpha_n)=\Cvect_{\mathrm{odd}} = \frac{ \er^{\ir \pi/4}}{\sqrt{2}}\begin{pmatrix} 1  \\ -\ir \end{pmatrix}$ into \eqref{eq:excepeqexplicit2}, we obtain after simplifications
 \[
\Im\left( \exp(\ir\ell_n\sigma)F_{n-1}(\balpha',\bell',\sigma)+\ir \exp(-\ir\ell_n\sigma)\overline{\tilde{F}_{n-1}(\balpha',\bell',\sigma)}\right)=\Im(F_n(\balpha,\bell,\sigma))=0,
 \]
 (where we again used  \eqref{eq:Fnrec} but now with $\cos\left(\frac{\pi^2}{2\alpha_{n}}\right)=\Odd(\alpha_n)=-1$), proving \eqref{eq:excepeqexplicit} in case (ii).
 
 The cases (iii) and (iv) are similar and are left to the reader.
 
\subsection{Zigzag quasi-eigenvalues as roots of trigonometric polynomials}
 
In this subsection we briefly discuss trigonometric equations whose roots give the quasi-eigenvalues of zigzags and zigzag domains. 

Let $n\ge 1$, let $\balpha=\balpha'=(\alpha_1,\dots,\alpha_{n-1})\in(\Pi\setminus\Eangles)^{n-1}$, let  $\bell=(\ell_1,\dots,\ell_n)\in\mathbb{R}_+^n$, and let $\mathcal{Z}=\mathcal{Z}(\balpha,\bell)$ be a curvilinear $n$ piece zigzag (domain). The quasi-eigenvalues of a corresponding $\aleph\beth$-zigzag $\mathcal{Z}^{(\aleph\beth)}$ are prescribed by Definition \ref{def:quasizigzag}. Set additionally $\bell'=(\ell_1,\dots,\ell_{n-1})\in\mathbb{R}_+^{n-1}$. 

\begin{theorem} The quasi-eigenvalues of a  $\aleph\beth$-zigzag $\mathcal{Z}^{(\aleph\beth)}$, $\aleph,\beth\in\{N,D\}$, are the non-negative roots of the trigonometric polynomials
\begin{equation}\label{eq:zigzagpolys}
\begin{aligned}
\sum_{\substack{\bm{\zeta}\in\pmset{n-1}\\\zeta_1=1}}\mathfrak{p}_{\bm{\zeta}}(\balpha')\sin\left((\bm{\ell}'\cdot\bm{\zeta}+\ell_n)\sigma\right)
-\mathfrak{q}_{\bm{\zeta}}(\balpha')\cos\left((\bm{\ell}'\cdot\bm{\zeta}-\ell_n)\sigma\right)
\quad&\text{if}\quad\aleph=N, \beth=N,
\\
\sum_{\substack{\bm{\zeta}\in\pmset{n-1}\\\zeta_1=1}}\mathfrak{p}_{\bm{\zeta}}(\balpha')\cos\left((\bm{\ell}'\cdot\bm{\zeta}+\ell_n)\sigma\right)
-\mathfrak{q}_{\bm{\zeta}}(\balpha')\cos\left((\bm{\ell}'\cdot\bm{\zeta}-\ell_n)\sigma\right)
\quad&\text{if}\quad\aleph=N, \beth=D,
\\
\sum_{\substack{\bm{\zeta}\in\pmset{n-1}\\\zeta_1=1}}\mathfrak{p}_{\bm{\zeta}}(\balpha')\cos\left((\bm{\ell}'\cdot\bm{\zeta}+\ell_n)\sigma\right)
+\mathfrak{q}_{\bm{\zeta}}(\balpha')\cos\left((\bm{\ell}'\cdot\bm{\zeta}-\ell_n)\sigma\right)
\quad&\text{if}\quad\aleph=D, \beth=N,
\\
\sum_{\substack{\bm{\zeta}\in\pmset{n-1}\\\zeta_1=1}}\mathfrak{p}_{\bm{\zeta}}(\balpha')\sin\left((\bm{\ell}'\cdot\bm{\zeta}+\ell_n)\sigma\right)
+\mathfrak{q}_{\bm{\zeta}}(\balpha')\cos\left((\bm{\ell}'\cdot\bm{\zeta}-\ell_n)\sigma\right)
\quad&\text{if}\quad\aleph=D, \beth=D
\end{aligned}
\end{equation}
where $\mathfrak{p}_{\bm{\zeta}}$ and $\mathfrak{q}_{\bm{\zeta}}$ are defined by  \eqref{eq:pfrakdefn} and \eqref{eq:qfrakdefn}.
\end{theorem} 

\begin{proof} We act as in the proof of Theorem \ref{thm:polygoneqn1}(b): we first use \eqref{eq:UviaBTprime} and then arrive at the analogue of \eqref{eq:excepeqexplicit1}, in which $\Cvect(\alpha_0)$ and  $\Cvect(\alpha_n)$ should be replaced by $\baleph$ and $\bbeth^\perp$, respectively. Polynomials \eqref{eq:zigzagpolys} are obtained directly from there after substituting in the expressions \eqref{eq:Fndefn} and \eqref{eq:tildeFndefn}, and some elementary manipulations.
\end{proof}
\clearpage\section{ A quantum graph interpretation of quasi-eigenvalues and the Riesz mean}\label{sec:quantum}

\subsection{Proof of Theorem \ref{thm:quantum}}\label{subsec:proofquantum}

We begin by proving Proposition \ref{prop:Dirac}.

\begin{proof}[Proof of Proposition \ref{prop:Dirac}] We give the proof in the case when there are no exceptional angles; the exceptional case is treated similarly, cf.\ Section \ref{subsec:bqm}. We first check that $\Dir$ is symmetric by the following direct calculation  for $\mathbf{f}$ and $\mathbf{g}$ in the domain of $\Dir$, using integration by parts, matching conditions \eqref{matchVj}, and the properties of matrices $\mathtt{A}(\alpha)$ from Remark \ref{rem:Aproperties}, and denoting $\mathtt{D}=\begin{pmatrix}1&0\\0&-1\end{pmatrix}=\mathtt{D}^*$, yielding
\[
\begin{split}
\left(\Dir \mathbf{f}, \mathbf{g}\right)_{(L^2(\mathcal{G}))^2}-\left(\mathbf{f}, \Dir \mathbf{g}\right)_{(L^2(\mathcal{G}))^2}
&=-\ir\sum_{j=1}^n \left.\mathtt{D}\mathbf{f}\cdot \mathbf{g}\right|_{V_j+0}^{V_{j+1}-0}\\
&=\ir\sum_{j=1}^n \left.\left(\mathtt{D} \mathtt{A}(\alpha_j)\mathbf{f}\cdot \mathtt{A}(\alpha_j)\mathbf{g}-\mathtt{D}\mathbf{f}\cdot \mathbf{g}\right)\right|_{V_j-0}\\
&=\ir\sum_{j=1}^n \left. \left(\mathtt{A}(\alpha_j)\mathtt{D} \mathtt{A}(\alpha_j)-\mathtt{D}\right)\mathbf{f}\cdot\mathbf{g}\right|_{V_j-0}=0
\end{split}
\]
(since $\mathtt{A}(\alpha_j)\mathtt{D} \mathtt{A}(\alpha_j)=\mathtt{D}$). The self-adjointness of $\Dir$ now follows by standard techniques similar to \cite{BK13}.

To prove the second part of the statement, suppose that $\sigma$ is an eigenvalue of $\Dir$; then a restriction of the corresponding eigenfunction $\mathbf{f}(s)$ to an edge $I_j$ has the form  $\begin{pmatrix} d_{j,1}\er^{\ir\sigma s} \\ d_{j,2}\er^{-\ir\sigma s} \end{pmatrix}$ with some constants $d_{j,1},d_{j,2}\in\mathbb{C}$ (which can be chosen so that $d_{j,2}=\overline{d_{j,1}}$, cf.\ Remark  \ref{rem:evfromConj} and the discussion in the proof of Proposition \ref{prop:almostorthogonality}). Set
\[
\mathbf{c}_j:=\left.\mathbf{f}(s)\right|_{V_j+0}.
\]
Then it is easily checked that the vectors  $\mathbf{c}_j$ satisfy
\[
\mathbf{c}_{j+1}=\mathtt{A}(\alpha_j)\mathtt{B}(\sigma\ell_j)  \mathbf{c}_{j}.
\]
Repeating now word by word the arguments of Section \ref{subsec:bqm} we see that the eigenvalues $\sigma$ of $\mathcal{D}$ are indeed the roots of  \eqref{eq:Ttrace}.
\end{proof}

We now proceed to the proof of Theorem \ref{thm:quantum}. Note that the vectors $\Cvect_\mathrm{odd}$ and $\Cvect_\mathrm{even}$ defined by \eqref{eq:orthogspecial1} are the eigenvectors of the matrix ${\mathtt A}(\alpha)$ with the eigenvalues
$\eta_1(\alpha)=\tan\left(\frac{\pi^2}{4\alpha}\right)$ and $\eta_2(\alpha)=\cot\left(\frac{\pi^2}{4\alpha}\right)$, respectively. Therefore, the matrix ${\mathtt A}(\alpha)$ in the basis 
\[
\left\{\frac{1}{\sqrt{2}}\Cvect_\mathrm{odd},\frac{1}{\sqrt{2}}\Cvect_\mathrm{even}\right\}
\] 
takes the diagonal form
\[
\begin{pmatrix}
\tan\left(\frac{\pi^2}{4\alpha}\right) & 0 \\ 0 & \cot\left(\frac{\pi^2}{4\alpha}\right)
\end{pmatrix}.
\]
Let us calculate the operator $\Dir$ in the same basis. The transition matrix is given by 
\[
{\mathtt W}=\frac{1}{2} \begin{pmatrix} 1+\ir & 1-\ir \\1-\ir & 1+\ir \end{pmatrix}, 
\]
and therefore the operator $\Dir$ in the new basis is given by the matrix
\[
{\mathtt W}^{-1}\Dir{\mathtt  W}=\begin{pmatrix} 0 & - \frac{\mathrm{d}}{\mathrm{d}s} \\ \frac{\mathrm{d}}{\mathrm{d}s} & 0  \end{pmatrix}.
\]
and its square $\Dir^2$ by the matrix
\[
{\mathtt W}^{-1}\Dir^2{\mathtt  W}=\begin{pmatrix} -\frac{\mathrm{d}^2}{\mathrm{d}s^2}&0 \\ 0&-\frac{\mathrm{d}^2}{\mathrm{d}s^2} \end{pmatrix}.
\]

Now, every $\mathbf{f}\in (L^2(\mathcal{G}))^2$ can be uniquely written as  
\[
\mathbf{f}=f_\mathrm{odd}\Cvect_\mathrm{odd}+f_\mathrm{even}\Cvect_\mathrm{even},
\] 
with $f_\mathrm{odd}, f_\mathrm{even}\in L^2(\mathcal{G})$. In other words, we have a direct sum decomposition,
\[
 (L^2(\mathcal{G}))^2=\mathbf{L}^2_\mathrm{odd}(\mathcal{G}) \oplus \mathbf{L}^2_\mathrm{even}(\mathcal{G}),
\]
 where
 \[
\mathbf{L}^2_\mathrm{odd}(\mathcal{G}):=  \Cvect_\mathrm{odd}L^2(\mathcal{G}),\qquad 
\mathbf{L}^2_\mathrm{even}(\mathcal{G}):=  \Cvect_\mathrm{even}L^2(\mathcal{G}).
\]
It is easily seen that both spaces $\mathbf{L}^2_\mathrm{odd/even}(\mathcal{G})$ are invariant for the operator $\Dir^2$, and therefore the spectrum of $\Dir^2$ is the union of the spectra of 
\[
\Dir^2_\mathrm{odd/even}:=\left.\Dir^2\right|_{\mathbf{L}^2_\mathrm{odd/even}(\mathcal{G})}.
\]

We claim that the spectra of $\Dir^2_\mathrm{odd}$ and $\Dir^2_\mathrm{even}$ are the same, and both coincide with the spectrum of $\Delta_\mathcal{G}$.
Obviously, $f_\mathrm{odd}\Cvect_\mathrm{odd}$ is in the domain of $\Dir^2$ if and only if both it and $\Dir(f_\mathrm{odd}\Cvect_\mathrm{odd})$ are in the domain of $\Dir$. A straightforward calculation shows that this happens exactly when conditions \eqref{matchcondboth} are satisfied with $f=f_\mathrm{odd}$, and we then have 
\[
\Dir^2_\mathrm{odd}\left(f_\mathrm{odd}\Cvect_\mathrm{odd}\right)=\left(\Delta_\mathcal{G}f_\mathrm{odd}\right)\Cvect_\mathrm{odd}. 
\]
Thus, the spectrum of $\Dir^2_\mathrm{odd}$ coincides with the spectrum of $\Delta_\mathcal{G}$.

A similar argument shows that the domain of $\Dir^2_\mathrm{even}$ consists of vector functions $f_\mathrm{even}\Cvect_\mathrm{even}$ satisfying the ``dual'' matching conditions
\begin{equation}
\label{matchcondbothdual}
\begin{split}
\cos\left(\frac{\pi^2}{4\alpha_j}\right)f|_{V_j+0}&=\sin\left(\frac{\pi^2}{4\alpha_j}\right)\,f|_{V_j-0},\\
\sin\left(\frac{\pi^2}{4\alpha_j}\right)f'|_{V_j+0}&=\cos\left(\frac{\pi^2}{4\alpha_j}\right)\,f'|_{V_j-0}
\end{split}
\end{equation} 
(with $f=f_\mathrm{even}$); these conditions are obtained from  \eqref{matchcondboth} by simply swapping sines and cosines. Denoting the quantum graph Laplacian subject to matching conditions \eqref{matchcondbothdual} by $\Delta_{\mathcal{G}'}$, we conclude that 
\[
\Dir^2_\mathrm{even}\left(f_\mathrm{even}\Cvect_\mathrm{even}\right)=\left(\Delta_{\mathcal{G}'}f_\mathrm{even}\right)\Cvect_\mathrm{even},
\]
and the spectrum of $\Dir^2_\mathrm{even}$ coincides with the spectrum of $\Delta_{\mathcal{G}'}$.

It remains to show that the spectra of $\Delta_{\mathcal{G}}$ and $\Delta_{\mathcal{G}'}$ coincide. 
It is easy to see that if $f(s)$ is an eigenfunction of $\Delta_{\mathcal{G}}$ corresponding to a non-zero eigenvalue (and therefore not a piecewise constant) then $f'(s)$ is an eigenfunction of $\Delta_{\mathcal{G}'}$  corresponding to the same eigenvalue. The same also holds the other way round. It is now enough to show that the multiplicities of eigenvalue zero coincide. 

By the variational principle, see Remark \ref{rem:varprinciple}, and its analogue for $\Delta_{\mathcal{G}'}$, the only possible eigenfunctions corresponding to eigenvalue zero are piecewise constants.  
In the non-exceptional case, it is easily checked from matching conditions  \eqref{matchcond:exc} or \eqref{matchcondbothdual} that zero is in the spectrum of either operator if and only if 
\[
\prod_{j=1}^n \tan\left(\frac{\pi^2}{4\alpha_j}\right)=\prod_{j=1}^n \cot\left(\frac{\pi^2}{4\alpha_j}\right)=1,
\]
and then it is a simple eigenvalue of either $\Delta_{\mathcal{G}}$ or $\Delta_{\mathcal{G}'}$. In the exceptional case, as follows from \eqref{matchcond:exc} and \eqref{matchcondbothdual}, the only exceptional components $\mathcal{Y}_\kappa$ which have the Neumann conditions at either end are those for which $\Odd\left(\alpha^\mathcal{E}_{\kappa-1}\right)=-\Odd\left(\alpha^\mathcal{E}_{\kappa}\right)=1$ in the case of $\Delta_{\mathcal{G}}$ and  $\Odd\left(\alpha^\mathcal{E}_{\kappa-1}\right)=-\Odd\left(\alpha^\mathcal{E}_{\kappa}\right)=-1$ in the case of  $\Delta_{\mathcal{G}'}$. In both cases the number of such components, and therefore the multiplicity of eigenvalue zero, is $\frac{\#\mathfrak{K}_\mathrm{odd}}{2}$, see also Remark \ref{rem:multiplicityexc}.

This completes the proof of Theorem \ref{thm:quantum}.

\subsection{Proof of Theorem \ref{thm:polygoneqn0}}\label{subs:quantumgrapheigenvalues}  As mentioned in Remark \ref{rem:secular}, we prove Theorem \ref{thm:polygoneqn0} by explicitly constructing the secular equation for the eigenvalues of the quantum graph $\mathcal{G}$ and invoking Theorem \ref{thm:quantum}. The method we use is standard, and we mostly follows \cite{KoSm, KN2, BK13, Ber17} and a private communication from P. Kurasov.

Suppose that $\nu=\sigma^2>0$ is an eigenvalue of $\Delta_{\mathcal{G}}$. We write a corresponding eigenfunction $f(s)$ on the edges $I_j$ and $I_{j+1}$ adjacent to the vertex $V_j$, $j=1,\dots,n$, using the local coordinate $s_j$ such that $s_j|_{V_j}=0$ (the coordinates near adjacent vertices are related by \eqref{eq:sjjm1}; cf.\ Figure \ref{fig:polygon_coords}):
\begin{align}
f|_{I_j}(s_j)&=a_1^{(j)}\er^{\ir\sigma s_j}+b_1^{(j)}\er^{-\ir\sigma s_j},\label{eq:fj}\\
f|_{I_{j+1}}(s_j)&=b_2^{(j)}\er^{\ir\sigma s_j}+a_2^{(j)}\er^{-\ir\sigma s_j},\label{eq:fjp1}
\end{align}
with some constants $a_k^{(j)}, b_k^{(j)}\in\mathbb{C}$, $k=1,2$. Substituting \eqref{eq:fj}--\eqref{eq:fjp1} into matching conditions \eqref{matchcondboth}, resolving with respect to $a_1^{(j)}, a_2^{(j)}$, and combining the results for $j=1,\dots,n$,  shows that the vectors
\[
\mathbf{a}:=\begin{pmatrix}a_1^{(1)}\\a_2^{(1)}\\\vdots\\a_1^{(n)}\\a_2^{(n)}\end{pmatrix}\in\mathbb{C}^{2n}\qquad\text{and}\qquad \mathbf{b}:=\begin{pmatrix}b_1^{(1)}\\b_2^{(1)}\\\vdots\\b_1^{(n)}\\b_2^{(n)}\end{pmatrix}\in\mathbb{C}^{2n}
\]
are related by the \emph{vertex scattering matrix} 
\begin{equation}\label{eq:scvg}
\mathtt{Sc}^\text{v}_\mathcal{G}=\mathtt{Sc}^\text{v}_\mathcal{G}(\balpha):=
\begin{pmatrix}
-\cos\frac{\pi^2}{2\alpha_1}&\sin\frac{\pi^2}{2\alpha_1}&&&&&\\
\sin\frac{\pi^2}{2\alpha_1}&\cos\frac{\pi^2}{2\alpha_1}&&&&&\\
&&-\cos\frac{\pi^2}{2\alpha_2}&\sin\frac{\pi^2}{2\alpha_2}&&&\\
&&\sin\frac{\pi^2}{2\alpha_2}&\cos\frac{\pi^2}{2\alpha_2}&&&\\
&&&&\ddots&&\\
&&&&&-\cos\frac{\pi^2}{2\alpha_n}&\sin\frac{\pi^2}{2\alpha_n}\\
&&&&&\sin\frac{\pi^2}{2\alpha_n}&\cos\frac{\pi^2}{2\alpha_n}
\end{pmatrix}
\end{equation}
as 
\begin{equation}\label{eq:ascvb}
\mathbf{a}=\mathtt{Sc}^\text{v}_\mathcal{G}\mathbf{b}.
\end{equation}
We note that $\mathtt{Sc}^\text{v}_\mathcal{G}$ is unitary and that $\det\mathtt{Sc}^\text{v}_\mathcal{G}=(-1)^n$.  Note also that the blocks of the vertex scattering matrix \eqref{eq:scvg} differ from the Peters solution scattering matrix \eqref{eq:scbeach} due to the change of basis given by $\mathtt{W}$.

We now re-write \eqref{eq:fj} in the variable $s_{j-1}$ using  \eqref{eq:sjjm1}:
\[
f|_{I_j}=b_2^{(j-1)}\er^{\ir\sigma s_{j-1}}+a_2^{(j-1)}\er^{-\ir\sigma s_{j-1}}=b_2^{(j-1)}\er^{\ir\sigma\ell_j}\er^{\ir\sigma s_j}+a_2^{(j-1)}\er^{-\ir\sigma\ell_j}\er^{-\ir\sigma s_j}.
\]
Comparing this with \eqref{eq:fj}, resolving the resulting equations with respect to $b_2^{(j-1)}$, $b_1^{(j)}$, and again combining  the results for $j=1,\dots,n$ gives a relation
\begin{equation}\label{eq:bscea}
\mathbf{b}=\mathtt{Sc}^\text{e}_\mathcal{G}\mathbf{a},
\end{equation}
where $\mathtt{Sc}^\text{e}_\mathcal{G}$ is the \emph{edge scattering matrix} 
\begin{equation}\label{eq:sceg}
\mathtt{Sc}^\text{e}_\mathcal{G}=\mathtt{Sc}^\text{e}_\mathcal{G}(\bell):=
\begin{pmatrix}
0&&&&&&\er^{-\ir\sigma\ell_1}\\
&0&\er^{-\ir\sigma\ell_2}&&&&\\
&\er^{-\ir\sigma\ell_2}&0&&&&\\
&&&\ddots&&&\\
&&&&0&\er^{-\ir\sigma\ell_n}&\\
&&&&\er^{-\ir\sigma\ell_n}&0&\\
\er^{-\ir\sigma\ell_1}&&&&&&0
\end{pmatrix}.
\end{equation}
Note that $\mathtt{Sc}^\text{e}_\mathcal{G}$ is also unitary. 

Combining \eqref{eq:ascvb} and \eqref{eq:bscea}, we arrive at the \emph{secular equation} for the quantum graph $\mathcal{G}$,
\begin{equation}\label{eq:secular}
\det(\mathtt{Sc}^\text{v}_\mathcal{G}\mathtt{Sc}^\text{e}_\mathcal{G}-\Id)=0.
\end{equation}
It is well known, see references above, that the positive roots of \eqref{eq:secular} are equal to the square roots of the positive eigenvalues of $\Delta_\mathcal{G}$; moreover, the multiplicity of $\sigma>0$ as a root of \eqref{eq:secular} coincides with the multiplicity of $\nu=\sigma^2$ as an eigenvalue of $\Delta_\mathcal{G}$.

We now proceed with evaluating the determinant in \eqref{eq:secular}. We remark that due to unitarity of $\mathtt{Sc}^\text{v}_\mathcal{G}$ and the fact that it is Hermitian, we have 
\begin{equation}\label{eq:secdet}
\det(\mathtt{Sc}^\text{v}_\mathcal{G}\mathtt{Sc}^\text{e}_\mathcal{G}-\Id)=
\frac{\det\left(\mathtt{Sc}^\text{e}_\mathcal{G}-(\mathtt{Sc}^\text{v}_\mathcal{G})^{-1}\right)}{\det\left((\mathtt{Sc}^\text{v}_\mathcal{G})^{-1}\right)}=
(-1)^n \det\left(\mathtt{Sc}^\text{v}_\mathcal{G}-\mathtt{Sc}^\text{e}_\mathcal{G}\right).
\end{equation}
The matrix $\mathtt{Sc}^\text{v}_\mathcal{G}-\mathtt{Sc}^\text{e}_\mathcal{G}$ is a tridiagonal circulant $2n\times 2n$ matrix, and determinants of such matrices can be evaluated using, for example, \cite[formula (1)]{Mol08}, which in our case after some simplifications reads
\begin{equation}\label{eq:Molinari}
\begin{split}
&\det\left(\mathtt{Sc}^\text{v}_\mathcal{G}-\mathtt{Sc}^\text{e}_\mathcal{G}\right)\\
&\quad=(-1)^n\er^{-\ir\sigma\sum_{j=1}^n\ell_j}\left(-2\prod_{j=1}^n \sin\frac{\pi^2}{2\alpha_j}
+\Tr\left( \tilde{\mathtt{C}}(\alpha_n,\ell_n,\sigma) \tilde{\mathtt{C}}(\alpha_{n-1},\ell_{n-1},\sigma)\cdots \tilde{\mathtt{C}}(\alpha_1,\ell_1,\sigma)\right)\right),
\end{split}
\end{equation}
where the matrices 
\[
\tilde{\mathtt{C}}(\alpha,\ell,\sigma):=
\begin{pmatrix}
\exp(\ir\ell\sigma)&-\ir\cos\frac{\pi^2}{2\alpha}\exp(-\ir\ell\sigma)\\
\ir\cos\frac{\pi^2}{2\alpha}\exp(\ir\ell\sigma)&\exp(-\ir\ell\sigma)
\end{pmatrix}
\]
are related to the matrices $\mathtt{C}(\alpha,\ell,\sigma)$ defined in \eqref{eq:Sdef1} by 
\[
\mathtt{C}(\alpha,\ell,\sigma)=\frac{1}{\sin\frac{\pi^2}{2\alpha}}\tilde{\mathtt{C}}(\alpha,\ell,\sigma).
\]
Repeating now word by word the proofs of Theorem \ref{thm:pnqnformulae} and Proposition \ref{prop:Fprelation} and dropping the sine factors in denominators gives 
\[
\Tr\left( \tilde{\mathtt{C}}(\alpha_n,\ell_n,\sigma) \tilde{\mathtt{C}}(\alpha_{n-1},\ell_{n-1},\sigma)\cdots \tilde{\mathtt{C}}(\alpha_1,\ell_1,\sigma)\right)=2F_\mathrm{even}(\balpha,\bell,\sigma),
\]
and \eqref{eq:secdet} becomes, with account of \eqref{eq:Molinari},
\[
\det(\mathtt{Sc}^\text{v}_\mathcal{G}\mathtt{Sc}^\text{e}_\mathcal{G}-\Id)=2\er^{-\ir\sigma\sum_{j=1}^n\ell_j}\left(F_\mathrm{even}(\balpha,\bell,\sigma)-\prod_{j=1}^n \sin\frac{\pi^2}{2\alpha_j}\right)=
2\er^{-\ir\sigma\sum_{j=1}^n\ell_j}F^{\mathcal{P}}(\balpha,\bell,\sigma).
\]
The first two statements of Theorem \ref{thm:polygoneqn0} now follow by dropping the non-zero factor $2\er^{-\ir\sigma\sum_{j=1}^n\ell_j}$ and using Theorem \ref{thm:quantum}. 

To prove the last statement of Theorem \ref{thm:polygoneqn0} concerning the multiplicity of the quasi-eigenvalue $\sigma=0$, we again  use Theorem \ref{thm:quantum} and \cite[Corollary 23]{FKW} which states that the algebraic multiplicity $\tilde{N}$ of $\sigma=0$  as a root of the secular equation \eqref{eq:secular} and the multiplicity $N_0$ of $\nu=0$ as an eigenvalue of $\Delta_\mathcal{G}$ are related by
\[
\tilde{N}=2N_0-|E|+D,
\]
where $|E|$ is the number of edges of the graph, and $D$ is the number of Dirichlet conditions. Since in our case $|E|=D=n$, the result follows immediately.

\subsection{Proof of Theorem \ref{thm:riesz}}

First, note that for any $\varepsilon<\varepsilon_0$,
\begin{equation}
\label{eq:rieszest}
\Riesz(\lambda)-\Riesz^q(\lambda)=O(\lambda^{1-\varepsilon}),
\end{equation}
where $\Riesz^q(\lambda):=\mathcal{R}(\{\sigma_m\}; \lambda)$ denotes the first Riesz mean for the sequence $\{\sigma_m\}$ quasi-eigenvalues of $\mathcal{P}$. Indeed, by Theorem 
\ref{thm:main} we have the estimate  $|\sigma_m-\lambda_m|=O(m^{-\varepsilon})$. At the same time, 
by Weyl's law \eqref{eq:Weylaw}, there are $O(\lambda)$ terms in the sums on the right-hand side of \eqref{eq:rieszdef} for either $\Riesz(\lambda)$ or $\Riesz^q(\lambda)$, and moreover $O(m^{-\varepsilon})$ is equivalent to $O(\lambda_m^{-\varepsilon})$. Putting this all together we get \eqref{eq:rieszest}. Therefore, it suffices to prove that
\begin{equation}
\label{Riesz:quasi}
\Riesz^q(\lambda)=\frac{|\partial \mathcal{P}|}{2 \pi}\lambda^2 + O(\lambda^{\frac{2n}{2n+1}}). 
\end{equation}

Let us  assume first  that all the side lengths of the curvilinear polygon $\mathcal{P}$ are (rationally) commensurable.  Then it follows from  equations \eqref{eq:quasieq1} and \eqref{eq:quasieq2bis}
that the sequence $\sigma_m$  is periodic: there exist $T, M>0$ such that $\sigma_{m+M}=\sigma_{m}+T$ for all $m\ge 1$ (in what follows, we refer to $T$ as the {\it period} of the sequence 
$\sigma_m$).
Moreover, in view of Remark \ref{rem:symmetric}, the roots of equations \eqref{eq:quasieq1} and \eqref{eq:quasieq2bis} are symmetric with respect to $\sigma=0$.  The algebraic multiplicity of $\sigma=0$ is always even, and according to Definitions \ref{def:quasimult} and \ref{def:quasiexcmult} exactly half of these zeros are counted as quasi-eigenvalues. This leads to the following observation: on any interval $[jT, (j+1)T]$, $j=0,1,2,\dots$,  the quasi-eigenvalues are located symmetrically with respect to the center of the interval, i.e. the midpoint of the period. Therefore, the sum of all the quasi-eigenvalues on each such  interval   is equal to $\frac{(2j+1)MT}{2}$. Note  that if $jT$, $j\ge 1$,  is a quasi-eigenvalue of some multiplicity (which is necessarily even in view of the observation above regarding the multiplicity of $\sigma=0$), then we assume that half of these eigenvalues contribute to the interval  $[(j-1), jT]$ and the other half to the interval $[jT,(j+1)T]$.

Assume that $\lambda=kT$ for some  $k\in \mathbb{N}$. Then the previous discussion implies that
\begin{equation}
\label{RieszkT}
\Riesz^q(kT)=MT\sum_{j=0}^{k-1}\frac{2j+1}{2}=\frac{MT k^2}{2}=\frac{M}{2T}\lambda^2,
\end{equation}
which proves \eqref{Riesz:quasi} in this case. Note that the equality $\frac{\pi M}{T}=|\partial \mathcal{P}|$ can be easily deduced from Weyl's law \eqref{eq:Weylaw}. 

Suppose now that $\lambda=kT+\varepsilon$ for some $0<\varepsilon<T$.
Then we have:
\[
\Riesz^q(\lambda)=\int_0^{kT} \mathcal{N}(\{\sigma_m\}; t)\, \dr t+\int_{KT}^\lambda  \mathcal{N}(\{\sigma_m\}; t)\, \dr t=
\frac{MTk^2}{2}+\int_{KT}^\lambda \frac{M}{T} t \, \dr t + O(T)=\frac{M}{2T}\lambda^2+O(T).
\]
Here we have used \eqref{RieszkT} as well as \eqref{eq:Weylaw} to obtain the second equality.
This completes the proof of \eqref{Riesz:quasi} (in fact, in this case the remainder is $O(T)$), and thus of \eqref{eq:rieszcorr} if all $\ell_1,\dots,\ell_n$ are commensurable. 

Next, suppose that $\ell_1,\dots, \ell_n$ are arbitrary real numbers. By the simultaneous version of Dirichlet's approximation theorem (in the form obtained via Minkowski's theorem, see \cite{Mat}), for any real $d>1$ there exist a $\zeta=\zeta(d) \in \mathbb{N}\cap(d,4d)$, and $\xi_j\in\mathbb{N}$,  $j=1,\dots, n$, such that 
with $\ell'_j:=\frac{\xi_j}{\zeta}$ we have
\begin{equation}
\label{approxsides}
|\ell_j-\ell'_j|<\frac{1}{d^{\frac{1}{n}} \zeta }<\frac{1}{d^\frac{n+1}{n}},\qquad j=1,\dots, n.
\end{equation} 
Later on, we will choose $d$ depending on the parameter $\lambda$ and will write $d=d(\lambda)$.

Denote by $\sigma'_m$ the quasi-eigenvalues of a curvilinear polygon $\mathcal{P}'$ with the side lengths $\ell'_1, \dots, \ell'_n$ and the same respective angles as 
$\mathcal{P}$. Assume 
\begin{equation}
\label{auxasslam}
d(\lambda)^\frac{n+1}{n}>\lambda.  
\end{equation}
Applying Corollary \ref{lemma:pertside} we have
\begin{equation}
\label{eq:pertside}
\left|\sigma_m-\sigma'_m\right|<\frac{C\lambda}{d(\lambda)^\frac{n+1}{n}}.
\end{equation}
for all  $\sigma_m<\lambda$ and some constant $C>0$.

Inequality \eqref{eq:pertside} together with Weyl's law \eqref{eq:Weylaw} implies that
\begin{equation}
\label{eq1r}
\Riesz^q(\lambda)=\sum_{\sigma_m\le \lambda} (\lambda- \sigma'_m)+(\sigma'_m-\sigma_m)=\mathcal{R}_{\mathcal{P}'}^q(\lambda)+O\left(\frac{\lambda^2}{d(\lambda)^\frac{n+1}{n}}\right). 
\end{equation}
At the same time, consider the polygon $\mathcal{P}'$. The lengths of all its sides are rational numbers with the common denominator $\zeta<4 d(\lambda)$. Let $T'$ be the period of the sequence  $\sigma'_m$.  It is easy to check from  equations   \eqref{eq:quasieq1} and \eqref{eq:quasieq2bis} that $T'=O(d(\lambda))$.
Therefore, by the result that we have already established for curvilinear polygons with rationally commensurable sides, 
\begin{equation}
\label{eq2r}
\mathcal{R}_{\mathcal{P}'}^q(\lambda)=\frac{\left| \partial\mathcal{P}' \right|}{2\pi}\lambda^2+O(d(\lambda))=\frac{\left| \partial \mathcal{P} \right|}{2\pi}\lambda^2+ 
O\left(\frac{\lambda^2}{d(\lambda)^\frac{n+1}{n}}\right)+O(d(\lambda)),
\end{equation}
where the last equality follows from \eqref{approxsides}. Combining \eqref{eq1r} and \eqref{eq2r} we get
\begin{equation}
\label{eq3r}
\Riesz^q(\lambda)=\frac{| \partial \mathcal{P} |}{2\pi}\lambda^2+ 
O\left(\frac{\lambda^2}{d(\lambda)^\frac{n+1}{n}}\right)+O(d(\lambda)).
\end{equation}
Let us now balance the error terms by choosing $d(\lambda)=\lambda^\frac{2n}{2n+1}$, which satisfies \eqref{auxasslam}. Substituting this into \eqref{eq3r} we obtain
\[
\Riesz^q(\lambda)=\frac{| \partial \mathcal{P} |}{2\pi}\lambda^2+ O\left(\lambda^\frac{2n}{2n+1}\right).
\]
With account of \eqref{eq:rieszest}, this completes the proof of Theorem \ref{thm:riesz}.
\clearpage\section{Layer Potentials}\label{sec:layer}
\subsection{Layer potential operators}

The aim of section \ref{sec:layer} is to extend the results obtained so far for partially curvilinear polygons to the fully curvilinear case. 
Throughout this section we assume that $\Omega_0$, $\Omega$, and $\tilde\Omega$ are curvilinear polygons, all with the same angles in $\Pi$ and the same side lengths in the same order. We also assume that $\Omega_0$ is partially curvilinear. The boundaries of all three domains are thus homeomorphic to the circle of length $L$, denoted $\mathbb{S}_L^1$, where $L$ is the common perimeter of $\Omega_0$, $\Omega$, and $\tilde\Omega$.  In this section, with a slight abuse of notation, we identify the sides $I_j$ of $\Omega_0$  of length $\ell_j$, $j=1,\dots, n$, with their images $I_j \subset \mathbb{S}_L^1$ (which are arcs of length $\ell_j$) under the above homeomorphism. Since the sides of $\Omega$ and $\tilde\Omega$ are of the same length as the sides of $\Omega_0$, the intervals $I_j$ of $\mathbb{S}_L^1$  correspond to the sides of these curvilinear polygons as well.

Let $s$ be a common arc length parameter, with the same orientation and with $s=0$ at the same vertex, on all three boundaries $\partial\Omega_0$, $\partial\Omega$, and $\partial\tilde\Omega$. Let $q_0(s)$, $q(s)$, and $\tilde q(s)$ be clockwise arc length parametrisations of the three boundaries, with outward unit normals $\mathbf{n}_0(s)$, $\mathbf{n}(s)$, and $\tilde{\mathbf n}(s)$. Also let $\gamma_0(s)$, $\gamma(s)$, and $\tilde\gamma(s)$ be the three (signed) curvatures. Our assumptions thus mean that $\gamma_0$ vanishes in a neighbourhood of each vertex. We will be interested in a situation when $q$ and $\tilde q$ (and thus $\gamma$ and $\tilde\gamma$) are close to each other in some $C^l$ norm. The outward normals and curvatures are defined at all points except the finitely many vertices.

We will be comparing the Steklov spectra of these polygons, so let $\mathcal D_0$, $\mathcal D$, and $\tilde{\mathcal D}$ be the Dirichlet-to-Neumann operators on each, with eigenvalues $\{\lambda_{0,m}\}$, $\{\lambda_{m}\}$, and $\{\tilde{\lambda}_m\}$. We will assume, as we can, that all these operators act in the same Hilbert space $L^2(\mathbb{S}_L^1)$. Throughout, we write $\CS^k$ or $\CS^k(\mathbb{S}_L^1)$, with $k=0$ or $k=1$,  for the direct sum $C^k(I_1)\oplus\cdots\oplus C^k(I_n)$; in particular, a function in $\CS^0(\mathbb{S}_L^1)$  need not be continuous at the ends of the intervals $I_j$.   At the same time, the Sobolev space $H^1(\mathbb{S}_L^1)$ is defined in the usual manner.

\begin{theorem}\label{thm:bigcurvthm} Fix a domain $\Omega$ of the type described above and let $\tilde\Omega$ vary within that class. Then there exist constants $C, \delta>0$, depending only on the geometry of $\Omega$, such that if $\tilde\Omega$ satisfies the condition
\begin{equation}\label{eq:gammadiff}
\|\gamma-\tilde\gamma\|_{\CS^1(\mathbb{S}_L^1)}\leq \delta,
\end{equation}
then $\mathcal D-\tilde{\mathcal D}$ is bounded as an operator $L^2(\mathbb{S}_L^1)\to L^2(\mathbb{S}_L^1)$ and further
\[
\|\mathcal D-\tilde{\mathcal D}\|_{L^2\to L^2}\leq C\|\gamma-\tilde\gamma\|_{\CS^1}.
\]
\end{theorem}

\begin{remark} This theorem states that within each class of curvilinear polygons we consider here, the dependence of the Dirichlet-to-Neumann operator (with respect to the operator norm) on the curvature of the boundary (with respect to the $C^1$ norm) is locally Lipschitz continuous.
\end{remark}

\begin{remark} Repeated  applications of Theorem \ref{thm:bigcurvthm} imply that even without assumption \eqref{eq:gammadiff},  $\mathcal D-\tilde{\mathcal D}$ is still bounded from $L^2(\mathbb{S}_L^1)$ to $L^2(\mathbb{S}_L^1)$.
\end{remark}

The proof of this theorem will occupy the remainder of the section. But first we explain how to use this Theorem to extend the results proved in previous sections for  partially curvilinear polygons to the fully curvilinear case. The following   proposition can be viewed as  a sort of a bootstrap argument.

\begin{proposition}\label{prop:autoimprovement} Suppose that $\Omega$ and $\tilde\Omega$ are curvilinear polygons as described above. Suppose further that $\|\mathcal D-\tilde{\mathcal D}\|<\pi/(6L)$, and let $\tilde{\delta}>1$ be as in Corollary \ref{cor:quasimodesfullycurv}. Then $|\sigma_m-\lambda_m|=o(1)$ implies $|\sigma_m-\tilde{\lambda}_m|=O\left(m^{\frac 12(1-\tilde\delta)}\right)$.
\end{proposition}

\begin{proof} The proof begins with the observation that the sequence $\{\sigma_m\}$ must have repeated spectral gaps of size greater than $\pi/2L$. Indeed, if this is not the case, then for all sufficiently large $\lambda$, the counting function for $\{\sigma_m\}$ would be at least $\frac{2L}{\pi}\lambda$, contradicting the Weyl law \cite[Lemma 3.7.4]{BK13}. 
Now partition the sequence $\{\sigma_m\}$ into clusters $\{\sigma_{a_k},\ldots,\sigma_{b_k}\}$, ending each cluster at the first new gap of size greater than $\pi/(2L)$, so that $a_{k+1}=b_k+1$ and $\sigma_{a_{k+1}}-\sigma_{b_{k}}>\pi/(2L)$. Since $|\sigma_m-\lambda_m|\to 0$, the disjoint intervals
\[
\left(\sigma_{a_k}-\frac{\pi}{12L},\sigma_{b_k}+\frac{\pi}{12L}\right),
\]
must eventually contain 
$\lambda_{a_k},\ldots,\lambda_{b_k}$, and no other eigenvalues of $\mathcal D$. On the other hand, $\|\mathcal D-\tilde{\mathcal D}\|<\frac{\pi}{6L}$. So every element of $\{\tilde{\lambda}_m\}$ must be within a distance $\frac{\pi}{6L}$ of some element of $\{\lambda_m\}$, which means that, for sufficiently large $k$, the intervals
\[
\left(\sigma_{a_k}-\frac{\pi}{4L},\sigma_{b_k}+\frac{\pi}{4L}\right),
\]
which are still disjoint, must contain 
$\tilde\lambda_{a_k},\ldots,\tilde\lambda_{b_k}$, and no other eigenvalues of ${\mathcal {\tilde D}}$.

We also note that, by Theorem \ref{thm:injectivity}, there is a map $j:\mathbb N\to\mathbb N$ which is eventually strictly increasing and for which $\left|\sigma_m-\tilde\lambda_{j(m)}\right|\leq Cm^{\frac 12(1-\tilde\delta)}\to 0$. 
These two observations imply that $j(m)=m$ for sufficiently large $m$. \end{proof}

Now we deduce the correct enumeration for any $\Omega$ from  the correct enumeration for $\Omega_0$ using a continuity argument. The key is the following Proposition.

\begin{proposition} Let $\Omega$ be a fully curvilinear polygon. There exists a continuous  family of curvilinear polygons $\Omega_t$, $t\in[0,1]$, all with the same angles and side lengths, with $\Omega_1=\Omega$, and $\Omega_0$ being partially curvilinear. 
\end{proposition}

\begin{proof} We proceed in three steps. First, we build a family $\Omega_{t}'$ with all the required properties except possibly for the preservation of the side lengths. This is easy to do by working locally in a neighbourhood of each vertex: assume that a vertex is at the origin, and that $\Omega$ in the vicinity  of the vertex coincides with 
\[
\left\{(x,y): x>0, y_{1,-}(x)<y<y_{1,+}(x)\right\},
\] 
where $y_\pm(0)=0$, see Figure \ref{fig:Figextra}.

\begin{figure}[htb]
\begin{center}
\includegraphics{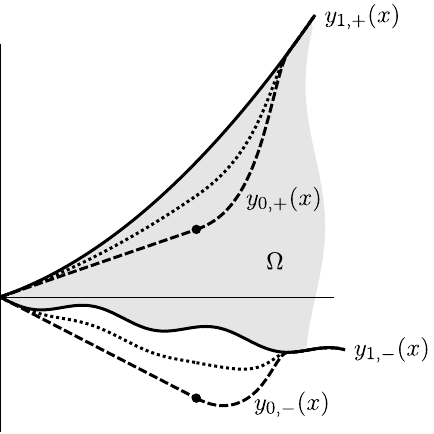}
\end{center}
\caption{A family $\Omega_t'$ (dotted boundary) near a vertex, with the original fully curvilinear polygon $\Omega=\Omega'_1$ (solid boundary) and a constructed partially curvilinear $\Omega'_0$ (dashed boundary; the black dots indicate the ends of the straight parts of the boundary)\label{fig:Figextra}}
\end{figure}

Let $\chi(x)$ denote a smooth nonnegative cut-off function satisfying \eqref{eq:cutofffunction}. We now set, with $\varepsilon>0$ small enough, 
\[
y_{0,\pm}(x):=\chi(x/\varepsilon) y_{1,\pm}'(0)x + (1-\chi(x/\varepsilon)) y_{1,\pm}(x)
\]
(so that these functions are linear near the origin), and choose $\Omega_0'$ to coincide locally with
\[
\left\{(x,y): x>0, y_{0,-}(x)<y<y_{0,+}(x)\right\} 
\]
(so its partially curvilinear), and $\Omega'_t$, $t\in[0,1]$, to coincide locally with
\[
\left\{(x,y): x>0, (1-t) y_{0,-}(x)+ty_{1,-}(x)<y<(1-t) y_{0,+}(x)+ty_{1,+}(x)\right\}. 
\]

On the second step, let $\rho_t$ be the maximum ratio of the corresponding side lengths of $\Omega_t'$ and $\Omega=\Omega_1'$,
\[
\rho_t:=\max\left\{\frac{\ell_1(\Omega_t')}{\ell_1(\Omega)}, \dots, \frac{\ell_n(\Omega_t')}{\ell_n(\Omega)}\right\},
\] 
with $\rho_1=1$. We now construct the continuous family of curvilinear polygons  $\Omega_t''$ as copies of $\Omega_t'$ scaled by a linear factor $\rho_t^{-1}$; this ensures that each side length $\ell_j(\Omega_t'')$ of the resulting family is  less than or equal to the corresponding side length $\ell_j(\Omega)$.

Finally, we adjust the domains $\Omega_t''$ to increase, if required, the side lengths $\ell_j(\Omega_t'')=\rho_t^{-1}\ell_j(\Omega_t')$ back to $\ell_j(\Omega)$ for all $j$, while leaving a neighbourhood of each vertex unchanged. This may be done, for example, by adding smooth oscillations with the appropriate $t$-dependent amplitude and frequency to each side away from the corners. The result is a continuous family $\Omega_t$ of curvilinear polygons with $\Omega_1=\Omega$, and $\Omega_0$ partially curvilinear, all $\Omega_t$ having the same angles and side lengths as $\Omega$.
\end{proof}

Along such a family $\Omega_t$, the curvature depends continuously on $t$ in the $\CS^1$-norm, and therefore by Theorem \ref{thm:bigcurvthm}, the operator $\mathcal{D}_t$ depends continuously on $t$ in the $L^2$ norm. Since $t\in[0,1]$ and $[0,1]$ is compact, this dependence is uniformly continuous, so there exists $\varepsilon>0$ such that $|t-t'|\leq\varepsilon$ implies that $\|\mathcal{D}-\tilde{\mathcal D}\|<\pi/(6L)$. Finally, by Theorem \ref{thm:main} for the \emph{partially} curvilinear polygon $\Omega_0$, which we proved in Section \ref{sec:completeness}, we know that for $\Omega_0$, $|\sigma_m-\lambda_m(\Omega_0)|=o(1)$. So by one use of Proposition \ref{prop:autoimprovement}, we conclude that for all $t\in[0,\varepsilon]$,
\[
|\sigma_m-\lambda_m(\Omega_t)|=O(m^{(-\tilde\delta+1)/2}).
\]
But then a second use gets us the same for all $t\in[0,2\varepsilon]$, and so on until we reach $1$ in finitely many steps. Therefore, we have
\begin{theorem} For any curvilinear polygon $\Omega$,
\[|\sigma_m-\lambda_m|=O(m^{(-\tilde\delta+1)/2}),\]
where $\tilde\delta$ is defined as in Corollary \ref{cor:quasimodesfullycurv}.
\end{theorem}

Our approach uses the theory of layer potentials. So we let $\SL$ and $\DL$ be the single- and double-layer potential operators on the boundary of our domain $\Omega$. Throughout, we use $\tilde\SL$, $\SL_0$, and analogous expressions to denote the same operators on $\tilde\Omega$ and $\Omega_0$ respectively. Recall that these operators are defined as follows:
\[
\begin{split}
\SL(f)(s)&=\int_{\mathbb{S}_L^1}K_{\SL}(s,s')f(s')\, \dr s';\qquad \DL(f)(s)=\int_{\mathbb{S}_L^1}\left(\frac{\partial}{\partial n(s')}K_{\SL}(s,s')\right)f(s')\, \dr s',\\
K_{\SL}(s,s')&=\frac{1}{2\pi}\log |q(s)-q(s')|,
\end{split}
\]
where $n(s')$ is the outward unit normal and the integral for $\DL$ is a principal value integral.

We now collect some basic facts about these operators. First, we have the \emph{Calderon jump relations}, stated in \cite[Chapter 7, (11.35)]{Tay} in the smooth context but also true in the general setting of Lipschitz domains (see for example \cite{Ag2011}):
\begin{equation}\label{eq:calderonjump}
\SL\mathcal D=-\frac 12(\operatorname{Id}-\DL).
\end{equation}
Now we give some information about the boundedness properties of these operators, which again hold for Lipschitz domains. Although we cite \cite{pp14}, some of these results are originally due to Verchota \cite{Verchota}.
\begin{proposition}[{\cite[Section 2 and Lemma 3.1]{pp14}}]\label{prop:putinar} The operators
\[
\SL: L^2(\partial\Omega)\to H^1(\partial\Omega),\quad
\DL: L^2(\partial\Omega)\to L^2(\partial\Omega),\quad 
\DL: H^1(\partial\Omega)\to H^1(\partial\Omega)
\]
are all bounded. Moreover, $\SL$ is invertible as long as the capacity of the domain $\Omega$ is not equal to one.
\end{proposition}

\begin{remark} The capacity of a domain (also called the \emph{logarithmic capacity}) scales linearly with the domain itself, and is bounded below by the inradius. We may thus safely assume that the capacity of each of our domains is greater than one. If not, to prove a theorem such as Theorem \ref{thm:bigcurvthm}, we simply scale up the domain(s) so that all inradii and thus all capacities are greater than one, apply the result there, and transform back. Since $\mathcal D$ is homogeneous under scaling, the result follows for all domains. This allows us to make the assumption, which we do throughout, that our single-layer operators $\SL$, $\SL_0$, and $\tilde\SL$ are all invertible. 
\end{remark}

\subsection{Proof of Theorem \ref{thm:bigcurvthm}}

Our first goal is to reduce the proof of Theorem \ref{thm:bigcurvthm} to the proof of the following lemma, which gives bounds on the differences of the single- and double-layer potential operators acting on different domains.
\begin{lemma}\label{lem:differences} There exist constants $C$ and $\delta$ depending only on the geometry of $\Omega$ such that if 
\[
\|\gamma(s)-\tilde\gamma(s)\|_{\CS^1}\leq\delta,
\]
then
\begin{enumerate}
\item $\SL-\tilde\SL$ is bounded from $H^{-1}(\mathbb{S}_L^1)\to H^1(\mathbb{S}_L^1)$, and $\|\SL-\tilde\SL\|_{H^{-1}\to H^1}
\leq C\|\gamma-\tilde \gamma\|_{\CS^1}$. 
\item $\DL-\tilde \DL$ is bounded from $L^2(\mathbb{S}_L^1)\to H^1(\mathbb{S}_L^1)$, and $\|\DL-\tilde\DL\|_{L^2\to H^1}
\leq C\|\gamma-\tilde\gamma\|_{\CS^1}$.
\end{enumerate}
\end{lemma}
We defer the proof of this Lemma to future subsections and now give the proof of Theorem \ref{thm:bigcurvthm}.

\begin{proof}[Proof of Theorem  \ref{thm:bigcurvthm}] The point is that the Calderon jump relations \eqref{eq:calderonjump} for $\Omega$ and $\tilde\Omega$ allow us to write $\mathcal D-\tilde{\mathcal D}$ in terms of $\SL-\tilde\SL$ and $\DL-\tilde\DL$. Specifically, since we assume $\SL$ is invertible, subtracting the jump relation for $\tilde\Omega$ from that for $\Omega$ and rearranging yields
 \begin{equation}\label{eq:differenceofdton}
\mathcal D-\tilde{\mathcal D}=\SL^{-1}\left(\frac 12(\DL-\tilde\DL)-(\SL-\tilde\SL)\tilde{\mathcal D}\right).
\end{equation}
We want the only tildes on the right to be in the differences of layer potential operators, and so a little more rearrangement yields the following formal expression:
\begin{equation}\label{eq:differenceofdton2}
\mathcal D-\tilde{\mathcal D}=\left(I-\SL^{-1}(\SL-\tilde\SL)\right)^{-1}\SL^{-1}\left(\frac 12(\DL-\tilde \DL)-(\SL-\tilde\SL)\mathcal D\right).
\end{equation}
This formal expression can now be justified. First, by Proposition \ref{prop:putinar} and the bounded inverse theorem, $\SL^{-1}$ is bounded from $H^1$ to $L^2$. By the same proposition, $\frac{1}{2}(\operatorname{Id}-\DL)$ is bounded from $H^1$ to itself, and thus by the Calderon jump relation, $\mathcal D$ is bounded from $H^1$ to $L^2$. By self-adjointness and duality, $\mathcal D$ is also bounded from $L^2$ to $H^{-1}$. By Lemma \ref{lem:differences}, the operator
\[\SL^{-1}((\DL-\tilde\DL)-(\SL-\tilde\SL)\mathcal D)\]
is bounded from $L^2$ to $L^2$. Further, its operator norm is bounded by $\|\gamma-\tilde\gamma\|_{\CS^1}$ times a constant depending only on the geometry of $\Omega$ (this absorbs the norms of $\mathcal D$ and $\SL^{-1}$, both of which depend only on $\Omega$).

Finally, as long as $\delta$ is chosen smaller than $\left(2C\|\SL^{-1}\|_{H^1\to L^2}\right)^{-1}$, which depends only on the geometry of $\Omega$, we have that 
\[
\left\|\SL^{-1}(\SL-\tilde\SL)\right\|_{L^2\to L^2}\leq \left\|\SL^{-1}(\SL-\tilde\SL)\right\|_{H^{-1}\to L^2}\leq \frac{1}{2}.
\]
Therefore $I-\SL^{-1}(\SL-\tilde\SL)$ is invertible on $L^2$, with inverse bounded by $2$. The required statement for $\mathcal D-\tilde{\mathcal D}$ now follows immediately from \eqref{eq:differenceofdton2}, and the proof of Theorem \ref{thm:bigcurvthm} is complete.
\end{proof}


\subsection{Proof of Lemma \ref{lem:differences}}

The proof proceeds via a careful analysis of the Schwartz kernels of the operators $\SL-\tilde{\SL}$ and $\DL-\tilde{\DL}$. These kernels are $K_{\SL-\tilde{\SL}}$ and $K_{\DL-\tilde{\DL}}$ respectively and of course the difference of kernels is the kernel of the difference. Each of these operators therefore has an extremely explicit Schwartz kernel which is conducive to direct analysis. 

Our Schwartz kernels live on $\mathbb{S}_L^1\times \mathbb{S}_L^1$,  and have input and output variables which we denote $s$ and $s'$ respectively. We decompose $\mathbb{S}_L^1\times \mathbb{S}_L^1$ as a union of rectangles of the form $I_j\times I_k$ and analyse each kernel on each rectangle separately. The critical results we need are as follows.

\begin{lemma}\label{lem:heresthecontent} There exist constants $C$ and $\delta$ depending only on the geometry of $\Omega$ so that if $||\gamma-\tilde\gamma||_{\CS^1}\leq\delta$, then the three operators with Schwartz kernels
\[
K_{\SL-\tilde\SL},\quad \partial_{s'}K_{\SL-\tilde\SL},\quad K_{\DL-\tilde\DL}
\]
have the following properties:
\begin{enumerate}
\item For each $j$ and $k$, each operator is bounded from $L^2(I_k)\to H^1(I_j)$, with the norm bounded by $C\|\gamma-\tilde\gamma\|_{\CS^1}$;
\item For each $j$ and $k$, each operator kernel is bounded on the rectangle $I_j\times I_k$.
\item For each $j$, each operator kernel is continuous on $I_j\times \mathbb{S}_L^1$, with the possible exception of the two points $V\times V$ where $V$ is an endpoint of $I_j$. In particular, as long as the first variable $s$ is in the interior of $I_j$, none of these kernels have a jump discontinuity as the second variable $s'$ crosses a corner.
\end{enumerate}
\end{lemma}
Now we complete the proof of Lemma \ref{lem:differences}, given Lemma \ref{lem:heresthecontent}.

\begin{proof}
First, we claim that each of the three operators in Lemma \ref{lem:heresthecontent} in fact defines a bounded operator from $L^2(I_k)\to H^1(\mathbb{S}_L^1)$ for each $k$, with norm bounded by $C\|\gamma-\tilde\gamma\|_{\CS^1}$. Each of these operators is the direct sum of the corresponding
operators from $L^2(I_k)$ to $H^1(I_j)$. 
The output of such an operator may not be in $H^1(\mathbb{S}_L^1)$ a priori, as it might not be continuous at the vertices. However, we claim that for any input in $L^2(I_k)$, the output is continuous at the vertices. Assuming this continuity, the $H^1(\mathbb{S}_L^1)$ norm of the output is the sum of the $H^1(I_j)$ norms of each piece, which is enough.

Indeed, the continuity follows from Lemma \ref{lem:heresthecontent}. Let $g(s)$ be any function in $L^2(I_k)$. The output is $\int_{I_k}K(s,s')g(s')\, \dr s'$, and its continuity is a real analysis exercise. Specifically, break this integral into a small ball around each endpoint and a large middle section. The middle integral is continuous since, by part 3 of Lemma \ref{lem:heresthecontent}, $K(s,s')$ is continuous for all values of $s'$ not at the endpoints. And since $K(s,s')$ is bounded by part 2 of the Lemma, the endpoint integrals are small, so a standard $\varepsilon/3$ argument  completes the proof of continuity.

It is now immediate that each of the three operators in fact is bounded from $L^2(\mathbb{S}_L^1)\to H^1(\mathbb{S}_L^1)$ with the same norm bounds (up to a factor of $n$), as their output is simply the sum of the outputs of the operators from $L^2(I_k)\to H^1(\mathbb{S}_L^1)$. This proves the second statement in Lemma \ref{lem:differences}.

For the first statement in Lemma \ref{lem:differences}, part 1 of Lemma \ref{lem:heresthecontent} shows that it is sufficient to prove that
\begin{equation}\label{eq:hminus1tol2}
\|\SL-\tilde\SL\|_{H^{-1}\to H^1}\leq C(\|K_{\SL-\tilde\SL}\|_{L^2\to H^1}+\|\partial_{s'}K_{\SL-\tilde\SL}\|_{L^2\to H^1}),
\end{equation}
with $C$ a constant depending only on the geometry of $\Omega$, and where the norms on the right-hand side are, in an abuse of notation, the operator norms of the operators with those kernels (so that, in particular, the operator denoted by $K_{\SL-\tilde\SL}$ in the right-hand side is the same as $\SL-\tilde\SL$ in the left-hand side).

To this end, let $\mathcal M$ be a Fourier multiplier operator on $\mathbb{S}_L^1$, multiplying each basis element $\mathrm{e}^{2\pi \mathrm{i}m s/L}$ by $1+|m|$. Then, up to a factor of $L$ which we  ignore as it can be absorbed into $C$, $\mathcal M$ is an isometric isomorphism from $L^2$ to $H^{-1}$. Thus
\[
\|\SL-\tilde\SL\|_{H^{-1}\to H^1}=\|(\SL-\tilde\SL)\mathcal M\|_{L^2\to H^1}=\sup_{f\in L^2, \|f\|=1}\left\|\int_{\mathbb{S}_L^1}K_{\SL-\tilde\SL}(s,s')(\mathcal Mf)(s')\, \dr s'\right\|_{H^1}. 
\]
Write out the Fourier expansions of $K_{\SL-\tilde\SL}(s,s')$ and of $f(s')$, ignoring all normalisation constants (which can be absorbed):
\[
K_{\SL-\tilde\SL}(s,s')=\sum_{m\in\mathbb Z}c_m(s)\mathrm{e}^{\mathrm{i}m s'},\quad f(s')=\sum_{m\in\mathbb Z} d_m\mathrm{e}^{\mathrm{i}ms'}.
\]
Then by a direct calculation,
\begin{equation}\label{eq:brokennorms}
\begin{split}
\left\|\SL-\tilde\SL\right\|_{H^{-1}\to H^1}&=\sup_{d_m:\sum d_m^2=1}\left\|\sum_{m\in\mathbb Z}(1+|m|)d_{-m}c_m(s)\right\|_{H^1}\\
&\leq \sup_{d_m:\sum d_m^2=1}\left\|\sum_{m\in\mathbb Z}d_{-m}c_m(s)\right\|_{H^1}+\sup_{d_m:\sum d_m^2=1}\left\|\sum_{m\in\mathbb Z}|m|d_{-m}c_m(s)\right\|_{H^1}.
\end{split}
\end{equation}
But also by direct calculations,
\begin{align*}
\left\|K_{\SL-\tilde\SL}\right\|_{L^2\to H^1}&=\sup_{d_m:\ \sum d_m^2=1}\left\|\sum_{m\in\mathbb Z}d_{-m}c_m(s)\right\|_{H^1},\\
\left\|\partial_{s'}K_{\SL-\tilde\SL}\right\|_{L^2\to H^1}&=\sup_{d_m:\ \sum d_m^2=1}\left\|\sum_{m\in\mathbb Z}(-m)d_{-m}c_m(s)\right\|_{H^1}.
\end{align*}
Although it does not initially look identical, this second norm is the same as the second term of \eqref{eq:brokennorms}, as the signs of the coefficients $d_m$ may be multiplied by $-1$ times sgn $m$. The equation \eqref{eq:hminus1tol2} follows immediately, and with it Lemma \ref{lem:differences}.
\end{proof}

\subsection{Proof of Lemma \ref{lem:heresthecontent}: kernel expressions}

The starting point for our proof of Lemma \ref{lem:heresthecontent} is work of Costabel \cite{cost}. Costabel analyses kernels of operators of the form $\SL-\SL_0$ and $\DL-\DL_0$, comparing a curvilinear polygon to a polygon with straight edges near the corners. From \cite{cost}, one extracts that each of the three operator kernels in Lemma \ref{lem:heresthecontent} is continuous on $I_j\times I_k$, except when $k=j\pm 1$, in which case there is a singularity at $V\times V$. Costabel carefully analyses the asymptotics of the kernels at each singular point and in fact, from his work one can deduce boundedness of each kernel in this singular case as well (we will also see it directly). By writing $\SL-\tilde\SL=(\SL-\SL_0)-(\tilde\SL-\SL_0)$, this proves part 2 of Lemma \ref{lem:heresthecontent}. As for part 3, this almost follows from Costabel, except that we need to prove the following:
\begin{proposition}\label{prop:kernelmatching} If $s$ is a vertex and $s'$ is not, then the kernels $K_{\SL-\tilde\SL}$, $\partial_{s'}K_{\SL-\tilde\SL}$, and $K_{\DL-\tilde\DL}$ are continuous at $(s,s')$.
\end{proposition}
This will follow from our explicit expressions for the kernels and therefore we postpone the proof for the moment, but once it is done, part 3 of Lemma \ref{lem:heresthecontent} follows immediately.

To prove part 1, and also complete the proof of Proposition \ref{prop:kernelmatching}, we must analyse the kernels directly. In order to show part 1 we take advantage of the definition of the $H^1$ norm, and observe that it suffices to show the following \emph{six} kernels induce bounded operators from $L^2(I_k)\to L^2(I_j)$ for each $j$ and $k$, with norms bounded by $C\|\gamma-\tilde\gamma\|_{\CS^1}$:
\[
K_{\SL-\tilde\SL},\quad \partial_{s'}K_{\SL-\tilde\SL},\quad \partial_sK_{\SL-\tilde\SL},\quad \partial_{s}\partial_{s'}K_{\SL-\tilde\SL},\quad K_{\DL-\tilde\DL},\quad\partial_s K_{\DL-\tilde\DL}.
\]
We will do this by showing explicitly that the absolute value of each of these kernels is bounded by $\|\gamma-\tilde\gamma\|_{\CS^1}$ times a simple function which induces a bounded operator from $L^2(I_k)\to L^2(I_j)$. In the case when $k\neq j\pm 1$ this function may be chosen to be a constant, but if $k=j\pm 1$ we must choose this function to be (mildly) singular at $V\times V$.

In order to do this, we write out the six kernels in question. Up to normalizing constants which we ignore as they are not relevant to the argument,
\begin{align}
\label{eq:noderivssingle} 
K_{\SL}(s,s')&=\log |q(s)-q(s')|;
\\
\label{eq:onederivfirstsingle}
\partial_s K_{\SL}(s,s')&=\frac{(q(s)-q(s'))\cdot\dot q(s)}{|q(s)-q(s')|^2};
\\
\label{eq:onederivsecondsingle}
\partial_{s'} K_{\SL}(s,s')&=-\frac{(q(s)-q(s'))\cdot\dot q(s')}{|q(s)-q(s')|^2};
\\
\label{eq:twoderivssingle}
\partial_s\partial_{s'} K_{\SL}(s,s')&=-\frac{\dot q(s)\cdot\dot q(s')}{|q(s)-q(s')|^2}+\frac{2((q(s)-q(s'))\cdot\dot q(s))((q(s)-q(s'))\cdot\dot q(s'))}{|q(s)-q(s')|^4};
\\
\label{eq:noderivsdouble}
K_{\DL}(s,s')&=-\frac{(q(s)-q(s'))\cdot \mathbf{n}(s')}{|q(s)-q(s')|^2};
\\
\label{eq:onederivdouble}
\partial_s K_{\DL}(s,s')&=-\frac{\dot q(s)\cdot\mathbf{n}(s')}{|q(s)-q(s')|^2}+\frac{2((q(s)-q(s'))\cdot\dot q(s))((q(s)-q(s'))\cdot\mathbf{n}(s'))}{|q(s)-q(s')|^4}.
\end{align}
The six kernels we need are these expressions minus the corresponding expressions with tildes. All dots denote derivatives. Note that here $\mathbf{n}(s')$ is a 90-degree rotation of $\dot q(s')$ (the sign is usually irrelevant) and that $\dot q(s)$ and $\dot q(s')$ are unit vectors since we have an arc length parametrisation.

We begin by proving Proposition \ref{prop:kernelmatching}. In fact this is easy because the kernels themselves are separately continuous, so their differences are continuous as well. At a point $(s,s')$ where $s$ is a vertex and $s'$ is not, the functions $q(s)$, $q(s')$ and $\dot q(s')$ are all continuous, though $\dot q(s)$ is not (it has a jump discontinuity at the vertex). By rotation, $\mathbf{n}(s')$ is also continuous. Moreover $|q(s)-q(s')|$ is continuous and nonzero. The same is true for all expressions with tildes. From this it is easy to see that the three expressions \eqref{eq:noderivssingle}, \eqref{eq:onederivsecondsingle}, and \eqref{eq:noderivsdouble} are all continuous, as are their analogues with tildes. This completes the proof of Proposition \ref{prop:kernelmatching}, leaving only the need to prove part 1 of Lemma \ref{lem:heresthecontent}, which will be done with a direct but lengthy calculation.

\subsection{Geometric preliminaries}

As a reminder, we have a fixed domain $\Omega$ and a domain $\tilde\Omega$ which is among the class of domains for which $\|\gamma-\tilde\gamma\|_{\CS^1}$ is bounded by some yet to be chosen small geometric constant $\delta$, depending only on the geometry of $\Omega$. We will, throughout, write
\[
\Gamma:=\|\gamma-\tilde\gamma\|_{\CS^0};\quad \Gamma_1:=\|\gamma-\tilde\gamma\|_{\CS^1},
\]
and our restriction on $\tilde\Omega$ is precisely that $\Gamma\leq\Gamma_1\leq\delta$ for some $\delta$ to be chosen later. Throughout we use $C$ to denote a constant \emph{depending only on the geometry of} $\Omega$ (and in particular not on the geometry of $\tilde\Omega$).

Note that all six kernels \eqref{eq:noderivssingle}--\eqref{eq:onederivdouble} are independent of rotation and translation of the domain $\tilde\Omega$ in $\mathbb R^2$. This allows us a degree of freedom of choice, and we will usually take advantage of this to ensure that $\dot{\tilde q}(s_0)=\dot q(s_0)$ for some specific value $s_0$, chosen conveniently. Of course we cannot ensure this for more than one point at a time, but that will not be necessary.

Our bounds on the kernels \eqref{eq:noderivssingle}--\eqref{eq:onederivdouble} are all based on the basic fact that the curvature $\gamma(s)$ is related to $\ddot q(s)$ by 
\begin{equation}\label{eq:curvdoubledot}\ddot q(s)=\gamma(s)\cdot\Rot(\dot q(s))=\gamma(s)\mathbf{n}(s),\end{equation}
where $\Rot$ is counterclockwise rotation by $\pi/2$. In particular we have $\ddot q(s)\perp\dot q(s)$. The equation \eqref{eq:curvdoubledot} may be used to estimate expressions involving $q$ and $\dot q$ because by the fundamental theorem of calculus,
\begin{equation}\label{eq:calc1}
q(s)=q(s')+(s-s')\dot q(s')+\int_{s'}^s\int_{s'}^t\ddot q(u)\, \dr u\, \dr t,
\end{equation}
\begin{equation}\label{eq:calc2}
\dot q(s)=\dot q(s')+\int_{s'}^s\ddot q(u)\, \dr u,
\end{equation}
with identical expressions for $\tilde q(s)$ and its derivative. Note that \eqref{eq:calc1} and \eqref{eq:calc2} hold only when $s$ and $s'$ are on the same side; otherwise there is a jump discontinuity in $\dot q$, which modifies the expressions in the way one would expect.

The following proposition reflects the fact that if the curvatures of $\partial\tilde{\Omega}$ and $\partial\Omega$ are close to one another, then their boundaries change direction in similar ways.

\begin{proposition}\label{prop:geomprelim} Suppose that $\tilde\Omega$ additionally satisfies the condition that there exists $s_0$ for which $\dot{\tilde q}(s_0)=\dot q(s_0)$. Then the following hold:
\begin{equation}\label{eq:prop95}
|(\dot{\tilde q}(s)-\dot q(s))-(\dot{\tilde q}(s')-\dot q(s'))|\leq C\Gamma|s-s'|;
\end{equation}
\begin{equation}\label{eq:cor96}
|\ddot{\tilde q}(s)-\ddot q(s)|\leq C\Gamma;
\end{equation}
\begin{equation}\label{eq:prop98}
|\dddot{\tilde q}(s)-\dddot q(s)|\leq C\Gamma_1.
\end{equation}
\end{proposition}
\begin{remark} In fact \eqref{eq:prop95} holds without that condition on $\tilde\Omega$, as the left-hand side of \eqref{eq:prop95} is invariant under rotation of $\tilde\Omega$. The others do not. \end{remark}

\begin{proof} Take \eqref{eq:calc2} for $\dot q$ and subtract it from the same for $\dot{\tilde q}$, then plug in \eqref{eq:curvdoubledot} to obtain
\[
(\dot{\tilde q}(s)-\dot q(s))-(\dot{\tilde q}(s')-\dot q(s'))=\int_{s'}^s\int_{s'}^s(\tilde\gamma(u)\Rot(\dot{\tilde q}(u))-\gamma(u)\Rot(\dot q(u)))\, \dr u.
\]
This holds for all $s$ and $s'$, not just those in the same boundary component; as the angles are the same, the jump discontinuities that are added to \eqref{eq:calc2} are the same for both $\Omega$ and $\tilde\Omega$ and therefore cancel. We estimate the integral using the usual analysis trick of adding and subtracting $\gamma(u)\Rot(\dot{\tilde q}(u))$ inside the integral. Observe also that rotation does not change the norm of a vector. Write $F(s)=\dot{\tilde q}(s)-\dot q(s)$ and $G_{s'}(s)=|F(s)-F(s')|$, so that the left-hand side of \eqref{eq:prop95} is $G_{s'}{s}$, and we obtain
\[
G_{s'}(s)\leq\int_{s'}^s(|\tilde\gamma(u)-\gamma(u)|\cdot|\dot{\tilde q}(u)|+|\gamma(u)|\cdot|F(u)|)\, \dr u.
\]
However $\dot{\tilde q}$ is a unit vector, so the first term is bounded by $\Gamma|s-s'|$. And $|\gamma(u)|$ is the curvature of $\Omega$ and thus bounded by $C$, so adding and subtracting $F(s)$ inside the second term gives
\[
G_{s'}(s)\leq \Gamma|s-s'|+C\int_{s'}^{s}|F(u)-F(s')+F(s')|\, \dr u\leq(\Gamma+C|F(s')|)|s-s'|+\int_{s'}^{s}CG_{s'}(u)\, \dr u.
\]
Since $F(s)$ is continuous so is $G_{s'}(s)$. We may therefore use the integral form of Gr\"onwall's inequality to obtain a bound for $G_{s'}(s)$:
\[
G_{s'}(s)\leq (\Gamma+C|F(s')|)|s-s'|+\int_{s'}^s(\Gamma+C|F(s')|)|u-s'|\cdot C\cdot\exp\left(\int_u^s C\, \dr u'\right)\, \dr u.
\]
A straightforward estimate gives
\begin{equation}\label{eq:gronwallcons}
G_{s'}(s)\leq (\Gamma+C|F(s')|)(|s-s'|+C|s-s'|^2\exp[C|s-s'|]).
\end{equation}
Now, in \eqref{eq:gronwallcons}, substitute $s_0$ for $s'$ and notice that by our assumption $F(s_0)=0$. The observation that $|s-s'|$ is universally bounded by a constant $C$ (namely $C=L$) yields a very crude bound for $|F(s)|$:
\[
|F(s)|=G_{s_0}(s)\leq\Gamma|s-s_0|+C\Gamma|s-s_0|^2\er^{C|s-s_0|}\leq C\Gamma.
\]
However, plugging this crude bound back into \eqref{eq:gronwallcons} and again using $|s-s'|\leq L$ on some (but not all) of the $|s-s'|$ terms gives a bound of $C\Gamma|s-s'|$ for $G_{s'}(s)$. This is \eqref{eq:prop95}.

To get \eqref{eq:cor96}, use \eqref{eq:curvdoubledot} and estimate the resulting difference as with the interior of the integral in the previous paragraph, by adding and subtracting $\gamma(s)\cdot\Rot(\dot{\tilde q}(s))$:
\begin{equation}\label{eq:randomultline}
\begin{split}
|\ddot{\tilde q}(s)-\ddot q(s)|&\leq |\tilde{\gamma}(s)\cdot\Rot(\dot{\tilde q}(s))-\gamma(s)\cdot\Rot(\dot{\tilde q}(s))|+|{\gamma}(s)\cdot\Rot(\dot{\tilde q}(s))-\gamma(s)\cdot\Rot(\dot q(s))|\\ 
&\leq\Gamma + |\gamma(s)|\cdot |\Rot(\dot{\tilde q}(s))-\Rot(\dot q(s))|\leq \Gamma+C|\dot{\tilde q}(s)-\dot q(s)|.
\end{split}
\end{equation}
By our assumption, $\dot{\tilde q}(s)-\dot q(s)$ equals zero for $s=s_0$. So \eqref{eq:prop95} demonstrates that
\[
|\dot{\tilde q}(s)-\dot q(s)|\leq C\Gamma|s-s_0|,
\]
which is crudely bounded by $C\Gamma$. This gives \eqref{eq:cor96}.

Finally, \eqref{eq:prop98} is obtained by differentiating both sides of \eqref{eq:curvdoubledot} and subtracting the non-tilde version from the tilde version, which yields
\[
|\dddot{\tilde q(s)}-\dddot q(s)|\leq 
|\dot{\tilde{\gamma}}(s)-\dot\gamma(s)|\cdot |\Rot(\dot q(s))|+|\gamma(s)|\cdot|\Rot(\ddot{\tilde q}(s))-\Rot(\ddot{q}(s))|.
\]
The first term is bounded by $\Gamma_1$ as the second factor is a unit vector. For the second term, note that $|\gamma(s)|\leq C$ and that the rotation $\ddot{\tilde q}(s)-\ddot q(s)$, by \eqref{eq:cor96}, is bounded in absolute value by $C\Gamma$. So the second term is bounded by $C\Gamma$ overall. Since of course $\Gamma\leq\Gamma_1$, \eqref{eq:prop98} follows. \end{proof}

\subsection{On-diagonal rectangles}
We now consider the on-diagonal rectangles, that is, each rectangle of the form $I_j\times I_j$ for fixed $j$. It is a well-known fact (see e.g. \cite{cost}) that although $K_{\SL}$ and $K_{\tilde\SL}$ both have logarithmic singularities at the diagonal $\{s=s'\}$, the difference does not, and is actually smooth if the boundary curve $I_j$ is smooth. 
This also will follow from our analysis. 
The kernel $K_{\DL}$ is actually smooth as well if the boundary is smooth.

We begin by analysing the differences of single layer kernels and their derivatives.
\begin{proposition}\label{prop:differenceofsinglelayers} There exists $\delta>0$ depending only on $\Omega$ such that if $\Gamma_1\leq\delta$, then for all $s,s'\in I_j$, we have
\begin{equation}\label{eq:diffofnoderiv}
\left|K_{\SL-\tilde\SL}(s,s')\right|\leq C\Gamma_1|s-s'|^2;
\end{equation}
\begin{equation}\label{eq:diffofonederivsprime}
\left|\partial_{s'} K_{\SL-\tilde\SL}\right|\leq C\Gamma_1|s-s'|;
\end{equation}
\begin{equation}\label{eq:diffofonederivs}
\left|\partial_{s} K_{\SL-\tilde\SL}\right|\leq C\Gamma_1|s-s'|;
\end{equation}
\begin{equation}\label{eq:diffoftwoderivs}
\left|\partial_s\partial_{s'} K_{\SL-\tilde\SL}\right|\leq C\Gamma_1.
\end{equation}
\end{proposition}
Observe that since $|s-s'|$ is bounded this immediately implies that all such kernels are bounded by a constant times $\Gamma_1$, and thus induce operators from $L^2(I_j)$ to $L^2(I_j)$ with norms bounded by a constant times $\Gamma_1$, as desired.

\begin{proof} First we fix an $s'$. As discussed, the kernels are invariant under a Euclidean motion of $\tilde\Omega$, so we assume without loss of generality that $\dot{\tilde q}(s')=\dot q(s')$, which also allows us to use Proposition \ref{prop:geomprelim}. 

To begin, observe that by \eqref{eq:calc1}, which holds since $s,s'$ are on the same side:
\[
|q(s)-q(s')|^2=(s-s')^2+2(s-s')\int_{s'}^s\int_{s'}^t\dot q(s')\cdot \ddot q(u)\, \dr u\, \dr t+\left|\int_{s'}^s\int_{s'}^t\ddot q(u)\, \dr u\, \dr t\right|^2.
\]
Write $\ddot q(u)=\ddot q(s')+\int_{s'}^u\dddot q(v)\, \dr v$; then since $\ddot q(s')$ and $\dot q(s')$ are orthogonal we have
\[
|q(s)-q(s')|^2=(s-s')^2+2(s-s')\int_{s'}^s\int_{s'}^t\int_{s'}^u\dot q(s')\cdot \dddot q(v)\, \dr v\, \dr u\, \dr t+\left|\int_{s'}^s\int_{s'}^t\ddot q(u)\, \dr u\, \dr t\right|^2.
\]
The same expression holds with tildes, and we can subtract the two to estimate the difference of distances-squared. The first terms are the same for each and thus cancel. The second terms' difference, since $\dot q(s')=\dot{\tilde q}(s')$, is
\[2(s-s')\int_{s'}^s\int_{s'}^t\int_{s'}^u\dot q(s')\cdot (\dddot q(v)-\dddot{\tilde q}(v))\, \dr v\, \dr u\, \dr t.\]
By \eqref{eq:prop98}, and the fact that $|\dot q(s')|=1$, this expression is bounded by $C\Gamma_1|s-s'|^4$. Finally, we need to estimate the difference of the last terms with and without tildes. By the usual add/subtract trick this difference of last terms is
\[
\begin{split}
&\left(\int_{s'}^s\int_{s'}^t\ddot q(u)\, \dr u\, \dr t\right)\cdot \left(\int_{s'}^s\int_{s'}^t(\ddot q(u)-\ddot{\tilde q}(u))\, \dr u\, \dr t\right) \\+ &\left(\int_{s'}^s\int_{s'}^t(\ddot q(u)-\ddot{\tilde q}(u))\, \dr u\, \dr t\right)\cdot\left( \int_{s'}^s\int_{s'}^t\ddot{\tilde q}(u)\, \dr u\, \dr t\right).
\end{split}
\]
Each piece of this can be estimated. For the first factor, $|\ddot q(u)|=|\gamma(u)|$, which is bounded by $C$ and so the first term is bounded by $C|s-s'|^2$. The second factor is bounded using \eqref{eq:cor96} by $C\Gamma|s-s'|^2$. The third factor is the same. And the fourth factor has 
\[
|\ddot{\tilde q}(u)|\leq |\ddot q(u)|+|\ddot{\tilde q}(u)-\ddot q(u)|\leq C+\Gamma\leq C+1\leq C,
\]
as long as we choose $\delta\leq 1$. So overall the difference of last terms is bounded by $C\Gamma|s-s'|^4$. Thus, overall, we have
\begin{equation}\label{eq:distanceofsquares}
||q(s)-q(s')|^2-|\tilde q(s)-\tilde q(s')|^2|\leq C\Gamma_1|s-s'|^4.
\end{equation}

Now observe that for some $c>0$ depending only on the geometry of $\Omega$, $|q(s)-q(s')|\geq c|s-s'|$, because all interior angles are positive and the boundary does not intersect itself. So $|q(s)-q(s')|\leq |s-s'|\leq C|q(s)-q(s')|$. Therefore, manipulating \eqref{eq:distanceofsquares} leads to
\begin{equation}\label{eq:distanceofsquareswith1}
\left|1-\frac{|\tilde q(s)-\tilde q(s')|^2}{|q(s)-q(s')|^2}\right|\leq C\Gamma_1|s-s'|^2.
\end{equation}
Taking logarithms, as long as $\Gamma_1$ is sufficiently small, 
yields \eqref{eq:diffofnoderiv}.

Before we analyse the difference of \emph{derivatives} of single layer kernels, a brief proposition.

\begin{proposition}\label{prop:auxdiffsingle} For all $s,s'\in I_j$,
\begin{equation}\label{eq:prop911}
|\ddot{\tilde q}(s)\cdot \dot{\tilde q}(s')-\ddot q(s)\cdot\dot q(s')|\leq C\Gamma|s-s'|.
\end{equation}
\end{proposition}
\begin{proof} Fix $s'$. 
As before, we may  assume that $\dot{\tilde q}(s')=\dot q(s')$. Use \eqref{eq:curvdoubledot} and the usual add/subtract trick to bound the left-hand side of \eqref{eq:prop911} by
\[|(\gamma(s)-\tilde\gamma(s))\Rot(\dot{q}(s))\cdot \dot{q}(s')|+|\tilde\gamma(s)(\Rot(\dot q(s))\cdot \dot q(s')-\Rot(\dot{\tilde q}(s))\cdot \dot{\tilde q}(s'))|.\]
For the first term, the first factor is bounded by $\Gamma$. The second factor $\Rot(\dot{q}(s))\cdot \dot{q}(s')$ is zero when $s=s'$ and has $s$-derivative equal to $\Rot(\ddot{q}(s))\cdot \dot{q}(s')$, which has absolute value bounded by $|\gamma(s)|\leq C$; thus the second factor is bounded by $C|s-s'|$. All together the first term is bounded by $C\Gamma|s-s'|$. As for the second term, the first factor $|\tilde\gamma (s)|$ is bounded by $C$ 
(assuming that $\delta\leq 1$). 
For the second factor, $\dot{\tilde q}(s')=\dot q(s')$, so the second factor is
\[
|(\Rot(\dot q(s))-\Rot(\dot{\tilde q}(s)))\cdot \dot{q}(s')|\leq |\dot q(s)-\dot{\tilde q}(s)|.
\]
But $\dot q(s)-\dot{\tilde q}(s)$ is zero when $s=s'$ and has $s$-derivative bounded in absolute value by $|\ddot q(s)-\ddot{\tilde q}(s)|$, which by \eqref{eq:cor96} is bounded by $C\Gamma$. We therefore get a bound of $C\Gamma|s-s'|$ here as well. This completes the proof.
\end{proof}

We use this to prove \eqref{eq:diffofonederivsprime}.
The kernel of $\partial_{s'} K_{\SL-\tilde\SL}$ is
\[-\frac{(q(s)-q(s'))\cdot\dot q(s')}{|q(s)-q(s')|^2}+\frac{(\tilde q(s)-\tilde q(s'))\cdot\dot{\tilde q}(s')}{|\tilde q(s)-\tilde q(s')|^2}.\]
Since $|q(s)-q(s')|^{-1}\leq C|s-s'|^{-1}$ as before, this is bounded in absolute value by
\[C|s-s'|^{-2}|(q(s)-q(s'))\cdot\dot q(s')-\frac{|q(s)-q(s')|^2}{|\tilde q(s)-\tilde q(s')|^2}(\tilde q(s)-\tilde q(s'))\cdot\dot{\tilde q}(s')|.\]
That ratio of squares is very close to $1$, so we add and subtract 1 from it. In addition, taking the dot product of \eqref{eq:calc1} with $\dot q(s')$ yields
\[(q(s)-q(s'))\cdot\dot q(s')=s-s'+\int_{s'}^s\int_{s'}^t\ddot q(u)\cdot \dot q(s')\, \dr u\, \dr t,\]
and the same is true with tildes. When we plug all this in, the main $s-s'$ terms cancel, and we get an upper bound for $\partial_{s'}K_{\SL-\tilde\SL}$ of
\begin{equation}\label{eq:firstupperonederiv}
\begin{split}
&C|s-s'|^{-2}\\
\cdot&\left |\int_{s'}^s\int_{s'}^t(\ddot q(u)\cdot \dot q(s') - \ddot{\tilde q}(u)\cdot \dot{\tilde q}(s'))\, \dr u\, \dr t 
+\left(1-\frac{|q(s)-q(s')|^2}{|\tilde q(s)-\tilde q(s')|^2}\right )(\tilde q(s)-\tilde q(s'))\cdot\dot{\tilde q}(s')\right |.
\end{split}
\end{equation}
The first of these two terms (including the pre-factor), by Proposition \ref{prop:auxdiffsingle}, is bounded  by $C|s-s'|^{-2}\cdot C\Gamma|s-s'|^{3}=C\Gamma|s-s'|$. As for the second, the factor of $1$ minus the fraction can be estimated with \eqref{eq:distanceofsquareswith1} and is bounded by $C\Gamma_1|s-s'|^2$. The other factors are bounded by $|s-s'|$ and $1$ respectively. So including the pre-factor we get a boound of $C\Gamma_1|s-s'|$ here as well. This proves \eqref{eq:diffofonederivsprime}.

Since the single layer kernels are symmetric, we also get \eqref{eq:diffofonederivs}.

Now we tackle the second derivatives of the single layer kernels.
The kernel $\partial_{s}\partial_{s'}K_{\SL-\tilde\SL}$ is given by \eqref{eq:twoderivssingle} minus the analogous expression with tildes. We consider the differences of the first and second terms of \eqref{eq:twoderivssingle} respectively.

The difference of the first terms can be handled nearly identically to the proof of \eqref{eq:diffofonederivsprime}. Following the first few steps of that proof, it has absolute value bounded by
\[
C|s-s|^{-2}\left|\dot q(s)\cdot \dot q(s')-\frac{|q(s)-q(s')|^2}{|\tilde q(s)-\tilde q(s')|^2}\dot{\tilde q}(s)\cdot\dot{\tilde q}(s')\right|.
\]
From \eqref{eq:calc2} we get
\[
\dot q(s)\cdot \dot q(s')=1+\int_{s'}^s\ddot q(u)\cdot q(s')\, \dr u.
\]
And the same trick of adding and subtracting $1$ from the ratio of squares gives an upper bound of
\begin{equation}
C|s-s'|^{-2}\left|\int_{s'}^s(\ddot q(u)\cdot \dot q(s')-\ddot{\tilde q}(u)\cdot\dot{\tilde q}(s'))\, \dr u + \left (1-\frac{|q(s)-q(s')|^2}{|\tilde q(s)-\tilde q(s')|^2}\right )\dot{\tilde q}(s)\cdot \dot{\tilde q}(s')\right|.
\end{equation}
The same estimates as before, namely Proposition \ref{prop:auxdiffsingle} and \eqref{eq:distanceofsquareswith1}, show that this is bounded, overall, by $C\Gamma_1$ as desired.

For the difference of the second terms of \eqref{eq:twoderivssingle}, observe that the second term is precisely $-2\partial_sK_{\SL}\partial_{s'}K_{\SL}$. So the difference of second terms is this minus the version with tildes, and we can use the usual add/subtract trick to bound this difference by
\begin{equation}\label{eq:diagonaltrick}
\left|\partial_sK_{\SL}\partial_{s'}K_{\SL-\tilde\SL}\right|+\left|\partial_{s}K_{\SL-\tilde\SL}\partial_{s'}K_{\tilde\SL}\right|.
\end{equation}
But by direct calculation, regardless of the parametrisations,
\[
\left|\partial_sK_{\SL}\right|\leq |s-s'|^{-1};\qquad \left|\partial_{s'}K_{\tilde\SL}\right|\leq |s-s'|^{-1}.
\]
Putting this together with \eqref{eq:diffofonederivsprime} and \eqref{eq:diffofonederivs} gives an overall bound of $C\Gamma_1$, and we have proven \eqref{eq:diffoftwoderivs}. This completes the proof of Proposition \ref{prop:differenceofsinglelayers}.
\end{proof}

We have dealt with the single layer potentials and their derivatives. Now we analyse the double layer potentials.

\begin{proposition}\label{prop:differenceofdoublelayers} There exists $\delta>0$ depending only on $\Omega$ such that if $\Gamma_1\leq\delta$, then for all $s,s'\in I_j$, we have
\begin{equation}\label{eq:diffofnoderiv2}
\left|K_{\DL-\tilde\DL}(s,s')\right|\leq C\Gamma_1;
\end{equation}
\begin{equation}\label{eq:diffofonederivs2}
\left|\partial_{s} K_{\DL-\tilde\DL}(s,s')\right|\leq C\Gamma_1.
\end{equation}
\end{proposition}
Once this is proven, we have completed the proof of part 1 of Lemma \ref{lem:heresthecontent} in the diagonal, $j=k$ case, as all six kernels will be bounded by $C\Gamma_1$ and thus will induce operators from $L^2(I_j)\to L^2(I_j)$ with norm less than or equal to $C\Gamma_1$.

\begin{proof}
As usual fix an $s'$ and assume $\dot{\tilde q}(s')=\dot q(s')$. Using \eqref{eq:noderivsdouble} and \eqref{eq:calc1}, along with the fact that $\dot q(s')\cdot\mathbf{n}(s')=0$, we get
\[
K_{\DL}(s,s')=-|q(s)-q(s')|^{-2}\left(\int_{s'}^s\int_{s'}^t\ddot q(u)\cdot \mathbf{n}(s')\, \dr u\, \dr t\right).
\]
Let us rewrite this using $\ddot q(u)=\ddot q(s')+\int_{s'}^u \dddot q(v)\, \dr v$ and the fact that $\ddot q(s')=\gamma(s')\mathbf{n}(s')$:
\begin{equation}\label{eq:betternoderivsdouble}
K_{\DL}(s,s')=-|q(s)-q(s')|^{-2}\left(\frac 12|s-s'|^2\gamma(s')+\int_{s'}^s\int_{s'}^t\int_{s'}^u\dddot q(v)\cdot\mathbf{n}(s')\, \dr v\, \dr u\, \dr t\right).\end{equation}
We can use this, and the adding/subtracting 1 trick, to write an expression for the difference:
\begin{equation}\label{eq:longequationdouble}
\begin{split}
K_{\DL-\tilde\DL}(s,s')=&-\frac 12\frac{|s-s'|^2}{|q(s)-q(s')|^2}\left((\gamma(s')-\tilde\gamma(s'))+\tilde\gamma(s')(1-\frac{|q(s)-q(s')|^2}{|\tilde q(s)-\tilde q(s')|^2})\right)
\\ &-\frac{1}{|q(s)-q(s')|^2}\left(\int_{s'}^s\int_{s'}^t\int_{s'}^u(\dddot q(v)\cdot\mathbf{n}(s')-\dddot{\tilde q}(v)\cdot\tilde{\mathbf{n}}(s'))\, \dr v\, \dr u\, \dr t \right)
\\
&-\frac{1}{|q(s)-q(s')|^2}\left(1-\frac{|q(s)-q(s')|^2}{|\tilde q(s)-\tilde q(s')|^2}\right)\int_{s'}^s\int_{s'}^t\int_{s'}^u\dddot{\tilde q}(v)\cdot\tilde{\mathbf{n}}(s')\, \dr v\, \dr u\, \dr t.
\end{split}
\end{equation}
It remains to bound this and its $s$-derivative, in absolute value, by $C\Gamma_1$. We do this for each of the three terms separately.

Consider the first line of \eqref{eq:longequationdouble}. The pre-factor is a $C^1$ function of $s$ and $s'$ on the rectangle $I_j$ and is independent of $\tilde\Omega$, so its $C^1$ norm is uniformly bounded by a constant $C$. The second factor is bounded, using \eqref{eq:distanceofsquareswith1}, by $\Gamma+(C+\Gamma)(C\Gamma_1|s-s'|)$, which is less than or equal to $C\Gamma_1$. As for the derivative of the second factor, by a direct calculation with logarithmic differentiation, we see that it is
\begin{equation}\label{eq:logthing}
\tilde\gamma(s')\left(-2\partial_sK_{\SL-\tilde\SL}\right)\frac{|q(s)-q(s')|^2}{|\tilde q(s)-\tilde q(s')|^2}.
\end{equation}
By \eqref{eq:distanceofsquareswith1} and \eqref{eq:diffofonederivs}, this is bounded by 
\[
(C+\Gamma)(C\Gamma_1|s-s'|)(1+C\Gamma_1|s-s'|^2)\leq C\Gamma_1.
\]
This is enough to bound the first line of \eqref{eq:longequationdouble} and its $s$-derivative by $C\Gamma_1$.

Now consider the second line. The pre-factor is bounded by $C|s-s'|^{-2}$, and the integrand, since $\mathbf{n}(s')=\tilde{\mathbf{n}}(s')$, is bounded in absolute value by $\Gamma_1$. Thus the second line itself is bounded by $C\Gamma_1|s-s'|$. As for its derivative, there are two terms. If the $s$-derivative hits the pre-factor we get $2|q(s)-q(s')|^{-4}(q(s)-q(s'))\cdot\dot q(s)$, which is bounded in absolute value by $C|s-s'|^{-3}$, yielding an overall bound of $C\Gamma_1$. If it hits the integral, it removes one of the integrals, so the bound on the integral becomes $C\Gamma_1|s-s'|^2$ instead of $C\Gamma_1|s-s'|^3$. Multiplied by $C|s-s'|^{-2}$ this still yields $C\Gamma_1$.

Finally, examine the third line. The first two factors are bounded by $C|s-s'|^{-2}$ and $C\Gamma_1|s-s'|^2$ respectively, as a consequence of \eqref{eq:distanceofsquareswith1}. For the integral we evaluate the inner integral and get a double integral of $\ddot{\tilde q}(u)-\ddot{\tilde q}(s')$ dotted with a unit vector. But $|\ddot{\tilde q}(u)|=|\tilde\gamma(u)|\leq C+\Gamma$, so that integral is less than $2(C+\Gamma)\frac 12|s-s'|^2$. Putting all three together, the third line is bounded by $C\Gamma_1|s-s'|^2$, which is certainly bounded by $C\Gamma_1$. As for the $s$-derivative, it can hit three different factors. If it hits the first factor it produces an extra factor of $|q(s)-q(s')|^{-2}(q(s)-q(s'))\cdot\dot q(s)$, which is bounded by $C|s-s'|^{-1}$. If it hits the second factor, the second factor turns into \eqref{eq:logthing} and thus the bound of $C\Gamma_1|s-s'|^2$ becomes $C\Gamma_1|s-s'|$ instead. And if it hits the integral, one of the integrals disappears and the bound again loses a factor of $|s-s'|$. But the $s$-derivative is still bounded by $C\Gamma_1|s-s'|$, which is more than enough. This completes the proof of Proposition \ref{prop:differenceofdoublelayers}, and with it the $j=k$ case of part 1 of Lemma \ref{lem:heresthecontent}.
\end{proof}

\subsection{Off-diagonal rectangles}

Now we assume $k\neq j$. We claim it is enough to consider the case where $k=j-1$. Indeed, the $k=j+1$ case is identical. For the other values of $k$, the geometric situation is the same as if we take the $k=j-1$ case and restrict the input to lie in a sub-interval of $I_{j-1}$ away from the intersection $V_j=I_j\cap I_{j-1}$, so the analysis here will cover those values of $k$ as well. In this $k=j-1$ case, the diagonal singularity is not an issue, so all of our kernels are smooth in the interior of $I_j\times I_{k}$. But there \emph{is} a singularity at the point $V_j\times V_j$, and as indicated in \cite{cost}, it has a more substantial effect than in the on-diagonal rectangles. Indeed not all of our kernels will be bounded near $V_j\times V_j$. However, we will be able to bound them, in absolute value, by a kernel which induces a bounded operator from $L^2(I_k)$ to $L^2(I_j)$.

So let $s$ and $s'$ be the usual arc length coordinates, and assume without loss of generality that $s=0$ at the vertex $V_j$. Thus we have $s\geq 0$ on $I_j$, and $s'\leq 0$ on $I_k=I_{j-1}$. Assume without loss of generality that $q(0)=\tilde q(0)=0$ and that for both $\partial\Omega$ and $\partial\tilde\Omega$, $I_{j-1}$ is tangent to the $x$-axis at $V_j$, with $I_j$ making an angle $\alpha$ with the $x$-axis for both. Now we define two vector-valued functions $\beta_{-}(s')$ and $\beta_{+}(s)$ by the equations
\[
q(s')=s'\begin{pmatrix}-1\\0\end{pmatrix}+\beta_{-}(s');\quad q(s)=s\begin{pmatrix}\cos\alpha\\\sin\alpha\end{pmatrix}+\beta_{+}(s).
\]
Define analogues with tildes the same way.
\begin{proposition} The following are true:
\begin{enumerate}
\item The function $\beta_-(s')$ is as smooth as $q(s')$ (at least $C^3$), is $O\left((s')^2\right)$, and its Taylor coefficient of $(s')^2$ at $s'=0$ is perpendicular to $I_{j-1}$. Similar statements hold for $\beta_+$, and the analogues with tildes also hold.
\item We have $\ddot\beta_{+}(s)=\ddot q(s)$, $\ddot\beta_{-}(s')=\ddot q(s')$, and the same are true for tildes and third derivatives.
\item We have the estimates
\begin{equation}\label{eq:betaestimates}
\begin{alignedat}{2}
\left\|\ddot\beta_{\pm}-\ddot{\tilde \beta}_{\pm}\right\|_{L^{\infty}}&\leq C\Gamma;&\qquad\left\|\dddot\beta_{\pm}-\dddot{\tilde \beta}_{\pm}\right\|_{L^{\infty}}&\leq C\Gamma_1;\\
\left|\dot\beta_{+}(s)-\dot{\tilde\beta}_{+}(s)\right|&\leq C\Gamma s;&\qquad\left|\dot\beta_{-}(s')-\dot{\tilde\beta}_{-}(s')\right|&\leq C\Gamma |s'|;\\ 
\left|\beta_{+}(s)-\tilde\beta_{+}(s)\right|&\leq\frac 12 C\Gamma s^2;&\qquad\left|\beta_{-}(s')-\tilde\beta_{-}(s')\right|&\leq\frac 12 C\Gamma (s')^2.
\end{alignedat}
\end{equation}
\end{enumerate}
\end{proposition}
\begin{proof} The first statement is obvious except for the orthogonality, but that follows from the fact that since $q(s')$ is an arc length parametrization, the vectors $\ddot q(0)$ and $\dot q(0)$ are orthogonal. The second statement is clear. The first two estimates in the third statement follow from \eqref{eq:cor96} and \eqref{eq:prop98}, and the others follow from integration and the fact that $\beta_{\pm}(0)=\tilde\beta_{\pm}(0)$ and $\dot\beta_{\pm}(0)=\dot{\tilde\beta}_{\pm}(0)$.
\end{proof}

Now define, as in \cite{cost},
\[
r(s,s'):=\left|s\begin{pmatrix}\cos\alpha\\\sin\alpha\end{pmatrix} - s'\begin{pmatrix}-1\\0\end{pmatrix}\right|.
\]
Its utility is the following
\begin{proposition} The kernel $r^{-1}(s,s')$ defines an operator which is bounded from $L^2(I_{j-1})$ to $L^2(I_j)$.
\end{proposition}
\begin{proof} This is a consequence of \cite{cost} and is essentially proved there. Specifically, apply the second statement in \cite[Theorem 4.2]{cost} with $s=0$ and $j=-1$, noting that $\tilde H^0=H^0=L^2$. The kernel $r^{-1}$ is an example of one of Costabel's kernels $K_j$ with homogeneity $j=-1$. If $\chi$ is a cutoff function localizing near $V_j$, then by \cite[Theorem 4.2]{cost}, $\chi(s) r^{-1}\chi(s')$ is bounded from $L^2$ to $L^2$. On the other hand, $1-\chi(s)r^{-1}\chi(s')$ is bounded on $I_{j}\times I_{j-1}$ and is also bounded from $L^2$ to $L^2$. Adding them completes the proof.
\end{proof}

The point of this is that we have the following proposition:
\begin{proposition}\label{prop:differenceoffdiagonal} There exists $\delta>0$ depending only on $\Omega$ such that if $\Gamma_1\leq\delta$, then for all $s\in I_j$ and $s'\in I_{j-1}$, we have
\begin{align}\label{eq:diffofnoderivod}
|K_{\SL-\tilde\SL}(s,s')|&\leq C\Gamma r;
\\
\label{eq:diffofonederivsprimeod}
|\partial_{s'} K_{\SL-\tilde\SL}|&\leq C\Gamma;
\\
\label{eq:diffofonederivsod}
|\partial_{s} K_{\SL-\tilde\SL}|&\leq C\Gamma;
\\
\label{eq:diffoftwoderivsod}
|\partial_s\partial_{s'} K_{\SL-\tilde\SL}|&\leq C\Gamma_1r^{-1};
\\
\label{eq:diffofnoderiv2od}
|K_{\DL-\tilde\DL}(s,s')|&\leq C\Gamma_1;
\\
\label{eq:diffofonederivs2od}
|\partial_{s} K_{\DL-\tilde\DL}(s,s')|&\leq C\Gamma_1r^{-1}.
\end{align}
\end{proposition}
Once we have proven this proposition, all of our kernels on $I_j\times I_{j-1}$ will be bounded in absolute value by $\Gamma_1$ times a kernel which defines a bounded operator from $L^2\to L^2$. This proves part 1 of Lemma \ref{lem:heresthecontent} and thereby completes the proof of the results in this section. It remains only to prove Proposition \ref{prop:differenceoffdiagonal}. 
\begin{proof}
First we note that $r$ is a good approximation to $|q(s)-q(s')|$ in the sense that
\begin{equation}\label{eq:qandrcomp}
C^{-1}r\leq |q(s)-q(s')|\leq Cr.
\end{equation}
In fact the ratio of $|q(s)-q(s')|$ and $r$ actually approaches 1 as $s,s'\to 0$, since the deviations of $q(s)$ and $q(s')$ from straight lines are quadratic, and if the deviations $\beta_{\pm}$ were identically zero then we would have $r(s,s')=|q(s)-q(s')|$. 

Now we claim:
\[
\begin{split}
\left|\, |q(s)-q(s')|-|\tilde q(s)-\tilde q(s')|\, \right|&\leq |q(s)-\tilde q(s)|+|q(s')-\tilde q(s')|\\
&=|\beta_{+}(s)-\tilde\beta_{+}(s)|+|\beta_{-}(s')-\tilde\beta_{-}(s')|\\\
&\leq \frac 12 C\Gamma(s^2+(s')^2)\leq C\Gamma r^2\leq C\Gamma r |q(s)-q(s')|.
\end{split}
\]
Indeed this follows from the definition of $\beta$, estimates \eqref{eq:betaestimates}, the fact that the ratio $r^2(s,s')/(s^2+(s')^2)$ is bounded by $C$, and \eqref{eq:qandrcomp}. As a consequence,
\begin{equation}\label{eq:1minusonemore}
\left|1-\frac{|\tilde q(s)-\tilde q(s')|}{|q(s)-q(s')|}\right|\leq C\Gamma r,
\end{equation}
and so as long as $\Gamma$ is sufficiently small,
\[
|K_{\SL-\tilde\SL}(s,s')|=\left|\log\frac{|\tilde q(s)-\tilde q(s')|}{|q(s)-q(s')|}\right|\leq C\Gamma r,
\]
which proves \eqref{eq:diffofnoderivod}.

To go after derivatives of the single layer potential, observe that, after doing the usual add and subtract 1 trick, the kernel of $\partial_sK_{\SL-\tilde\SL}$ is
\begin{equation}\label{eq:offdiagkernelonederiv}
\begin{split}
&|q(s)-q(s')|^{-2}\bigg(((q(s)-q(s'))\cdot\dot q(s)-(\tilde q(s)-\tilde q(s'))\cdot\dot{\tilde q}(s))
\\&+\left(1-\frac{|q(s)-q(s')|^2}{|\tilde q(s)-\tilde q(s')|^2}\right)(\tilde q(s)-\tilde q(s'))\cdot\dot{\tilde q}(s)\bigg).
\end{split}
\end{equation}
The pre-factor is bounded by $Cr^{-2}$, so we need to show the sum of two terms is bounded by $C\Gamma r^2$. The first of these terms, using an add-subtract trick, is bounded by
\[
C\left(|(q(s)-q(s'))\cdot(\dot q(s)-\dot{\tilde q}(s))|+|((q(s)-q(s'))-(\tilde q(s)-\tilde q(s')))\cdot\dot{\tilde q}(s)|\right).
\]
By \eqref{eq:qandrcomp} and rearrangement, this is bounded by
\[
Cr|\dot q(s)-\dot{\tilde q}(s)|+|(q(s)-\tilde q(s))+(q(s')-\tilde q(s'))|.
\]
Switching from $q$ to $\beta_{\pm}$, then applying \eqref{eq:betaestimates}, yields an upper bound of
\[Cr\Gamma s+\frac 12\Gamma(s^2+(s')^2),\]
which is at most $C\Gamma r^2$, as $s\leq |s-s'|\leq Cr$.  
Now the second of the two terms in \eqref{eq:offdiagkernelonederiv} is bounded by
\[
C\left|1-\frac{|q(s)-q(s')|^2}{|\tilde q(s)-\tilde q(s')|^2}\right|\cdot|\tilde q(s)-\tilde q(s')|.
\]
By \eqref{eq:1minusonemore}, for sufficiently small $\Gamma$, the first factor is bounded by $C\Gamma r$. By \eqref{eq:1minusonemore} and \eqref{eq:qandrcomp}, the second term is bounded by $Cr$. Putting these together gives what we want and proves \eqref{eq:diffofonederivsod}. Since our geometric setup is symmetric with respect to interchange of $s$ and $s'$, we also get \eqref{eq:diffofonederivsprimeod}.

For the second derivative of the single layer potential, as with the diagonal case, we consider the differences between the tilde and non-tilde versions of the first and second terms of \eqref{eq:twoderivssingle} separately. The first two terms have difference which is bounded by
\[
Cr^{-2}\left(\left|\dot q(s)\cdot \dot q(s')-\dot{\tilde q}(s)\cdot\dot{\tilde q}(s')\right|+\left|\left(1-\frac{|q(s)-q(s')|^2}{|\tilde q(s)-\tilde q(s')|^2}\right)\dot{\tilde q}(s)\cdot\dot{\tilde q}(s')\right|\right).
\]
The right-most portion of this is bounded by $C\Gamma r$, so with the pre-factor, that gives $C\Gamma r^{-1}$ as desired. The left-most term is bounded, using an add-subtract trick, by
\[
|\dot q(s)\cdot (\dot q(s')-\dot{\tilde q}(s'))|+|(\dot q(s)-\dot{\tilde q}(s))\cdot \dot{\tilde q}(s')|=|\dot \beta_{-}(s')-\dot{\tilde \beta}_{-}(s')|+|\dot \beta_{+}(s)-\dot{\tilde \beta}_{+}(s)|,
\]
and by \eqref{eq:betaestimates} this is bounded by $C\Gamma(s-s')\leq C\Gamma r$. So the first terms of \eqref{eq:twoderivssingle} differ by $C\Gamma r^{-1}$. As for the second terms of \eqref{eq:twoderivssingle}, the same trick as on the diagonal yields a bound of \eqref{eq:diagonaltrick}, and by \eqref{eq:diffofonederivsod} and \eqref{eq:diffofonederivsprimeod} this is bounded by
\[C\Gamma(|\partial_s K_{\SL}(s,s')|+|\partial_{s'}K_{\tilde\SL}(s,s')|).\]
But each of the terms in brackets, by \eqref{eq:onederivfirstsingle} and \eqref{eq:onederivsecondsingle}, is bounded by $C|q(s)-q(s')|^{-1}$, which by \eqref{eq:qandrcomp} is bounded by $Cr^{-1}$, yielding an overall bound of $C\Gamma r^{-1}$. This proves \eqref{eq:diffoftwoderivsod}.

For the double layer potential, the kernel $K_{\DL-\tilde\DL}(s,s')$ is \eqref{eq:offdiagkernelonederiv} but with all $\dot q(s)$ replaced by $\mathbf{n}(s)$, same for the tildes. Most of the analysis is identical, except that now we need to replace the bound $|\dot q(s)-\dot{\tilde q}(s)|\leq C\sigma s$ with the same bound for $|\mathbf{n}(s)-\tilde{\mathbf{n}}(s)|$. But this is just a 90-degree rotation, which leaves the magnitude unchanged, so the same bound applies, proving \eqref{eq:diffofnoderiv2od}.

Finally we need to analyse $\partial_sK_{\DL}(s,s')$ and do so by dealing with the first and second terms of \eqref{eq:onederivdouble} separately. For the first terms, the proof is precisely analogous to the proof of \eqref{eq:diffoftwoderivsod}, with the same replacement of $\dot q(s)$ by $\mathbf{n}(s)$, and the rotation trick we just used in the previous paragraph. For the second terms, observe that the second term of \eqref{eq:onederivdouble} is precisely $-2(\partial_sK_{\SL})K_{\DL}$. By the same trick as usual, the difference of terms is bounded by
\[
|\partial_s K_{\SL}(s,s')K_{\DL-\tilde\DL}(s,s')|+|\partial_s K_{\SL-\tilde\SL}(s,s')K_{\tilde\DL}(s,s')|.
\]
Using \eqref{eq:diffofonederivsod} and \eqref{eq:diffofnoderiv2od}, this is bounded by
\[
C\Gamma(|\partial_s K_{\SL}(s,s')|+|K_{\tilde\DL}(s,s')|).
\]
By direct calculation, the first term is bounded by $|q(s)-q(s')|^{-1}$, and the second by the same with tildes. By \eqref{eq:qandrcomp} and \eqref{eq:1minusonemore} both of these are bounded by $Cr^{-1}$, yielding \eqref{eq:diffofonederivs2od}. This completes the proof of Proposition \ref{prop:differenceoffdiagonal}, and with it all the results in this section.
\end{proof}
\clearpage\section{Further examples and numerics}\label{sec:numerics}

\subsection{General setup and benchmarking}\label{sec: numsetup}  The examples in this Section extend those in subsection \ref{sec:examples}. In most cases, the Steklov eigenvalues are computed using the Finite Element package  \emph{FreeFEM} (earlier versions known as \emph{FreeFem++}), see \cite{Freefem} and short notes \cite{Freefemweb}.  
In most cases, we choose a uniform mesh with $300$ mesh points per unit length on the boundary.
Roots of trigonometric polynomials are found using \emph{Mathematica}  operating with double precision. 

In order to benchmark the performance of the finite element solver, we compare the numerically computed Steklov eigenvalues $\lambda^\mathrm{num}$ of the unit square $\mathcal{P}_4\left(\frac{\pi}{2},1\right)$ with the exact eigenvalues \cite[section 3.1]{GP2017} 
\[
\begin{split}
\lambda\in\left\{0,2\right\}&\cup\left\{2t\tanh t\mid \tan t=-\tanh t\text{ or }\tan t=\coth t, t>0\right\}\\
&\cup\left\{2t\coth t\mid \tan t=\tanh t\text{ or }\tan t=-\coth t, t>0\right\}.
\end{split},
\] 
We also compare the numerically computed Steklov eigenvalues and the exact eigenvalues $\sigma_{2m}=\sigma_{2m+1}=m$ for the unit disk $\mathbb{D}_1$. 
Figure \ref{fig:Fig17} shows the relative numerical error 
\[
\varepsilon^\mathrm{num}_m:=\left|\frac{\lambda_m^\mathrm{num}}{\lambda_m}-1\right|
\]
for the square and the disk, and also the relative asymptotic error
\[
\varepsilon^\mathrm{asy}_m:=\left|\frac{\sigma_m}{\lambda_m}-1\right|
\]
for the disk.

\begin{figure}[htb]
\begin{center}
\includegraphics{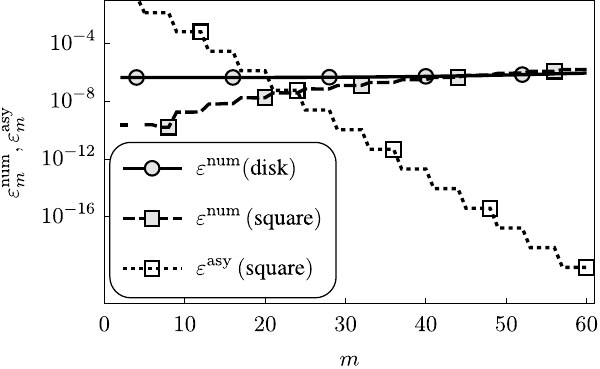}
\end{center}
\caption{Relative FEM errors for the disk and the square, and asymptotics error for the square\label{fig:Fig17}}
\end{figure} 

One can see that with the chosen mesh size, the relative error $\varepsilon^\mathrm{num}_m$ does not exceed approximately $10^{-6}$ for the eigenvalues $\lambda_m$, $m=2,\dots,100$.  Although it is well known that adaptive FEM are better suited for Steklov eigenvalue problems, see e.g.\ \cite{Garau}, they are processor-time costly and harder to implement. As we conduct the numerical experiments purely for illustrative purposes in order to demonstrate the practical effectiveness of the asymptotics,  the use of uniform meshes already gives very good results as shown above.  For an alternative method of calculating Steklov or mixed Steklov-Dirichlet-Neumann eigenvalues, see, e.g., \cite{ABIN19,AIN19}.

As in the examples which follow the exact eigenvalues are not known, we redefine from now on the relative asymptotic error as
\[
\varepsilon^\mathrm{asy}_m:=\left|\frac{\sigma_m}{\lambda_m^\mathrm{num}}-1\right|
\]
and use these quantities for all illustrations.

\subsection{Example \ref{ex:specialorexceptional} revisited}

Before proceeding to concrete examples, we fist demonstrate formulae \eqref{eq:allspecial} when all angles are special. Recalling Definition \ref{def:quasi}, formula \eqref{eq:Tdef1} and Remark \ref{rem:Aproperties}(c), we get in this case 
\[
\mathtt{T}(\balpha,\bell,\sigma)=\prod_{j=1}^n \Odd(\alpha_j)\begin{pmatrix}\er^{\ir |\partial\mathcal{P}| \sigma}&0\\0&\er^{-\ir |\partial\mathcal{P}| \sigma}\end{pmatrix},
\]
and so 
\[
\Tr\mathtt{T}=2\cos\left( |\partial\mathcal{P}| \sigma\right) \prod_{j=1}^n \Odd(\alpha_j)=2
\]
with $\sigma\ge 0$ if and only if \eqref{eq:allspecial}  holds. The statement on multiplicities, as well as the statement in case (b) of Example \ref{ex:specialorexceptional}  when some exceptional angles are present, are easily checked.

Switching to particular examples, we consider, in addition to right-angled triangles $T_1$ and $T_2$, a family of curvilinear triangles $\mathcal{T}(\alpha)$ constructed according to Figure \ref{fig:Fig18}. For each $\alpha\in\left(0,\frac{\pi}{3}\right)$, the vertices of $\mathcal{T}(\alpha)$  coincide with the vertices of an equilateral triangle of side one, two sides are straight, and the third (curved) side is given by the  equation shown in  Figure \ref{fig:Fig18}. Thus $\mathcal{T}(\alpha)=\mathcal{P}\left(\left(\frac{\pi}{3},\frac{\pi}{3},\alpha\right),\left(1,1,\ell_\alpha\right)\right)$, where the length $\ell_\alpha$ of the curved side has to be found numerically. We consider further two particular cases $\mathcal{T}_3=\mathcal{T}\left(\frac{\pi}{5}\right)$ (two angles are odd special and one is even special) and $\mathcal{T}_4=\mathcal{T}\left(\frac{\pi}{7}\right)$ (all three angles are odd special), for which $\ell_{\frac{\pi}{5}}\approx 1.0130$ and $\ell_{\frac{\pi}{7}}\approx 1.0296$, respectively.

\begin{figure}[htb]
\begin{center}
\includegraphics{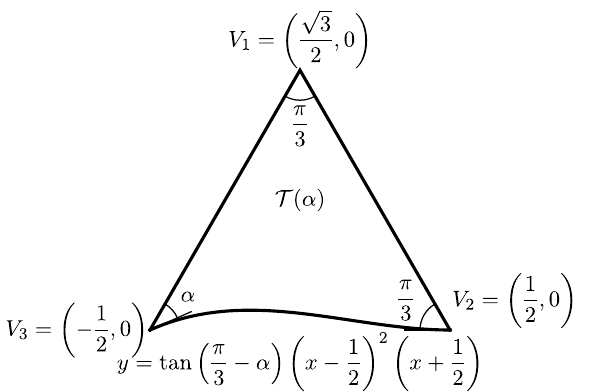}
\end{center}
\caption{Family $\mathcal{T}(\alpha)$ of curvilinear triangles\label{fig:Fig18}}
\end{figure}

The asymptotic accuracy for $T_1$, $T_2$, $\mathcal{T}_3$, and $\mathcal{T}_4$ is plotted in Figure \ref{fig:Fig19}.

\begin{figure}[htb]
\begin{center}
\includegraphics{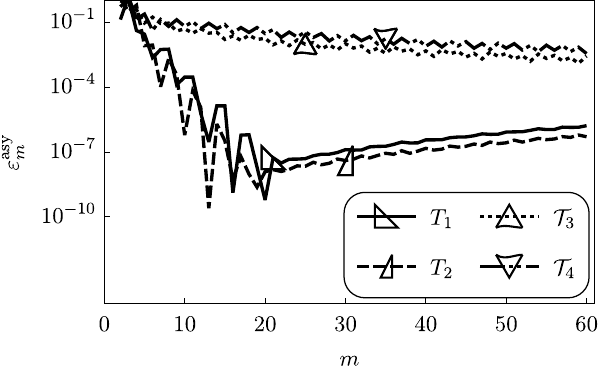}
\end{center}
\caption{Asymptotic accuracy for $T_1$, $T_2$, $\mathcal{T}_3$, and $\mathcal{T}_4$\label{fig:Fig19}}
\end{figure} 

\subsection{Example \ref{ex:quasi-regular} revisited} We start with the proof of Proposition \ref{prop:regnonexc}. For a quasi-regular $n$-gon $\mathcal{P}_n(\alpha,\ell)$ with a non-exceptional angle $\alpha$, we have, by \eqref{eq:Tdef1},
\[
\mathtt{T}(\balpha,\bell,\sigma)=\mathtt{C}(\alpha,\ell,\sigma)^n,
\]
where $\mathtt{C}$ is given by \eqref{eq:Sdef1}. Thus $\mathtt{T}$ will have an eigenvalue one if and only if $\mathtt{C}(\alpha,\ell,\sigma)$ has an eigenvalue $c$ equal to one of the complex $n$-roots of one, $\er^{2\ir q/n}$, $q\in\mathbb{Z}$. As $\det \mathtt{C}=1$, the other eigenvalue of $\mathtt{C}$ is $\frac{1}{c}$, therefore to cover all the distinct cases we need to take $q=0,\dots,\left[\frac{n}{2}\right]$; moreover, the condition can be then equivalently re-written as 
\begin{equation}\label{eq:eigreg}
\Tr \mathtt{C}(\alpha,\ell,\sigma)=2\csc\left(\frac{\pi^2}{2\alpha}\right)\cos(\ell\sigma)=c+1/c=2\cos\left(\frac{2q}{n}\right).
\end{equation}
Solving \eqref{eq:eigreg} for non-negative $\sigma$ gives the expressions for quasi-eigenvalues in the statement of Proposition \ref{prop:regnonexc}.

To prove the statement on multiplicities, we remark that if $c\ne\pm 1$, the matrix  $\mathtt{C}(\alpha,\ell,\sigma)$ has two linearly independent eigenvectors, and so does $\mathtt{T}(\balpha,\bell,\sigma)$. The rest of the statement follows from the careful analysis of the dimension of the eigenspace of $\mathtt{C}(\alpha,\ell,\sigma)$ coresponding to eigenvalues $c=\pm 1$ when $\sigma$ is a root of \eqref{eq:eigreg}.

As an illustration, we present numerical data for the equilateral triangle $P_3=\mathcal{P}_3\left(\frac{\pi}{3},1\right)$, the regular pentagon $P_5=\mathcal{P}_5\left(\frac{3\pi}{5},1\right)$,  the regular hexagon $P_6=\mathcal{P}_6\left(\frac{2\pi}{3},1\right)$, and a Reuleaux triangle  $\mathcal{R}=\mathcal{P}_3\left(\frac{2\pi}{3},\frac{\pi}{3}\right)$ (whose boundary is the union of three arcs constructed on the sides of an equilateral triangle of side one as chords, with centres at the opposite vertices), see Figure \ref{fig:Fig20}.

\begin{figure}[htb]
\begin{center}
\includegraphics{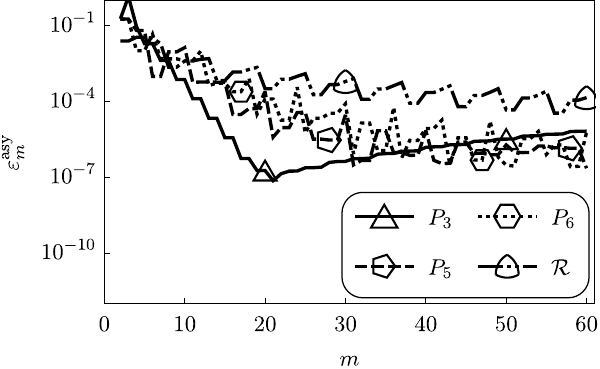}
\end{center}
\caption{Asymptotic accuracy for $P_3$, $P_5$, $P_6$, and $\mathcal{R}$\label{fig:Fig20}}
\end{figure} 

Additionally, we consider a family of (non-symmetric) one-angled droplets $\mathcal{D}_\alpha=\mathcal{P}_1(\alpha, \ell_\alpha)$ shown in Figure \ref{fig:Fig21}; the perimeter $\ell_\alpha$ needs to be calculated numerically. The quasi-eigenvalues $\sigma$ are listed in Example \ref{ex:quasi-regular}(a$_1$).

\begin{figure}[htb]
\begin{center}
\includegraphics{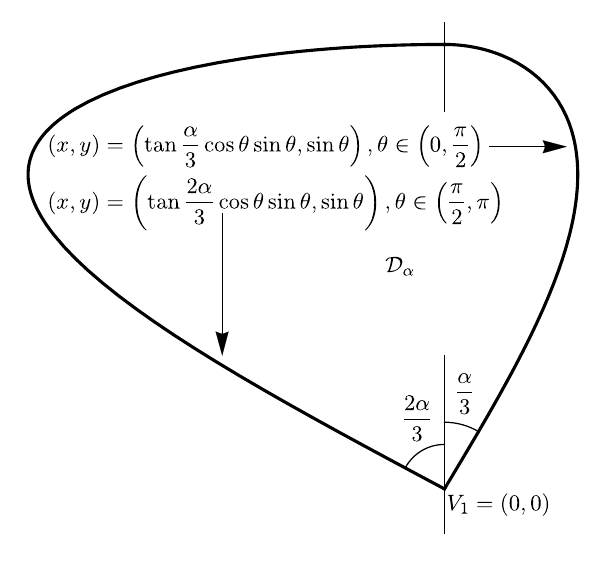}
\end{center}
\caption{A family of one-gons (droplets) $\mathcal{D}_\alpha$\label{fig:Fig21}}
\end{figure} 

Asymptotic accuracy for a selection of droplets is shown in Figure \ref{fig:Fig22}.

\begin{figure}[htb]
\begin{center}
\includegraphics{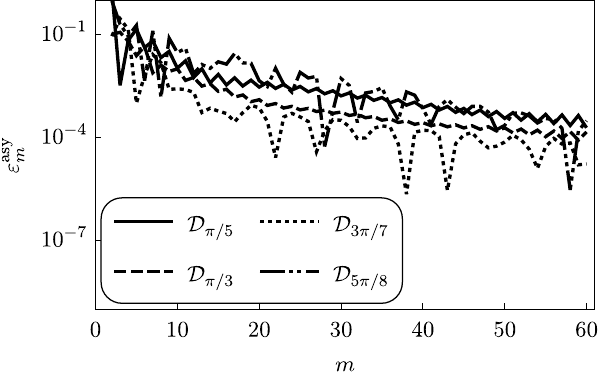}
\end{center}
\caption{Asymptotic accuracy for droplets $\mathcal{D}_\alpha$, $\alpha\in\left\{\frac{\pi}{5},\frac{\pi}{3},\frac{3\pi}{7},\frac{5\pi}{8}\right\}$\label{fig:Fig22}}
\end{figure}

\subsection{Discussion, and going beyond the theorems} When analysing the numerical data presented in Figures \ref{fig:Fig19}, \ref{fig:Fig20}, and \ref{fig:Fig22}, one should exercise caution in interpreting the results. For example, the asymptotic accuracy curves for $T_1$ and $T_2$ in Figure \ref{fig:Fig19}, and for $P_3$ in Figure \ref{fig:Fig20}, sharply bend upwards around $m\approx 20$. This means that for higher eigenvalues the errors of numerical computations exceed asymptotic errors (with the asymptotics in these cases converging rather rapidly), and the results become unreliable.

We make the following empirical observations on the speed of convergence of quasi-eigenvalues $\sigma_m$ to the actual eigenvalues $\lambda_m$ as $m\to\infty$ based on numerical results:
\begin{itemize}
\item convergence is more rapid for straight polygons compared to (partially) curvilinear polygons, for which it is in turn faster than for fully curvilinear polygons;
\item the rate of convergence becomes somewhat slower as the number of vertices increases.
\end{itemize}
\begin{remark} 
\label{convrate}
In view of the results  of \cite{Dav,Ur} for the sloshing problem, one could suggest that the curvature at the corner points may contribute to lower order terms in the  spectral asymptotics.
In particular, for fully curvilinear polygons, it is likely that  $\lambda_m-\sigma_m=O\left(\frac{1}{m}\right)$,  and that this estimate cannot be improved in general.  
At the same time, one can  show using the methods of Section \ref{sec:quasieigenvalues} and \cite[Section 3]{sloshing}, that for the triangles  $T_1$, $T_2$ and $P_3$ with all angles being special or exceptional, 
the error term in the spectral asymptotics decays superpolynomially (and, in fact, similar behaviour is expected for any partially curvilinear polygon with
all the angles which are either special or exceptional). 
\end{remark}

We also emphasise that all our theoretical results are only applicable to curvilinear polygons with angles less than $\pi$, see Remark \ref{rem:alphabiggerpi}. Consider, however, the family of sectors
\[
\mathcal{S}_\alpha=\{z=\rho\er^{\ir\theta}, 0<r<1, 0<\theta<\alpha\}=\mathcal{P}\left(\left(\alpha, \frac{\pi}{2},\frac{\pi}{2}\right),(1,1,\alpha)\right).
\]   
For $\alpha<\pi$, Theorem \ref{thm:polygoneqn1}(b) is applicable, giving three series of simple quasi-eigenvalues
\begin{equation}\label{eq:qesectors}
\sigma=\left\{\begin{aligned}
&\frac{\pi}{\alpha}\left(m-\frac{1}{2}\right),\\
&\frac{1}{2}\arccos\left(\cos\left(\frac{\pi^2}{2\alpha}\right)\right)+2\pi (m-1),\\
&-\frac{1}{2}\arccos\left(\cos\left(\frac{\pi^2}{2\alpha}\right)\right)+2\pi m,
\end{aligned}
\right.
\quad m\in\mathbb{N}.
\end{equation}

Numerical experiments indicate, however, that formulae \eqref{eq:qesectors} give good approximations of eigenvalues even when $\alpha>\pi$, see Figure \ref{fig:Fig23}. 
Together with further numerical experiments (we omit the details) this gives a good indication that Theorem \ref{thm:polygoneqn1} may be applicable (possibly with worsened remainder estimates) to \emph{all} curvilinear polygons with angles less than $2\pi$. 

\begin{figure}[htb]
\begin{center}
\includegraphics{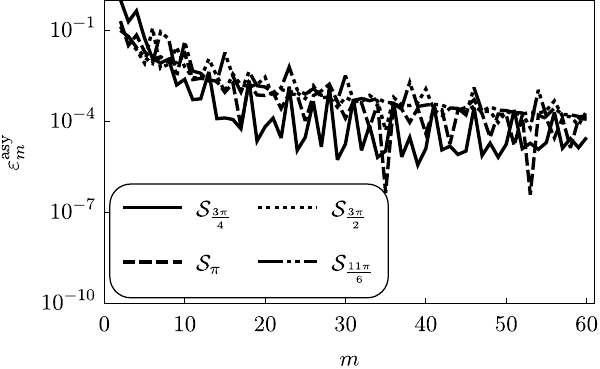}
\end{center}
\caption{Asymptotic accuracy for sectors $\mathcal{S}_\alpha$, $\alpha\in\left\{\frac{3\pi}{4},\pi,\frac{3\pi}{2},\frac{11\pi}{6}\right\}$\label{fig:Fig23}}
\end{figure} 

Finally, we note that it is straightforward to extend our results to not necessarily simply connected domains $\Omega$ for which all boundary components are either smooth curves or curvilinear polygons with interior (with respect to $\Omega$) angles less than $\pi$: the set of quasi-eigenvalues for such a domain is just a union of the sets of quasi-eigenvalues generated by individual boundary components taken with multiplicities, cf.\ \cite{GPPS}.

\end{document}